\documentclass[12pt]{article}

\usepackage[latin1]{inputenc} 

\usepackage[tbtags]{amsmath}    %%% cette option permet que le numero des equations `splittees'
\usepackage{amsfonts,amssymb,amsthm,euscript,makeidx,pifont,boxedminipage,fullpage}
\usepackage{bbm}
\usepackage{array}
\usepackage{enumerate}
\usepackage{braket}
\usepackage{stmaryrd}
\usepackage{subfigure}
\usepackage{centernot}
\usepackage[normalem]{ulem}
\usepackage{soul}
\usepackage{fancybox}

\usepackage{ifpdf}
\ifpdf
   \usepackage[pdftex,bookmarks=true,bookmarksnumbered,%
    plainpages=false,hypertexnames=false,pdfpagelabels,%
    colorlinks=true,linkcolor=black,urlcolor=red,citecolor=red,%Olivier : chang\'e le urlcolor=blue en red
    pdfstartview=FitH]{hyperref}
   \usepackage[pdftex]{graphicx,color}
\else
    \usepackage[dvips,bookmarks=true,bookmarksnumbered,plainpages=false,%
     colorlinks=true,linkcolor=black,urlcolor=blue,citecolor=red,%
     hypertexnames=false,pdfstartview=FitH,breaklinks=true]{hyperref}
    \usepackage{pstricks,pst-node,pst-text,pst-3d}
    \usepackage[dvips]{graphicx}
\fi
\usepackage{color}
\usepackage[active]{srcltx}%
\usepackage{subfigure}
\usepackage{multirow}

\def \cov{\mathop{\rm Cov}}

\newcommand{\smallO}[1]{\ensuremath{\mathop{}\mathopen{}{\scriptstyle\mathcal{O}}\mathopen{}\left(#1\right)}}

\makeatletter

\newtheorem{theorem}{Theorem}[section]
\newtheorem{lemma}[theorem]{Lemma}
\newtheorem{proposition}[theorem]{Proposition}
\newtheorem{corollary}[theorem]{Corollary}

\newtheorem{remark}[theorem]{Remark}

\newcommand{\Z}{\mathbbm{Z}} %set of signed integers
\newcommand{\N}{\mathbbm{N}} %set of integers
\newcommand{\R}{\mathbbm{R}} %set of real numbers
\newcommand{\T}{\mathcal{T}} %set of trajectories
\newcommand{\J}{\mathcal{J}}% synaptic weights without scaling

\newcommand{\mP}{\mathcal{P}}%set of probability measures
\newcommand{\PS}{\mP_S} %stationary probabilty measures
\newcommand{\emp}{\hat{\mu}_{n}} %empirical measure
\newcommand{\empp}[1]{\hat{\mu}_{#1}}%parameterized empirical measure
\newcommand{\norm}[1]{\left\| #1 \right\|}

 \newcommand{\lsup}[1]{\underset{#1\to\infty}{\overline{\lim}}}
\newcommand{\linf}[1]{\underset{#1\to\infty}{\underline{\lim}}}

\newcommand{\bms}{\begin{multline*}}
\newcommand{\fms}{\end{multline*}}

\newcommand{\limprob}[3]{\mathrel{\mathop{\kern 0pt#1}\limits_{#2}^{#3}}}

\newcommand{\Exp}{\mathbb{E}} % Expectation

\makeatother
\begin{document}

\title{The mean-field limit of a  network of Hopfield neurons with correlated synaptic weights}

\author{Olivier Faugeras$^{1}$ and James Maclaurin $^2$ and Etienne Tanr\'e$^3$}

\footnotetext[1]{Universit\'e C\^ote d'Azur, Inria, CNRS, LJAD, France.
	 \texttt{Olivier.Faugeras@inria.fr}, }
\footnotetext[2]{New-Jersey Institute of Technology, USA.  \texttt{james.n.maclaurin@njit.edu}, }

\footnotetext[3]{Universit\'e C\^ote d'Azur, Inria, France.  \texttt{Etienne.Tanre@inria.fr}}

\maketitle

\begin{abstract}
We study the asymptotic behaviour for asymmetric neuronal dynamics in a network of Hopfield neurons. The randomness in the network is modelled by random couplings which are centered Gaussian correlated random variables. We prove that the annealed law of the empirical measure satisfies a large deviation principle without any condition on time. We prove that the good rate function of this large deviation principle achieves its minimum value at a unique Gaussian measure which is not Markovian. This implies almost sure convergence of the empirical measure under the quenched law. We prove that the limit equations are expressed as an infinite countable set of linear non Markovian SDEs.
\end{abstract}

\noindent
{\em AMS Subject of Classification (2010)}:\\
 60F10, 60H10, 60K35, 82C44, 82C31, 82C22, 92B20
 
 \vspace{0.5cm}

\section{Introduction}
We revisit the problem of characterizing the large-size limit of a network of Hopfield neurons. Hopfield \cite{hopfield:82} defined a broad class of neuronal networks and characterized some of their computational properties \cite{hopfield:84,hopfield-tank:86}, i.e. their ability to perform computations. Inspired by his work Sompolinsky and co-workers studied the thermodynamic limit of these networks when the interaction term is linear \cite{crisanti-sompolinsky:87} using the dynamic mean-field theory developed in \cite{sompolinsky-zippelius:82} for symmetric spin glasses. The method they use is a functional integral formalism used in particle physics and produces the self-consistent mean-field equations of the network. This was later extended to the case of a nonlinear interaction term, the nonlinearity being an odd sigmoidal function \cite{sompolinsky-crisanti-etal:88}. A recent revisit of this work can be found in \cite{crisanti2018path}. Using the same formalism the authors established the self-consistent mean-field equations of the network and the dynamics of its solutions which featured a chaotic behaviour for some values of the network parameters. A little later the problem was picked up again by mathematicians. Ben Arous and Guionnet applied large deviation techniques to study the thermodynamic limit of a network of spins interacting linearly with i.i.d. centered Gaussian weights. The intrinsic spin dynamics (without interactions) is a stochastic differential equation where the drift is the gradient of a potential. They prove that the annealed (averaged) law of the empirical measure satisfies a large deviation principle and that the good rate function of this large deviation principle achieves its minimum value at a unique measure which is not Markovian \cite{guionnet:95,ben-arous-guionnet:95,guionnet:97}. They also prove averaged propagation of chaos results. Moynot and Samuelides \cite{moynot-samuelides:02} adapt their work to the case of a network of Hopfield neurons with a nonlinear interaction term, the nonlinearity being a sigmoidal function, and prove similar results in the case of discrete time. The intrinsic neural dynamics is the gradient of a quadratic potential.

We extend this paradigm by including correlations in the random distribution of network connections. There is an excellent motivation for this, because it is commonly thought that neural networks have a small-world architecture, such that the connections are not completely random, but display a degree of clustering \cite{sporns:11}. It is thought that this clustering could be a reason behind the correlations that have been observed in neural spike trains \cite{buzsaki:06}.

We propose a different method to obtain the annealed LDP to previous work by Ben Arous and Guionnet \cite{ben-arous-guionnet:95,guionnet:97}, Faugeras and MacLaurin \cite{faugeras-maclaurin:14}. The analysis of these papers centres on the Radon-Nikodym derivative between the coupled state and the uncoupled state, demonstrating that this converges as the network size asymptotes to infinity. By contrast, our analysis centres on the SDE governing the finite-dimensional annealed system. It bears some similarities to the coupling method developed by Sznitman \cite{sztnitman:91} for interacting particle systems, insofar as we demonstrate that the finite-dimensional SDE converges to the limiting system superexponentially quickly. 

Our method is more along the lines of recent work that uses methods from stochastic control theory to determine the Large Deviations of interacting particle systems \cite{budhiraja-dupuis-etal:12}. It is centered on the idea of constructing an exponentially good approximation of the annealed  law of the empirical measure under the averaged law of the finite size system.
\section{Outline of model and main result}
Let $I_n=[-n \cdots n]$, $n \geq 0$ be the set of $2n+1$ integers between $-n$ and $n$, $N:=2n+1$.

For any positive integer $n$, let $J_n=(J^{ij}_n)_{i,\,j \in I_n} \in \R^{N \times N}$,  and consider the system $\mathcal{S}^{N}(J_n)$ of  $N$ stochastic differential equations
\begin{equation}\label{eq:model}
\mathcal{S}^{N}(J_n):= \left\{
\begin{array}{lcl}
dV^i_t & = & \sum_{j \in I_n} J^{ij}_n f(V^j_t) dt+\sigma dB^i_t \quad  i \in I_n\\
V_0^i& = & 0
\end{array}
\right.
\end{equation}
where $(B^i)_{i \in I_n}$ is an $N$-dimensional vector of independent Brownian motions. We assume for simplicity that $V_0^i = 0$, $i \in I_n$. 
 $\sigma$ is a positive number.
The function $f: \R \to \R^+$ is bounded and Lipschitz continuous. We may assume without loss of generality that  $f(\R) \subset [0,1]$ and that its Lipschitz constant is equal to 1. A typical example is
\begin{equation}\label{eq:f}
f(x)=\frac{1}{1+e^{-4x}}.
\end{equation}

The weights $J_n := \big( J_n^{jk}\big)_{j,k \in I_n}$ are, under the probability $\gamma$ on $(\Omega, \mathcal{A})$, centered correlated Gaussian random variables  with a shift invariant covariance function given by
\begin{equation}\label{eq:cov}
\Exp^\gamma\left[J^{ij}_{n}J^{kl}_{n}\right] =
\frac{1}{N} R_\J((k-i) \bmod I_n, (l-j) \bmod I_n) 
\end{equation}
\begin{remark}
Expectations w.r.t. $\gamma$ are noted $\Exp^\gamma$ throughout the paper.
\end{remark}
\begin{remark}
Model \eqref{eq:model} is a slightly simplified version of the full Hopfield model which includes a linear term and a general initial condition:
\begin{equation}\label{eq:fmodel}
\mathcal{S}^{N}_{full}(J_n):= \left\{
\begin{array}{lcl}
dV^i_t & = & -\alpha V^i_tdt + \sum_{j \in I_n} J^{ij}_n f(V^j_t) dt+\sigma dB^i_t \quad  i \in I_n.\\
\text{Law}(V_0) & = & \mu_0^{\otimes N}
\end{array}
\right.
\end{equation}
$\alpha$ is a positive constant and $\mu_0$ is a probability measure on $\R$ with finite variance.

Adding the extra linear term and a more general initial condition does not change the nature of the mathematical problems
we address but complicates the notations.
\end{remark}
Here $R_{\J}$ is independent of $n$ and such that 
\begin{enumerate}
\item
\begin{equation}\label{eq:RJdef}
| R_\J(k,l) | \leq a_k b_l
\end{equation} 
where the two positive sequences $(a_k)$ and $(b_l)$ are such that
\begin{equation}\label{eq:akbl}
a_k = \smallO{1/|k|^3}, \quad \text{and}  \quad \sum_{l \in \Z} b_l < \infty
%\sum_{k \in \Z} |k| a_k < \infty \quad 
\end{equation}
We note $a$ and $b$ the sums of the two series $(a_k)_{k \in \Z}$ and $(b_k)_{k \in \Z}$,
\begin{equation}\label{eq:ab}
a:= \sum_{k \in \Z} a_k \quad b:= \sum_{k \in \Z} b_k
\end{equation}
\item There exists a centered Gaussian stationary process $(J^{ij})_{i,j \in \Z}$ with autocorrelation $R_\J$. Because of \eqref{eq:RJdef} this process has a spectral density noted $\tilde{R}_\J$ given by
\begin{equation}\label{eq:RJtilde}
\tilde{R}_\J(\varphi_1,\varphi_2)= \sum_{k,l \in \Z} R_\J(k,l) e^{-i k\varphi_1}e^{-i l\varphi_2},
\end{equation}
with $i=\sqrt{-1}$. We assume that this spectral density is strictly positive:
\begin{equation}\label{eq:bound on spectrum}
\tilde{R}_\J(\varphi_1,\varphi_2) > 0
\end{equation}
for all $\varphi_1,\,\varphi_2 \in [-\pi, \pi[$.
\end{enumerate}
\begin{remark}\label{rem:twicediff}
The hypotheses \eqref{eq:akbl} guarantee that the Fourier transform
\[
\tilde{R}(\varphi,0)=\sum_{k,l \in \Z} R_\J(k,l) e^{-i k \varphi} 
\]
is three times continuously differentiable on $[-\pi, \pi]$. We provide a short proof.
\begin{proof}
Define $Q_\J(k):=\sum_{l \in \Z} R_\J(k,l)$. This is well defined since the series in the right hand side is absolutely convergent. Because $|Q_\J(k)| \leq b a_k$, $Q_\J(k)$ is $\smallO{1/|k|^3}$ and hence its Fourier transform $\tilde{R}_\J(\varphi,0)$ (see \eqref{eq:RJtilde}) is three times continuously differentiable.
\end{proof}
\end{remark}
We have the following Proposition.
\begin{proposition}\label{prop:SNJchi}
For each $J_n \in \R^{N \times N}$ , $\mathcal{S}^{N}(J_n)$ has a unique weak solution.
\end{proposition}

\begin{proof}
For each \(J_n\), we have a standard system of stochastic differential equations with smooth coefficient (Lipschitz continuous). Existence and uniqueness of the solution is well known.
\end{proof}

The solution $V_n := (V^j)_{j\in I_n}$ to the above system defines a $\T^N$-valued random variable, where $\T = \mathcal{C}([0,T],\R)$. 

Given a metric space $\mathfrak{X}$, in what follows $\mathfrak{X}=\T,\,\T^N,$ or $\T^\Z$, and the corresponding distance $d$ we consider the measurable space $(\mathfrak{X}, \mathcal{B}_d)$, where $\mathcal{B}_d$ is the Borelian $\sigma$-algebra induced by the topology defined by $d$, and note $\mP(\mathfrak{X})$ the set of probability measures on $(\mathfrak{X}, \mathcal{B}_d)$.

We note $P \in \mP(\T)$, the law of each scaled Brownian motion $\sigma B^i$, $P^{\otimes N} \in \mP(\T^N)$ the law of $N$ independent scaled Brownian motions $\sigma B^j$, $j \in I_n$, and $P^{\otimes \Z} \in \mP(\T^\Z)$ the law of $(\sigma B^j_t)_{j \in \Z}$.
We also note $P^N(J_n) \in \mP(\T^N)$ the law of the solution to $\mathcal{S}^{N}(J_n)$.

We note $u=(u^i)_{i \in \Z}$ an element of $\T^\Z$ and $u_n=(u^i)_{i \in I_n}$ its projection on $\T^N$.

Given $\mu \in \mP(\T^\Z)$ we note $\mu^{I_n} \in \mP(\T^N)$ its marginal over the set of coordinates of $u_n$.

Because of the shift invariance of the covariance $R_\J$ we are naturally led to consider stationary probability measures on $\T^\Z$.
For this, let $S^i$ be the shift operator acting on $\T^{\Z}$ by
\[
(S^iu)^j=u^{i+j},\,u\in \T^{\Z},\,i,j \in \Z,
\]
and let $\mathcal{P}_S\big(\T^{\Z}\big)$ be the space of all probability 
measures that are invariant under $S$. 
This property obviously implies the invariance under $S^i$, for all integers $i$.
The periodic empirical measure $\emp:\T^{N} \to \PS(\T^\Z)$ is defined to be\begin{equation}\label{eq:empirical}
\empp{n}(u_n)=\frac{1}{N}\sum_{i \in I_{n}} \delta_{S^iu_{n,p}},
\end{equation}
where $u_{n,p} \in \T^{\Z}$ is the periodic interpolant of $u_n$, i.e. such that $u_{n,p}^j := u_n^{j\mod I_n}$. Let $\Pi^n(J_n)=P^N(J_n) \circ \emp^{-1} \in \mathcal{P}\big(\mathcal{P}_S(\T^{\Z})\big)$ be the (quenched) law of $\hat{\mu}_n(V_n)$ under $P^N(J_n)$, and $\Pi^n := \Exp^{\gamma}[\Pi^n(J_n)] =$ $ \Exp^\gamma \left[  P^N(J_n) \right]  \circ \emp^{-1}  \in  \mathcal{P}\big(\mathcal{P}_S(\T^{\Z})\big)$ be the annealed (averaged) law of $\hat{\mu}_n(V_n)$ under the averaged law $Q^n:=\Exp^\gamma [ P^N(J_n) ]$.  Finally let $\Pi_0^n = P^{\otimes N}  \circ \emp^{-1} $ be the law of $\emp(\sigma B_n)$, i.e. the law of the empirical measure under $P^{\otimes N}$. 

We metrize the weak topology on $\T^\Z$ with the following distance
\begin{equation}\label{eq:metric2}
d_T(u,v)=\sum_{i \in \Z} b_i \norm{f(u^i)-f(v^i)}_T 
\end{equation}
where $\norm{f(u^i)-f(v^i)}_T=\sup_{t \in [0,T]} |f(u^i_t)-f(v^i_t)|$ and the positive sequence $b_i$ is defined by \eqref{eq:RJdef}. 

We use the Wasserstein-1 distance to metrize the weak topology on $\mathcal{P}(\T^{\Z})$:\\
 given $\mu$, $\nu \in \mP(\T^\Z)$ we define
\begin{equation}\label{eq:dPT}
D_T(\mu,\nu)=\inf_{\xi \in C(\mu,\nu)} \int d_T(u,v)\,d\xi(u,v),
\end{equation}
where $C(\mu,\nu)$ denotes the set of probability measures on $\T^\Z \times \T^\Z$ with marginals $\mu$ and $\nu$ on the first and second factors  (couplings).
\label{sec:firstlook}

The following is our main result.
%\begin{theorem}\label{th:main}
\begin{theorem}\label{Theorem:Main}
\hspace{2em}
	\begin{enumerate}[(i)]
		\item The sequence of laws $\big(\Pi^n\big)_{n\in\Z^+}$ satisfies a Large Deviation Principle with respect to the weak topology on $\mathcal{P}_S(\T^{\Z})$, with good rate function $H(\mu) :\PS\big(\T^{\Z}\big) \to \R$. 
		\item The rate function $H$ has the following structure. If it is not the case that $\mu^{I_n} \ll P^{\otimes N}$ for all $n$, then $H(\mu) = \infty$, otherwise
		\begin{equation}
		\label{eq:Hmu2}
		H(\mu) = \inf_{\zeta \in \PS(\T^\Z): \Psi(\zeta) = \mu}\left\lbrace I^{(3)}(\zeta) \right\rbrace,
		\end{equation}
		where the measurable function $\Psi: \mP_S(\T^\Z) \to \mP_S(\T^\Z)$ is defined in Section \ref{Section Definition Tilde Psi qm rm}.
		and \(I^{(3)}\) in Theorem~\ref{Theorem: Pi 0 LDP}.
		\item $H$ has a unique zero $\mu_*=\Psi(P^{\otimes \Z})$.
		\item  $\mu_*$  is the law of the unique weak solution $Z$ of the following system of McKean-Vlasov-type equations,
		\begin{align}\label{eq: limit equations} 
		Z^j_t &= \sigma W^j_t + \sigma \int_0^t \theta^j_s ds  \\
		\theta^j_t &= \sigma^{-2}\sum_{i\in \Z} \int_0^t L_{\mu_*}^{ i-j}(t,s)dZ^{i}_s. \nonumber 
		\end{align}
		The sequence of processes $\big(\sigma W^j\big)_{j\in\Z}$ is distributed as $P^{\otimes \Z}$, and $L_{\mu_*}$ is defined in Remark~\ref{rem:Lmukt} and Appendix \ref{app:covcont}.
		Furthermore $\mu_*$ is Gaussian.
	\end{enumerate}
\end{theorem}

The proof of this theorem uses the following, classical, theorem \cite{bryc-dembo:96} and \cite[Section 6]{dembo-zeitouni:97}. Recall that $\Pi^n_0$ is the law of the empirical measure under $P^{\otimes N}$.

\begin{theorem}\label{Theorem: Pi 0 LDP}
The sequence of laws $\big(\Pi^n_0\big)_{n\in\Z^+}$ satisfies a large deviation principle with good rate function $I^{(3)}$ on $\PS(\T^{\Z})$.  The specific relative entropy is
\begin{equation}\label{eq: I3 definition}
I^{(3)}(\mu) = \lim_{n\to\infty}\frac{1}{N}I^{(2)}\big(\mu^{I_n} | P^{\otimes N} \big),
\end{equation}
where, for measures \(\nu\) and \(\rho\) on \(\R^N\), the relative entropy \(I^{(2)}\) is defined by
	\[
	I^{(2)}(\rho|\nu) = 
	\begin{cases}\displaystyle\int_{\R^N}\log\dfrac{d\rho}{d\nu}(x)\nu(dx)\quad\text{ if } \rho \ll \nu\\
	+\infty \quad\quad \text{ otherwise},
	\end{cases}
	\]
see e.g. \cite{ellis:85}.

The unique zero of $I^{(3)}$ is $P^{\otimes \Z}$.
\end{theorem}

A standard argument yields that the averaged LDP of the previous theorem implies almost sure convergence of the empirical measure under the quenched law \cite{ben-arous-guionnet:95}. This is stated in the following corollary. 
\begin{corollary}\label{Corr: Correlated Chaos}
For almost every realization of the weights and Brownian motions, 
\[
\emp(V_n) \to \mu_*\quad \mbox{as }N \to \infty.
\]
\end{corollary}
\begin{proof}
The proof is standard. It follows from an application of Borel-Cantelli's Lemma to Proposition \ref{prop:halfLDP}.
\end{proof}
\begin{remark}
Note that this implies that for all $f \in C_b(\T^\Z)$ and for almost all $\omega \in \Omega$.
\begin{equation}
\label{eq:questionolivier3}
\lim_{N \to \infty} \frac{1}{N} \int_{\T^N} \sum_{i \in I_n} \mathbb{E}^{P^N(J_n)(\omega)}f(S^i V_{n,p})  = \int_{\T^\Z} f(v) \, d\mu_*(v)\
\end{equation}
\end{remark}

\begin{proposition}\label{prop:halfLDP}
For any closed set $F$ of $\mP_S(\T^\Z)$ and for almost all $J_n$,
\[
\limsup_{N \to \infty} \frac{1}{N} \log \left[ P^N(J_n)(\emp \in F) \right] \leq - \inf_{\mu \in F} H(\mu).
\]
\end{proposition}
\begin{proof}
The proof, found in \cite[Th. 2.7]{ben-arous-guionnet:95}, follows from an application of Borel-Cantelli's Lemma.
\end{proof}
\begin{remark}
Note that in the case we assume the synaptic weights to be uncorrelated, equations \eqref{eq: limit equations} reduce to
\begin{equation}\label{eq: limit equations guionnet} 
		Z_t = \sigma W_t + \sigma^{-1} \int_0^t  \int_0^s L_{\mu_*}(s,u)dZ_u ds  \\
\end{equation}
which is exactly the one found in \cite[Th. 5.14]{ben-arous-guionnet:95}.
\end{remark}
\section{Proof of Theorem~\ref{Theorem:Main}}
Our strategy is partially inspired from the one in \cite{ben-arous-guionnet:95,guionnet:97}. We apply Girsanov's Theorem to $\mathcal{S}^{N}(J_n)$ to obtain the Radon-Nikodym derivative of the measure $P^N(J_n)$ with respect to the measure $P^{\otimes N}$ of the system of $N$ uncoupled neurons. We then show that the average $Q^n$ of  $P^N(J_n)$  w.r.t. to the weights is absolutely continuous w.r.t. $P^{\otimes N}$ and compute the corresponding Radon-Nikodym derivative which characterizes the averaged (annealed) process. As in the work of Ben Arous and Guionnet 
\cite{ben-arous-guionnet:95}, the idea is to deduce our LDP from the one satisfied by the sequence $(\Pi_0^n)_{n \in \N}$. We differ from the work of Ben Arous and Guionnet in that in order to obtain the Large Deviation Principle that governs this process we approximate the averaged system of SDEs with a system with piecewise constant in time coefficients by discretizing the time interval $[0,T]$ into $m$ subintervals of size $T/m$, for $m$ an integer. This system allows us to construct a sequence of continuous maps $\Psi^m: \mP_S(\T^\Z) \to \mP_S(\T^\Z)$ and a measurable map $\Psi: \mP_S(\T^\Z) \to \mP_S(\T^\Z)$ such that the sequence $\Psi^m$ converges uniformly toward $\Psi$ on the level sets of the good rate function of the LDP satisfied by $\Pi_0^n$. We then show that for a specific choice $m(n)$ of $m$ as a function of $n$ the sequence $\Pi_0^n \circ (\Psi^{m(n)})^{-1}$ is an exponentially good approximation of the sequence $\Pi^n$. The LDP for $\Pi^n$ and the corresponding good rate function then follow from a Theorem by Dembo and Zeitouni, \cite[Th.~4.2.23]{dembo-zeitouni:97}.

In more details, we use Girsanov's Theorem to establish in Section \ref{subsection:annealed equations} the SDEs whose solution's law is the averaged law $Q^n$. In Section \ref{Section Definition Tilde Psi qm rm} we construct an approximation of these equations by a) discretizing the time interval $[0,T]$ with $m$ subintervals and b) cutting off the spatial correlation of the weights so that it extends over $[-q_m,q_m]$ rather than over $[-n,n]$, $q_m \leq n$. We then use this approximation to construct the family $(\Psi^m)_{m \in \N}$ of continuous maps. Section \ref{subsection:proofT2.4} contains the proof of our main 
%Theorem~\ref{th:main}. 
Theorem~\ref{Theorem:Main}. 
This proof contains two main ingredients, the exponential tightness of $(\Pi^n)_{n \in \Z^+}$ proved in Section \ref{sect:exponential tightness}, and the existence of an exponential approximation of the family of measures $(\Pi^n)_{n \in \Z^+}$ by the family of measures $(\Pi^{m,n})_{m, n \in \Z^+}=\Pi_0^n \circ (\Psi^m)^{-1}$ constructed from the law of the solutions to the approximate equations. The existence of this exponential approximation and the possible choices for $m$ and $q_m$ as functions of $n$ are  proved in Section \ref{subsec:main lemma}. The unique minimum of the rate function is characterized in Section \ref{Section Limiting Process}.
\subsection{The SDEs governing the Finite-Size Annealed Process}\label{subsection:annealed equations}

For every $J_n  \in \R^{N \times N}$, $P^N(J_n)$ is a probability measure on $\T^N$ and as a consequence of Girsanov's theorem

\begin{multline*}
\left.\frac{d P^{N}(J_n)}{d P^{\otimes N}}
\right|_{\mathcal{F}_T}=\\\exp  \left\{ \frac{1}{\sigma}\sum_{i \in I_{n}}\int_0^T \left( \sum_{j \in I_{n}} J^{ij}_n f(X^j_{t}) \right) \,dB^i_t\right. \left.-\frac{1}{2\sigma^2} \sum_{i \in I_{n}}\int_0^T \left( \sum_{j \in I_{n}} J^{ij}_n f(X^j_{t}) \right)^2\,dt \right\},
\end{multline*}
 where
 \begin{equation}\label{eq:dBt}
 X^j_t = \sigma B^j_t
 \end{equation}
In Proposition \ref{prop:RNderiv1}  below, we demonstrate that the
Radon-Nikodym derivative of $Q^n$ w.r.t. \(P^{\otimes N}\) exists and 
is a function of the empirical measure. 
To facilitate this, we must introduce intermediate 
centered Gaussian Processes $(G^i_t)_{i\in I_n, t\in [0,T]}$, for which 
it turns 
out that their probability law is entirely determined by the empirical measure, 
i.e.
 \begin{equation}\label{eq:Gnit}
 G^i_t=\sum_{j \in I_{n}} J_{n}^{ij} f(X^j_t),\,i \in I_{n}.
 \end{equation}

  It can be verified that the covariance is entirely determined by the empirical measure, i.e., according to equation \eqref{eq:cov}
 \begin{multline}\label{eq:covGiGk}
 \Exp^\gamma \left[G_{t}^{i}G_s^{k}\right]= \int_\Omega G^i_t(\omega)G^k_s(\omega)\,d\gamma(\omega)= \\
 \frac{1}{N}\sum_{l,j \in I_{n}} R_\J((k-i) \bmod I_n, (l-j) \bmod I_n) f(X^j_t)f(X^l_s)=\\
\quad   \sum\limits_{m \in I_{n}} R_\J((k-i) \bmod I_n,m) \frac{1}{N}\sum\limits_{j \in I_n}f(X^j_t)f(X^{(j+m) \bmod I_n}_s)=\\
\sum\limits_{m \in I_{n}} R_\J((k-i) \bmod I_n,m)\int_{\T^{\Z}} f(v^{0}_t)f(v^{m}_s)\,d\emp(X_n)(v):=K_{\empp{n}(X_n)}^{k-i}(t,s).
 \end{multline}
\begin{remark}\label{rem:Gnit}
Note that we have shown that under $\gamma$, the sequence $G^i$, $i \in I_n$, is centered, stationary with  covariance 
$K_{\empp{n}(X_n)}$. To make this dependency explicit we write $\gamma^{\hat{\mu}_n(X_n)}$ the law under which the Gaussian process $(G^i_t)_{i \in I_n, t \in [0,T]}$ has mean 0 and covariance $K_{\empp{n}(X_n)}$. 
\end{remark}
Before we prove the following proposition which is key to the whole approach we need to introduce a few more notations. We note
\begin{equation}\label{eq:defLambdat}
\Lambda_t(G) :=\frac{\exp\left\{ -\frac{1}{2\sigma^2} \sum_{i \in I_n}\int_0^t \left(G^i_s\right)^2\,ds\right\}}{\Exp^{\gamma^{\emp(X_n)}} \left[ \exp\left\{ -\frac{1}{2\sigma^2} \sum_{i \in I_n}\int_0^t \left(G^i_s\right)^2\,ds\right\} \right]},
\end{equation}
and define the new probability law
\begin{equation}\label{eq:gammatilde}
\bar{\gamma}^{\hat{\mu}_n(X_n)}_t:=\Lambda_t(G) \cdot \gamma^{\hat{\mu}_n(X_n)}.
\end{equation}
\begin{remark}\label{rem:Kmuk}
More generally given a measure $\mu$ in $\mP_S(\T^\Z)$ we note $\gamma^\mu$ the law under which the Gaussian process $(G^i_t)_{i \in I_n, t \in [0,T]}$ has mean 0 and covariance $K_\mu$ such that
\[
K_\mu^k(t,s)= \sum\limits_{m \in I_{n}} R_\J(k,m)\int_{\T^{\Z}} f(v^{0}_t)f(v^{m}_s)\,d\mu(v)
\]
and
\[
\bar{\gamma}^\mu_t := \Lambda_t^\mu(G) \cdot \gamma^\mu,
\]
where
\[
\Lambda_t(G) :=\frac{\exp\left\{ -\frac{1}{2\sigma^2} \sum_{i \in I_n}\int_0^t \left(G^i_s\right)^2\,ds\right\}}{\Exp^{\gamma^{\mu}} \left[ \exp\left\{ -\frac{1}{2\sigma^2} \sum_{i \in I_n}\int_0^t \left(G^i_s\right)^2\,ds\right\} \right]}.
\]
The properties of $K_\mu$ are proved in Appendix \ref{app:covariances}.
Note that we do not make explicit the dependency of $\Lambda$ on $\mu$ since it is always clear from the context, see next remark.
\end{remark}
\begin{remark}\label{rem:Lmukt}
To each covariance $K_\mu$ defined in Remark~\ref{rem:Kmuk} we associate a new covariance $L_\mu^t$ such that
\[
L_\mu^{t,k}(s,u) = \Exp^{\gamma^\mu} \left[ \Lambda_t(G) G_s^0 G_u^k  \right] = \Exp^{\bar{\gamma}^\mu_t} \left[  G_s^0 G_u^k  \right]
\]
for all $0 \leq s,\,u \ \leq t$. The properties of $L_\mu^t$, in particular the fact that it is a covariance, are stated and proved in Appendix \ref{app:covariances}. For the sake of simplicity and because it is always clear from the context, we drop the upper index $t$ and write $L_\mu^k$ instead of $L_\mu^{t,k}$.
\end{remark}
\begin{proposition}\label{prop:RNderiv1} 
The measures $Q^n$ and $P^{\otimes N}$ are equivalent, with Radon-Nikodym derivative over the time interval $[0,t]$ equal to
\begin{align}\label{eq:Girsanov derivative 1}
\left.\frac{dQ^n}{dP^{\otimes N}}\right|_{\mathcal{F}_t} &= \exp\bigg(\sum_{j\in I_n}\int_0^t \theta_s^j dB^j_s - \frac{1}{2}\sum_{j\in I_n}\int_0^t \big(\theta_s^j\big)^2 ds \bigg) \text{, where}\\
\theta^j_t &= \sigma^{-2} \Exp^{\bar{\gamma}^{\emp(X_n)}_t}\left[ \sum_{i\in I_n}G_t^{j}\int_0^t  G_s^{i} dB^{i}_s\right].\label{eq:Girsanov derivative 2}
\end{align}
\end{proposition}
\begin{proof}
 As stated above, by the Girsanov's Theorem we have
\begin{multline*}
\left.\frac{d P^{N}(J_n)}{d P^{\otimes N}}\right|_{\mathcal{F}_t}=\\
\exp  \left\{ \frac{1}{\sigma}\sum_{i \in I_n}\int_0^t \left(\sum_{j \in I_n} J^{ij}_n f(X^j_s) \right)dB_s^i\right. \left. -\frac{1}{2\sigma^2} \sum_{i \in I_n}\int_0^t \left(  \sum_{j \in I_n} J^{ij}_n f(X^j_s) \right)^2\,ds \right\}.
\end{multline*}
Applying the Fubini-Tonelli theorem to the positive measurable function $\frac{d P^{N}(J_n)}{d P^{\otimes N}}$ we find that $Q^n <\hspace{-2mm}< P^{\otimes N}$ and
\begin{multline*}
\left.\frac{d Q^n}{d P^{\otimes N}}\right|_{\mathcal{F}_t}=\Exp^\gamma \Bigg[ \exp  \Bigg\{ \frac{1}{\sigma}\sum_{i \in I_n}\int_0^t  \left( \sum_{j \in I_n} J^{ij}_n f(X^j_s) \right) \,dB_s^i-\\
\frac{1}{2\sigma^2}\sum_{i \in I_n} \int_0^t \left(  \sum_{j \in I_n} J^{ij}_n f(X^j_s) \right)^2\,ds \Bigg\} \Bigg]\end{multline*}
\newcommand{\mtp}{{t^{(q)}}}
Moreover, under $\gamma$, 
$\left\{\sum_{j \in I_n} J^{ij}_n f(X^j_{t}),\,i \in I_n,\, t \leq T  \right\}$ is a centered Gaussian process with
 covariance $K_{\emp(X_n)}$, thanks to \eqref{eq:Gnit} and \eqref{eq:covGiGk}.
Therefore we have:
\[
\left.\frac{d Q^n}{d P^{\otimes N}}\right|_{\mathcal{F}_t}=\Exp^{\gamma^{\hat{\mu}_n(X_n)}}\left[ \exp\left\{ \frac{1}{\sigma} \sum_{i \in I_n} \int_0^t G_s^i\,dB_s^i\right\} \times \exp\left\{ -
\frac{1}{2\sigma^2}  \sum_{i \in I_n} \int_0^t \left( G_s^i \right)^2\,ds \right\}\right].
\]
Divide and multiply the right hand side by $\Exp^{\gamma^{\hat{\mu}_n(X_n)}} \left[ \exp\left\{ -
\frac{1}{2\sigma^2}  \sum_{i \in I_n} \int_0^t \left( G_s^i \right)^2\,ds \right\} \right]$ to obtain, thanks to \eqref{eq:defLambdat} and \eqref{eq:gammatilde}:
\begin{multline}\label{eq:formula}
\Exp^{\gamma^{\hat{\mu}_n(X_n)}}\left[ \exp\Bigg\{ \frac{1}{\sigma} \sum_{i \in I_n} \int_0^t G_s^i\,dB_s^i-
\frac{1}{2\sigma^2}  \sum_{i \in I_n} \int_0^t \left( G_s^i \right)^2\,ds \Bigg\}\right]=\\
\Exp^{\gamma^{\hat{\mu}_n(X_n)}} \left[ \exp\left\{ -\frac{1}{2\sigma^2} \sum_{i \in I_n} \int_0^t \left( G_s^i \right)^2\,ds \right\} \right] \times
\Exp^{\bar{\gamma}^{\hat{\mu}_n(X_n)}_t}\left[ \exp \left\{ \frac{1}{\sigma} \sum_{i \in I_n} \int_0^t G_s^i\,dB_s^i\right\} \right]
\end{multline}
By Gaussian calculus and \eqref{eq:gammatilde}
\begin{multline*}
\Exp^{\bar{\gamma}^{\hat{\mu}_n(X_n)}_t}\left[ \exp \left\{ \frac{1}{\sigma} \sum_{i \in I_n} \int_0^t G_s^i\,dB_s^i\right\} \right]=
\exp \left\{ \frac{1}{2\sigma^2} \Exp^{\bar{\gamma}^{\hat{\mu}_n(X_n)}_t}\left[\left(\sum_{i \in I_n} \int_0^t G_s^i\,dB_s^i\right)^2\right]\right\}=\\
\exp \left\{ \frac{1}{2\sigma^2} \Exp^{\gamma^{\hat{\mu}_n(X_n)}}\left[\left(\sum_{i \in I_n} \int_0^t G_s^i\,dB_s^i\right)^2\Lambda_t(G) \right]\right\}%\label{eq:Zt}
\end{multline*}
This shows that
\begin{multline}\label{eq:girsanov1}
\left.\frac{d Q^n}{d P^{\otimes N}}\right|_{\mathcal{F}_t}=\\
\Exp^{\gamma^{\hat{\mu}_n(X_n)}} \left[ \exp\left\{ -\frac{1}{2\sigma^2} \sum_{i \in I_n} \int_0^t \left( G_s^i \right)^2\,ds \right\} \right] \times
\exp \left\{ \frac{1}{2\sigma^2} \Exp^{\gamma^{\hat{\mu}_n(X_n)}}\left[\left(\sum_{i \in I_n} \int_0^t G_s^i\,dB_s^i\right)^2\Lambda_t(G) \right]\right\}
\end{multline}
The above expression demonstrates that $Q^n_{|\mathcal{F}_t}$ is equivalent to $P^{\otimes N}_{|\mathcal{F}_t}$ for all $t\in [0,T]$, since the above exponential cannot be zero on any set $A \in \mathcal{B}(\T^N)$ such that $P^{\otimes N}(A) \neq 0$. 
Thus by Girsanov's Theorem \cite{revuz-yor:91},
\[
Z_t = \exp\bigg(\sum_{j\in I_n}\int_0^t \theta_s^j dB^j_s - \frac{1}{2}\sum_{j\in I_n}\int_0^t \big(\theta_s^j\big)^2 ds \bigg),
\]
where $Z_t = \left.\frac{dQ^n}{dP^{\otimes N}}\right|_{\mathcal{F}_t}$, and $\theta^j_t = \frac{d}{dt}\langle \log Z_\cdot,B^j_\cdot\rangle_t$. \begin{multline}\label{eq:Girsanov Temporary}
\theta^j_t  = \frac{d}{dt}\bigg\langle B^j_\cdot,\frac{1}{2\sigma^2} \Exp^{\gamma^{\hat{\mu}_n(X_n)}}\left[\left(\sum_{i \in I_n} \int_0^\cdot G_s^i\,dB_s^i\right)^2\Lambda_\cdot\right]\bigg\rangle_t \\
+ \frac{d}{dt}\bigg\langle B^j_\cdot,\log\left(\Exp^{\gamma^{\hat{\mu}_n(X_n)}} \left[ \exp\left\{ -\frac{1}{2\sigma^2} \sum_{i \in I_n} \int_0^\cdot \left( G_s^i \right)^2\,ds\right\} \right] \right)\bigg\rangle_t.
\end{multline}
the second bracket only contains a finite variation process, so its bracket with \(B^j\) is \(0\).
Furthermore the probability measure $\gamma^{\hat{\mu}_n(X_n)} \in \PS(\T^{\Z})$ does not change with time, hence we may commute the bracket and expectation as follows,
\begin{align}
\theta^j_t  
=&  \Exp^{\gamma^{\hat{\mu}_n(X_n)}}\left[\frac{d}{dt} \bigg\langle B^j_\cdot,\frac{1}{2\sigma^2} \Lambda_\cdot(G) \left(\sum_{i \in I_n} \int_0^\cdot G_s^i\,dB_s^i\right)^2\bigg\rangle_t\right]\nonumber \\
=&\frac{1}{2\sigma^2} \Exp^{\gamma^{\hat{\mu}_n(X_n)}}\left[ 2\sum_{i\in I_n}\Lambda_t(G) G_t^{j}\int_0^t  G_s^{i} dB^{i}_s\right] \label{eq:thetaj},
\end{align}
since $\Lambda_t$ is time-differentiable, and we have used Ito's Lemma. To be sure, we have carefully double checked (using multiple applications of Ito's Formula) that the time-differentiable terms in \eqref{eq:Girsanov Temporary} are of the correct form. We thus have proved the Proposition, using \eqref{eq:gammatilde} again.
\end{proof}
\begin{remark}
By writing $G^j$, $G^i$ and $\Lambda_t(G)$ as functions of the synaptic weights in \eqref{eq:thetaj} and using their stationarity, $\theta^j_t$ can be rewritten as
\begin{multline*}
\theta^j_t= \sigma^{-2}\sum_{i\in I_n} \Exp^{\bar{\gamma}^{\hat{\mu}_n(X_n)}_t} \left[ G^0_t\int_0^t  G^i_s dB^{i+j}_s \right] =
\sigma^{-2}  \sum_{i\in I_n} \Exp^{\gamma^{\emp(X_n)}} \left[ \Lambda_t(G)G^0_t\int_0^t  G^i_s dB^{i+j}_s \right] =\\
\left\{
\begin{array}{ll}
\sigma^{-2}  \sum_{i\in I_n} \Exp^\gamma \left[\Lambda_t(G) G^0_t\int_0^t  G^i_s dB^{i+j}_s \right]\\
\text{and } G^i_t = \sum_{k \in I_n} J^{ik}_n f(X^k_t)
\end{array}
\right.,
\end{multline*}
with indexes taken modulo $I_n$.
\end{remark}
Since $Q^n$ and $P^{\otimes N}$ are equivalent, by Girsanov's Theorem we obtain the following immediate corollary of Proposition \ref{prop:RNderiv1}. Part (ii) of the corollary is immediate from the definitions. 
\begin{corollary}\label{Corollary Measure Representation}\ \\
(i) Let $V_n \in \T^N$ have law $Q^n$. There exist processes $W^j_t$ that are independent Brownian motion under $Q^n$ and such that $V_n$ is the unique weak solution to the following equations
\begin{align}
V^j_t &= \sigma W^j_t + \sigma \int_0^t \theta^j_s ds \label{eq: theta SDE} \\
\theta^j_t &= \sigma^{-2}\sum_{i\in I_n} \Exp^{\bar{\gamma}^{\emp(V_n)}_t}\bigg[G^0_t\int_0^t  G^i_s dV^{i+j}_s \bigg]
 .\label{eq: theta SDE1} 
\end{align}
\\
(ii) The law of $\emp(\sigma W_n)$ under $Q^n$ is $\Pi^n_0$.
\end{corollary}

\subsection{Approximation of the Finite-Size Annealed Process and construction of the sequence of maps $\Psi^m$}\label{Section Definition Tilde Psi qm rm}
It is well known that Large Deviations Principles are preserved under 
continuous transformations. However we cannot in general find a continuous 
mapping $\Gamma^n$ on $\PS(\T^{\Z})$ such that $\Gamma^n\big(\hat{\mu}_n(\sigma W_n)\big) 
= \hat{\mu}_n(V_n)$, 
where $V_n$ is defined in 
Corollary~\ref{Corollary Measure Representation}. Therefore to prove the LDP, we will use 
`exponentially equivalent approximations'. This technique approximates the 
mapping $\hat{\mu}_n(\sigma W_n) \to \hat{\mu}_n(V_n)$ by a sequence of continuous 
approximations. Our next step therefore is to define the continuous map 
$\Psi^m: \PS(\T^{\Z}) \to \PS(\T^{\Z})$ (for positive integers 
$m$), which will be such that for any $\delta > 0$, the probability that 
$D_T\big(\Psi^m(\hat{\mu}_n(\sigma W_n)),\hat{\mu}_n(V_n)\big) > 
\delta$ is superexponentially small. These approximations will converge to the map $\Psi$ that is defined in the proof of Theorem \ref{Theorem:Main}. This is done in two steps: First approximate the system \eqref{eq: theta SDE}-\eqref{eq: theta SDE1} by discretizing the time and cutting off the correlation between the synaptic weights and, second, by using this approximation to construct the map $\Psi^m$ from $\mP_S(\T^\Z)$ to itself.
\subsubsection{Approximation of the system of equations \eqref{eq: theta SDE}-\eqref{eq: theta SDE1}}\label{subsub:approx}
	To this aim, we use an Euler scheme type approximation:  the integrand of 
	$V^j_t$ is replaced by a piecewise constant in time version. Let $\Delta_m$, $m$ a strictly positive integer,
	be a partition of $[0,T]$ with steps \(\eta_m := \frac{T}{m}\) into the 
	$(m+1)$ points $p\eta_m$, for $p=0$ to $m$, and for any 
	$t\in [0,T]$, write $t^{(m)} := p\eta_m$ such that 
	$t\in [p\eta_m,(p+1)\eta_m)$.

To obtain the Large Deviation Principle, we need to approximate the 
expression for $V_n$ in Corollary~\ref{Corollary Measure Representation}  
by a continuous map. 
The approximate system has finite-range 
spatial interactions.
The spatial interactions have range 
$Q_m=2q_m+1$ (with $0 < q_m < n$).
The parameters $m$ and $q_m$ are specified as functions of $n$ in 
Remark~\ref{rem:choice of m} in the proof of Lemma~\ref{lem:alpha2}.

More precisely, following \eqref{eq: theta SDE}, the approximate system is of the form, for $j \in I_n$
\begin{equation}\label{eq:SD}
V^{m,j}_t  =  \sigma^{-1} \sum_{i \in I_{q_m}} \int_0^t \Exp^{\gamma^{\emp(V_n^m)}}\bigg[ \Lambda_{s^{(m)}}(G^m)G^{m,0}_{s^{(m)}}\int_0^{s^{(m)}}  G^{m,i}_{u^{(m)}} dV^{m,i+j}_u \bigg]\,ds+ \sigma W^j_t
\end{equation}
indexes \(i+j\) are taken modulo $I_n$. 
The $I_{q_m}$-periodic centered stationary Gaussian process $(G^{m, i}_t)_{i \in I_{q_m}, t\in [0,T]}$ is defined by
 \begin{equation}\label{eq:Hdef}
 G^{m,i}_t = \sum_{k \in I_n} J^{i k}_{n,m} f(V^{m,k}_t),\,i \in I_{q_m},
\end{equation}
where  the  $\lbrace J^{ik}_{n,m} \rbrace_{i \in I_{q_m},k\in I_n}$ are centered Gaussian Random variables with covariance (remember \eqref{eq:cov})
\begin{equation}\label{lccl}
\Exp^\gamma \left[J^{ij}_{n,m} J^{kl}_{n,m}\right] =  \frac{1}{N} R_\J \left( k - i \mod I_{q_m} , l - j\mod I_n \right) \mathbbm{1}_{I_{q_m}} ( l - j \mod I_n ),
\end{equation}
where $\mathbbm{1}_{I_{q_m}}$ is the indicator function of the set $I_{q_m}$.
Note that the sum in \eqref{eq:Hdef} is for \(k\in I_n\).

The $W^j_t$s are Brownian motions and (remember \eqref{eq:defLambdat})
\begin{equation}\label{eq:defLambdatm}
\Lambda_{t^{(m)}}(G^m):= \displaystyle \frac{\exp\left\{ -\frac{1}{2\sigma^2} \int_0^{t^{(m)}} \sum_{i \in I_{q_m}} (G^{m,i}_s)^2 ds\right\} }{\Exp^{\gamma^{\emp(V_n^m)}} \left[ \exp\left\{ -\frac{1}{2\sigma^2} \int_0^{t^{(m)}}  \sum_{i \in I_{q_m}} (G^{m,i}_s)^2 \,ds\right\} \right]},
\end{equation}

It is important for the upcoming definition of the map $\Psi^m$ that the covariance between the Gaussian variables $(G^{m, i}_t)$ can be written as a function of the empirical measure $\emp(V^m_n)$ which we now demonstrate. One verifies easily that
\begin{multline}\label{eq:covHjsHkt}
\cov(G^{m,i}_t, G^{m,k}_s )=\sum_{j, l \in I_n} \cov(J^{ij}_{n,m},J^{kl}_{n,m})f(V^{m,j}_t)f(V^{m,l}_s)=\\
\frac{1}{N}\sum_{j, l \in I_n} R_\J \left( k - i \mod I_{q_m} , l - j\mod I_n \right) \mathbbm{1}_{I_{q_m}} ( l - j \mod I_n ) f(V^{m,j}_t)f(V^{m,l}_s)=\\
\sum_{K \in I_{q_m}} R_\J(k-i \mod I_{q_m},K)\frac{1}{N}\sum_{j \in I_n} f(V^{m,j}_t)f(V^{m,j+K}_s)=\\
\sum_{K \in I_{q_m}} R_\J(k - i \mod I_{q_m},K) \int f(w^0_t)f(w^{K}_s)\,d\emp(V^m_n)[w]=\\
\sum_{K \in I_{q_m}} R_\J(k - i \mod I_{q_m},K) \Exp^{\emp(V^m_n)} \left[ f(w^0_t) f(w^K_s)   \right].
\end{multline}
This implies that \eqref{eq:SD} can be rewritten
\begin{equation}\label{eq:SD1}
\begin{array}{lcl}
V^{m,j}_t  & = & \sigma^{-1} \sum_{k \in I_{q_m}} \int_0^t \Exp^{\bar{\gamma}_{s^{(m)}}^{\emp(V_n^m)}}\bigg[ G^{m,0}_{s^{(m)}}\int_0^{s^{(m)}}  G^{m,k}_{u^{(m)}} dV^{m,k+j}_u \bigg]\,ds+ \sigma W^j_t,\, j \in I_n\\
\end{array}
\end{equation}
or
\begin{equation}
\label{eq:vmtthetamt}
\begin{cases}
V^{m,j}_t &= \sigma W^j_t + \sigma \displaystyle\int_0^t \theta^{m,j}_s ds \\
\theta^{m,j}_t &= \sigma^{-2}\displaystyle\sum_{k\in I_{q_m}} \Exp^{\bar{\gamma}^{\emp(V^m_n)}_{t^{(m)}}}\bigg[G^{m,0}_{t^{(m)}}
\displaystyle\int_0^{t^{(m)}}  G^{m,k}_{s^{(m)}} dV^{m, k+j}_s \bigg],\,j \in I_n
\end{cases}
\end{equation}
\subsubsection{Construction of the sequence of maps $\Psi^m$ }
In order to construct the map $\Psi^m$ we rewrite \eqref{eq:SD1} in terms of the increment of $V^m_t - V^m_{t^{(m)}}$ of the process $V^m$:
\begin{equation}\label{eq:SD1inc}
V^{m,j}_t  = V^{m,j}_{t^{(m)}} + \sigma^{-1} \sum_{k \in I_{q_m}} \int_{t^{(m)}}^t \Exp^{\bar{\gamma}^{\emp(V_n^m)}_{s^{(m)}}}\bigg[ G^{m,0}_{s^{(m)}}\int_0^{s^{(m)}}  G^{m,k}_{u^{(m)}} dV^{m,k+j}_u \bigg]\,ds+ \sigma ( W^j_t - W^j_{t^{(m)}}), \, j \in I_n.
\end{equation}
We can now generalize \eqref{eq:SD1inc} by considering a general measure $\nu$ in $\mP_S(\T^\Z)$ and simply replacing $\bar{\gamma}^{\emp(V_n^m)}_s$ by $\bar{\gamma}^\nu_s$ in this equation. This is the basic idea but we have to be slightly more careful. 

In detail, following Remark \ref{rem:Gnit}, given $\nu=(\nu_1,\,\nu_2) \in 
\mP_S((\T^\Z)^2)$ we define the $I_{q_m}$-periodic centered stationary Gaussian process $(G^{m,i}_t)_{i \in I_{q_m}, t\in [0,T]}$, i.e. its covariance  function, by (patterning after \eqref{eq:covHjsHkt})
\begin{align}
\cov(G^{m,i}_t, G^{m,k}_s ) & = \Exp^{\gamma^{\nu_1}}\big[ G^{m,i}_t G^{m,k}_s \big]\nonumber\\
& =\sum_{K \in I_{q_m}} R_\J(k - i \mod I_{q_m},K) \Exp^{\nu_1} \left[ f(w^0_t) f(w^K_s)   \right].
\label{eq:covHjsHktnu}
\end{align}
 Given two elements $X$ and $Y$ of $\T^\Z$ we define the $m$ elements $Z^u$ of $\T^\Z$
 for $u = 0, \cdots, m-1$ by
\begin{multline}\label{eq:Ymu}
\forall t \in [u\eta_m,\,(u+1) \eta_m],\, j \in \Z,\\
Z^{u, j}_t = Y^j_{u\eta_m}+
\sigma^{-1} \sum_{i \in I_{q_m}} \int_{u\eta_m}^t \Exp^{\bar{\gamma}^{\nu_1}_{u\eta_m}}\bigg[  G^{m,0}_{u\eta_m}
\sum_{v = 0}^{u-1}\int_{v \eta_m}^{(v+1)\eta_m} G^{m,i}_{ v \eta_m}  dY^{u, i+j}_v \bigg]\,ds \\ + \sigma (X^j_t - X^j_{u\eta_m})
\end{multline}
\[
Z^{u, j}_t = Y^j_t,\,t \leq u\eta_m,\quad u > 0
\]
and
\[
Z^{u, j}_t = Z^{u, j}_{(u+1) \eta_m},\,t \geq (u+1) \eta_m.
\]
\begin{remark}\label{rem:Z periodic}
Note that 
\begin{enumerate}[(a)]
\item 
if $X_t^j$ and $Y_t^j$ are $N$-periodic, so is $Z^j_t$.
\item \label{rem:range of expectation} the expected value $\Exp^{\bar{\gamma}^{\nu_1}_{u\eta_m}}$ in \eqref{eq:Ymu} acts only on the Gaussian random variables $G^m$ and not on the $Y$s.
\end{enumerate}
\end{remark}
This defines the sequence of mappings $\psi^m_u: \mP_S((\T^\Z)^2) \times (\T^\Z)^2 \to (\T^\Z)^2$,  $u = 0, \cdots, m-1$, by
\begin{equation}\label{eq:psimu}
\psi^m_u(\nu, Y, X) = (Z^u, X),
\end{equation}
the sequence of mappings $\Psi^m_u: \mP_S((\T^\Z)^2) \to \mP_S((\T^\Z)^2)$, $u = 0,\cdots,m-1$ by
\begin{equation}\label{eq:Psimu}
\Psi^m_u(\nu) = \nu \circ \psi^m_u(\nu, \cdot, \cdot)^{-1},
\end{equation}
and finally the mapping $\Psi^m: \mP_S(\T^\Z) \to \mP_S(\T^\Z)$ by
\begin{equation}\label{eq:Psim}
\Psi^m(\mu) = (\Psi^m_{m-1} \circ \cdots \circ \Psi^m_0 \circ \Psi^0(\mu))^1,
\end{equation}
where $\Psi^0 : \mP(\T^\Z) \to \mP((\T^\Z)^2)$ is defined by
\begin{equation}\label{eq:Psi0}
\Psi^0(\mu) = \mu  \circ \iota, 
\end{equation}
and $\iota: \T^\Z \to (\T^\Z)^2$ is defined as
\begin{equation}\label{eq:iota}
\iota(x)^j = (0, x^j)
\end{equation}
We then have the following Lemma.
\begin{lemma}\label{Lem: Psim continuous}
The function $\Psi^m$ defined by \eqref{eq:Psim} is continuous in $(\mP_S(\T^\Z),\,D_T)$ and satisfies
\[
\Psi^m(\emp(\sigma W_n)) = \emp(V^m_n),
\]
where $V^m_n$ is the solution to \eqref{eq:SD1}.
\end{lemma}
\begin{proof} 
{\bf $\Psi^m$ is continuous:}\\
Recall the formula  \eqref{eq:Ymu} for $Z^{u, j}_t$:
\begin{multline*}
Z^{u, j}_t = Y^j_{u\eta_m}+
\sigma^{-1} \sum_{i \in I_{q_m}} \int_{u\eta_m}^t \Exp^{\bar{\gamma}^{\nu_1}_{u\eta_m}}\bigg[  G^{m,0}_{u\eta_m}
\sum_{v = 0}^{u-1}\int_{v \eta_m}^{(v+1)\eta_m} G^{m,i}_{ v \eta_m}  dY^{u, i+j}_v \bigg]\,ds  + \sigma (X^j_t - X^j_{u\eta_m})
\end{multline*}
Note that
\[
\int_{v \eta_m}^{(v+1)\eta_m} G^{m,i}_{ v \eta_m}  dY^{u, i+j}_v= G^{m,i}_{ v \eta_m} \left( Y^{u, i+j}_{(v+1)\eta_m}-  Y^{u, i+j}_{v \eta_m} \right),
\]
and hence
\begin{multline*}
 \Exp^{\bar{\gamma}^{\nu_1}_{u\eta_m}}\left[  G^{m,0}_{u\eta_m}
\sum_{v = 0}^{u-1}\int_{v \eta_m}^{(v+1)\eta_m} G^{m,i}_{ v \eta_m}  dY^{u, i+j}_v \right]\\
= \Exp^{\gamma^{\nu_1}}\left[ \Lambda_{u \eta_m}(G^m)G^{m,0}_{u\eta_m} \sum_{v=0}^{u-1}G^{m,i}_{ v \eta_m} \left( Y^{u, i+j}_{(v+1)\eta_m}-  Y^{u, i+j}_{v \eta_m} \right) \right]\\
=\sum_{v=0}^{u-1}\Exp^{\gamma^{\nu_1}}\left[ \Lambda_{u \eta_m}(G^m)G^{m,0}_{u\eta_m} G^{m,i}_{ v \eta_m} \right] \left( Y^{u, i+j}_{(v+1)\eta_m}-  Y^{u, i+j}_{v \eta_m} \right),
\end{multline*}
since $\Exp^{\gamma^{\nu_1}}$ does not operate on $Y^{u, i+j}$, see Remark~\ref{rem:Z periodic}(\ref{rem:range of expectation}).
Using Remark~\ref{rem:Lmukt} we can conclude that
\begin{multline*}
Z^{u, j}_t = Y^j_{u \eta_m} + \sigma^{-1} \sum_{i \in I_{q_m}} \int_{u \eta_m}^t \sum_{v=0}^{u-1} L_{\nu_1}^{i}(u \eta_m, v \eta_m) (Y^{u,\,i+j}_{(v+1) \eta_m} - Y^{u,\,i+j}_{v \eta_m}) \, ds \\+ \sigma (X^j_t - X^j_{u \eta_m}),\, t \in [u \eta_m, (u+1) \eta_m]
\end{multline*}
%}\\

The quantities $L_{\nu_1}^{i}(u \eta_m, v \eta_m)$ are defined in 
Remark~\ref{rem:Lmukt} and in Appendix~\ref{app:covariances}.
The continuity of $\psi^m_u$ follows from the facts that this equation is linear in $X$, $Y$ and $Z$, and the mapping $\nu \to L_{\nu_1}$ is continuous, see Proposition \ref{prop:Amununif}. The continuity of $\Psi^m_u$ follows from \eqref{eq:Psimu} and that of $\Psi^m$ from \eqref{eq:Psim} and the continuity of $\Psi^0$ defined by \eqref{eq:Psi0} and \eqref{eq:iota}.
\\
{$\mathbf{ \Psi^m(\emp(\sigma W_n)) = \emp(V^m_n)}$, {\bf where} $\mathbf{V_n^m}$ {\bf is the solution to \eqref{eq:SD1}}:}\\
We use the following Lemma.
\begin{lemma}\label{lem:empirical}
\hspace{2em}
\begin{enumerate}[(i)]
	\item \label{lem:empirical:item1} We have $\emp(X_n)  \circ \iota = \emp(0_n, X_n) \in \mP_S((\T^\Z)^2)$ for all $X_n \in \T^N$, where $0_n = (0,\cdots,0) \in \T^N$.
	\item \label{lem:empirical:item2} Let $X_{n,2}=(X^1_n,X^2_n)$ be an element of $(\T^N)^2$,  and $\emp(X_{n, 2})=\frac{1}{N} \sum_{i \in I_n} \delta_{(S^i\, X^1_{n,p}, S^i\,X^2_{n,p})}$ (remember \eqref{eq:empirical}) the corresponding empirical measure in $\mP_S((\T^\Z)^2)$. Let $\varphi: (\T^\Z)^2 \to (\T^\Z)^2$, be a measurable function. Then it is true that
	\[
	\emp(X_{n,2}) \circ \varphi^{-1} = \emp(\varphi(X_{n,2})),
	\]
	where, with a slight abuse of notation, if $X_{n,2,p} \in (\T^\Z)^2$ is the periodic extension of $X_{n,2} \in (\T^N)^2$,  i.e. $(X^1_{n,p}, X^2_{n,p})$, and $\varphi(X_{n,2,p})=(Y^1,Y^2) \in (\T^\Z)^2$ we define
	\[
	\varphi(X_{n,2}) = \varphi(X_{n,2,p})=((Y^1_{-n},\cdots,Y^1_n), (Y^2_{-n},\cdots,Y^2_n)).
	\]
\end{enumerate}
\end{lemma}
We first prove that Lemma~\ref{lem:empirical} is enough to conclude 
	the proof of Lemma~\ref{Lem: Psim continuous}.
First, statement \eqref{lem:empirical:item1} of Lemma~\ref{lem:empirical} implies
\[
\Psi^0(\emp(\sigma W_n)) = \emp(\sigma W_n) \circ \iota =   \emp(0_n,\,\sigma W_n).
\]
Going one step further, and using the definition \eqref{eq:Psimu} and 
statement \eqref{lem:empirical:item2} of Lemma~\ref{lem:empirical}
\begin{multline*}
\Psi^m_0(\Psi^0(\emp(\sigma W_n))) = \Psi^m_0(\emp(0_n,\,\sigma W_n)) = 
\emp(0_n,\,\sigma W_n) \circ \psi^m_0(\emp(0_n,\,\sigma W_n), \cdot, \cdot)^{-1} = \\
\emp( \psi^m_0(\emp(0_n,\,\sigma W_n), 0_n, W_n)) = \emp(\,^0 V^m, \sigma W_n),
\end{multline*}
where $\,^0 V^m$ is equal to the solution of \eqref{eq:SD1} on the time interval $[0, \eta_m]$. According to Remark~\ref{rem:Z periodic}, $\,^0 V^{m,j}_t$ is $N$-periodic 
in the variable \(j\) for $t \in [0, \eta_m]$.

Next we have
\begin{align*}
\Psi^m_1(\Psi^m_0(\Psi^0(\emp(\sigma W_n)))) &= \Psi^m_1(\emp(\,^0 V^m, \sigma W_n)) = \emp(\,^0 V^m, \sigma W_n) \circ \psi^m_1(\emp(\,^0 V^m, \sigma W_n), \cdot, \cdot)^{-1} \\
&=\emp(\psi^m_1(\emp(\,^0 V^m, \sigma W_n), \,^0 V^m, \sigma W_n)) = \emp(\,^1 V^m, \sigma W_n),
\end{align*}
where $\,^1 V^m$ is equal to the solution of \eqref{eq:SD1} on the time interval $[0, 2\eta_m]$, and again $N$-periodic.

One concludes that
\[
\Psi^m_{m-1} \circ \cdots \circ \Psi^m_0 \circ \Psi^0(\emp(\sigma W_n)) = \emp(\,^{m-1} V^m, \sigma W_n),
\]
where $\,^{m-1} V^m$ is equal to the $N$-periodic solution of \eqref{eq:SD1} on the time interval $[0, m\eta_m] = [0, T]$ i.e. $V^m_n$.

Therefore,
\begin{align*}
\Psi^m(\emp(\sigma W_n)) &= (\Psi^m_{m-1} \circ \cdots \circ \Psi^m_0 \circ \Psi^0(\emp(\sigma W_n)))^1 = (\emp(\,^{m-1} V^m, \sigma W_n))^1 \\
&= \emp(\,^{m-1} V^m) = \emp(V^m_n)
\end{align*}
\end{proof}
We now prove Lemma \ref{lem:empirical}.
\begin{proof}[Proof of Lemma \ref{lem:empirical}]\ \\
	\eqref{lem:empirical:item1} For any Borelian of $(\T^\Z)^2$ we have $\emp(X_n) \circ \iota (A)=\emp(X_n)(\iota^{-1}(A))=\emp(X_n)((A \cap \{ 0 \times \T^\Z \})_2)$, where $(A \cap \{ 0 \times \T^\Z \})_2$ is the second coordinate $y$ of the elements of $A$ of the form $(0,y)$. This means that $\emp(X_n) \circ \iota =\emp(0_n,X_n)$.\\
\eqref{lem:empirical:item2} Let $A$ be a Borelian of $(\T^\Z)^2$. We have
\[
(\emp(X_{n,2}) \circ \varphi^{-1}) (A) = \emp(X_{n,2})(\varphi^{-1} (A)) = \emp(\varphi(X_{n,2}))(A),
\]
and the conclusion of the Lemma  follows.
\end{proof}
\subsection{Proof of Theorem \ref{Theorem:Main}.(i)-(iii)}\label{subsection:proofT2.4}
It turns out to be convenient, in order to prove the Theorem, to use the $L^2$ distance on $\T^\Z$ given by
\begin{equation}\label{eq: L squared distance}
d_{L^2}(u,v) = \sum_{i \in \Z} b_i \norm{f(u^i)-f(v^i)}_{L_2}%=\frac{1}{3}\sum_{i \in \Z}  2^{-|i|} (d_{L^2}^1(u^i,v^i) \wedge 1),
\end{equation}
where 
\[
\norm{f(u^i)-f(v^i)}_{L^2}^2 = \int_0^T \big(f(u^i_t) - f(v^i_t)\big)^2 dt.
\]
The reason for this is that we are then able to use the tools of Fourier analysis since the measures we consider are shift invariant, i.e. invariant to spatial translations.

Let $D_{T,L^2}$ be the corresponding Wasserstein-1 metric on $\mP(\T^\Z)$ induced by $d_{L^2}(u,v) $.

\begin{remark}\label{rem:topologies}
The topology induced by $D_{T,L^2}$ on $\mP(\T^\Z)$ is coarser than the one induced by $D_T$.  Hence
it will suffice for us to prove the LDP with respect to the  topology on $\mP(\T^\Z)$ induced by the metric $D_{T,L^2}$. 
This is because we prove in Lemma \ref{Lem: exponentially tight} that the sequence $\Pi^n$ is exponentially tight for the topology induced by $D_T$ on $\mP(\T^\Z)$.
We can then use \cite[Corollary 4.2.6]{dembo-zeitouni:97} which states that if $\Pi^n$ satisfies an LDP for a coarser topology, then it does satisfy the same LDP for a finer topology.
Lemma \ref{Lem: exponentially tight} is proved in Section \ref{sect:exponential tightness}.
\end{remark}

We use  \cite[Th. 4.2.23]{dembo-zeitouni:97} to prove the LDP for $\emp(V_n)$ on $\mP_S(\T^\Z)$ induced by the metric $D_{T,L^2}$. The common probability space in which we perform the exponentially equivalent approximations is $(\T^N,Q^n)$ which contains the  random variable $(V^j_t)$, as well as (as explained in Corollary~\ref{Corollary Measure Representation}) the random variables $(\sigma W^j_t)$ which are distributed as $P^{\otimes N}$. We approximate $\emp(V_n)$ by $\Psi^m\big(\emp(\sigma W_n) \big)$. It is noted in Lemma \ref{Lem: Psim continuous} that the approximations $\Psi^m$ are continuous with respect to the topology induced by $D_T$, so that they must also be continuous with respect to the topology induced by $D_{T,L^2}$. 

The proof is based on Lemma \ref{Lemma Exp Equivalent 1}.
According to this Lemma for any $j \in \N^*$, we have 
\[
\lim_{m\to\infty}\lsup{n}\frac{1}{N}\log Q^n\bigg( D_{T,L^2}\big(\Psi^m\big(\emp(\sigma W_n)\big),\emp(V_n) \big)>2^{-j-1}   \bigg)=-\infty.
\]
We define $m_j$ to be the smallest integer strictly bigger than \(m_{j-1}\) such that
\begin{equation}\label{eq: exponentially equivalent 11}
\sup_{m\geq m_j}\lsup{n}\frac{1}{N}\log Q^n\bigg( D_{T,L^2}\big(\Psi^{m}\big(\emp(\sigma W_n)\big),\emp(V_n) \big) > 2^{-j-1} \bigg) \leq -j.
\end{equation}
By construction, the sequence $(m_j)_{j\geq 1}$ is strictly increasing and hence $\lim_{j \to \infty} m_j = \infty$.

Next define the sets
\begin{equation}\label{eq:frAj}
\mathfrak{A}_j = \bigg\lbrace \mu : D_{T,L^2}\big(\Psi^{m_j}(\mu),\Psi^{m_{j+1}}(\mu) \big) \leq 2^{-j}\bigg\rbrace, \, j \in \N^*,
\end{equation}
and  the set 
\begin{equation}\label{eq:frA}
\mathfrak{A} = \liminf_k \mathfrak{A}_k = \bigcup_{j\in \N^+} \bigcap_{k\geq j} \mathfrak{A}_k.
\end{equation}
The following Lemma shows that $\mathfrak{A}$ is not empty.
\begin{lemma}\label{lem:level sets}
If $I^{(3)}(\mu) < \infty$, then $\mu \in \mathfrak{A}$.
\end{lemma}
\begin{proof}
	We prove that if \(I^{(3)}(\mu) <  j\), then \(\mu\in \mathfrak{A}_j\) and so we also have $\mu \in \bigcap_{k\geq j}\mathfrak{A}_k$.
	By Theorem~\ref{Theorem: Pi 0 LDP}, we know that
\begin{equation}\label{eq:i3fracac}
- \inf_{\mu \in \mathfrak{A}^c_j}I^{(3)}(\mu) \leq \linf{n}\frac{1}{N} \log \Pi^n_0 \big( \mathfrak{A}^c_j \big) \leq 
\lsup{n}\frac{1}{N} \log \Pi^n_0 \left( \mathfrak{A}^c_j \right).
\end{equation}
But 
\[
\mathfrak{A}^c_j \subset \{\mu, D_{T,L^2}\big(\Psi^{m_j}(\mu,\emp(V_n) \big)>2^{-(j+1)}\}
\cup\{\mu, D_{T,L^2}\big(\Psi^{m_{j+1}}(\mu),\emp(V_n) \big)>2^{-(j+1)}\}.
\]
We deduce by Corollary~\ref{Corollary Measure Representation} that
\begin{align*}
\Pi^n_0 \left( \mathfrak{A}^c_j \right) \leq&\\
&\Pi^n_0\left(D_{T,L^2}\big(\Psi^{m_{j}}(\mu),\emp(V_n) \big)>2^{-(j+1)}\right)+
\Pi^n_0\left(D_{T,L^2}\big(\Psi^{m_{j+1}}(\mu),\emp(V_n) \big)>2^{-(j+1)}\right)
\\
\leq& 
Q^n \left( D_{T,L^2}\big(\Psi^{m_j}\big(\emp(\sigma W_n)\big),\emp(V_n)\big) > 2^{-(j+1)} \right) \\
&+ Q^n \left( D_{T,L^2}\big(\Psi^{m_{j+1}}\big(\emp(\sigma W_n)\big),\emp(V_n)\big) > 2^{-(j+1)} \right)\\
 \leq &                       
Q^n \left( D_{T,L^2}\big(\Psi^{m_j}\big(\emp(\sigma W_n)\big),\emp(V_n)\big) > 2^{-(j+1)} \right) \\
&+ Q^n \left( D_{T,L^2}\big(\Psi^{m_{j+1}}\big(\emp(\sigma W_n)\big),\emp(V_n)\big) > 2^{-(j+2)} \right)
\end{align*}
In addition, using \(\log(a+b)\leq \log(2\max(a,b)) = \log(2) + \max(\log(a),\log(b))\) and \eqref{eq: exponentially equivalent 11}, we obtain
\begin{multline*}
\lsup{n}\frac{1}{N} \log \Pi^n_0 \left( \mathfrak{A}^c_j \right) \leq \\
\max \bigg\lbrace \lsup{n}\frac{1}{N}  \log Q^n \left( D_{T,L^2}\big(\Psi^{m_j}\big(\emp(\sigma W_n)\big),\emp(V_n)\big) > 2^{-(j+1)} \right),\\
\lsup{n}\frac{1}{N}  \log Q^n \left( D_{T,L^2}\big(\Psi^{m_{j+1}}\big(\emp(\sigma W_n)\big),\emp(V_n)\big) > 2^{-(j+2)} \right)   \bigg\rbrace \leq \\
\max \bigg\lbrace -j, -(j+1) \bigg\rbrace =-j .
\end{multline*}
Then, by \eqref{eq:i3fracac} we conclude  that \(\forall \mu\in\mathfrak{A}^c_j\) we have \(I^{(3)}(\mu) \geq j\). It ends the proof.
\end{proof}
We define $\Psi : \mathfrak{A} \to \mathcal{P}_S\big(\T^{\Z}\big)$ as follows
\begin{equation}\label{eq:Psiinfty defin}
\Psi(\mu ) = \lim_{j\to\infty}\Psi^{m_j}(\mu ),
\end{equation}
It follows from the  definitions \eqref{eq:frAj} and \eqref{eq:frA}
	that $\big( \Psi^{m_j}(\mu)\big)_{j\in \N^*}$ is Cauchy so that the limit in \eqref{eq:Psiinfty defin} exists. In effect given $j \geq 0$ it is true that $\mu \in \bigcap_{k\geq j} \mathfrak{A}_k$. since, by the triangle inequality:
\[
D_{T, L^2}(\Psi^{m_j}(\mu),\Psi^{m_{j+k}}(\mu)) \leq \sum_{l=0}^{k-1} D_{T, L^2}(\Psi^{m_{j+l}}(\mu),\Psi^{m_{j+l+1}}(\mu)) \leq 
\sum_{l=0}^{k-1} 2^{-(j+l)} \leq  2^{-(j-1)},
\]
it is true that $\lim_{j, k \to \infty} D_{T, L^2}(\Psi^{m_j}(\mu),\Psi^{m_{j+k}}(\mu)) =0$.

In the notation of \cite[Th. 4.2.23]{dembo-zeitouni:97}, $\epsilon = N^{-1}$, $\tilde{\mu}_\epsilon = \Pi^n$, $f := \Psi$, $\mu_\epsilon = \Pi^n_0$ and $f^j := \Psi^{m_j}$. 

\textit{Step 1: Exponential equivalence}

The `exponentially equivalent' property requires that for any $\delta > 0$, and 
recalling the definition of $V_n$ in 
Corollary~\ref{Corollary Measure Representation} and the fact that the 
law of $\emp(V_n)$ is $\Pi^n$ (also in 
Corollary~\ref{Corollary Measure Representation}),
\begin{equation}\label{eq: exponentially equivalent 11 2}
\lim_{j\to\infty}\lsup{n}\frac{1}{N}\log Q^n\bigg( D_{T,L^2}\big(\Psi^{m_j}\big(\emp(\sigma W_n)\big),\emp(V_n) \big)>\delta \bigg)=-\infty.
\end{equation}
This is an immediate consequence of \eqref{eq: exponentially equivalent 11} which in turn follows from Lemma \ref{Lemma Exp Equivalent 1}.

\textit{Step 2: Uniform Convergence on Level Sets of $I^{(3)}$}

The second property required for \cite[Th. 4.2.23]{dembo-zeitouni:97} is the uniform convergence on level sets,  $\mathcal{L}_{I^{(3)}}(\alpha):=\left\{\mu : I^{(3)}(\mu) \leq \alpha \right\}$,  of $I^{(3)}$, that is we must prove that for any $\alpha > 0$,
\begin{equation}\label{eq:cont on level sets}
\lim_{j\to \infty}\sup_{\mu\in\mathcal{L}_{I^{(3)}}(\alpha)}\left\lbrace D_{T,L^2}\big(\Psi^{m_j}(\mu),\Psi(\mu) \big)\right\rbrace = 0.
\end{equation}
Note that the fact that for all $j \geq \lfloor \alpha \rfloor+1$,
\begin{equation}\label{eq: to show j}
\sup_{\mu\in\mathcal{L}_{I^{(3)}}(\alpha)}\left\lbrace D_{T,L^2}\big(\Psi^{m_j}(\mu),\Psi^{m_{j+1}}(\mu) \big)\right\rbrace \leq 2^{-j}.
\end{equation}
follows from Lemma \ref{lem:level sets} and this suffices because
\begin{equation}
 D_{T,L^2}\big(\Psi^{m_j}(\mu),\Psi(\mu) \big) \leq \sum_{k=j}^\infty D_{T,L^2}\big(\Psi^{m_k}(\mu),\Psi^{m_{k+1}}(\mu) \big) \leq \sum_{k=j}^\infty 2^{-k} 
 \limprob{\longrightarrow}{j \to \infty}{}0.
% 
% 
% \overset{j \to \infty}{\to} 0
\end{equation}
for all $\mu \in \mathcal{L}_{I^{(3)}}(\alpha)$.

\textit{Step 3: Rate Function}
We have thus established the LDP. It remains for us to prove that the rate function is of the form noted in the theorem, and its unique minimum is given by $\mu_*$. According to \cite[Th. 4.2.23]{dembo-zeitouni:97} ,
\begin{equation}
\label{eq:Hmu1}
H(\mu) = \inf_{\zeta \in \PS(\T^\Z): \Psi(\zeta) = \mu}\left\lbrace I^{(3)}(\zeta) \right\rbrace,
\end{equation}
where $H(\mu) := \infty$ if there does not exist $\zeta \in \mP_S(\T^\Z)$ such that $\Psi(\zeta) = \mu$. Since the unique zero of $I^{(3)}$ is $P^{\otimes\Z}$, we can immediately infer that the unique zero of $H$ is $\Psi \big(P^{\otimes\Z}\big)$, which is $\mu_*$. In Section \ref{Section Limiting Process} we prove that this satisfies the McKean-Vlasov stochastic differential equation stated in the Theorem.
\begin{remark}
Theorem 4.2.23 of \cite{dembo-zeitouni:97} requires $\Psi$ to be defined and measurable in $\mP(\T^\Z)$, not only in $\mathfrak{A}$. Since $\mathfrak{A}$ is non empty thanks to Lemma \ref{lem:level sets}, measurable as a countable union of closed sets, we can extend $\Psi$ to a measurable function in $\mP(\T^\Z)$ by simply setting it to an arbitrary measure, say $P^{\otimes \Z}$, in $\mathfrak{A}^c$.
\end{remark}

\subsection{Exponential Tightness of $(\Pi^n)_{n\in\Z^+}$ on $\big( \mathcal{P}_S(\T^{\Z}),D_{T}\big)$}\label{sect:exponential tightness}
In this section we prove in Lemma \ref{Lem: exponentially tight} the exponential tightness of $(\Pi^n)_{n\in\Z^+}$ for the topology induced by $D_T$ on $ \mathcal{P}_S(\T^{\Z})$. As pointed out in Remark \ref{rem:topologies} it is necessary to prove Theorem~\ref{Theorem:Main}.

Lemma~\ref{Lemma: bound theta} is crucial for comparing the system with correlations with the uncorrelated system via Girsanov's Theorem. It is used in the proof of the exponential tightness of $\big(\Pi^n\big)_{n\in\Z^+}$ in  Lemma \ref{Lem: exponentially tight} and is used, as well as Lemma  \ref{Lemma: bound eta}, several times in the sequel. 

Just as for several of the Lemmas below it makes good use of the Discrete Fourier Transform (DFT) of the relevant variables. The corresponding material and notations are presented in Appendix \ref{app:DFTs}. As a general notation, given an $I_n$-periodic sequence $(\beta^j)_{j \in I_n}$, we note $(\tilde{\beta}^p)_{p \in I_n}$ its length $N$ DFT defined by
\[
\tilde{\beta}^p = \sum_{j \in I_n} \beta^j F_N^{-jp}\quad F_N= e^{\frac{2 i \pi }{N}}\ \mbox{ with }\ i^2 = -1.
\]
\begin{lemma}\label{Lemma: bound theta}
For any $M >0$, there exists $C_M > 0$ such that 
\begin{equation}
\lsup{n}\frac{1}{N}\log Q^n\bigg(\frac{1}{N}\sup_{t\in [0,T]}\sum_{j\in I_n}\big(\theta^j_t\big)^2 \geq C_M \bigg) \leq -M.
\end{equation}
\end{lemma}
\begin{proof}
	The proof is rather typical of many of the proofs in this paper. It uses some definitions and results that are given in Appendix \ref{app:DFTs}. It follows three steps.\\
		{\bf Step 1: Go to the Fourier domain}\\
		By Parseval's Theorem,
		\begin{equation}
		\frac{1}{N}\sum_{j\in I_n}\big(\theta^j_t\big)^2 = \frac{1}{N^2}\sum_{p\in I_n}\big|\tilde{\theta}^p_t\big|^2. 
		\end{equation}
		Taking Fourier transforms in \eqref{eq: theta SDE} and using Lemma \ref{lem:translation}, we find that
		\begin{equation}\label{eq: tilde V 0}
		\tilde{V}^p_t = \sigma \tilde{W}^p_t +  \sigma \int_0^t \tilde{\theta}^p_s ds,
		\end{equation}
		where 
		\begin{equation}\label{eq:tildethetap1}
		\tilde{\theta}^p_s =
		\sigma^{-2}  \,\Exp^{\gamma^{\emp(V_n)}}\bigg[\Lambda_s ( G ) G^0_s \int_0^s\tilde{G}^{-p}_r d\tilde{V}^{p}_r \bigg].
		\end{equation}
		Next we write $G^0_s =\frac{1}{N}\sum_{q \in I_n}\tilde{G}^q_s$.
				\begin{align*}
		\tilde{\theta}^p_s &= \sigma^{-2}\Exp^{\gamma^{\emp(V_n)}}\bigg[ 
		\Lambda_s ( G )
		(\frac{1}{N}\sum_{q \in I_n}\tilde{G}^q_s) \int_0^s\tilde{G}^{-p}_r d\tilde{V}^{p}_r \bigg]  \\
&=		\frac{1}{N} \sigma^{-2} \sum_{q \in I_n} \Exp^{\gamma^{\emp(V_n)}}\bigg[ \Lambda_s ( G ) \tilde{G}^q_s 
		\int_0^s\tilde{G}^{-p}_r d\tilde{V}^{p}_r \bigg] .
		\end{align*}
			According to Corollary \ref{cor:tilde Gp independent tilde Gq} and its proof
		\begin{align*}
		\sum_{q \in I_n} \Exp^{\gamma^{\emp(V_n)}}\bigg[ \Lambda_s ( G ) \tilde{G}^q_s 
		\int_0^s\tilde{G}^{-p}_r d\tilde{V}^{p}_r \bigg] &=  
		\Exp^{\gamma^{\emp(V_n)}}\bigg[ \Lambda_s ( G ) \tilde{G}^p_s 
		\int_0^s\tilde{G}^{-p}_r d\tilde{V}^{p}_r \bigg]  \\
&=		\Exp^{\gamma^{\emp(V_n)}} \bigg[  \tilde{\Lambda}^{|p|}_s(\tilde{G}) \tilde{G}^p_s
		\int_0^s\tilde{G}^{-p}_r d\tilde{V}^{p}_r  \bigg].
		\end{align*}
		This allows us
		to rewrite \eqref{eq:tildethetap1} as
		\begin{equation}\label{eq:tildethetap}
		\tilde{\theta}^p_s =
		N^{-1} \sigma^{-2}  \,\Exp^{\gamma^{\emp(V_n)}}\bigg[\tilde{\Lambda}_s^{| p |}( \tilde{G} ) \tilde{G}^p_s \int_0^s\tilde{G}^{-p}_r d\tilde{V}^{p}_r \bigg].
		\end{equation}
		We substitute \eqref{eq: tilde V 0} into the right hand side of \eqref{eq:tildethetap} and obtain
		\begin{align}
		\tilde{\theta}^p_t =& \frac{1}{\sigma N} \int_0^t \Exp^{\gamma^{\emp(V_n)}}\big[\tilde{\Lambda}^{|p|}_t\big(G\big)\tilde{G}^p_t \tilde{G}^{-p}_s\big]   \tilde{\theta}^p_s ds + \frac{1}{\sigma N}\Exp^{\gamma^{\emp(V_n)}}\bigg[\tilde{\Lambda}^{|p|}_t\big(G\big)\tilde{G}^p_t \int_0^t \tilde{G}^{-p}_s d\tilde{W}^p_s\bigg]\label{eq: theta Fourier 1}
		\end{align}
		{\bf Step 2: Find an upper bound for the Fourier transformed quantities:}\\
		Applying twice the Cauchy-Schwarz inequality to \eqref{eq: theta Fourier 1},
		\begin{multline*}
		\big|\tilde{\theta}^p_t\big|^2 \leq  \frac{2 t}{ \sigma^2  N^2} \int_0^t \bigg|\Exp^{\gamma^{\emp(V_n)}}\big[\tilde{\Lambda}^{|p|}_t (\tilde{G})\tilde{G}^p_t \tilde{G}^{-p}_s\big]\bigg|^2   \big|\tilde{\theta}^p_s\big|^2 ds + \frac{2}{\sigma^2  N^2}\bigg| \Exp^{\gamma^{\emp(V_n)}}\bigg[\tilde{\Lambda}^{|p|}_t(\tilde{G})\tilde{G}^p_t \int_0^t \tilde{G}^{-p}_s d\tilde{W}^p_s\bigg]\bigg|^2.
		\end{multline*}
		By Lemma~\ref{Lemma Fourier Bound on R Jn},
		\[
		\bigg|\Exp^{\gamma^{\emp(V_n)}}\big[\tilde{\Lambda}^{|p|}_t (\tilde{G}) \tilde{G}^p_t \tilde{G}^{-p}_s\big]\bigg|^2  \leq (C_\J)^2\sum_{j \in I_n}f(V^j_s)^2 \sum_{k \in I_n}  f(V^k_t)^2  \leq N^2 (C_\J)^2
		\]
		and 
		\begin{align*}
		\bigg| \Exp^{\gamma^{\emp(V_n)}}\bigg[\tilde{\Lambda}^{|p|}_t(\tilde{G})\tilde{G}^p_t \int_0^t \tilde{G}^{-p}_s d\tilde{W}^p_s\bigg]\bigg|^2
		&\leq (C_\J)^2 \sum_{j \in I_n} f(V^j_t)^2 \sum_{k \in I_n}\bigg| \int_0^t f(V^k_s) d\tilde{W}^p_s\bigg|^2 \\
		&\leq N(C_\J)^2\sum_{k\in I_n}  \bigg| \int_0^t f(V^k_s) d\tilde{W}^p_s\bigg|^2.
		\end{align*}
		Applying Parseval's Theorem to the right hand side of the previous inequality,
		\[
		\sum_{p\in I_n} \bigg| \Exp^{\gamma^{\emp(V_n)}}\bigg[\tilde{\Lambda}^{|p|}_t(\tilde{G})\tilde{G}^p_t \int_0^t \tilde{G}^{-p}_s d\tilde{W}^p_s\bigg]\bigg|^2
		\leq N^2(C_\J)^2\sum_{j,k\in I_n}  \left( \int_0^t f(V^k_s) dW^j_s\right)^2.
		\]
		This means that
		\begin{equation*}
		\frac{1}{N^2}\sum_{p\in I_n}\big|\tilde{\theta}^p_t\big|^2 \leq 2 \sigma^{-2} (C_\J)^2t\int_0^t \frac{1}{N^2}\sum_{p\in I_n}\big|\tilde{\theta}^p_s\big|^2 ds +  \frac{2}{N^2} \sigma^{-2} (C_\J)^2\sum_{j,k\in I_n}  \left( \int_0^t f(V^k_s) dW^j_s\right)^2.
		\end{equation*}
		We thus find through Gronwall's Inequality that 
		\begin{align*}
		\frac{1}{N^2}\sum_{p\in I_n} \big|\tilde{\theta}^p_t\big|^2 \leq &  \frac{2}{\sigma^2N^2}(C_\J)^2\exp\big(2 \sigma^{-2} (C_\J)^2 T^2\big)\sup_{r \in [0,t]}\sum_{j,k\in I_n}  \left( \int_0^r f(V^k_s) dW^j_s\right)^2.
		\end{align*}
		{\bf Step 3: Apply Doob's submartingale inequality}:\\
		Now $\sum_{j,k\in I_n}  \left( \int_0^t f(V^k_s) dW^j_s\right)^2$ is a submartingale, hence, for any $\kappa > 0$,
		\[
		\zeta_t := \exp\bigg(  \kappa\,2 \sigma^{-2} N^{-1}(C_\J)^2\exp\big(2 \sigma^{-2} (C_\J)^2 T^2\big)\sum_{j,k\in I_n}  \left( \int_0^t f(V^k_s) dW^j_s\right)^2\bigg)
		\]
		is also a submartingale. By Doob's submartingale inequality,  for an $K > 0$,
		\begin{align*}
		Q^n \bigg(\sup_{t \in [0,T]} \frac{1}{N^2}\sum_{p\in I_n} \big|\tilde{\theta}^p_t\big|^2 \geq K \bigg) &= 
		Q^n\bigg(\sup_{t \in [0,T]}  \exp\bigg(\frac{\kappa}{N}\sum_{p\in I_n} \big|\tilde{\theta}^p_t\big|^2\bigg) \geq \exp\big( \kappa NK\big) \bigg) \\
		&\leq Q^n\bigg(\sup_{t \in [0,T]} \zeta_t \geq \exp\big( \kappa NK\big) \bigg) \\
		&\leq \exp\big( -\kappa NK\big)\Exp\big[ \zeta_T\big].
		\end{align*}
		Now for $\kappa$ small enough, by Lemma \ref{lem:BG} and the boundedness of $f$ there exists a constant $C$ such that $\Exp\big[ \zeta_T\big] \leq \exp\big( NC\big)$ for all $N \in \Z^+$. We thus find that
		\begin{align*}
		Q^n\bigg(\frac{1}{N}\sup_{t\in [0,T]}\sum_{j\in I_n}\big(\theta^j_t\big)^2 \geq K \bigg) =& Q^n\bigg(\frac{1}{N^2}\sup_{t \in [0,T]} \sum_{p\in I_n} \big|\tilde{\theta}^p_t\big|^2 \geq K \bigg)\\ \leq &\exp\bigg(N(C - \kappa K)\bigg),
		\end{align*}
		from which we can conclude the Lemma by taking $K$ to be sufficiently large.

\end{proof}
We have a similar result for $\theta^{m,j}$ defined in \eqref{eq:vmtthetamt}.
\begin{lemma}\label{Lemma: bound eta}
For any $M >0$, there exists $C_M > 0$ such that 
\begin{equation}
\lsup{n}\frac{1}{N}\log Q^n\bigg(\frac{1}{N}\sup_{t\in [0,T]}\sum_{j\in I_n}\big(\theta^{m,j}_t\big)^2 \geq C_M \bigg) \leq -M.
\end{equation}
\end{lemma}
\begin{proof}
The proof is similar to that of Lemma~\ref{Lemma: bound theta} and is left to the reader. 
\end{proof}
Note that the DFT $\tilde{V}^{m,p}$ of the approximation \(V^{m,j}\) satisfies the following system of SDEs, analog to \eqref{eq: tilde V 0}:
\begin{equation}\label{eq: tilde Vm 0}
\tilde{V}^{m, p}_t = \sigma \tilde{W}^p_t +  \sigma \int_0^t \tilde{\theta}^{m, p}_s ds
\end{equation}

As pointed out in the introduction to Section \ref{subsection:proofT2.4} the exponential tightness is a key step in proving the LDP for $\Pi^n$. 
\begin{lemma}\label{Lem: exponentially tight}
The family of measures $\big(\Pi^n \big)_{n\in \Z^+}$ is exponentially tight 
with respect to the topology on $\mathcal{P}_S(\T^{\Z})$ induced by $D_T$. That 
is, for any $M > 0$,
 there exists a compact set 
$K_M \subset \mathcal{P}_S(\T^{\Z})$ such that
\[
\lsup{n}\frac{1}{N}\log \Pi^n\big(K_M^c \big) \leq - M.
\]
\end{lemma}
\begin{proof}
Consider the event $\mathfrak{K}_{n,M}$ defined by
\begin{equation}\label{eq: mathfrak K n}
\mathfrak{K}_{n,M} = \bigg\lbrace  \frac{1}{N}\sup_{t\in [0,T]}\sum_{j\in I_n}\big(\theta^j_t\big)^2 \geq C_M \bigg\rbrace 
\end{equation}
By Lemma \ref{Lemma: bound theta}, we can find  $C_M$ such that
\begin{equation}\label{eq:minuskappa}
\lsup{n}\frac{1}{N}\log Q^n\big(\mathfrak{K}_{n,M} \big) \leq -M.
\end{equation}
For any compact set $K_M$ of $\PS(\T^{\Z})$, we have $\Pi^n\big(K_M^c\big)=Q^n\left( \emp^{-1}(K_M^c) \right)$ so that, by \eqref{eq:minuskappa}
\begin{align*}
\lsup{n}\frac{1}{N}\log \Pi^n\big(K_M^c\big) &\leq \max\bigg\lbrace\lsup{n}\frac{1}{N}\log Q^n\big(\emp^{-1}(K_M^c) \cap\mathfrak{K}_{n,M}^c \big) ,
 \lsup{n}\frac{1}{N}\log Q^n\big( \mathfrak{K}_{n,M}\big)\bigg\rbrace \\
&\leq \max\bigg\lbrace \lsup{n}\frac{1}{N}\log Q^n\big(\emp^{-1}(K_M^c) \cap\mathfrak{K}_{n,M}^c \big),\,- M \bigg\rbrace,
\end{align*}
so that it suffices for us to prove that
\begin{equation}\label{eq: to prove Qn}
\lsup{n}\frac{1}{N}\log Q^n\big(\emp^{-1}(K_M^c) \cap\mathfrak{K}_{n,M}^c\big) \leq - M.
\end{equation}
By Proposition \ref{prop:RNderiv1}, and using the Cauchy-Schwarz Inequality,
\begin{align*}
 Q^n\big(\emp^{-1}(K_M^c) \cap\mathfrak{K}_{n,M}^c \big) = \int_{\emp^{-1}(K_M^c) \cap\mathfrak{K}_{n,M}^c}\exp\bigg(\sum_{j\in I_n}\int_0^T \theta_s^j dB^j_s - \frac{1}{2}\sum_{j\in I_n}\int_0^T \big(\theta_s^j\big)^2 ds \bigg) dP^{\otimes N}(B)\\
\leq \bigg\lbrace\int_{\emp^{-1}(K_M^c) \cap\mathfrak{K}_{n,M}^c}\exp\bigg(2\sum_{j\in I_n}\int_0^T \theta_s^j dB^j_s -2\sum_{j\in I_n}\int_0^T \big(\theta_s^j\big)^2 ds \bigg) dP^{\otimes N}(B)\bigg\rbrace^{\frac{1}{2}}\\ \times \bigg\lbrace\int_{\emp^{-1}(K_M^c) \cap\mathfrak{K}_{n,M}^c}\exp\bigg(\sum_{j\in I_n}\int_0^T \big(\theta_s^j\big)^2 ds \bigg) dP^{\otimes N}(B)\bigg\rbrace^{\frac{1}{2}}.
\end{align*}
Now using the properties of a supermartingale,
\begin{multline*}
\int_{\emp^{-1}(K_M^c) \cap\mathfrak{K}_{n,M}^c}\exp\bigg(2\sum_{j\in I_n}\int_0^T \theta_s^j dB^j_s -2\sum_{j\in I_n}\int_0^T \big(\theta_s^j\big)^2 ds \bigg) dP^{\otimes N}(B) \\
\leq \int_{\T^{N}}\exp\bigg(2\sum_{j\in I_n}\int_0^T \theta_s^j dB^j_s -2\sum_{j\in I_n}\int_0^T \big(\theta_s^j\big)^2 ds \bigg) dP^{\otimes N}(B) \leq 1.
\end{multline*}
Using the definition of $\mathfrak{K}_{n,M}$ in \eqref{eq: mathfrak K n}, and since $\Pi_0^n = P^{\otimes N} \circ \emp^{-1}$
\begin{multline*}
\int_{\emp^{-1}(K_M^c) \cap\mathfrak{K}_{n,M}^c}\exp\bigg(\sum_{j\in I_n}\int_0^T \big(\theta_s^j\big)^2 ds \bigg) dP^{\otimes N}(B)
\leq \exp\big(NT C_M \big)P^{\otimes N}\big( \emp^{-1}(K_M^c)\big)\\
=\exp\big(NT C_M \big)\Pi_0^n\big( K_M^c \big).
\end{multline*}
Now $\big( \Pi_0^n \big)_{n\in\Z^+}$ is exponentially tight (a direct consequence of Theorem \ref{Theorem: Pi 0 LDP}), which means that we can choose $K_{M}$ to be such that
\[
\lsup{n}\frac{1}{N} \log \Pi_0^n\big(K_{M}^c \big) \leq -\big( 2 M + T C_M\big),
\]
so that we can conclude \eqref{eq: to prove Qn} as required.
\end{proof}

\subsection{Exponentially Equivalent Approximations using $\Psi^m$}
\label{subsec:main lemma}
The following Lemma, which is central in the proof of Theorem \ref{Theorem:Main}, is the main result of this section. Its proof is long and technical and uses four auxiliary Lemmas, Lemmas~\ref{lem:alpha1}-\ref{lem:alpha4} whose proofs are found in Appendix~\ref{app:proofofalphas}.

%Recall that $D_{T,L^2}$ is the Wasserstein-1 distance on $\mathcal{P}_S(\T^{\Z})$ induced by the metric on $\T^\Z$ defined by \eqref{eq: L squared distance}.
%
%For the following Lemma, recall that under $Q^n$, $V_n$ and $W_n$ satisfy the SDE given in Corollary \ref{Corollary Measure Representation}, and the law of $\sigma W_n=(\sigma W^j)_{j\in I_n}$ is $P^{\otimes N}$.
\begin{lemma}\label{Lemma Exp Equivalent 1}
For any $\delta > 0$,
\begin{equation}\label{eq: exponentially equivalent 4}
\lsup{m}\lsup{n}\frac{1}{N}\log Q^n\left( D_{T,L^2}\left(\Psi^m\left(\emp(\sigma W_n)\right),\emp(V_n) \right)>\delta   \right)=-\infty.
\end{equation}
\end{lemma}
\begin{proof}
The proof uses the following ideas.\\
By Lemma~\ref{Lem: Psim continuous}, $\Psi^m\big( \hat{\mu}_n(\sigma W_n)\big)=\hat{\mu}_n(V^m_n)$. By Lemma~\ref{lem:useful ineq}, we can find an  upperbound of
$D_{T,L^2}(\hat{\mu}_n(V^m_n),\emp(V_n))$ using the $L^2$ distance between 
$V^m_n$  and $V_n$, so that the proof boils down to comparing the solution 
$V_n$ to the system of equations \eqref{eq: theta SDE} and 
\eqref{eq: theta SDE1} to the solution $V^m_n$ to the approximating system of 
equations \eqref{eq:vmtthetamt}  constructed in 
Section~\ref{subsub:approx} by an $L^2$ distance. 
By equations \eqref{eq:vmtthetamt} and \eqref{eq: theta SDE} this is equivalent to 
comparing the $L^2$ distance between $\theta^m$ and $\theta$. 
As already mentioned, it is technically easier to work in the Fourier domain 
with the $L^2$  distance between $\tilde{\theta}^{m,p}$ and $\tilde{\theta}^p$, 
$p \in I_n$, the Fourier transforms of $(\theta^{m,j})_{j \in I_n}$ and 
$(\theta^j)_{j \in I_n}$. This distance naturally brings in the operators 
$\bar{L}_{\emp(V_n)}^t$ and $\bar{L}_{\emp(V^m_n)}^t$ defined in 
Appendix~\ref{app:covariances}, in effect their Fourier transforms.

The following Lemma (proved page~\pageref{prooflemusefulineq}) relates the Wasserstein distance $D_{T,L^2}$ between two empirical measures associated with two elements of $\T^N$ to the $L^2$ distance between these elements.
\begin{lemma}\label{lem:useful ineq}
For all $X_n,\,Y_n \in \T^N$ we have
\[
D_{T, L^2}(\emp(X_n),\emp(Y_n))^2 \leq \frac{b^2}{N} \sum_{k \in I_n} \norm{X^k-Y^k}_{L^2}^2
\]
	where $b$ is defined by \eqref{eq:ab}.
\end{lemma}
We now follow our plan for the proof of Lemma \ref{Lemma Exp Equivalent 1}.\\
By Lemmas~\ref{Lem: Psim continuous} and \ref{lem:useful ineq} we write
\[
D_{T,L^2}\left(\Psi^m\left(\emp(\sigma W_n)\right),\emp(V_n) \right)^2 \leq \frac{b^2}{N}\sum_{j\in I_n} \norm{V^{m,j} - V^j}_{L^2}^2.
\]
By Parseval's Theorem,
\begin{align*}
\frac{1}{N}\sum_{j\in I_n} \norm{V^{m,j} - V^j}_{L^2}^2 = \frac{1}{N}\sum_{j\in I_n}\int_0^T \big| V^{m,j}_t - V^j_t \big|^2 dt 
 = \frac{1}{N^2}\sum_{p\in I_n}\int_0^T \big| \tilde{V}^{m,p}_t - \tilde{V}^p_t \big|^2 dt,
\end{align*}
 In order to prove \eqref{eq: exponentially equivalent 4} it therefore suffices for us to prove that for any arbitrary $M,\delta > 0$, \emph {which are now fixed throughout the rest of this proof},
\begin{equation}
\lsup{m}\,\lsup{n}\frac{1}{N}\log Q^n\bigg( \sup_{t\in [0,T]}\frac{1}{N^2}\sum_{p\in I_n} \big|\tilde{V}^{m, p}_t - \tilde{V}_t^p\big|^2 >\delta^2 / T  \bigg)\leq -M.\label{eq:final 0}
\end{equation}
Using the expression in \eqref{eq: tilde V 0}, it follows from the Cauchy-Schwarz inequality that for any $t\in [0,T]$,
\begin{equation}\label{eq: tilde YV}
 \frac{1}{N^2}\sum_{p\in I_n} \big| \tilde{V}^{m, p}_t - \tilde{V}^p_t \big|^2
 \leq \frac{t \sigma^2}{N^2}\sum_{p\in I_n}\int_0^t \big| \tilde{\theta}^p_s - \tilde{\theta}^{m, p}_s \big|^2 ds \leq \frac{T \sigma^2}{N^2}\sum_{p\in I_n}\int_0^t \big| \tilde{\theta}^p_s - \tilde{\theta}^{m, p}_s \big|^2 ds.
 \end{equation}
In order to continue our plan we introduce the discrete time approximation $^m\tilde{\theta}^p_{s^{(m)}}$ of $\tilde{\theta}^p_s$
% que penses-tu de cette notation ?}  
%of $\tilde{\theta}^p_s$ 
\begin{equation}
\label{eq:def mthetaps}
^m\tilde{\theta}^p_{s^{(m)}} = \frac{1}{ N\sigma^{2}} \Exp^{\gamma^{\emp(V_n)}} \left[ \tilde{\Lambda}_{s^{(m)}}^{|p|}(\tilde{G})\tilde{G}^p_{s^{(m)}} \int_0^{s^{(m)}} \tilde{G}^{-p}_{r^{(m)}}  d\tilde{V}^{p}_r \right].
\end{equation}
We obtain in the following Lemma a characterization of \(^m\tilde{\theta}^p_{s^{(m)}}\)
\begin{lemma}\label{lem:tildethetamp}
Assume $s^{(m)} = v \eta_m$, $v=0,\cdots,m$. We have
\[
 ^m\tilde{\theta}^p_{v \eta_m} =  \sigma^{-2} \left( \bar{\tilde{L}}^{p}_{\emp(V_n)}\delta \tilde{V}^p \right)(v \eta_m),
\]
where $\bar{\tilde{L}}^{ p}_{\emp(V_n)}$ is the $(v+1) \times (v+1)$ matrix $(\tilde{L}^{p}_{\emp(V_n)}(w \eta_m, u \eta_m))_{w,\,u= 0,\cdots,v}$ defined by
\[
\tilde{L}^{p}_{\emp(V_n)}(v \eta_m, w \eta_m)= N^{-1} \Exp^{\gamma^{\emp(V_n)}}\left[ \tilde{\Lambda}_{v \eta_m}^{|p|}(\tilde{G})\tilde{G}^p_{v \eta_m} \tilde{G}^{-p}_{w \eta_m}  \right],
\]
and $\delta \tilde{V}^p$ is the $v+1$-dimensional vector
\begin{equation}\label{eq:deltatildeVp}
 \delta \tilde{V}^p_w = \left\{
 \begin{array}{ll}
 0 &  w = 0\\
 \tilde{V}^p_{w \eta_m}- \tilde{V}^p_{(w-1) \eta_m} & w=1,\cdots, v
 \end{array}
 \right.
 \end{equation}
\end{lemma}
\begin{proof}
We give a short proof.
 Since $s^{(m)} = v \eta_m$, $v=0,\cdots,m$, and using Remark~\ref{rem:Lmukt} and the notations of Appendix~\ref{app:covariances}
 \begin{align*}
 ^m\tilde{\theta}^p_{v \eta_m} &= \sigma^{-2} N^{-1} \Exp^{\gamma^{\emp(V_n)}} \left[ \tilde{\Lambda}_{v \eta_m}^{|p|}(\tilde{G})\tilde{G}^p_{v \eta_m} \int_0^{v \eta_m} \tilde{G}^{-p}_{r^{(m)}}  d\tilde{V}^{p}_r \right] \\
& = \sigma^{-2} \sum_{w = 0}^{v-1}  N^{-1} \Exp^{\gamma^{\emp(V_n)}}\left[ \tilde{\Lambda}_{v \eta_m}^{|p|}(\tilde{G})\tilde{G}^p_{v \eta_m} \tilde{G}^{-p}_{w \eta_m}  \right] 
 (\tilde{V}^p_{(w+1) \eta_m}- \tilde{V}^p_{w \eta_m}) \\
 & = \sigma^{-2} \sum_{w = 0}^{v-1} \tilde{L}^{ p}_{\emp(V_n)}(v \eta_m, w \eta_m)  (\tilde{V}^p_{(w+1) \eta_m}- \tilde{V}^p_{w \eta_m}) \\
& = \sigma^{-2} \left( \bar{\tilde{L}}^{p}_{\emp(V_n)}\delta \tilde{V}^p \right)(v \eta_m),
 \end{align*}
 where $\bar{\tilde{L}}^{p}_{\emp(V_n)}$ is the $(v+1) \times (v+1)$ matrix $(\tilde{L}^{p}_{\emp(V_n)}(w \eta_m, u \eta_m))_{w,\,u= 0,\cdots,v}$ defined in Remark~\ref{rem:Lmukt} and Appendix~\ref{app:covdiscrete}.
\end{proof}
The autocorrelation function $L_{\emp(V_n)}$ (resp. $L_{\emp(V^m_n)}$) involved in the sequence $(V^j)_{j \in I_n}$  (resp. $(V^{m,j})_{j \in I_n}$) and hence in the sequence $(\theta^j)_{j \in I_n}$ (resp. $(\theta^{m,j})_{j \in I_n}$)  arises from the values of the autocorrelation function $R_\J$, defined in \eqref{eq:cov}, on a grid $I_n \times I_n$ (resp. $I_{q_m} \times I_n $). Since we are working in the discrete Fourier domain, it is natural, as explained in Appendix \ref{app:covdiscrete}, and in fact necessary, to consider the following four operators (in the discrete time setting, matrixes) in order to compare $\tilde{\theta}^p$ and $\tilde{\theta}^{m,p}$. In detail, $\bar{\tilde{L}}^{ p}_{\emp(V_n)}$, (resp. $\bar{\tilde{L}}^{p}_{\emp(V^m_n)}$), $p \in I_n$ is obtained by taking the length $N$ DFT of the length $N$ sequence $(L_{\emp(V_n)}^{i})_{i \in I_n}$ (resp. $(L_{\emp(V^m_n)}^{i})_{i \in I_n}$). Similarly, $\bar{\tilde{L}}^{q_m, p}_{\emp(V_n)}$, (resp. $\bar{\tilde{L}}^{q_m, p}_{\emp(V^m_n)}$), $p \in I_n$ is obtained by taking the length $N$ DFT of the length $Q_m$ sequence $(L_{\emp(V_n)}^{i})_{i \in I_{q_m}}$ (resp. $(L_{\emp(V^m_n)}^{i})_{i \in I_{q_m}}$) padded with $N-Q_m$ zeros.\label{page:explicationInIQm}
 
 We then use the following decomposition
 \begin{multline*}
  \left|  \tilde{\theta}^p_s - \tilde{\theta}^{m,p}_s \right| \leq  \left| \tilde{\theta}^p_s -{^m}\tilde{\theta}^p_{v \eta_m} \right|+
  \sigma^{-2} \left| \left( \left( \bar{\tilde{L}}^{p}_{\emp(V_n)} - \bar{\tilde{L}}^{q_m, p}_{\emp(V_n)} \right) \delta \tilde{V}^p \right)(v \eta_m) \right|+\\
   \sigma^{-2} \left| \left( \left( \bar{\tilde{L}}^{q_m, p}_{\emp(V_n)}-\bar{\tilde{L}}^{p}_{\emp(V_n^m)} \right)\delta \tilde{V}^p \right)(v \eta_m) \right|+
    \sigma^{-2} \left| \left( \bar{\tilde{L}}^{ p}_{\emp(V_n^m)} \left( \delta \tilde{V}^p-\delta \tilde{V}^{m, p} \right) \right)(v \eta_m) \right|+\\
 \left|   \sigma^{-2} \left( \bar{\tilde{L}}^{p}_{\emp(V_n^m)}\delta \tilde{V}^{m, p} \right)(v \eta_m)-\tilde{\theta}^{m,p}_s  \right|,
 \end{multline*}

Each term on the right hand side performs a specific comparison:
\begin{description}
	\item[First term:] Allows to compare $\tilde{\theta}^p_s$ and its time discretized version ${^m}\tilde{\theta}^p_{s^{(m)}}$ which is equal, thanks to Lemma  \ref{lem:tildethetamp}, to $\sigma^{-2}\left( \bar{\tilde{L}}^{p}_{\emp(V_n)}\delta \tilde{V}^p \right)(v \eta_m)$.
	\item[Second term:] Allows to compare the operator $\bar{\tilde{L}}^{p}_{\emp(V_n)}$ with its space/correlation truncated and Fourier interpolated version  $\bar{\tilde{L}}^{q_m, p}_{\emp(V_n)}$.
	\item[Third term:] Allows to compare the operator $\bar{\tilde{L}}^{q_m, p}_{\emp(V_n)}$ with the operator $\bar{\tilde{L}}^{p}_{\emp(V_n^m)}$ corresponding to the approximated solution.
	\item[Fourth term:] Allows to compare the time discretized versions of the $\tilde{V}_n$ and $\tilde{V}^m_n$ processes.
	\item[Fifth term] Allows to compare the space/correlation truncated and Fourier interpolated opertor  $\bar{\tilde{L}}^{q_m, p}_{\emp(V_n^m)}$ with its Fourier interpolation $\tilde{\theta}^{m,p}$.
\end{description}
By slightly changing the order of the terms we write, remember that $s^{(m)} = v \eta_m$,
 \begin{align}
\nonumber\frac{1}{N^2} \sum_{p \in I_n} \left|  \tilde{\theta}^p_s - \tilde{\theta}^{m, p}_s \right|^2 &\leq    \frac{5}{N^2} \sum_{p \in I_n}  \left| \tilde{\theta}^p_s -{^m}\tilde{\theta}^p_{v \eta_m} \right|^2&\Big\}\alpha^1_s\hspace{4mm}\\
\nonumber&+  \frac{5}{N^2\sigma^4} \sum_{p \in I_n}  \left| \left( \left( \bar{\tilde{L}}^{ p}_{\emp(V_n)} - \bar{\tilde{L}}^{q_m,  p}_{\emp(V_n)} \right) \delta \tilde{V}^p \right)(v \eta_m) \right|^2&\Big\}\alpha^2_{v \eta_m} \\ 
\nonumber& +    \frac{5}{N^2} \sum_{p \in I_n}  \left| \sigma^{-2} \left(   \bar{\tilde{L}}^{ p}_{\emp(V_n^m)}\delta \tilde{V}^{m, p} \right)(v \eta_m)-\tilde{\theta}^{m,p}_s \right|^2 &\Big\}\alpha^3_{v \eta_m}\\
\nonumber& +  \frac{5}{N^2\sigma^4} \sum_{p \in I_n} \left| \left( \left( \bar{\tilde{L}}^{q_m,  p}_{\emp(V_n)}-\bar{\tilde{L}}^{ p}_{\emp(V_n^m)} \right)\delta \tilde{V}^p \right)(v \eta_m) \right|^2&\Big\}\alpha^4_{v \eta_m}\\
& + \frac{5}{N^2\sigma^4} \sum_{p \in I_n}  \left|\left( \bar{\tilde{L}}^{ p}_{\emp(V_n^m)} \left( \delta \tilde{V}^p-\delta \tilde{V}^{m, p} \right) \right)(v \eta_m) \right|^2&\Big\} \alpha^5_{v \eta_m}.\label{eq:tildetheta-eta}
\end{align}
  Our first action is to remove the term $\alpha^5$ through the use of Gronwall's Lemma.
  
  Since, by Proposition~\ref{prop:Lkmuregular},  $| \tilde{L}^{ p}_{\emp(V_n^m)}(v\eta_m, w\eta_m) |$  is uniformly bounded
     by some constant $K > 0$ independent of $w,\, v,\, p,\, q_m, \, n,\, V_n^m$, and according to equations \eqref{eq: tilde V 0}, \eqref{eq: tilde Vm 0} and \eqref{eq:deltatildeVp}
  \begin{align*}
   \alpha^5_{v \eta_m}  &\leq  \frac{5  K^2}{N^2\sigma^{4}} \sum_{p \in I_n} \left( \sum_{w=1}^{v}  | \delta \tilde{V}^p_w-\delta \tilde{V}^{m, p}_w | \right)^2 = 
   \frac{ 5 K^2}{N^2\sigma^{2}} \sum_{p \in I_n} \left( \sum_{w=0}^{v-1} \left| \int_{ w \eta_m}^{ (w+1) \eta_m} ( \tilde{\theta}^p_r -  \tilde{\theta}^{m, p}_r )\,dr \right| \right)^2\\
& \leq     \frac{5 v \eta_m  K^2}{N^2\sigma^{2}} \sum_{p \in I_n}  \int_0^{v \eta_m}  \left| \tilde{\theta}^p_r -  \tilde{\theta}^{m, p}_r  \right|^2 \,dr \leq   \frac{ 5 T   K^2}{N^2\sigma^{2}} \sum_{p \in I_n} \int_0^s  \left| \tilde{\theta}^p_r- \tilde{\theta}^{m, p}_r  \right|^2 \,dr
  \end{align*}
  Inserting this uppper bound for $ \alpha^5_{v \eta_m}$ in the right hand side of \eqref{eq:tildetheta-eta} and applying Gronwall's Lemma we obtain
  \[
  \frac{1}{N^2} \sum_{p \in I_n}  \left|  \tilde{\theta}^p_s - \tilde{\theta}^{m, p}_s \right|^2 \leq C \sup_{r \in  [0, s]} \left\{  \alpha^1_r + \sum_{i = 2}^4 \alpha^i_{r^{(m)}}  \right\},
  \]
  with 
  \begin{equation}\label{eq:C1}
  C =  \exp \left( 5 T^2  \sigma^{-2} K^2 \right).
  \end{equation}
   Hence, by \eqref{eq: tilde YV}
  \begin{align}\nonumber
   \frac{1}{N^2}\sum_{p\in I_n} \big| \tilde{V}^{m, p}_t - \tilde{V}^p_t \big|^2
 &\leq T C \sigma^2 \int_0^t \sup_{r \in  [0, s]} \left\{ \alpha^1_r + \sum_{i = 2}^4 \alpha^i_{r^{(m)}} \right\} \,ds  \\
 &\leq
 TC \sigma^2 \left(  \int_0^t  \sup_{r \in [0,s]} \alpha^1_r \,ds + \sum_{i = 2}^4 \int_0^t  \sup_{r \in [0,s]} \alpha^i_{r^{(m)}} \,ds \right).
 \label{eq:upper}
  \end{align}
The next step in the proof is the definition of the following stopping time. For $\mathfrak{c} > 0$ and $\epsilon \leq \exp\big(-\mathfrak{c}T\big)\delta^2 / T$, define 
\begin{equation}\label{eq:deftau}
\tau(\epsilon,\mathfrak{c}) = \inf\bigg\lbrace  t \in [0,T]:   \frac{1}{N^2}\sum_{p\in I_n} \big| \tilde{V}^{m, p}_t - \tilde{V}^p_t \big|^2 = \epsilon \exp\big(t\mathfrak{c}\big)\bigg\rbrace.
\end{equation}
\begin{remark}
The random time $\tau(\epsilon,\mathfrak{c})$ is the time at which the $L^2$ distance between the $N$ trajectories $V_n$ and $V^m_n$  differ on average by more than
$\exp\left(-\mathfrak{c}(T-t)\right)\delta^2 / T (\leq \delta^2/T)$.
\end{remark}
%Indeed, this allows us to circumvent in the proof of the upcoming Lemma \ref{lem:alpha4} the fact that the functions $(t,s) \to L^{t,k}_\mu(t,s)$ are not in general Lipschitz continuous.
The crucial idea of the proof is to upper bound the left hand side of \eqref{eq:final 0}  by 
\[
\lsup{n} \frac{1}{N} \log \left( m \max_{u=0,\cdots,m-1} Q^n\left( \left\{  \tau(\epsilon, \mathfrak{c}) \in [u \eta_m, (u+1) \eta_m] \right\} \right) \right),
\] 
see \eqref{eq:bound with tau} below.

The proof proceeds iteratively through the time steps: we show that if $\tau(\epsilon,\mathfrak{c})\geq u \eta_m$, for $u=0,\cdots,m-1$ then with very high probability $\tau(\epsilon,\mathfrak{c})\geq (u+1)\eta_m$.
We show in the proof of Lemma \ref{lem:alpha4} that there exists $\mathfrak{c} > 0$ such that for any $\epsilon < \delta^2 \exp(-\mathfrak{c}T) / T$, for all $m$ sufficiently large, for all $0\leq u < m$,
\begin{equation}
\lsup{n}\frac{1}{N} \log Q^n\bigg(\tau(\epsilon,\mathfrak{c}) \in \big[ u \eta_m,(u+1) \eta_m] \bigg) \leq -M. \label{eq: toshow tau hitting time}
\end{equation}	
	Indeed this suffices for proving Lemma~\ref{Lemma Exp Equivalent 1}.
	We have
	\[
	\frac{1}{N^2}\sum_{p\in I_n} \big|\tilde{V}^{m,p}_t - \tilde{V}^p_t\big|^2 =\frac{\delta^2}{T} \Longrightarrow\tau(\epsilon, \mathfrak{c}) \leq t.
	\]
	So
\[
	\left\{ \sup_{t \in [0,T]} 	\frac{1}{N^2}\sum_{p\in I_n} \big|\tilde{V}^{m,p}_t - \tilde{V}^p_t\big|^2 \geq \delta^2/T \right\} 
	\subset \left\{  \tau(\epsilon, \mathfrak{c})  \leq T \right\},
	\]
	and we can conclude that
	\[
	Q^n\left(  \left\{ \sup_{t \in [0,T]} \frac{1}{N^2}\sum_{p\in I_n} \big|\tilde{V}^{m,p}_t - \tilde{V}^p_t\big|^2 > \delta^2/T \right\} \right) \leq \sum_{u=0}^{m-1} Q^n\left( \left\{  \tau(\epsilon, \mathfrak{c}) \in [u \eta_m, (u+1) \eta_m] \right\} \right).
	\]
	This commands that
	\begin{multline}\label{eq:bound with tau}
	\lsup{n}\frac{1}{N}\log Q^n\bigg( \sup_{t\in [0,T]}\frac{1}{N^2}\sum_{p\in I_n} \big|\tilde{V}^{m, p}_t - \tilde{V}^p_t\big|^2 >\delta^2 /T  \bigg)\\
	\leq \lsup{n} \frac{1}{N} \log \left( m \max_{u=0,\cdots,m-1} Q^n\left( \left\{  \tau(\epsilon, \mathfrak{c}) \in [u \eta_m, (u+1) \eta_m] \right\} \right) \right) \leq -M,
	\end{multline}
	by \eqref{eq: toshow tau hitting time}, so that we may conclude that \eqref{eq:final 0} holds.

It remains to prove \eqref{eq: toshow tau hitting time} which requires the four technical Lemma \ref{lem:alpha1} to \ref{lem:alpha4} below.\\
\textit{Proof of \eqref{eq: toshow tau hitting time}:}
Fix $\epsilon < \delta^2 \exp\big(-\mathfrak{c}T\big) / T$. We first establish that
\begin{equation}\label{eq:Qnbound preliminaryzero}
Q^n\bigg(\tau(\epsilon,\mathfrak{c}) \in \big[ u \eta_m,(u+1) \eta_m] \bigg) \leq Q^n\bigg(\bigcup_{j = 1}^3 (\mathfrak{B}^j)^c \bigcup_{v = 0}^{u} (\mathfrak{B}^4_v)^c\text{ and } \tau(\epsilon, \mathfrak{c}) \geq u \eta_m\bigg),
\end{equation}
for the following events 
\begin{align}
\mathfrak{B}^j =& \bigg\lbrace \sup_{s \in [0,T]} \alpha^j_{s^{(m)}} \leq \frac{\epsilon}{3 T^{2} C \sigma^2} \bigg\rbrace\label{eq: cond1},\, j =1, 2, 3\\
\mathfrak{B}^4_v =& \bigg\lbrace \alpha^4_{v \eta_m} \leq \frac{\epsilon \mathfrak{c}}{  T C \sigma^2} \exp \left (v \eta_m \mathfrak{c} \right)  \bigg\rbrace,\,v=0,\cdots,u,  \label{eq: cond4}
\end{align}
the constant $C$ being defined in \eqref{eq:C1}.
Taking the complements of the events, 
\eqref{eq:Qnbound preliminaryzero} is equivalent to
\[
Q^n\bigg(\bigcap_{j = 1}^3 \mathfrak{B}^j \bigcap_{v = 0}^{u} \mathfrak{B}^4_v \text{ or } \tau(\epsilon, \mathfrak{c}) < u \eta_m \bigg)     \leq Q^n\bigg(\tau(\epsilon,\mathfrak{c}) \notin \big[ u \eta_m,(u+1) \eta_m) \bigg).
\]
Now, using the equality $\mathbb{P}(A \cup B) = \mathbb{P}(A \cap B^c) + \mathbb{P}(B)$,
\begin{multline}
Q^n\bigg(\bigcap_{j = 1}^3 \mathfrak{B}^j \bigcap_{v = 0}^{u} \mathfrak{B}^4_v \text{ or } \tau(\epsilon, \mathfrak{c}) < u \eta_m \bigg) = \\
Q^n\bigg(\bigcap_{j = 1}^3 \mathfrak{B}^j \bigcap_{v = 0}^{u} \mathfrak{B}^4_v \text{ and } \tau(\epsilon, \mathfrak{c}) \geq u \eta_m \bigg)+Q^n\bigg( \tau(\epsilon, \mathfrak{c}) < u \eta_m \bigg), 
\end{multline}
and
\[
 Q^n\bigg(\tau(\epsilon,\mathfrak{c}) \notin \big[ u \eta_m,(u+1) \eta_m) \bigg)=Q^n\bigg( \tau(\epsilon, \mathfrak{c}) < u \eta_m \bigg)+Q^n\bigg( \tau(\epsilon, \mathfrak{c}) \geq (u+1) \eta_m \bigg).
\]
It therefore suffices for us to prove that
\begin{equation}\label{eq:Qnbound preliminary}
Q^n\bigg(\bigcap_{j = 1}^3 \mathfrak{B}^j \bigcap_{v = 0}^{u} \mathfrak{B}^4_v \text{ and } \tau(\epsilon, \mathfrak{c}) \geq u \eta_m \bigg) \leq Q^n\bigg( \tau(\epsilon, \mathfrak{c}) \geq (u+1) \eta_m \bigg).
\end{equation}
Indeed, if the above  conditions $\lbrace \mathfrak{B}^j\rbrace, j=1,2,3$  it follows from \eqref{eq:upper} and \eqref{eq: cond1}, that for $t \in [u \eta_m, (u+1) \eta_m]$, i.e. for $t^{(m)} = u \eta_m$,
\begin{align}\nonumber
 \frac{1}{N^2}\sum_{p\in I_n} \big| \tilde{V}^{m, p}_t - \tilde{V}^p_t \big|^2 &\leq TC \sigma^2 \left( \int_0^t \sup_{r \in [0, s]} \alpha^1_r \, ds + \sum_{j=2}^3 \int_0^t \sup_{r \in [0, s]} \alpha^j_{r^{(m)}} \, ds+
  \int_0^t \sup_{r \in [0, s]} \alpha^4_{r^{(m)}} \, ds \right)\\
& \leq   \epsilon + TC \sigma^2 \int_0^t \sup_{r \in [0, s]} \alpha^4_{r^{(m)}} \, ds .\label{eq:bounding}
\end{align}
Because the conditions \eqref{eq: cond4}, $\lbrace \mathfrak{B}^4_v \rbrace$, $v=0,\cdots, u$, are all satisfied we can write
\begin{multline*}
\int_0^t \sup_{r \in [0, s]} \alpha^4_{r^{(m)}} \, ds = \sum_{v=0}^{u-1} \int_{v \eta_m}^{(v+1) \eta_m} \sup_{r \in [0, s]} \alpha^4_{r^{(m)}} \, ds + \int_{u \eta_m}^t \sup_{r \in [0, s]} \alpha^4_{r^{(m)}} \, ds = \\
\eta_m \sum_{v=0}^{u-1}  \sup_{r \in [0, v \eta_m]} \alpha^4_{r^{(m)}} + \int_{u \eta_m}^t \sup_{r \in [0, s]} \alpha^4_{r^{(m)}} \, ds \leq 
\frac{\epsilon \mathfrak{c} \eta_m}{TC \sigma^2}  \sum_{v=0}^{u-1}  \exp{\mathfrak{c} v \eta_m} + \int_{u \eta_m}^t \sup_{r \in [0, s]} \alpha^4_{r^{(m)}} \, ds = \\
\frac{\epsilon \mathfrak{c} \eta_m}{TC \sigma^2}  \sum_{v=0}^{u-1} \exp{\mathfrak{c} v \eta_m} + \int_{u \eta_m}^t \sup_{r \in [0, u \eta_m]} \alpha^4_{r^{(m)}} \, ds \leq 
\frac{\epsilon \mathfrak{c} \eta_m}{TC \sigma^2} \sum_{v=0}^{u-1}  \exp{\mathfrak{c} v \eta_m} + (t-u \eta_m) \frac{\epsilon \mathfrak{c}}{TC \sigma^2}  \exp{\mathfrak{c} u \eta_m} \leq \\
\frac{\epsilon \mathfrak{c} \eta_m}{TC \sigma^2}   \sum_{v=0}^{u}  \exp{\mathfrak{c} v \eta_m} =
 \frac{\epsilon \mathfrak{c} \eta_m}{TC \sigma^2} \frac{ \exp{\mathfrak{c} (u+1) \eta_m} - 1 }{\exp{\mathfrak{c} \eta_m}-1}.
\end{multline*}
Since $x \leq \exp{x}-1$ for $x \geq 0$, it follows that
\[
\int_0^t \sup_{r \in [0, s]} \alpha^4_{r^{(m)}} \, ds 
\leq \frac{\epsilon}{TC\sigma^2} \left( \exp{\mathfrak{c} (u+1) \eta_m}-1 \right),
\]
and, because of \eqref{eq:bounding},
 \begin{equation} \label{eq:taugttup1}
  \frac{1}{N^2} \sum_{p\in I_n} \big| \tilde{V}^{m, p}_t - \tilde{V}^p_t \big|^2  \leq   \epsilon \exp \mathfrak{c} (u+1) \eta_m.
\end{equation}
for $t \in [u \eta_m, (u+1) \eta_m]$.

This means that if conditions \eqref{eq: cond1}-\eqref{eq: cond4} are satisfied, and $\tau(\epsilon, \mathfrak{c}) \geq u \eta_m$, then $\tau(\epsilon,\mathfrak{c}) \geq (u+1) \eta_m$, and we have established \eqref{eq:Qnbound preliminary}.

Now
\begin{multline}\label{eq:Qnbound preliminary2}
Q^n\bigg(\bigcup_{j = 1}^3 (\mathfrak{B}^j)^c \bigcup_{v = 0}^{u} (\mathfrak{B}^4_v)^c\text{ and } \tau(\epsilon, \mathfrak{c}) \geq u \eta_m\bigg) \leq \\ 
\sum_{j=1}^3 Q^n\bigg( (\mathfrak{B}^j)^c\bigg)+\sum_{v=0}^uQ^n\bigg( (\mathfrak{B}^4_v)^c \text{ and } \tau(\epsilon, \mathfrak{c}) \geq u \eta_m\bigg).
\end{multline}
We use the following four Lemmas
 \begin{lemma}\label{lem:alpha1}
For any  $M > 0$, for all $m\in \N$ sufficiently large,
\[
\lsup{n}\frac{1}{N} \log Q^n\bigg(\sup_{s\in [0,T]}\alpha^1_{s} \geq  \frac{\epsilon}{3 T C \sigma^2}  \bigg) \leq -M.
\]
\end{lemma} 
\begin{lemma}\label{lem:alpha2}
For any $M > 0$, for all $m\in \N$ sufficiently large,
\begin{equation}
\lsup{n}\frac{1}{N} \log Q^n\bigg(\sup_{s\in [0,T]}\alpha^2_{s^{(m)}} \geq  \frac{\epsilon}{3 T C \sigma^2}  \bigg) \leq -M,
\end{equation}
if the function $\psi(n,q_m): \N \to \R^+$ defined in the proof is such that $\lim_{n,m \to \infty} N m \psi(n,q_m) = 0$.
\end{lemma}
\begin{lemma}\label{lem:alpha3}
For any $M > 0$, for all $m\in \N$ sufficiently large,
\[
\lsup{n}\frac{1}{N} \log Q^n\bigg(\sup_{s\in [0,T]}\alpha^3_{s^{(m)}} \geq  \frac{\epsilon}{3 T C \sigma^2}  \bigg) \leq -M.
\]
\end{lemma}
\begin{lemma}\label{lem:alpha4}
For any $M > 0$, there exists a constant $\mathfrak{c}$ such that for all $m \in \N$ sufficiently large, all $0 \leq u \leq m$ and all $0 \leq v \leq u$ and all $\epsilon \leq \exp\big(-\mathfrak{c}T\big)\delta^2 / T$,
\[
\lsup{n}\frac{1}{N} \log Q^n\bigg(\alpha^4_{v\eta_m} \geq \frac{\epsilon \mathfrak{c}}{T C \sigma^2} \exp \left (v \eta_m \mathfrak{c} \right)  \text{ and } \tau(\epsilon,\mathfrak{c}) \geq u \eta_m\bigg) \leq -M.
\]
\end{lemma}
It follows from Lemmas~\ref{lem:alpha1} to \ref{lem:alpha3} that
\begin{equation*}
\lsup{n}\frac{1}{N}\log Q^n\big((\mathfrak{B}^j)^c \big) \leq -M,\,j = 1, 2, 3
\end{equation*}
and from Lemma~\ref{lem:alpha4} that
\begin{equation*}
\lsup{n}\frac{1}{N}\log Q^n\big((\mathfrak{B}^4_v)^c\text{ and } \tau(\epsilon, \mathfrak{c}) \geq u \eta_m\big) \leq -M,
\end{equation*}
for all $0\leq v \leq u$, for $m$ sufficiently large.  This means that
\begin{multline*}
\lsup{n}\frac{1}{N} \log Q^n\bigg(\bigcup_{j = 1}^{3} (\mathfrak{B}^j)^c \bigcup_{v=0}^u (\mathfrak{B}^4_v)^c\text{ and } \tau(\epsilon, \mathfrak{c}) \geq u \eta_m\bigg)  \\
 \leq\lsup{n}\frac{1}{N} \log \left( \sum_{j=1}^{3} Q^n\bigg( (\mathfrak{B}^j)^c\bigg)+\sum_{v=0}^u Q^n\bigg( (\mathfrak{B}^4_v)^c \text{ and } \tau(\epsilon, \mathfrak{c}) \geq u \eta_m\bigg) \right)
   \\
 \leq \lsup{n}\frac{1}{N} \log (u+4) \max_{j,v} \left\{ Q^n\bigg( (\mathfrak{B}^j)^c\bigg),\,Q^n\bigg( (\mathfrak{B}^4_v)^c \text{ and } \tau(\epsilon, \mathfrak{c}) \geq u \eta_m\bigg)\right\}  \\
 = \lsup{n} \max_{j,v} \left\{ \frac{1}{N} \log Q^n\bigg( (\mathfrak{B}^j)^c\bigg), \, \frac{1}{N} \log Q^n\bigg( (\mathfrak{B}^4_v)^c \text{ and } \tau(\epsilon, \mathfrak{c}) \geq u \eta_m\bigg)\right\} \leq
 -M.
\end{multline*}
We can therefore conclude \eqref{eq: toshow tau hitting time}, and this finishes the proof of Lemma \ref{Lemma Exp Equivalent 1}.
\end{proof}
%%%%%%%%
%%%%%%%%
\begin{proof}[Proof of Lemma~\ref{lem:useful ineq}]
\label{prooflemusefulineq} By \eqref{eq: L squared distance} we write
\[
D_{T, L^2}(\emp(X_n),\emp(Y_n))\leq  \sum_{i \in \Z} b_i \int \norm{f(u^i)-f(v^i)}_{L^2}  \,d\xi(u,v),
\]
for all stationary couplings $\xi$ between $\emp(X_n)$ and $\emp(Y_n)$. Because of the stationarity of $\xi$ and the Lipschitz continuity of $f$ we have
\begin{multline*}
D_{T, L^2}(\emp(X_n),\emp(Y_n)) \leq  b \int \norm{f(u^0)-f(v^0)}_{L^2}  \,d\xi(u,v) \leq \\
b \int \norm{u^0-v^0}_{L^2}  \,d\xi(u,v) \leq b \left( \int \norm{u^0-v^0}_{L^2}^2 \,d\xi(u,v) \right)^{1/2},
\end{multline*}
where $b$ is defined by \eqref{eq:ab}.

Consider the set $\mathcal{S}_n$ of permutations $s$ of the set $I_n$. If $X_n=(X^{-n},\cdots,X^n)$, we note $s(X_n)$ the element $(X^{s(-n)},\cdots,X^{s(n)})$.  The knowledge of $\emp(X_n)$ does not imply that of $X_n$, in effect it implies the knowledge of all $s(X_n)$s without knowing which permutation is the correct one. Choose one such element, say $s_0(X_n)$. Similarly choose $s_1(Y_n)$. There exists a family of couplings\footnote{For example $\xi^s(u,v)=\frac{1}{N} \sum_{i \in I_n} \delta_{S^i s_0(X_n)}(u)\delta_{S^i s(s_1(Y_n))}(v)$.} $\xi^s$ such that
\[
\int \norm{u^0-v^0}_{L^2}^2 \,d\xi^s(u,v)=\frac{1}{N} \sum_{k \in I_n} \norm{X^{s_0(k)}-Y^{s(s_1(k))}}_{L^2}^2,
\]
from which we obtain, for $s=s_0s_1^{-1}$
\[
D_{T, L^2}(\emp(X_n),\emp(Y_n))^2 \leq \frac{b^2}{N}\sum_{k \in I_n} \norm{X^k-Y^k}_{L^2}^2,
\]
which is the announced result.
\end{proof}
The proofs of Lemma \ref{lem:alpha1}-\ref{lem:alpha4} are found in Appendix \ref{app:proofofalphas}.

\subsection{Characterization of the Limiting Process}\label{Section Limiting Process}
We prove in this Section that the limit equations are given by \eqref{eq: limit equations}, i.e. Theorem \ref{Theorem:Main}{.iv}. This is achieved by first showing that the solution to  \eqref{eq: limit equations}, without the condition that $\mu_*$ is the law of $Z$, is unique and has a closed form expression as a function of the Brownian motions $W^j$. This is the content of the following Lemma whose proof can be found in Appendix \ref{app:stochvolterra}. This proof is based on an adaptation of the theory of Volterra equations of the second type \cite{tricomi:57} to our, stochastic, framework.
\begin{lemma}\label{lem:limit equations}
Let $\mu \in \mP_S(\T^\Z)$. The system of equations \eqref{eq: limit equations} 
\begin{align*}
V^j_t &=\sigma W^j_t + \sigma \int_0^t \theta^j_s ds \\
\theta^j_t &= \sigma^{-2}\sum_{i\in \Z} \int_0^t L_{\mu}^{i-j}(t,s)dV^{i}_s\ .\nonumber 
\end{align*}
has a unique solution given by
\begin{multline}\label{eq:Msol}
V_t^j = \sigma W_t^j + \sum_{i \in \Z}  \int_0^t 
\left( \int_0^s L^{ i}_{\mu}(s, u) \, dW_u^{i+j} \right) \, ds+\\
 \sigma^{-1} \sum_{i, \ell \in \Z}\int_0^t \left( \int_0^s M_{\mu}^i(s, u) \left(  \int_0^u L^{\ell-i}_{\mu}(u, v) \, dW_v^{\ell+j} \right) \, du \right) \, ds,
\end{multline}
where $M^k_{\mu}$ is defined in the proof and satisfies
\[
\sup_{s,u \in [0,t]} \sum_k \left| M^k_{\mu}(s,u) \right| < \infty.
\]
\end{lemma}

Note $Q^{m, n}$ the law of the solution to \eqref{eq:SD}. Lemma \ref{Lemma Exp Equivalent 1} indicates that $\Pi^{m,n} = Q^{m,n} \circ \emp(V^m_n)$ satisfies an LDP with the same good rate function $H$ as $\Pi^n$.
\begin{lemma}\label{lem:limit law of Qmn}
The limit law of $Q^{m, n}$ when $m,\,n \to \infty$ is $\mu_*$, the unique zero of the rate function $H$. Moreover, for all $k \in I_n$, $t,\,s \in [0,\,T]$
\[
\lim_{m,\, n \to \infty} \int_{\T^N} L_{\emp(V^m_n)}^k(t,\,s) \, dQ^{m, n}(V_n^m) = L_{\mu_*}^k(t,\,s).
\]
\end{lemma}
\begin{proof}
We know that $H$ has a unique zero, noted $\mu_*$. This implies that $\Pi^{m, n}$ converges weakly to $\delta_{\mu_*}$ and therefore, for all $F \in C_b(\mP(\T^\Z))$,
\[
\lim_{m,\,n \to \infty} \int_{\mP(\T^\Z)} F(\mu) \, d\Pi^{m, n}(\mu) =  F(\mu_*).
\] 
From the relation $\Pi^{m, n} = Q^{m, n} \circ \emp(V^m_n)^{-1}$ we infer that
\[
\lim_{m, \,n \to \infty} \int_{\T^N} F(\emp(V^m_n)) \, dQ_{m, n}(V^m_n) = F(\mu_*).
\]
Let us choose a function $f \in C_b(\T^\Z)$ and define $F : \mP(\T^\Z) \to \R$ by
\[
F(\mu) = \int_{\T^\Z} f(V) \, d\mu(V),
\]
so that we have
\begin{multline*}
\lim_{m, \, n \to \infty} \int_{\T^N} \int_{\T^\Z} f(V) \, d\emp(V^m_n)(V) \, dQ^{m, n}(V^m_n) =\\
\lim_{m, \, n \to \infty}\frac{1}{N} \sum_{i \in I_n} \int_{\T^N} f(S^i V^m_n) dQ^{m, n}(V^m_n) =
\int_{\T^\Z} f(V) \, d\mu_*(V)
\end{multline*}
We note that $Q^{m, n}$ is invariant under a uniform shift of the indexes, i.e. satisfies
\[
Q^{m, n} \circ S^i = Q^{m, n}
\]
for all $i \in I_n$, so that
\[
\frac{1}{N} \sum_{i \in I_n} \int_{\T^N} f(S^i V^m_n) dQ^{m, n}(V^m_n) = \int_{\T^N} f(V^m_n) \, dQ^{m, n}(V^m_n),
\]
and therefore
\[
\lim_{m, \, n \to \infty} \int_{\T^N} f(V^m_n) \, dQ^{m, n}(V^m_n) = \int_{\T^\Z} f(V) \, d\mu_*(V).
\]
Since this is true for all $f \in C_b(\T^\Z)$ we have proved that the limiting law of $Q^{m, n}$ is $\mu_*$.

Next consider the function $F: \mP(\T^\Z) \to \R$
\[
F(\mu) = L_\mu^{k}(t, s)
\]
for a given $k \in I_n$ and $t,\,s \in [0, T]$.  We have
\[
\lim_{m, \, n \to \infty} \int_{\T^N} L_{\emp(V^m_n)}^{ k} (t, s) \, dQ^{m, n}(V^m_n) = L_{\mu_*}^{ k} (t, s),
\]
which also reads
\[
\lim_{m, \, n \to \infty} \Exp \left[   L_{\emp(V^m_n)}^{ k} (t, s) -  L_{\mu_*}^{ k} (t, s)   \right] =0.
\]
\end{proof}
We now prove Theorem \ref{Theorem:Main}{.iii}
\begin{theorem}
The equations describing the unique 0, $\mu_*$, of the rate function $H$ are \eqref{eq: limit equations}.
\end{theorem}
\begin{proof}
We prove that for all $n \geq 0$
\[
\lim_{m,n \to \infty} \Exp \left[ \sup_{s \in [0,t]} \left| \theta^j_s -\theta^{m,j}_s  \right|^2 \right] = 0.
\]
Indeed, as shown below, this is sufficient to prove that
\[
\lim_{m,n \to \infty} \Exp \left[ \sup_{s \in [0,t]} \left| V^j_s -V^{m,j}_s  \right| \right] = 0.
\]
We recall that the equations \eqref{eq:vmtthetamt}  satisfied by $V^m$ are, for $j \in I_n$, 
\begin{align*}
V^{m,j}_t &= \sigma W^j_t + \sigma \int_0^t \theta^{m,j}_s ds \\
\theta^{m,j}_t &= \sigma^{-2}\sum_{i\in I_{q_m}} \Exp^{\bar{\gamma}^{\emp(V^m_n)}_t}\bigg[G^{m,0}_{t^{(m)}}\int_0^{t^{(m)}}  G^{m,i}_{s^{(m)}} dV^{m, i+j}_s \bigg]  \\
& = \sigma^{-2}\sum_{i\in I_{q_m}} \int_0^{t^{(m)}} L_{\emp(V^m_n)}^{i}(t^{(m)}, s^{(m)}) \, dV^{m, i+j}_s.
\end{align*}
We also have, for $j \in \Z$. 
\[
\theta^j_t = \sigma^{-2}\sum_{i\in \Z} \int_0^t L^{i}_{\mu_*}(t, s) dV^{i+j}_s.
\]
Write 
%\begin{multline*}
%\theta^j_t - \theta^{m,j}_t = 
%\sigma^{-2} \underbrace{\sum_{i \in I_{q_m}} \int_0^t ( L_{\mu_*}^{i}(t,s) - L_{\mu_*}^{i}(t^{(m)},s^{(m)}) ) \, dV^{i+j}_s }_{\alpha^{j, 1}_t} +\\
%\sigma^{-2} \underbrace{\sum_{i \in I_{q_m}}   \int_{t^{(m)}}^t L_{\mu_*}^{i}(t^{(m)},s^{(m)})  \, dV^{i+j}_s }_{\alpha^{j, 2}_t}+\\
%\sigma^{-2} \underbrace{\sum_{i \in I_{q_m}}  \int_0^{t^{(m)}} \left( L_{\mu_*}^{ i}(t^{(m)},s^{(m)})  - L_{\emp(V^m_n)}^{ i}(t^{(m)}, s^{(m)}) \right) \, dV^{i+j}_s}_{\alpha^{j, 3}_t} +\\
%\sigma^{-2} \underbrace{\sum_{i \in \Z/I_{q_m}} \int_0^t  L_{\mu_*}^{i}(t,s) \, dV^{i+j}_s }_{\alpha^{j, 4}_t}\\
%\sigma^{-2} \sum_{i \in I_{q_m}}   \int_0^{t^{(m)}} L_{\emp(V^m_n)}^{i}(t^{(m)}, s^{(m)}) (\theta^{i+j}_s - \theta^{m, i+j}_s)\,ds,
%\end{multline*}
\begin{align*}
\theta^j_t - \theta^{m,j}_t  =&  
\sigma^{-2} \sum_{i \in I_{q_m}} \int_0^t ( L_{\mu_*}^{i}(t,s) - L_{\mu_*}^{i}(t^{(m)},s^{(m)}) ) \, dV^{i+j}_s &\Big\} \alpha^{j, 1}_t\\
%\end{align*}
& + \sigma^{-2} \sum_{i \in I_{q_m}}   \int_{t^{(m)}}^t L_{\mu_*}^{i}(t^{(m)},s^{(m)})  \, dV^{i+j}_s & \Big\}\alpha^{j, 2}_t\\
&+ \sigma^{-2} \sum_{i \in I_{q_m}}  \int_0^{t^{(m)}} \left( L_{\mu_*}^{ i}(t^{(m)},s^{(m)})  - L_{\emp(V^m_n)}^{ i}(t^{(m)}, s^{(m)}) \right) \, dV^{i+j}_s & \Big\}\alpha^{j, 3}_t \\
&+\sigma^{-2} \sum_{i \in \Z/I_{q_m}} \int_0^t  L_{\mu_*}^{i}(t,s) \, dV^{i+j}_s & \Big\}\alpha^{j, 4}_t\\
& +\sigma^{-2} \sum_{i \in I_{q_m}}   \int_0^{t^{(m)}} L_{\emp(V^m_n)}^{i}(t^{(m)}, s^{(m)}) (\theta^{i+j}_s - \theta^{m, i+j}_s)\,ds,& 
\end{align*}

so that we have
\[
\theta^j_t - \theta^{m,j}_t = \sum_{k=1}^{4} \alpha^{j,k}_t+\sigma^{-2} \sum_{i \in I_{q_m}}   \int_0^{t^{(m)}} L_{\emp(V^m_n)}^{i}(t^{(m)}, s^{(m)}) (\theta^{i+j}_s - \theta^{m, i+j}_s)\,ds.
\]
To simplify notations further we write $L_n^i(t^{(m)}, s^{(m)})$ for $L_{\emp(V^m_n)}^{ i}(t^{(m)}, s^{(m)})$ since there is no ambiguity, and define
\[
\Phi^j_t := \theta^j_t - \theta^{m,j}_t \quad j \in I_n \quad t \in [0,T].
\]
The previous equation writes
\begin{equation}\label{eq:volterratheta}
\Phi^j_t = \sum_{k=1}^{4} \alpha^{j,k}_t+\sigma^{-2} \sum_{i \in I_{q_m}}   \int_0^{t^{(m)}} L_n^i(t^{(m)}, s^{(m)})\Phi^{i+j}_s \,ds
\end{equation}
This is a Volterra equation of the second type \cite{tricomi:57}. We solve it for $\Phi$ as a function of the $\alpha$s and use
the following  Lemma whose proof can be found in Appendix \ref{app:alphajs}.
\begin{lemma}\label{lem:alphaj}
For all $\varepsilon > 0$, there exists $m_0(\varepsilon)$ in $\N$ such that for all $m \geq m_0$
\[
\Exp \left[ \max_{k=1,2,3,4} \sup_{s \in [0,t]} \left| \alpha_s^{j,k} \right|^2  \right] \leq C \varepsilon 
\]
for some positive constant $C$ independent of $j$.
\end{lemma}

%%%%
%%%%

%%%%
%%%%

Since equation \eqref{eq:volterratheta} is affine we solve it for each $\alpha^{k,j}$, $k=1,2,3,4$ and add the four solutions. In what follows we thus drop the $k$ index and solve
\[
\Phi^j_t = \alpha^{j}_t+\sigma^{-2} \sum_{k \in I_{q_m}}   \int_0^{t^{(m)}} L_n^k(t^{(m)}, s^{(m)})\Phi^{k+j}_s \,ds.
\]
We take continuous Fourier transforms of both sides to obtain
\[
\tilde{\Phi}_t(\varphi) = \tilde{\alpha}_t(\varphi)+\sigma^{-2}  \int_0^{t^{(m)}} \tilde{L}_n^{*}(\varphi)(t^{(m)}, s^{(m)})\tilde{\Phi}_s(\varphi)\,ds,
\]
where $\ ^*$ indicates complex conjugate and, for example
\[
\tilde{\Phi}_t(\varphi)=\sum_{j \in I_n} \Phi^j_t e^{-ij\varphi},\ \varphi \in [-\pi, \pi[,
\]
and, as explained page~\pageref{page:explicationInIQm}, the Fourier transform of \(L^j\) is given by:
\[
\tilde{L}_n(\varphi)(t^{(m)}, s^{(m)}) = \sum_{j \in I_n} \mathbbm{1}_{I_{q_m}}(j)L_n^j(t^{(m)}, s^{(m)}) e^{-ij\varphi},\ \varphi \in [-\pi, \pi[.
\] 
We use standard results on Volterra equations \cite{tricomi:57} to write
\begin{equation}\label{eq:voltsolved}
\tilde{\Phi}_t(\varphi) = \tilde{\alpha}_t(\varphi)+\lambda \int_0^{t^{(m)}} \tilde{H}(\varphi)(t^{(m)}, s^{(m)},\lambda) \tilde{\alpha}_s(\varphi)\,ds,
\end{equation}
where we have noted $\lambda = \sigma^{-2}$, the ``resolvent kernel'' $\tilde{H}(\varphi)(t,s,\lambda)$ is given by the series of iterated kernels
\begin{equation}\label{eq:series}
\tilde{H}(\varphi)(t^{(m)}, s^{(m)},\lambda)=\sum_{\ell=0}^\infty \lambda^\ell \tilde{L}_{n,\ell+1}^{*}(\varphi)(t^{(m)}, s^{(m)}),
\end{equation}
and
\[
\tilde{L}_{n,\ell+1}^{*}(\varphi)(t^{(m)}, s^{(m)}) = \int_0^{t^{(m)}} \tilde{L}_{n}^{*}(\varphi)(t^{(m)}, u^{(m)})\tilde{L}_{n,\ell}^{*}(\varphi)(u^{(m)}, s^{(m)})\,du.
\]
The convergence of the series \eqref{eq:series} is guaranteed by the fact that the two functions
\[
A_n(\varphi,t)^2 = \int_0^T \left| \tilde{L}_n(\varphi)(t,s) \right|^2 \,ds \text{  and  } B_n(\varphi,s)^2 = \int_0^T \left| \tilde{L}_n(\varphi)(t,s) \right|^2 \,dt
\]
are upperbounded by $T^2a^2b^2$ {\em independently} of $n$, thanks to Proposition~\ref{prop:Lkmuregular}.
The theory of Volterra equations then guarantees that
\[
\tilde{H}(\varphi)(t^{(m)}, s^{(m)},\lambda) \leq C
\]
for some positive constant $C$ independent of $n,\,m$.

Equation \eqref{eq:voltsolved} then implies that
\[
\left| \tilde{\Phi}_t(\varphi) \right|^2 \leq 2\left|   \tilde{\alpha}_t(\varphi) \right|^2 +2 \lambda^2 C^2 \int_0^t  \left|   \tilde{\alpha}_s(\varphi) \right|^2\,ds.
\]
By Parseval's Theorem
\[
\sum_{j \in I_n} \left| \Phi^j_t \right|^2 \leq 2 \sum_{j \in I_n} \left| \alpha^j_t \right|^2 + 2 \lambda^2 C^2 \int_0^t \sum_{j \in I_n} \left| \alpha^j_s \right|^2 \,ds.
\]
Taking the expected value of both sides and using the spatial stationarity of $(\Phi^j_t)_{j \in I_n}$ and $(\alpha^j_t)_{j \in I_n}$ we have for any $j \in I_n$
\[
\Exp\left[ \left| \Phi^j_t \right|^2 \right] \leq 2 \Exp\left[ \left|  \alpha^j_t \right|^2 \right] + 2 \lambda^2 C^2 \int_0^t \Exp\left[ \left| \alpha^j_s \right|^2 \right]\,ds.
\]
Since by Lemma \ref{lem:alphaj} 
 \[
\Exp \left[ \max_{k=1,2,3,4} \sup_{s \in [0,t]} \left| \alpha_s^{j,k} \right|^2  \right] \rightarrow 0,
\] 
we conclude that
\[
\lim_{m,n \to \infty} \Exp \left[ \sup_{s \in [0,t]} \left| \theta^j_s - \theta^{m,j}_s \right|^2  \right] = 0,
\]
and therefore
\[
 \sup_{s \in [0,t]} \left| V^j_s - V^{m,j}_s \right|  \leq    \int_0^t \sup_{\rho \in [0,s]} \left|  \theta^j_\rho - \theta^{m,j}_\rho \right|  \,ds   \leq \sqrt{t}\ \left( \int_0^t \sup_{\rho \in [0,s]} \left( \theta^j_\rho - \theta^{m,j}_\rho \right)^2\,ds \right)^{1/2} ,
\]
and by Cauchy-Schwarz again
\begin{align*}
\Exp \left[  \sup_{s \in [0,t]} \left| V^j_s - V^{m,j}_s \right| \right]  \leq &   \sqrt{t}\ \left(  \Exp \left[  \int_0^t \sup_{\rho \in [0,s]} \left( \theta^j_\rho - \theta^{m,j}_\rho \right)^2\,ds \right] \right)^{1/2} \\& = \sqrt{t}\ \left(   \int_0^t \Exp \left[ \sup_{\rho \in [0,s]} \left( \theta^j_\rho - \theta^{m,j}_\rho \right)^2 \right]\,ds  \right)^{1/2}.
\end{align*}
We conclude that
\[
\lim_{m,n \to \infty} \Exp \left[ \sup_{s \in [0,t]} \left| V^j_s - V^{m,j}_s \right| \right] = 0
\]
for all $j \in \Z$ and $t \in [0,T]$.
\end{proof}

\section*{Acknowledgements}
This research has received funding from the Europe
an Union's Horizon 2020 Framework Programme
for Research and Innovation under the Specific Grant Agreement No. 785907 (Human Brain Project SGA2).
 \appendix

\section{A martingale Expectation Inequality}
We recall the result used in \cite{ben-arous-guionnet:95}:
\begin{lemma}\label{lem:BG}
Consider \(B^{-n}, \cdots, B^{n}\) \(N\) independent Brownian motions and 
\(h^{-n}, \cdots, h^n\) \(N\) previsible processes 
such that $N^{-1}\sum_{j\in I_n}\big(h^j_t\big)^2 \leq 1$.
Then, for all $\varepsilon < 1/ (2\sqrt{T})$, we have
\[
\Exp\bigg[\exp\bigg\lbrace\frac{\varepsilon^2}{2N}\sum_{i,j\in I_n}\bigg(\int_0^T h^i_t dB^j_t \bigg)^2 \bigg\rbrace \bigg] \leq  (1-4\varepsilon^2 T)^{-N/4}.
\]
\end{lemma}
\begin{proof}
Define $\alpha:=\frac{\varepsilon^2}{2N}$, $X^{ij}_t=\int_0^t h^i_s\,dB^j_s$, $S_t:=\sum_{i \in I_n}(h^i_t)^2$, and $Y_t:=\sum_{i,j \in I_n} (X^{ij}_t)^2$. Using It\^o's rule we obtain
\[
Y_t=2 \sum_{i,j \in I_n} \int_0^t h^i_s\left( \int_0^s h^i_u \,dB^j_u \right)\,dB^j_s+N\int_0^t S_u\,du.
\]
Define the martingale
\[
Z_t:=\sum_{i,j \in I_n} \int_0^t h^i_s\left( \int_0^s h^i_u \,dB^j_u \right)\,dB^j_s.
\]
Using the fact that $\langle B^j,\,B^l \rangle_t=\delta_{jl} t$, we have
\[
\langle Z,\,Z \rangle_t= \sum_{i,j,k \in I_n} \int_0^t h^i_s h^k_s \left( \int_0^s h^i_u\,dB^j_u \right) \left( \int_0^s h^k_u\,dB^j_u \right) \,ds.
\]
Apply Cauchy-Schwarz to obtain
\[
\sum_{i \in I_n} \left|   h^i_s\left( \int_0^s h^i_u \,dB^j_u \right) \right| \leq \left( \sum_{i \in I_n}(h^i_s)^2 \right)^{1/2}\,\left( \sum_{i \in I_n}\left( \int_0^s h^i_u\,dB^j_u \right)^2 \right)^{1/2},
\]
from which it follows that
\[ 
\langle Z,\,Z \rangle_t=| \langle Z,\,Z \rangle_t | \leq \int_0^t S_u\,Y_u\,du.
\]
Now we have
\[
e^{\alpha Y_t}=e^{2\alpha Z_t+\alpha N \int_0^t S_s\,ds}=e^{2\alpha Z_t - 4\alpha^2 \langle Z,\,Z \rangle_t} \times e^{4\alpha^2 \langle Z,\,Z \rangle_t+\alpha N\int_0^t S_u\,du}.
\]
Apply Cauchy-Schwarz again to obtain
\[
\Exp\left[ e^{\alpha Y_t} \right]^2 \leq \Exp \left[ e^{4\alpha Z_t - 8\alpha^2 \langle Z,\,Z \rangle_t} \right] \times \Exp \left[ e^{8\alpha^2 \int_0^t S_u\,Y_u\,du+2\alpha N\int_0^t S_u\,du} \right].
\]
By supermartingale properties, the first expected value in the right hand side of the previous inequality is bounded by \(1\), hence
\[
\Exp\left[ e^{\alpha Y_t} \right]^2 \leq \Exp \left[ e^{8\alpha^2 \int_0^t S_u\,Y_u\,du+2\alpha N\int_0^t S_u\,du} \right].
\]
Now use the fact that $S_u \leq N$ uniformly in $u$ to conclude
\[
\Exp\left[ e^{\alpha Y_t} \right]^2 \leq e^{2\alpha N^2  t} \, \Exp\left[ e^{8 \alpha^2 N \int_0^t Y_u \,du } \right]= e^{\varepsilon^2 N t} \, \Exp\left[ e^{4 \varepsilon^2 \alpha \int_0^t Y_u \,du } \right],
\]
then use Jensen's inequality to obtain
\[
\Exp\left[ e^{\alpha Y_t} \right]^2 \leq e^{\varepsilon^2 N t} \, \frac{1}{t}\int_0^t \Exp\left[ e^{4 \varepsilon^2 t \alpha  Y_u } \right]\,du.
\]
If $4 \varepsilon^2 t  < 1$ we can use again Jensen's inequality
\[
\Exp\left[ e^{\alpha Y_t} \right]^2 \leq e^{\varepsilon^2 N t} \, \frac{1}{t}\int_0^t \left( \Exp\left[ e^{ \alpha  Y_u } \right]\right)^{4 \varepsilon^2 t} \,du=\frac{e^{\varepsilon^2 N t}}{t} \, \int_0^t \left(\left( \Exp\left[ e^{ \alpha  Y_u } \right]\right)^{2}\right)^{2 \varepsilon^2 t} \,du.
\] 
Define $g(t):=\Exp\left[ e^{\alpha Y_t} \right]^2$, the above inequality reads
\[
g(t) \leq \frac{e^{\varepsilon^2 N t}}{t} \int_0^t (g(s))^{2 \varepsilon^2 t}\,ds.
\]
Since $4 \varepsilon^2 t  < 1$ implies $2 \varepsilon^2 t  < 1$ we can apply Bihari's Lemma \cite[Chap. 1, Th. 8.2]{mao:08} to obtain 
\[
\Exp \left[ e^{\alpha Y_t} \right] \leq \left( (1-2\varepsilon^2 t)e^{\varepsilon^2 t N} \right)^{\frac{1}{2(1-2\varepsilon^2 t)}} \leq e^{\frac{\varepsilon^2t}{2(1-2\varepsilon^2 t)}N},
\]
and, since $1-4\varepsilon^2t < 1-2\varepsilon^2t$, 
\[
\Exp \left[ e^{\alpha Y_t} \right] \leq e^{\frac{\varepsilon^2 t}{(1-4\varepsilon^2 t)}N}=e^{\frac{4\varepsilon^2 t}{(1-4\varepsilon^2 t)}N/4},
\]
and, since $-\frac{x}{1-x} > \log(1-x)$, $0 < x <1$
\begin{align*}
\Exp \left[ e^{\alpha Y_t} \right] & \leq e^{-\frac{N}{4} \log(1-4\varepsilon^2 t))}\\
\Exp \left[ e^{\alpha Y_T} \right] & \leq (1-4\varepsilon^2 T)^{-N/4}.
\end{align*}
\end{proof}
%%%%%%%
\section{Discrete Fourier Transforms (DFT) of Gaussian processes}
\label{app:DFTs}
Define $F_N:=e^{\frac{2\pi i}{N}}$. Let $a:=(a^j)_{j \in I_n}$ be an $N$-periodic complex sequence. Its DFT $\tilde{a}:=(\tilde{a}^p)_{p \in I_n}$ is defined by
\[
\tilde{a}^p = \sum_{j \in I_n} a^j F_N^{-jp},
\]
from which the original sequence can be recovered by the inverse DFT (IDFT)
\[
a^j = \frac{1}{N} \sum_{p \in I_n} \tilde{a}^p F_N^{jp}.
\]
We need two Lemmas about the DFT of $N$-periodic sequences defined on $I_n$. 
The first one is about the DFT of a translated sequence.
\begin{lemma}\label{lem:translation}
The DFT of the sequence $a_k:=(a^{j+k})_{j \in I_n}$, $k \in \Z$ is given by 
\[
DFT(a_k)^p=F_N^{kp}\, \tilde{a}^p
\]
\end{lemma}
\begin{proof}
The proof is left to the reader.
\end{proof}
The second Lemma is about the DFT of the convolution of two sequences. Let $(a^j)_{j \in I_n}$ and $(b^j)_{j \in I_n}$. We define their (circular or periodic) convolution as
\[
(a \star b)^j=\sum_{k \in I_n}a^k b^{j-k}=\sum_{k \in I_n}a^{j-k}b^k,
\]
where indexes are taken modulo $I_n$.
We have the Lemma.
\begin{lemma}\label{lem:convolution}
\[
DFT^{-1}(\tilde{a} \star \tilde{b})^j=N\, a^j b^j,
\]
and hence
\[
(\tilde{a} \star \tilde{b})^p=N \,DFT(ab)^p
\]
\end{lemma}
\begin{proof}
The proof is left to the reader.
\end{proof}
We derive some properties of the Fourier transforms of the synaptic weights 
$(J^{ij}_n)_{i,j \in I_n}$ and the Gaussian processes $G^j_t$. We define 
$(\tilde{R}_\J(p,l))_{p,l \in I_n}$ to be the length $N$ DFT w.r.t to the first 
index of the sequence $(R_\J(k,l)_{k,l \in I_n})$, that \footnote{
There is no conflict with the definition \eqref{eq:RJtilde} since they are always used in different contexts.
}
\[
\tilde{R}_\J(p,l)=\sum_{k \in I_n} R_\J(k,l) W_N^{-kp}.
\]
 We first characterize the 
joint laws of the synaptic weights under $\gamma$.

\begin{lemma}\label{lem:Fourier coeffs uncorrelated}
Define
\[
\tilde{J}_n^{pk}:=\sum_{j \in I_n} J^{jk}_n F_N^{-jp},
\]
to be the DFT of the synaptic weights $J^{jk}$ w.r.t the first index. Their covariance is
\[ 
\Exp^\gamma\left[\tilde{J}_n^{pk}\tilde{J}_n^{ql}\right]=\left\{
\begin{array}{lcl}
\tilde{R}_\J(p,k-l \mod I_n) & \text{if} & p+q=0\\
0 & \text{otherwise} &\  
\end{array}
\right.
\]
\end{lemma}
\begin{proof}
By \eqref{eq:cov} and the symmetry of $R_\J$
\begin{multline*}
\Exp^\gamma\left[\tilde{J}_n^{pk}\tilde{J}_n^{ql}\right]=\sum_{j,h \in I_n} \Exp^\gamma \left[J_n^{jk}J_n^{hl}\right]F_N^{-jp}F_N^{-hq}=\frac{1}{N} \sum_{j,h \in I_n} R_{\J}(h-j,l-k)F_N^{-jp}F_N^{-hq}=\\
\frac{1}{N} \sum_{j,h \in I_n} R_{\J}(j-h,k-l)F_N^{-jp}F_N^{-hq}.
\end{multline*}
By Lemma \ref{lem:translation} we have
\[
\sum_{j \in I_n} R_{\J}(j-h,k-l)F_N^{-jp}=\tilde{R}_\J(p,k-l)F_N^{-hp},
\]
and, since $\sum_{h \in I_n} F_N^{-h(p+q)} = N\delta_{p+q}$,
\[
\frac{1}{N} \sum_{j,h \in I_n} R_{\J}(h-j,l-k)F_N^{-jp}F_N^{-hq}=\left\{
\begin{array}{lcl}
\tilde{R}_\J(p,k-l) & \text{if} & p+q=0\\
0 & \text{otherwise} &\  
\end{array}
\right.,
\]
\end{proof}
\begin{remark}
In the terminology of complex Gaussian vectors to be found, e.g. in \cite{gallager:13}, Lemma \ref{lem:Fourier coeffs uncorrelated} states the following. Consider the $N$ centered complex $N$-dimensional Gaussian vectors $\tilde{J}_n^p=(\tilde{J}_n^{pk})_{k \in I_n}$, $p \in I_n$. Note that the complex conjugate $\tilde{J}_n^{p *}$ of $\tilde{J}_n^p$ is $\tilde{J}_n^{-p}$, $p \in I_n$. If $p \neq 0$ $\tilde{J}_n^p$ is such that its pseudo-covariance matrix $\Exp^\gamma \left[ \tilde{J}_n^p \,^t \tilde{J}_n^p  \right] =0$ and its covariance matrix $\Exp^\gamma \left[ \tilde{J}_n^p \,^t \tilde{J}_n^{-p}  \right]$ is equal to the circulant matrix
$C^p_n:=(R_\J(p,k-l))_{k,l \in I_n}$. If $p=0$ $\tilde{J}_n^0$ is in effect real and its covariance and pseudo-covariance matrixes are both equal to $C^0$.
\end{remark}
\begin{remark}\label{rem:Cppositive}
Note that the covariance matrices 
$C^p_n=\Exp^\gamma \left[\tilde{J}_n^{p}\,\,^t\tilde{J}_n^{-p}\right]$, 
$p \in I_n$, are circulant Hermitian, i.e. $C^p_n=\,^tC^{p\,*}_n$, because 
$R_\J$ is even. They are positive definite because, being circulant, their 
eigenvalues are the values of the length $N$ DFT of the sequence 
$(\tilde{R}_\J(p,k))_{k \in I_n}$ which are positive because $R_\J$ is an 
autocorrelation function hence has a positive spectrum. 
Hypothesis 
\eqref{eq:bound on spectrum} guarantees that for $N$ large enough these 
eigenvalues are strictly positive, hence $C^p_n$ is invertible.
\end{remark}
\begin{remark}
Complex Gaussian calculus indicates that the probability density function under $\gamma$ of $\tilde{J}_n^p$, $p \neq 0$ is
\[
p_\gamma(\tilde{J}_n^p)=\frac{1}{\pi^N |det(C^p_n)|} \exp \left\{ -\frac{1}{2} \left[ \,^t\tilde{J}_n^{-p} \ \,^t\tilde{J}_n^p \right] 
\left[
\begin{array}{cc}
C^p_n & 0\\
0 & C^{p*}_n
\end{array}
\right]^{-1}
\left[
\begin{array}{c}
\tilde{J}_n^p\\
\tilde{J}_n^{-p}
\end{array}
\right]
 \right\},
\]
and, since $C^p_n$ is invertible (see Remark \ref{rem:Cppositive}),
\begin{equation}\label{eq:pgammaJp}
p_\gamma(\tilde{J}_n^p)=\frac{1}{\pi^N |det(C^p_n)|} \exp \left\{ -\frac{1}{2} \left[ \,^t\tilde{J}_n^{-p} \ \,^t\tilde{J}_n^p \right] 
\left[
\begin{array}{cc}
(C^p_n)^{-1} & 0\\
0 & (C^{p*}_n)^{-1}
\end{array}
\right]
\left[
\begin{array}{c}
\tilde{J}_n^p\\
\tilde{J}_n^{-p}
\end{array}
\right]
 \right\},
\end{equation}
\end{remark}
\begin{remark}\label{rem:JpJqindependent}
Note that Lemma \ref{lem:Fourier coeffs uncorrelated} implies that the complex centered Gaussian vectors $\tilde{J}_n^p$ and $\tilde{J}_n^q$ are independent under $\gamma$ if $p+q \neq 0$. Indeed, complex Gaussian calculus indicate that the four jointly Gaussian $N$-dimensional centered real vectors ${\rm Re}(\tilde{J}_n^p), {\rm Im}(\tilde{J}_n^p),{\rm Re}(\tilde{J}_n^q), {\rm Im}(\tilde{J}_n^q)$ are independent if $p+q \neq 0$.
\end{remark}
Given a Hermitian matrix $A$ of size $N$, we note  $\lambda_1(A) \geq \cdots \geq \lambda_N(A)$ its eigenvalues. 
As a consequence of Lemma \ref{lem:Fourier coeffs uncorrelated} we obtain a useful upper bound.
\begin{corollary}\label{cor:boundgamma}
For all $n\in \Z^+$, all $p \in I_n$ and all vectors $\zeta=(\zeta^j)_{j \in I_n}$ and $\xi=(\xi^j)_{j \in I_n}$ of $\R^N$,
\begin{align*}
\sup_{p\in I_n}\left| \sum_{j,k \in I_n}\Exp^\gamma \bigg[\tilde{J}_n^{pj}\tilde{J}_n^{- pk} \bigg]\zeta^j \xi^k \right|\leq a b   \norm{\zeta}_2 \norm{\xi}_2
\end{align*}
$a$ and $b$ are defined in \eqref{eq:ab}.
\end{corollary}
\begin{proof}
According to Remark \ref{rem:Cppositive} we have
\[
\left| \sum_{j,k \in I_n}\Exp^\gamma \bigg[\tilde{J}_n^{pj}\tilde{J}_n^{- pk} \bigg]\zeta^j \xi^k \right| = \left| \ ^t \zeta C^p_n \xi \right| \leq \norm{C^p_n}_2 \norm{\zeta}_2 \norm{\xi}_2
\]
Next we have $\norm{C^p_n}_2 = \lambda_1(C^p_n)$,  where $\lambda_1(A)$ is the largest eigenvalue of the Hermitian matrix $A$.  By Remark \ref{rem:Cppositive} the eigenvalues of the circulant matrix $C^p_n$ are the values of the  DFT of the sequence $(\tilde{R}_\J(p,k))_{k\in I_n}$. According to \eqref{eq:RJdef} and \eqref{eq:ab} they are  all upperbounded in magnitude by $ab$, and so is $\norm{C^p_n}_2$.

\end{proof}
Let $(Z^j_t)$, $j \in I_n$ be an element of $\T^N$. We recall the definition of the centered Gaussian field $(G^j_t)$:
\[
G^j_t = \sum_{l\in I_n} J_n^{jl} f(Z^l_t).
\]
Taking the length $N$ DFT of the $I_n$-periodic sequence $(G^j_t)_{j \in I_n}$, we introduce the following $I_n$-periodic stationary sequence of centered complex Gaussian processes $(\tilde{G}^p)_{p \in I_n}$  
\begin{equation}\label{eq:Gtildep}
\tilde{G}^p_t = \sum_{l \in I_n} \tilde{J}_n^{pl} f(Z^l_t).
\end{equation}
We have the following independence result.
\begin{lemma}\label{lem:tildeGpGq}
If $p+q \neq 0$,  under $\gamma^{\emp(Z_n)}$, the centered complex Gaussian processes $(\tilde{G}^p)_t$ and $(\tilde{G}^q)_s$ are independent on $[0,T]$ and
\[
\Exp^{\gamma^{\emp(Z_n)}}\left[\tilde{G}^p_t\tilde{G}^q_s\right]=\left\{
\begin{array}{lcl}
\sum_{l,k \in I_n} \tilde{R}_\J(p,l-k)f(Z^l_t)f(Z^k_s) & \text{if} & p+q=0\\
0 & \text{otherwise} &
\end{array}
\right.
\]
\end{lemma}
\begin{proof}
We write
\[
\tilde{G}^p_t = \sum_{l\in I_n} \tilde{J}_n^{pl} f(Z^l_t),\ \ \tilde{G}^q_s = \sum_{k \in I_n} \tilde{J}_n^{qk} f(Z^k_s),
\]
The independence under $\gamma^{\emp(Z_n)}$ follows from the independence under $\gamma$ of $\tilde{J}_n^p$ and $\tilde{J}_n^q$ if $p+q \neq 0$ proved in Remark \ref{rem:JpJqindependent}.
Moreover
\[
\Exp^{\gamma^{\emp(Z_n)}}\left[\tilde{G}^p_t\tilde{G}^q_s\right]= \sum_{l,k \in I_n} \Exp^\gamma \left[\tilde{J}_n^{pl} \tilde{J}_n^{qk} \right] f(Z^l_t)f(Z^k_s).
\]
The result follows from Lemma \ref{lem:Fourier coeffs uncorrelated}.
\end{proof}
We recall the expression \eqref{eq:defLambdat} for $\Lambda_t(G)$
\[
\Lambda_t(G):= \frac{\exp\left\{ -\frac{1}{2\sigma^2} \int_0^t \sum_{k \in I_n} \left( G^k_s\right)^2\,ds\right\} }{\Exp^{\gamma^{\hat{\mu}_n(Z_n)}} \left[ \exp\left\{ -\frac{1}{2\sigma^2} \int_0^t \sum_{k \in I_n} \left( G^k_s\right)^2\,ds\right\} \right]},
\]
and define
\begin{equation}\label{eq:Lambdapt}
\tilde{\Lambda}^p_t(\tilde{G}):= \frac{\exp\left\{ -\alpha_p \displaystyle\int_0^t  \left| \tilde{G}^p_s\right|^2\,ds\right\} }{\Exp^{\gamma^{\hat{\mu}_n(Z_n)}} \left[ \exp\left\{ -\alpha_p 
	\displaystyle\int_0^t  \left| \tilde{G}^p_s\right|^2\,ds\right\} \right]}\  p \in I_n,
\end{equation}
where $\alpha_0=\frac{1}{2\sigma^2 N}$, $\alpha_p=\frac{1}{\sigma^2 N}$, $p \neq 0$.

Define also $U_t^n$ to be the $N\times N$ symmetric positive semi-definite matrix with elements 
	$U^{n,jk}_t = \int_0^t f(Z^j_s)f(Z^k_s)ds$, $j,\,k \in I_n$.
\begin{lemma}\label{lem:Lambdatindep}
The $\tilde{\Lambda}^p_t(\tilde{G})$, $p \in I_n, p \geq 0$, are independent under $\gamma^{\emp(Z_n)}$ and
we have
\begin{equation}\label{eq:LambdaLambdap}
\Lambda_t(G)=\prod_{p \in I_n,\, p \geq 0} \tilde{\Lambda}^p_t(\tilde{G})
\end{equation}
\end{lemma}
\begin{proof}
By Parseval's theorem
\[
\sum_{k \in I_n} \left( G^k_s\right)^2=\frac{1}{N} \sum_{p \in I_n} \left| \tilde{G}^p_s\right|^2,
\]
since the $G^k$s are real, $\tilde{G}^p_s=\tilde{G}^{-p\,*}_s$, and we have
\[
\sum_{k \in I_n} \left( G^k_s\right)^2=\frac{1}{N} \left| \tilde{G}^0_s\right|^2+\frac{2}{N} \sum_{p \in I_n, p > 0} \left| \tilde{G}^p_s\right|^2,
\]
so that
\[
-\frac{1}{2\sigma^2} \int_0^t \sum_{k \in I_n} \left( G^k_s\right)^2\,ds = -\sum_{p=0}^n \alpha_p \int_0^t \left| \tilde{G}^p_s\right|^2\,ds
\]
Note that
\[
 \int_0^t \left| \tilde{G}^p_s \right|^2\,ds = \,^t \tilde{J}_n^{-p} U_t^n \tilde{J}_n^p,
\]
implying that
\[
\exp\left\{ -\frac{1}{2\sigma^2} \int_0^t \sum_{k \in I_n} \left( G^k_s\right)^2\,ds\right\}=\prod_{p=0}^n \exp\left\{-\alpha_p  \,^t \tilde{J}_n^{-p} U_t^n \tilde{J}_n^p \right\},
\]
and hence
\[
\Exp^{\gamma^{\hat{\mu}_n(Z_n)}} \left[ \exp\left\{ -\frac{1}{2\sigma^2} \int_0^t \sum_{k \in I_n} \left( G^k_s\right)^2\,ds\right\} \right] =
\Exp^\gamma \left[ \prod_{p=0}^n \exp\left\{-\alpha_p  \,^t \tilde{J}_n^{-p} U_t^n \tilde{J}_n^p \right\}  \right]
\]
Because of the independence under $\gamma$, proved in Remark \ref{rem:JpJqindependent}, of $\tilde{J}_n^p$ and $\tilde{J}_n^q$ if $p+q \neq 0$, we have
\begin{align*}
\Exp^\gamma \left[ \prod_{p=0}^n \exp\left\{-\alpha_p  \,^t \tilde{J}_n^{-p} U_t^n \tilde{J}_n^p \right\}  \right]= \prod_{p=0}^n \Exp^\gamma \left[  \exp\left\{-\alpha_p  \,^t \tilde{J}_n^{-p} U_t^n \tilde{J}_n^p \right\}  \right] \\
= \prod_{p=0}^n \Exp^{\gamma^{\hat{\mu}_n(Z_n)}} \left[ \exp\left\{ -\frac{1}{N\sigma^2} \int_0^t  \left| \tilde{G}^p_s\right|^2\,ds\right\} \right],
\end{align*}
and \eqref{eq:LambdaLambdap} follows.

The independence under $\gamma^{\hat{\mu}_n(Z_n)}$ of the $\tilde{\Lambda}^p_t(\tilde{G})$, $p=0,\cdots,n$, follows from the independence under $\gamma$, proved in Remark \ref{rem:JpJqindependent}, of $\tilde{J}_n^p$ and $\tilde{J}_n^q$ if $p+q \neq 0$. This concludes the proof of the Lemma.
\end{proof}
We next characterize the law of  $(\tilde{J}_n^{p}, p \in I_n)$ under the law $\bar{\gamma}^{\emp(Z_n)}_t=\Lambda_t(G) \cdot \gamma^{\emp(Z_n)}$. 
\begin{proposition}\label{prop:gamma tilde p t}
For any $Z_n$ in $\T^N$, any $p,\,q \in I_n$, $p+q \neq 0$, $\tilde{J}_n^p$ and $\tilde{J}_n^q$ are, under $\bar{\gamma}^{\emp(Z_n)}_t$ independent centered complex Gaussian vectors. The covariance of $\tilde{J}_n^p$ under $\bar{\gamma}^{\emp(Z_n)}_t$ is given by
\[
\Exp^{\bar{\gamma}^{\emp(Z_n)}_t} \left[ \tilde{J}_n^{-p} \,^t\tilde{J}_n^p \right]=((C^p_n)^{-1}+\alpha_p U_t^n)^{-1}
\]
\end{proposition}
\begin{proof}
By Lemma \ref{lem:Lambdatindep} and Remark \ref{rem:JpJqindependent} we write
\[
p_{\bar{\gamma}^{\emp(Z_n)}_t}(\tilde{J}_n^p,\tilde{J}_n^q)=p_\gamma(\tilde{J}_n^p,\tilde{J}_n^q) \tilde{\Lambda}^p_t(\tilde{G})\tilde{\Lambda}^q_t(\tilde{G})=
p_\gamma(\tilde{J}_n^p) \tilde{\Lambda}^p_t(\tilde{G}) \times p_\gamma(\tilde{J}_n^q) \tilde{\Lambda}^q_t(\tilde{G}),
\]
and the independence follows.

Next we have
\[
\tilde{\Lambda}^p_t(\tilde{G}) = \frac{\exp \left\{ -\alpha_p \,^t \tilde{J}_n^{-p} U_t^n \tilde{J}_n^p  \right\}  }{ \Exp^\gamma \left[ \exp \left\{ -\alpha_p \,^t \tilde{J}_n^{-p} U_t^n \tilde{J}_n^p  \right\}   \right]  },
\]
and since $\alpha_p$ and $U_t^n$ are real and $U_t^n$ is symmetric
\[
\tilde{\Lambda}^p_t(\tilde{G}) = \frac{\exp \left\{  -\frac{\alpha_p}{2} \left[ \,^t\tilde{J}_n^{-p} \ \,^t\tilde{J}_n^p \right] 
\left[
\begin{array}{cc}
U_t^n & 0\\
0 & U_t^n
\end{array}
\right]
\left[
\begin{array}{c}
\tilde{J}_n^p\\
\tilde{J}_n^{-p}
\end{array}
\right]  \right\}  }{ \Exp^\gamma \left[ \exp \left\{ -\alpha_p \,^t \tilde{J}_n^{-p} U_t^n \tilde{J}_n^p  \right\}   \right]  }.
\]
Combining this equation with \eqref{eq:pgammaJp}, we write
\begin{multline*}
p_\gamma(\tilde{J}_n^p) \tilde{\Lambda}^p_t = 
\frac{1}{\pi^N |det(C^p_n)|\Exp^\gamma \left[ \exp \left\{ -\alpha_p \,^t \tilde{J}_n^{-p} U_t^n \tilde{J}_n^p  \right\}   \right] }\\
\times \exp \left\{ -\frac{1}{2} \left[ \,^t\tilde{J}_n^{-p} \ \,^t\tilde{J}_n^p \right] 
\left[
\begin{array}{cc}
(C^p_n)^{-1}+\alpha_p U_t^n & 0\\
0 & (C^{p*}_n)^{-1}+\alpha_p U_t^n
\end{array}
\right]
\left[
\begin{array}{c}
\tilde{J}_n^p\\
\tilde{J}_n^{-p}
\end{array}
\right]
 \right\},
\end{multline*}
which shows that, under $\bar{\gamma}^{\emp(Z_n)}_t$, $\tilde{J}_n^p$ is centered complex Gaussian with covariance $((C^p_n)^{-1}+\alpha_p U_t^n)^{-1}$
\end{proof}
\begin{corollary}\label{cor:tilde Gp independent tilde Gq}
The centered processes $\tilde{G}^p_t$ and $\tilde{G}^q_s$, $p, q \in I_n$ are still Gaussian and independent under $\bar{\gamma}^{\emp(Z_n)}_t$ for all $s \leq t$ except for $p+q=0$. 
Moreover
\[
\Exp^{\bar{\gamma}^{\emp(Z_n)}_t} \left[ \tilde{G}^p_t \tilde{G}^{-p}_s   \right]=\Exp^{\gamma^{\emp(Z_n)}} \left[ \tilde{\Lambda}^{|p|}_t(\tilde{G})  \tilde{G}^p_t \tilde{G}^{-p}_s \right].
\]
\end{corollary}
\begin{proof}
By Lemma \ref{lem:Lmun} the process $(G^k_t)_{k \in I_n, t \in [0,T]}$ is Gaussian centered under $\bar{\gamma}_t^{\emp(Z_n)}$ and therefore so is the process $(\tilde{G}^p_t)_{p \in I_n, t \in [0,T]}$.
By Lemma \ref{lem:Lambdatindep}
\begin{align*}
\Exp^{\bar{\gamma}^{\emp(Z_n)}_t} \left[ \tilde{G}^p_t \tilde{G}^q_s   \right] &= \Exp^{\gamma^{\emp(Z_n)}} \left[ \Lambda_t(G) \tilde{G}^p_t \tilde{G}^q_s  \right]=
\Exp^{\gamma^{\emp(Z_n)}} \left[ \tilde{G}^p_t \tilde{G}^q_s \prod_{r=0}^n \tilde{\Lambda}^r_t(\tilde{G}) \right] \\
&= \Exp^{\gamma^{\emp(Z_n)}} \left[ \tilde{G}^p_t \tilde{\Lambda}^{|p|}_t(\tilde{G})   \tilde{G}^q_t \tilde{\Lambda}^{|q|}_t(\tilde{G}) \right]
\end{align*}
By rewriting the last term in the right hand side of the previous equation as a function of $\tilde{J}_n^p$ and $\tilde{J}_n^q$ and applying Proposition~\ref{prop:gamma tilde p t} one finds that if $p+q \neq 0$
\[
\Exp^{\bar{\gamma}^{\emp(Z_n)}_t} \left[ \tilde{G}^p_t \tilde{G}^q_s   \right] = \Exp^{\gamma^{\emp(Z_n)}} \left[ \tilde{G}^p_t \tilde{\Lambda}^{|p|}_t(\tilde{G}) \right] 
\Exp^{\gamma^{\emp(Z_n)}} \left[ \tilde{G}^q_t \tilde{\Lambda}^{|q|}_t(\tilde{G}) \right] =0.
\]
Therefore, for all $p, q \in I_n$, $p+q \neq 0$
\[
\Exp^{\bar{\gamma}^{\emp(Z_n)}_t} \left[ \tilde{G}^p_t \tilde{G}^q_s   \right]=\Exp^{\bar{\gamma}^{\emp(Z_n)}_t} \left[ \tilde{G}^p_t \tilde{G}^{-q}_s   \right]=0.
\]
This implies that the four real and imaginary parts of $\tilde{G}^p_t$ and $\tilde{G}^q_s$ are uncorrelated and therefore, being Gaussian, independent.
If $p+q=0$
\[
\Exp^{\bar{\gamma}^{\emp(Z_n)}_t} \left[ \tilde{G}^p_t \tilde{G}^{-p}_s   \right]=\Exp^{\gamma^{\emp(Z_n)}} \left[ \Lambda_t(G) \tilde{G}^p_t \tilde{G}^{-p}_s  \right],
\]
and by Proposition~\ref{prop:gamma tilde p t}
\[
\Exp^{\bar{\gamma}^{\emp(Z_n)}_t} \left[ \tilde{G}^p_t \tilde{G}^{-p}_s   \right]=\Exp^{\gamma^{\emp(Z_n)}} \left[ \tilde{\Lambda}^{|p|}_t(\tilde{G})  \tilde{G}^p_t \tilde{G}^{-p}_s \right]
\]
for all $p \in I_n$ and all $0 \leq s \leq t \leq T$.
\end{proof}
\begin{remark}\label{rem:upeigen}
Note that since $C^p_n$ is Hermitian positive definite, it is invertible and its inverse is also Hermitian positive definite. $U_t^n$ is real symmetric positive hence also Hermitian positive.
The sum $(C^p_n)^{-1}+\alpha_p U_t^n$ is therefore Hermitian positive. The dual Weyl inequality \cite{tao:12} commands that
\[
\lambda_{i+j-N}((C^p_n)^{-1}+\alpha_p U_t^n) \geq \lambda_i((C^p_n)^{-1})+\lambda_j(\alpha_p U_t^n)
\]
whenever $1 \leq i,\,j$, $i+j-N \leq N$. Since $(C^p_n)^{-1}$ is Hermitian positive definite for $N$ large enough, and $\alpha_p U_t^n$ is Hermitian positive, this inequality implies that $\lambda_N((C^p_n)^{-1}+\alpha_p U_t^n)  > 0$ and hence that $(C^p_n)^{-1}+\alpha_p U_t^n$ is invertible. 

Next we have 
\[
\lambda_1(((C^p_n)^{-1}+\alpha_p U_t^n)^{-1})=\frac{1}{\lambda_N((C^p_n)^{-1}+\alpha_p U_t^n)} \leq \frac{1}{\lambda_i((C^p_n)^{-1})+\lambda_j(\alpha_p U_t^n)},
\]
for $1 \leq i,\,j \leq N$ and $i+j=N$.
Since $\lambda_j(\alpha_p U_t^n) \geq 0$ for $j=1,\cdots,N$ and $\lambda_i((C^p_n)^{-1}) \geq \lambda_N((C^p_n)^{-1})  = \lambda_1(C^p_n)^{-1} > 0$ for all $i=1,\cdots,N$ we conclude that
\begin{equation}\label{eq:CJ}
\lambda_1(((C^p_n)^{-1}+\alpha_p U_t^n)^{-1}) \leq \frac{1}{\lambda_1(C^p_n)} \leq C_\J
\end{equation}
for some positive constant $C_\J$ independent of $N$ and $p$.
\end{remark}
In several places we use the following Lemma.
\begin{lemma}\label{Lemma Fourier Bound on R Jn}
For all $n\in \Z^+$ and $Z \in \T^N$, and all vectors $\zeta=(\zeta^j)_{j \in I_n}$ and $\xi=(\xi^j)_{j \in I_n}$ of $\R^N$,
\begin{align*}
\sup_{p\in I_n}\left| \sum_{j,k \in I_n}\Exp^{\gamma^{\emp(Z_n)}}\bigg[\tilde{\Lambda}^p_t(\tilde{G})\tilde{J}_n^{pj}\tilde{J}_n^{- pk} \bigg]\zeta^j \xi^k \right|\leq C_\J   \norm{\zeta}_2 \norm{\xi}_2
\end{align*}
where $C_\J$ is defined in \eqref{eq:CJ}.
$\tilde{\Lambda}^p_t(\tilde{G})$ is defined by \eqref{eq:Lambdapt}, $\tilde{G}^p_t = \sum_{l\in I_n} \tilde{J}_n^{pl} f(Z^l_t)$, and $\norm{\,}_2$ is the usual Euclidean norm.
\end{lemma}
\begin{proof}
We use Proposition~\ref{prop:gamma tilde p t}. We define $D^p_n:=((C^p_n)^{-1}+\alpha_p U_t^n)^{-1}$, $p \in I_n$, $p \geq 0$. We can write
\[
\sum_{j,k \in I_n}\Exp^{\gamma^{\emp(Z_n)}}\bigg[\tilde{\Lambda}^p_t(\tilde{G})\tilde{J}_n^{pj}\tilde{J}_n^{-pk} \bigg]\zeta^j \xi^k={\,^t}\zeta D^p_n \xi,
\]
hence
\[
\left| \sum_{j,k \in I_n}\Exp^{\gamma^{\emp(Z_n)}}\bigg[\tilde{\Lambda}^p_t(\tilde{G})\tilde{J}_n^{pj}\tilde{J}_n^{-pk} \bigg]\zeta^j \xi^k \right| = \left| {\,^t}\zeta D^p_n \xi \right|.
\]
Considering the Euclidean norm  in $\R^N$ and the corresponding matrix norm, both noted $\norm{\ }_2$, we have
\[
\left| {\,^t}\zeta D^p_n \xi \right| \leq \norm{D^p_n}_2 \norm{\zeta}_2 \norm{\xi}_2.
\]
By definition of the Euclidean norm, $\norm{D^p_n}_2=\lambda_1(D^p_n)\leq C_\J$, by Remark \ref{rem:upeigen}.
\end{proof}
%%%%%%%%%%%%%
%%%%%%%%%%%%%
\section{Covariance functions}\label{app:covariances}
\subsection{Time continuous setting}\label{app:covcont}
One of the basic constructions in this paper is the following. Given a measure $\mu \in \mP_S(\T^\Z)$, an integer $n$ (possibly infinite),  and a time $t \in [0,T]$, define the following
sequence of functions $K_\mu^k: [0,\,t]^2 \to \R$
\begin{equation}\label{eq:Kmudef}
K_\mu^k(s,u) = \sum_l R_\J(k,l) \int_{\T^\Z} f(v^0_s)f(v^l_u)\, d\mu(v),
\end{equation}
for $s,\,u \in [0,\,t]$. The summation w.r.t $l$ in the right hand side is either over the set $I_n$ for finite $n$ or over $\Z$. The index $k$ in the left hand side has the same range as $l$. In case of $n$ infinite, the right hand side is well defined because of the absolute summability of the sequences $(R_\J(k,l))_{l \in \Z}$ for all $k \in \Z$ and the fact that $0 \leq f \leq 1$. In the case of $n$ finite, the sequence $(K_\mu^k)_{k \in I_n}$, noted  $K_\mu^{n,k}$,   is $N$-periodic.

It is easy to check that the sequence $(K^k_\mu(s,u))_k$ of functions is the covariance of a centered stationary Gaussian process noted $G^j_s$, with $s \in [0,\,t]$ and $j$ is in $I_n$ for finite $n$ or in $\Z$ otherwise. There are several possible representations of this process. In the case of finite $n$ we use
\begin{equation}\label{eq:Gmu}
G^j_s = \sum_{k \in I_n} J^{jk}_n f(v^k_s),
\end{equation}
and noted $\gamma^{\mu^{I_n}}$ the law under which it has covariance $K_\mu^n$, i.e.
\[
\Exp^{\gamma^{\mu^{I_n}}} \left[ G^i_s G^j_u \right] = K_\mu^{n, j-i}(s,u),
\]
see the proof of Lemma \ref{lem:L2positive} below.
A second representation is provided by the consideration of the operator defined by the sequence $K_\mu^{k}$. This operator is defined on the Hilbert space $L^2(\Z \times [0,\,t]):=\bigoplus_{i \in \Z} L^2([0,t])$ (or $L^2(I_n \times [0,\,t])$) of infinite (or finite) sequences of measurable square integrable complex functions $g^k_s$ on $[0,\,t]$ such that
\[
\sum_k \int_0^t \left| g^k_s \right|^2\,ds < \infty,
\]
where, as usual, the summation w.r.t. $k$ is over $I_n$ for $n$ finite or over $\Z$ otherwise. In the sequel we treat only the case of infinite $n$, i.e. $I_n = \Z$, the case of $n$ finite being easily deduced from this one.

We prove in Lemma \ref{lem:L2} that the operator $\bar{K}_\mu$ acting on $L^2(\Z \times [0,\,t])$ by
\begin{equation}\label{eq:barKmu}
\left( \bar{K}_\mu\,g \right)^k_s = \sum_l \int_0^t K_\mu^{k-l}(s,u) g^l_u \, du,\ g \in L^2(\Z \times [0,\,t]),
\end{equation}
is continuous, self-adjoint, and compact.

Note that by Fourier transform the space $L^2(\Z \times [0,t])$ is isomorphic to the space $L^2([-\pi,\pi] \times [0,t])$.
Each element $g$ of $L^2(\Z \times [0,\,t])$ features a Fourier transform $\tilde{g}$ such that
\[
\tilde{g}(\varphi)(s) = \sum_k g^k_s e^{-i k \varphi},
\]
where the series in the right hand side is absolutely convergent. For each $\varphi \in [-\pi,\pi[$, $\tilde{g}(\varphi) \in  L^2([0,t])$. 

By the convolution theorem, the operator $\bar{K}_\mu$ on $L^2(\Z \times [0,t])$ induces an operator $\bar{\tilde{K}}_\mu$ on $L^2([-\pi,\pi] \times [0,t])$ acting on such functions by
\[
\left( \bar{\tilde{K}}_\mu\tilde{g}\right)(\varphi) (s) = \int_0^t \tilde{K}_\mu(\varphi)(s,u) \tilde{g}(\varphi)(u)\,du,
\]
where
\[
\tilde{K}_\mu(\varphi)(s,u) = \sum_k K_\mu^k(s,u) e^{-ik\varphi}.
\]
\begin{lemma}\label{lem:L2}
The linear operator $\bar{K}_\mu$ defined by \eqref{eq:barKmu} maps $L^2(\Z \times [0,\,t])$  to itself and is continuous, self-adjoint, and compact. Its norm is upperbounded by $abt$.
\end{lemma}
\begin{proof}\ \\
1) Well-defined and continuous:\\
We prove that $\bar{K}_\mu$ maps $L^2(\Z \times [0,\,t])$ onto itself. In effect, by Cauchy-Schwarz
\begin{equation}\label{eq:Ineed}
\left| \left( \bar{K}_\mu\,g \right)^k_s \right| \leq \sum_l \left( \int_0^t \left| K_\mu^{ k-l}(s,u) \right|^2 \, du \right)^{1/2} \left( \int_0^t  \left| g^l_u \right|^2 \, du \right)^{1/2} .
\end{equation}
By Young's convolution Theorem, $0 \leq f \leq 1$, \eqref{eq:RJdef} and \eqref{eq:Kmudef}
\begin{multline*}
\left( \sum_k \left| \left( \bar{K}_\mu\,g \right)^k_s \right|^2 \right)^{1/2} \leq \sum_k \left| \left( \bar{K}_\mu\,g \right)^k_s \right| \leq \\
\sum_k \left( \int_0^t \left| K_\mu^{k}(s,u) \right|^2 \, du \right)^{1/2} \times \left( \sum_k  \int_0^t  \left| g^k_u \right|^2 \, du  \right)^{1/2} \leq
a b \sqrt{t} \norm{ g  }_{  L^2(\Z \times [0,\,t])}
\end{multline*}
so that, 
\[
\norm{\bar{K}_\mu g}_{  L^2(\Z \times [0,\,t])}^2 = \sum_k \int_0^t \left| \left( \bar{K}_\mu\,g \right)^k_s \right|^2\,ds \leq
a^2 b^2 t^2 \norm{ g  }_{  L^2(\Z \times [0,\,t])}^2,
\]
and therefore $\bar{K}_\mu$ is well-defined as a linear mapping from $L^2(\Z \times [0,\,t])$ to itself , bounded and therefore continuous with $\norm{ \bar{K}_\mu}_{  L^2(\Z \times [0,\,t])} \leq a b t$.\\
2) Self-adjoint:\\
This follows directly from the identity $K_\mu^{k}(u,s) = K_\mu^{-k}(s,u)$.
\\
3) Compactness:\\
We sketch the proof. We use the Kolmogorov-Riesz-Fr\'echet Theorem \cite[Th. 4.26]{brezis:10} for the compactness of bounded set of $L^p(\R^n)$, the analog of the Ascoli-Arzel\`a Theorem for continuous functions.

 Let $\tilde{g} \in L^2([-\pi,\pi] \times [0,t])$. Let $h=(h_1,h_2) \in \R^2$. We define the operator $\tau_h : L^2([-\pi,\pi] \times [0,t]) \to L^2([-\pi,\pi] \times [0,t])$ by 
\[
(\tau_h \tilde{g})(\varphi,s)=\tilde{g}(\varphi+h_1, s+h_2),
\]
where the values are taken modulo $2\pi$ and modulo $t$, respectively. Given a 
bounded sequence $(\tilde{g}^k)_{k \in \N}$ of $L^2([-\pi,\pi] \times [0,t])$ 
we want to prove that the set $(\bar{\tilde{K}}_\mu \tilde{g}^k)_k$ 
is 
relatively compact. 
According to the Kolmogorov-Riesz-Fr\'echet Theorem, it is 
sufficient to prove that
\begin{equation}\label{eq:KRF}
\lim_{|h| \to 0} \norm{\tau_h (\bar{\tilde{K}}_\mu \tilde{g}^k) - (\bar{\tilde{K}}_\mu\tilde{g}^k)}_{L^2([-\pi,\pi] \times [0,t])} = 0
\end{equation}
uniformly in $k$. In effect we have
\begin{multline}\label{eq:thenorm}
\norm{\tau_h (\bar{\tilde{K}}_\mu \tilde{g}^k) - (\bar{\tilde{K}}_\mu\tilde{g}^k)}_{L^2([-\pi,\pi] \times [0,t])}^2=\\
\int_{-\pi}^\pi \int_0^t \left| \int_0^t (\tilde{K}_\mu(\varphi+h_1)(s+h_2,u)-\tilde{K}_\mu(\varphi)(s,u))\tilde{g}^k(\varphi,u)\,du  \right|^2\,d\varphi\,ds.
\end{multline}
We write, by \eqref{eq:Kmudef},
\begin{multline}\label{eq:thenorm1}
\tilde{K}_\mu(\varphi+h_1)(s+h_2,u)-\tilde{K}_\mu(\varphi)(s,u) \\
=
\sum_{l \in \Z} \tilde{R}_\J(\varphi+h_1,l)\int_\T f(v^0_{s+h_2})f(v^l_u)\,d\mu(v)-\tilde{R}_\J(\varphi,l)\int_\T f(v^0_{s})f(v^l_u)\,d\mu(v)\\
=
\sum_{l \in \Z} (\tilde{R}_\J(\varphi+h_1,l)-\tilde{R}_\J(\varphi,l))\int_\T f(v^0_{s+h_2})f(v^l_u)\,d\mu(v)\\+ \sum_{l \in \Z} \tilde{R}_\J(\varphi,l)\int_\T(f(v^0_{s+h_2})-f(v^0_{s}))f(v^l_u)\,d\mu(v),
\end{multline}
where we have noted
\[
\tilde{R}_\J(\varphi,l) = \sum_k R_\J(k,l) e^{-ik\varphi}.
\]
We first upperbound the magnitude of the first term in the right hand side of \eqref{eq:thenorm1}.
By the mean value theorem and \eqref{eq:akbl}
\[
\left| \tilde{R}_\J(\varphi+h_1,l)-\tilde{R}_\J(\varphi,l) \right| \leq |h_1| \sum_k  |k| \left| R_\J(k,l) \right| \leq |h_1|\, b_l \sum_k |k| a_k.
\]
Because of $0 \leq f \leq 1$ and \eqref{eq:akbl} again, 
we have
\begin{equation}\label{eq:maj1}
\left|  \sum_{l \in \Z} (\tilde{R}_\J(\varphi+h_1,l)-\tilde{R}_\J(\varphi,l))\int_{\T^\Z} f(v^0_{s+h_2})f(v^l_u)\,d\mu(v) \right| \leq C_1 |h_1|,
\end{equation}
for some positive constant $C_1$.

	We next  upperbound the magnitude of the second term in the right hand side of \eqref{eq:thenorm1}. First, thanks to the Dominated Convergence Theorem, the function
	$s \to \int_\T f(v^0_s) \, d\mu(v)$ is continuous on $[0,t]$
	$0 \leq t \leq T$, and hence uniformly continuous, 
	\begin{equation}\label{eq:maj2}
	\forall \varepsilon >0\ \exists \delta(\varepsilon) \geq 0,\,|h_2| \leq \delta \Rightarrow \left|  \int_{\T^\Z}(f(v^0_{s+h_2})-f(v^0_{s}))\,d\mu(v) \right| \leq \varepsilon
	. 
	\end{equation}
	Second, 
	$|\tilde{R}_\J(\varphi,l)| \leq ab_l$.

Combining \eqref{eq:thenorm}-\eqref{eq:maj2} with the fact that $(\tilde{g}^k)_k$ is bounded and Cauchy-Schwarz implies \eqref{eq:KRF}.
\end{proof}
We now prove that $\bar{K}_\mu$ is non negative.
\begin{lemma}\label{lem:L2positive}
The linear operator $\bar{K}_\mu$ defined by \eqref{eq:barKmu} is non negative.
\end{lemma}
\begin{proof}
Consider
\[
G^i_s=\sum_{j \in I_n} J^{ij}_n f(v^j_s).
\]
This implies, because of \eqref{eq:cov} and the stationarity of $\mu$, that
\begin{align*}
\Exp^{\gamma^{\mu^{I_n}}} \left[ G^i_s G^k_u   \right] &= \Exp^{\gamma^{\mu^{I_n}}} \left[ \sum_{j,l \in I_n} J^{ij}_n J^{kl}_n f(v^j_s) f(v^l_u) \right] 
=\frac{1}{N} \sum_{j,l \in I_n}  R_\J(k-i,l-j) \Exp^{\mu^{I_n}} \left[  f(v^j_s) f(v^l_u) \right] \\&= \frac{1}{N} \sum_{j,l \in I_n}  R_\J(k-i,l-j) \Exp^{\mu^{I_n}} \left[  f(v^0_s) f(v^{l-j}_u) \right] \\
& = \sum_{l \in I_n}  R_\J(k-i,l) \Exp^{\mu^{I_n}} \left[  f(v^0_s) f(v^l_u) \right] = K_\mu^{n,k-i}(s,u),
\end{align*}
from which it follows that
\begin{align*}
\left \langle \bar{K}_\mu^n g, g   \right \rangle_{L^2(I_n \times [0,t])} &= \sum_{k,l \in I_n} \int_0^t \left( \int_0^t K_\mu^{n,k-l}(s,u) g^l_u \, du  \right) (g^k_s)^*\,ds \\
&= \sum_{k,l \in I_n} \int_0^t \int_0^t \Exp^{\gamma^{\mu^{I_n}}} \left[ G^k_s G^l_u   \right] g^l_u (g^k_s)^*\,du\,ds\\&=
\Exp^{\gamma^{\mu^{I_n}}} \left[ \left| \sum_{k \in I_n} \int_0^t G^k_s g^k_s \,ds \right|^2 \right] \geq 0.
\end{align*}
We conclude that $\bar{K}_\mu^n$ is positive as an operator on $L^2(I_n \times [0,t])$  and hence, taking the limit $n \to \infty$ that $\bar{K}_\mu$ is a positive operator on $L^2( \Z \times [0,t]  )$.
\end{proof}

We have the following Lemma related to the Fourier representation of the sequence $(K_\mu^{k}(s,u))_{k \in \Z}$.
\begin{lemma}\label{lem:FTKmut}
The sequence $(K_\mu^{k}(s,u))_{k \in \Z}$ is the Fourier series of a three times continuously differentiable periodic function $[-\pi,\pi[ \to \R$, $\varphi \to \tilde{K}_\mu(\varphi)( s,u)$ which is continuous w.r.t. $(s,u)$. This implies that the $K^k_\mu(s,u)$ are $\smallO{1/ |k|^3}$. Furthermore this convergence is uniform in $s,\,u,\,\mu$.
\end{lemma}
\begin{proof}
It follows from Lemma \ref{lem:L2}  that for all $s,\,u \in [0,t]$ that the sequence $(K_\mu^{k}(s,u))_{k \in \Z}$ is the Fourier series of a continuous periodic function $[-\pi,\pi[ \to \R$, $\varphi \to \tilde{K}_\mu(\varphi)( s,u)$ which is continuous w.r.t. $(s,u)$.
By definition
\[
 \tilde{K}_\mu(\varphi)( s,u) = \sum_k K_\mu^{k}(s,u) e^{-i k \varphi},
\]
where the series in the right hand side is absolutely convergent. By \eqref{eq:Kmudef} we have
\[
\tilde{K}_\mu(\varphi)( s,u)=\sum_l \tilde{R}_\J(\varphi,l) \int_{\T^\Z} f(v^0_s)f(v^l_u)\, d\mu(v),
\]
and the order three differentiability of $\tilde{K}_\mu(\varphi)( s,u)$ follows from Remark~\ref{rem:twicediff} as well as the uniform convergence of $K^k_\mu(s,u)$.
\end{proof}
We have the following useful result.
\begin{lemma}\label{lem:Kmusquare}
 We have
 \[
  \left| \tilde{K}_\mu(\varphi)(s,u)  \right| \leq  a b \quad \forall s,\,u \in [0,t],\, \varphi \in [-\pi,\pi[.
 \]
 \end{lemma}
 \begin{proof}
  By \eqref{eq:Kmudef}
  \[
 \left| \tilde{K}_\mu(\varphi)(s,u)  \right| \leq \sum_{l \in \Z} \left| \tilde{R}_\J(\varphi,l)  \right|,
 \]
 where
 \[
 \tilde{R}_\J(\varphi,l) = \sum_{k \in \Z} R_\J(k,l) e^{-ik\varphi}.
 \]
 This implies that
 \[
 \left| \tilde{K}_\mu(\varphi)(s,u)  \right|^2 \leq \left( \sum_{l \in \Z}\left|   \tilde{R}_\J(\varphi,l)  \right| \right)^2,
 \]
 and, since by \eqref{eq:RJdef}, \eqref{eq:akbl}
 \[
 \left|   \tilde{R}_\J(\varphi,l)  \right| \leq \sum_{k \in \Z} \left| R_\J(k,l)  \right| \leq b_l \sum_{k \in \Z} a_k =ab_l.
 \]
 We conclude that
 \[
 \left| \tilde{K}_\mu(\varphi)(s,u)  \right|^2 \leq a^2 b^2.
 \]
 \end{proof}
 By Lemmas \ref{lem:L2} and \ref{lem:L2positive} it follows that the spectrum of $\bar{K}_\mu$ is discrete and composed of non negative eigenvalues noted $\lambda_m^\mu$, $m \in \N$. Let $(h^\mu_m)$ be a corresponding  orthonormal basis of eigenvectors i.e. such as
\[
 \bar{K}_{\mu} h_m^{\mu} =\lambda^{\mu}_m \,  h_m^{\mu}, \quad \langle h_m^{\mu}, h_{m'}^{\mu} \rangle = \delta_{mm'}\ \ \forall m,\,m' \in \N.
 \]
Next define $g_m^{\mu}=\sqrt{\lambda^{\mu}_m} h_m^{\mu}$, $m \in \N$. 
One has the following ``SVD'' decomposition of the operator $\bar{K}_\mu$.
 \[
 K_{\mu}^{k}(s,u)=\sum_{m \in \N} \sum_l g_m^{\mu}(l,s)g_m^{\mu}(l+k,u).
 \]
 Given a covariance $(K_\mu^k)_{k \in \Z}$ we know that there exists a centered Gaussian process $(\Omega,\mathcal{A},\gamma,(G^k_t)_{k \in \Z} )$ with covariance $(K_\mu^k)_{k \in \Z}$. For any such process, if $H_\mu$ denotes the Gaussian space associated (the closed linear span of $(G^k_t)_{k \in \Z}$ in $L^2(\Omega,\mathcal{A},\gamma)$), then $H_\mu$ is isomorphic to the autoreproducing Hilbert space $\mathcal{H}_\mu$ associated to $(K_\mu^k)_{k \in \Z}$ by
 \[
 \begin{array}{rcl}
 \phi : H_\mu & \to & \mathcal{H}_\mu\\
 Z  & \to & \Exp^\gamma\left[  Z G^\cdot_\cdot \right].
 \end{array}
 \]
 The space $\mathcal{H}_\mu \subset L^2(\Z, [0,T])$ admits $(g_m^\mu)_{m \geq 0}$ as an orthonormal basis. If $\xi^\mu_m=\phi^{-1}(g^\mu_m)$, then $(\xi^\mu_m)_{m \geq 0}$ is a sequence of i.i.d. $\mathcal{N}(0,1)$ random variables in $H_\mu$ 
and   we have the following representation for the Gaussian process $G^i_s$:
 \[
 G^{i}_s=\sum_{m \geq 0} g^{\mu}_m (i, s) \xi^{\mu}_m,
 \]
 where the convergence is in $L^2(\Omega,\mathcal{A},\gamma)$. We note $\gamma^\mu$ the law on $(\Omega, \mathcal{A})$ under which the sequence $(G^i_s)$, $i \in \Z$, $s \in [0,t]$ has covariance $K_\mu^{k}$.
 \begin{remark}\label{rem:opcomposition}
 Note that given two measures $\mu_1$ and $\mu_2$ in $\mP_S(\T^\Z)$ and the corresponding operators $\bar{K}_{\mu_1}$ and $\bar{K}_{\mu_2}$, the operator $\bar{K}:=\bar{K}_{\mu_1} \circ \bar{K}_{\mu_2}$ has the following kernel
 \[
 K^k(s,u)=\sum_l \int_0^t K_{\mu_1}^{k-l}(s,v) K_{\mu_2}^{l}(v,u)\,dv,
 \]
 or, in the (continuous) Fourier domain
 \[
 \tilde{K}(\varphi)(s,u)=\int_0^t \tilde{K}_{\mu_1}(\varphi)(s,v) \tilde{K}_{\mu_2}(\varphi)(v,u)\,dv,
 \]
 and in the discrete case
 \[
 \tilde{K}^p(s,u)=\int_0^t \tilde{K}_{\mu_1}^{p}(s,v) \tilde{K}_{\mu_2}^{p}(v,u)\,dv,\ p \in I_n.
 \]
 \end{remark}
 \noindent
Consider the new  self-adjoint positive compact operator $\bar{L}_{\mu}$ on $L^2(\Z \times [0,t])$ defined by
 \begin{equation}\label{eq:Lmubarn}
 \bar{L}_{\mu}=({\rm Id}+\sigma^{-2} \bar{K}_{\mu})^{-1} \bar{K}_{\mu},
 \end{equation}
 and let $L_{\mu}$ be its kernel:
 \[
 L_{\mu}^{k}(s,u)=\sum_{m \geq 0} \frac{1}{1+\frac{\lambda^{\mu}_m}{\sigma^2}} \sum_lg^{\mu}_m(l,s)g^{\mu}_m(l+k,u).
 \]
% \begin{remark}
% As stated in Remark \ref{rem:Lmukt}, because it is always clear from the context, we drop the upper index $t$.
% \end{remark}
 \begin{remark}\label{rem:commute}
 Note that $({\rm Id}+\sigma^{-2} \bar{K}_{\mu})^{-1}$ and $\bar{K}_{\mu}$ commute, i.e.,
 \[
  \bar{L}_{\mu}=({\rm Id}+\sigma^{-2} \bar{K}_{\mu})^{-1} \bar{K}_{\mu} =\bar{K}_{\mu}  ({\rm Id}+\sigma^{-2} \bar{K}_{\mu})^{-1},
 \]
 as can be readily seen by noticing that both sides of the previous equality are equal to $\sigma^2 \left( {\rm Id} -   \left( {\rm Id}+\sigma^{-2}\bar{K}_\mu \right)^{-1} \right)$, so that we also have
 \begin{equation}\label{eq:Lmubarn1}
  \bar{L}_{\mu} = \sigma^2 \left( {\rm Id} -   \left( {\rm Id}+\sigma^{-2}\bar{K}_\mu \right)^{-1} \right).
 \end{equation}
 \end{remark}
 \begin{remark}
 Just as for the operator $\bar{K}_\mu$ we also use the finite size version $\bar{L}_\mu^{n}$ of $\bar{L}_\mu$ whose kernel is written $L_\mu^{n,k}$, $k \in I_n$.
 \end{remark}
 We have the analog of Lemma \ref{lem:FTKmut} for the Fourier transform $\tilde{L}_\mu(\varphi)$ of $L_\mu^{k}$.
 \begin{proposition}\label{prop:Lkmuregular}
 The sequence $(L_\mu^{k}(s,u))_{k \in \Z}$ is the Fourier series of a three times continuously differentiable periodic function $\varphi \to \tilde{L}_\mu(\varphi)( s,u)$ which is continuous w.r.t. $(s,u)$. The Fourier coefficients of $\tilde{L}_\mu(\varphi)(s,u)$, i.e. the kernel $(L_\mu^{k}(s,u))_{k \in \Z}$ of the operator $\bar{L}_\mu$, is $\smallO{1/|k|^3}$, uniformly in $s$, $u$ in $[0,t]$ and $\mu$. Therefore there exist constants $C$ and $D$ independent of $\mu$ such that $\forall s,\,u \in [0,t],\,\forall \varphi \in [-\pi, \pi)$, 
 \begin{align*}
 \sum_{k \in \Z} |L^{k}_\mu(s,u)| \leq C \\
 \sum_{k \in \Z} \left( L_\mu^{k}(s,u)  \right)^2 \leq D\\
  \left| \tilde{L}_\mu(\varphi) (s,u) \right| \leq \sqrt{D} \quad .
 \end{align*}
 \end{proposition}
 \begin{proof}
 It follows from \eqref{eq:Lmubarn} and Remark \ref{rem:commute} that
\begin{equation}\label{eq:Ltildebar}
\bar{\tilde{L}}_\mu(\varphi)=\left( {\rm Id} + \sigma^{-2} \bar{\tilde{K}}_\mu(\varphi)   \right)^{-1} \bar{\tilde{K}}_\mu(\varphi)=\bar{\tilde{K}}_\mu(\varphi) \left( {\rm Id} + \sigma^{-2} \bar{\tilde{K}}_\mu(\varphi)   \right)^{-1}.
\end{equation}
The order three continuous differentiability of $\tilde{L}_\mu(\varphi)(s,u)$ w.r.t. $\varphi$ follows from that of $\tilde{K}_\mu(\varphi)(s,u)$ proved in Lemma \ref{lem:FTKmut}. We also obtain the fact that the $L^k_\mu(s,u)$ are $\smallO{1/|k|^3}$ uniformly in $s,\,u$ in $[0,t]$ and $\mu$.
 \end{proof}
We have the following important Lemma which establishes that the kernels $L_\mu^{k}(s,u)$ are the covariance of the centered Gaussian field defined by \eqref{eq:Gmu} under another probability law than $\gamma^\mu$.
 \begin{lemma}\label{lem:Lmun}
 For all $t \in [0,T]$ and all $s,\,u \in [0,t]$, under the new law $
 \Lambda_t(G) \cdot \,\gamma^\mu$, the family of processes $(G^i_s)$ is still centered and Gaussian with covariance $L_\mu$ given by
 \begin{equation}\label{eq:Ktilde}
 \Exp^{\gamma^{\mu}} \left[ \Lambda_t(G) G_s^{0} G_u^{k} \right]=L_\mu^{t, k}(s,u),
 \end{equation}
 where
 \[
\Lambda_t(G) =  \frac{\exp \left\{ -\frac{1}{2\sigma^2}  \sum_j \int_0^t (G_u^{j})^2 \, du \right\}  }{\Exp^{\gamma^{\mu}} \left[ \exp \left\{ -\frac{1}{2\sigma^2} \sum_j  \int_0^t (G_u^{j})^2 \, du  \right\} \right] }.
 \]
 In the above, the summation w.r.t. $j$ is over $I_n$ for finite $n$ or over $\Z$ otherwise.
 
 In agreement with \eqref{eq:gammatilde} and Remark \ref{rem:Kmuk} we note $\bar{\gamma}_t^{\mu}$ the corresponding probability law on $(\Omega,\mathcal{A})$
 \end{lemma}
 \begin{proof}
 Let $\delta$ be a real number and $G_t^{M,k}=\sum_{m=0}^M  g^{\mu}_m(k,t) \xi^{\mu}_m$. Using the properties of the basis $(g^{\mu}_m)_{m \geq 0}$ we have
 \begin{multline*}
 \Exp^{\gamma^{\mu}} \left[ \exp \left \{ \delta G^{M,k}_t -\frac{1}{2\sigma^2} \sum_j \int_0^t (G^{M, j}_s)^2 \, ds \right\} \right] \\
 =\Exp^{\gamma^\mu} \left[ \exp \left \{ \delta \sum_{m=0}^M  g^{\mu}_m(k,t) \xi^{\mu}_m -\frac{1}{2\sigma^2}  \sum_{m=0}^M \lambda^{\mu}_m  (\xi^{\mu}_m)^2 \right\} \right].
 \end{multline*}
 Because of the independence of the $\xi_n^{\mu}$, this is equal to
\[
 \prod_{m=0}^M  \Exp^{\gamma^\mu} \left[ \exp \left \{ \delta g^{\mu}_m (k,t) \xi^{\mu}_n -\frac{1}{2\sigma^2}  \lambda^{\mu}_m (\xi^\mu_m)^2 \right\} \right],
 \]
 and, using standard Gaussian calculus, we obtain 
 \begin{multline*}
  \Exp^{\gamma^{\mu}} \left[ \exp \left \{ \delta G^{M,k}_t -\frac{1}{2\sigma^2} \sum_j \int_0^t (G^{M, j}_s)^2 \, ds \right\} \right] \\
	= \prod_{m=0}^M \left(1+\frac{\lambda^{\mu}_m}{\sigma^2}\right)^{-1/2} \, \exp \left(\frac{\delta^2}{2}  \sum_{m=0}^M \frac{1}{1+ \frac{\lambda^{\mu}_m}{\sigma^2}}\left(g^{\mu}_m(k,t)\right)^2 \right).
 \end{multline*}
  In particular
 \begin{align}\label{eq:Gamma1nM}
  \Exp^{\gamma^{\mu}} \left[ \exp \left \{  -\frac{1}{2\sigma^2} \sum_j \int_0^t (G^{M, j}_s)^2 \, ds \right\} \right] &= \left(\prod_{m=0}^M (1+ \frac{\lambda^{\mu}_m}{\sigma^2})\right)^{-1/2}\\
  \frac{  \Exp^{\gamma^{\mu}} \left[ \exp \left \{ \delta G^{M,k}_t -\frac{1}{2\sigma^2} \sum_j \int_0^t (G^{M, j}_s)^2 \, ds \right\} \right]}{  \Exp^{\gamma^{\mu}} \left[ \exp \left \{  -\frac{1}{2\sigma^2} \sum_j \int_0^t (G^{M, j}_s)^2 \, ds \right\} \right]} & =  \exp \frac{\delta^2}{2} \left\{ \sum_{m=0}^M \frac{1}{1+\frac{\lambda^{\mu}_m}{\sigma^2}}\left(g^\mu_m(k,t)\right)^2 \right\}.
 \nonumber 
 \end{align}
 The same formula shows that the sequence $\exp \left \{ \delta G^{M,k}_t -\frac{1}{2\sigma^2} \sum_j \int_0^t (G^{M, j}_s)^2 \, ds \right\}$ is bounded in $L^{1+\rho}(\Omega,\mathcal{A},\gamma)$ for any positive real $\rho$ so that this sequence is uniformly integrable. It converges in probability to $\exp \left \{ \delta G^{k}_t -\frac{1}{2\sigma^2} \sum_j \int_0^t (G^j_s)^2 \, ds \right\}$. 
 We conclude that
 \begin{align}\label{eq:Gamma1nn}
 \Exp^{\gamma^{\mu}} \left[ \exp \left \{  -\frac{1}{2\sigma^2} \sum_j \int_0^t (G^j_s)^2 \, ds \right\} \right]   &= \prod_{m \in \N} \left(1+ \frac{\lambda^{\mu}_m}{\sigma^2}\right)^{-1/2}
\\
  \frac{  \Exp^{\gamma^{\mu}} \left[ \exp \left \{ \delta G^{k}_t -\frac{1}{2\sigma^2} \sum_j \int_0^t (G^{j}_s)^2 \, ds \right\} \right]}{  \Exp^{\gamma^{\mu}} \left[ \exp \left \{  -\frac{1}{2\sigma^2} \sum_j \int_0^t (G^{j}_s)^2 \, ds \right\} \right]} & =  \exp \frac{\delta^2}{2} \left\{ \sum_{m \in \N} \frac{1}{1+\frac{\lambda^{\mu}_m}{\sigma^2}}\left(g^\mu_m(k,t)\right)^2 \right\}.
  \label{eq:moments1}
 \end{align}
 We have computed the moment generating function of $G^i_s$ under the new law $\Lambda_t(G) \cdot \gamma^\mu$.
 It is still Gaussian centered with covariance obtained by deriving \eqref{eq:moments1} twice at $\delta=0$ to obtain:
 \[
  \frac{\Exp^{\gamma^{\mu}} \left[ (G^{k}_t)^2 \exp \left \{ -\frac{1}{2\sigma^2} \sum_j \int_0^t (G^j_s)^2 \, ds \right\} \right]}{\Exp^{\gamma^{\mu}} \left[ \exp \left \{ -\frac{1}{2\sigma^2} \sum_j \int_0^t (G^j_s)^2 \, ds\right\}  \right]}=  \sum_{m \in \N} \frac{1}{1+\frac{\lambda^{\mu}_m}{\sigma^2}}\left(g^\mu_m(k,t)\right)^2 ,
 \]
 which yields \eqref{eq:Ktilde} by polarization.
 \end{proof}
  \begin{proposition}\label{prop:Amununif}
 The application $\mu \to L_{\mu}$ is Lipschitz continuous: There exists a positive constant $C_{t}$ such that
 \[
  | L_{\mu}^{k}(s,u) - L_{\nu}^{k}(s,u) |  \leq \smallO{1/|k|^3} C_{t} D_t(\mu, \nu)\ \forall s,\,u \in [0,t]
 \]
 for all $k \in \Z$.
 \end{proposition}
 \begin{proof}
 According to \eqref{eq:Lmubarn1} we have
 \begin{align*}
 \bar{L}_\mu - \bar{L}_\nu & = \sigma^2 \left( \left( {\rm Id}+\sigma^{-2}\bar{K}_\nu \right)^{-1} -   \left( {\rm Id}+\sigma^{-2}\bar{K}_\mu \right)^{-1} \right)\\
 &=
 \left( {\rm Id}+\sigma^{-2}\bar{K}_\nu \right)^{-1} \left( \bar{K}_\mu - \bar{K}_\nu  \right)  \left( {\rm Id}+\sigma^{-2}\bar{K}_\mu \right)^{-1}.
 \end{align*}
 Define $\bar{H}_\mu = \left( {\rm Id}+\sigma^{-2}\bar{K}_\mu \right)^{-1}$ and $\bar{H}_\nu = \left( {\rm Id}+\sigma^{-2}\bar{K}_\nu \right)^{-1}$. Using Remark \ref{rem:opcomposition}
 we have
 \[
 L_\mu^k(s,u) - L_\nu^k(s,u) = \sum_{l,j} \int_0^t \int_0^t  H_\nu^{k-l}(s,s_1) ( K_\mu^{l-j}(s_1,s_2) - K_\nu^{l-j}(s_1,s_2) ) H_\mu^{j}(s_2,u)\,ds_1\,ds_2.
 \]
 Let $\xi$ be a coupling  between $\mu$ and $\nu$, \eqref{eq:Kmudef} commands that 
 \begin{multline*}
 \left| L_\mu^k(s,u) - L_\nu^k(s,u)  \right| \leq  \\
 \sum_{l,j,m} \int_0^t \int_0^t  \left| H_\nu^{k-l}(s,s_1) \right| \left| R_\J(l-j,m) \right| \Exp^\xi \left[  \left| f(w^0_{s_1})f(w^m_{s_2}) - f( w^{'0}_{s_1})f(w^{'m}_{s_2}) \right| \right]  \left| H_\mu^{j}(s_2,u) \right|\,ds_1\,ds_2.
 \end{multline*}
  Observing that $f(w^0_{s_1})f(w^m_{s_2}) -f(w^{'0}_{s_1})f(w^{'m}_{s_2})=f(w^0_{s_1})(f(w^m_{s_2})-f(w^{'m}_{s_2}))+f(w^{'m}_{s_2})(f(w^0_{s_1})-f(w^{'0}_{s_1}))$, we obtain, using $0 \leq f \leq 1$
 \begin{multline*}
 \sum_{m \in \Z}\left | R_{\J}(l-j,m) \right |\int_0^t \int | f(w^0_{s_1})f(w^m_{s_2}) -f(w^{'0}_{s_1})f(w^{'m}_{s_2}) | \,d\xi(w,w')  \\
  \leq\left( \sum_{m \in \Z}\left | R_{\J}(l-j,m) \right | \right) \int | f(w^0_{s_1})-f(w^{'0}_{s_1}) |\,d\xi(w,w')\\+\sum_{m \in \Z}  \left | R_{\J}(l-j,m) \right |\int | f(w^m_{s_2})-f(w^{'m}_{s_2}) |\,d\xi(w,w').
 \end{multline*}
Equations \eqref{eq:RJdef} and  \eqref{eq:metric2} imply 
\begin{align*}
\left( \sum_{m \in \Z}\left | R_{\J}(l-j,m) \right | \right) 
\int | f(w^0_{s_1})-f(w^{'0}_{s_1}) |\,d\xi(w,w') &\leq  \frac{b}{b_0} a_{l-j} 
\int d_t(w,w')\, d\xi(w,w')\\
\sum_{m \in \Z}  \left | R_{\J}(l-j,m) \right |\int | f(w^m_{s_2})-f(w^{'m}_{s_2}) |\,d\xi(w,w') &\leq  a_{l-j}  \int d_t(w,w')\, d\xi(w,w').
\end{align*}
This commands that
\begin{multline}\label{eq:lmuk-lnuk}
\left| L_\mu^k(s,u) - L_\nu^k(s,u)  \right| \\ \leq C \sum_{l,j} \int_0^t \int_0^t  \left| H_\nu^{k-l}(s,s_1) \right| a_{l-j} \left| H_\mu^{j}(s_2,u) \right|\,ds_1\,ds_2 \times  \int d_t(w,w')\, d\xi(w,w')
\end{multline}
for some constant $C > 0$. We use Proposition \ref{prop:Lkmuregular}, which clearly applies to $\bar{H}_\mu$ and $\bar{H}_\nu$.
Since convolving two sequences $(c_k)_{k \in \Z}$ and $(d_k)_{k \in \Z}$ whose terms are $\smallO{1/|k|^3}$ results in a sequence which is also $\smallO{1/|k|^3}$ it follows from \eqref{eq:lmuk-lnuk} that
\[
\left| L_\mu^k(s,u) - L_\nu^k(s,u)  \right| \leq \smallO{1/|k|^3} C t^2 D_t(\mu,\nu).
\]
 \end{proof}

 \subsection{Discrete time setting}\label{app:covdiscrete}
 In several parts of the paper we use time-discretized versions of these operators. Two cases occur. The first is that of a general measure in $\mP_S(\T^\Z)$, typically the limit measure $\mu_*$. The second is that of an empirical measure $\emp(V_n)$ or $\emp(V^m_n)$. Given a partition of $[0,T]$ into the $(m+1)$ points $v \eta_m = v\frac{T}{m}$, with $\eta_m := T/m$, for $v=0$ to $m$ we deal with the operators $\bar{K}_{\mu}$ and $\bar{L}_{\mu}$. It will be clear from the context whether these operators are defined by a finite, e.g. $(\bar{K}_{\mu}^{i})_{i \in I_n}$, or infinite, e.g. $(\bar{K}_{\mu}^{i})_{i \in \Z}$,sequence. In the finite case these operators are $Nv \times Nv$ matrixes which are block Toeplitz for $\bar{K}_\mu$ and
 $\bar{L}_\mu$. 
 
 We also consider several Fourier transforms of these operators. The continuous one noted $\bar{\tilde{K}}_\mu(\varphi)$,
 $\varphi \in [-\pi, \pi[$ in both the infinite and finite cases, and the discrete one.  In the continuous case we have
 \[
\tilde{K}_\mu(\varphi)=\sum_{j \in I_n} K_{\mu}^{ j} e^{-i j \varphi},\quad i^2=-1.
 \]
 For the discrete case, and this applies only to $\mu=\emp(V_n)$ and $\mu=\emp(V^m_n)$, the operators $\bar{K}_{\mu}$ and $\bar{L}_{\mu}$ are defined by the $N$ $v \times v$ matrixes $K_\mu^{j}$, $j \in I_n$. We consider their length $N$ Discrete Fourier Transform (DFT), i.e. the sequence of $N$ $v \times v$ matrixes $\tilde{K}_\mu^{p}$, $p \in I_n$ with
 \[
 \tilde{K}_\mu^{p}=\sum_{j \in I_n} K_\mu^{j} F_N^{-jp},
 \]
 the corresponding operator, noted $\bar{\tilde{K}}_\mu^{v \eta_m}$, is  block diagonal, the blocks having size $v \times v$.
 
 We also consider the sequence of $Q_m$ $v \times v$ matrixes, noted $K_\mu^{q_m, j}$, $j \in I_{q_m}$, pad it with $N-Q_m$ nul matrixes, and consider its length $N$ Discrete Fourier Transform (DFT), i.e. the sequence of $N$ $v \times v$ matrixes noted $\tilde{K}_\mu^{q_m, p}$, $p \in I_n$ with
 \[
 \tilde{K}_\mu^{q_m, p}=\sum_{j \in I_{q_m}} K_\mu^{j} F_N^{-jp},
 \]
 the corresponding operator, noted $\bar{\tilde{K}}_\mu^{q_m}$, is  also block diagonal, the blocks having also size $v \times v$.
 
 Note that we have
 \begin{equation}\label{eq:Kqmp}
 \tilde{K}_\mu^{ p} = \tilde{K}_\mu\left( \frac{2 \pi p}{N} \right), \quad p \in I_{n},
 \end{equation}
 and
\begin{equation}\label{eq:Kqmnp}
  \tilde{K}_\mu^{q_m,  p} = \tilde{K}_\mu^{q_m}\left( \frac{2 \pi p}{N} \right), \quad p \in I_n.
\end{equation}
All this holds mutatis mutandis if we replace $K_\mu$ by $L_\mu$.

Also note that the following relations hold
\begin{equation}\label{eq:tildeLdefdis}
\tilde{L}_{\emp(Z_n)}^{ p}(v \eta_m, w \eta_m) = \frac{1}{N} \Exp^{\gamma^{\emp(Z_n)}} \left[ \tilde{\Lambda}^{|p|}_{v \eta_m}(\tilde{G}) \tilde{G}^{-p}_{v \eta_m}\tilde{G}^p_{w \eta_m}  \right],\ p \in I_n,\,w \leq v \in \{0, \cdots, m\},
\end{equation}
where $Z_n=V_n$ or $V^m_n$.
We provide a short proof
\begin{proof}
According to \eqref{eq:Ktilde} we have, taking the length $N$ DFT of both sides,
\[
\tilde{L}_{\emp(Z_n)}^{ p}(v \eta_m, w \eta_m) = \Exp^{\gamma^{\emp(Z_n)}} \left[ \Lambda_{v\eta_m}(G) G^0_{v \eta_m} \tilde{G}^p_{w \eta_m}  \right].
\]
Using the inverse DFT relation,
\[
G^0_{v \eta_m} = \frac{1}{N} \sum_{q \in I_n} \tilde{G}^q_{v \eta_m},
\]
so that
\[
\tilde{L}_{\emp(Z_n)}^{ p}(v \eta_m, w \eta_m) = \frac{1}{N} \sum_{q \in I_n} \Exp^{\gamma^{\emp(Z_n)}} \left[ \Lambda_{v \eta_m}(G)  \tilde{G}^q_{v \eta_m} \tilde{G}^p_{w \eta_m}  \right].
\]
By Proposition~\ref{prop:gamma tilde p t} and Corollary~\ref{cor:tilde Gp independent tilde Gq} we have
\[
\tilde{L}_{\emp(Z_n)}^{ p}(v \eta_m, w \eta_m) = \frac{1}{N} \Exp^{\gamma^{\emp(Z_n)}} \left[ \tilde{\Lambda}^{|p|}_{v \eta_m}(\tilde{G}) \tilde{G}^{-p}_{v \eta_m} \tilde{G}^p_{w \eta_m} \right],
\]
which ends the proof.
\end{proof}
%%%%%%%%%%%%%%%%%
%%%%%%%%%%%%%%%%%
\section{Proof of Lemmas \ref{lem:alpha1}-\ref{lem:alpha4}}\label{app:proofofalphas}
\begin{proof}[{\bf Proof of Lemma \ref{lem:alpha1}}]
 We recall from \eqref{eq:tildetheta-eta} that
 \[
 \alpha^1_{s} = \frac{1}{N^2} \sum_{p \in I_n}  \left| \tilde{\theta}^p_s -{^m}\tilde{\theta}^p_{s^{(m)}} \right|^2.
 \]
 The proof is based on decomposing the right hand side of this equation into four terms.
  Using \eqref{eq:tildethetap} we write,
  \begin{multline}\label{eq:alpha1jpdef}
 \tilde{\theta}^p_s -{^m}\tilde{\theta}^p_{s^{(m)}}=  \sigma^{-2} N^{-1} \Exp^{\gamma^{\emp(V_n)}}\left[ \tilde{\Lambda}_s^p (\tilde{G}) \tilde{G}^p_s \int_0^s\tilde{G}^{-p}_r d\tilde{V}^{p}_r -\tilde{\Lambda}_{s^{(m)}}^p(\tilde{G})\tilde{G}^p_{s^{(m)}} \int_0^{s^{(m)}} \tilde{G}^{-p}_{r^{(m)}}  d\tilde{V}^{p}_r\right] = \\
 \sigma^{-2} N^{-1} \underbrace{\Exp^{\gamma^{\emp(V_n)}}\left[\left( \tilde{\Lambda}_s^p(\tilde{G})-\tilde{\Lambda}_{s^{(m)}}^p(\tilde{G})\right)\tilde{G}^p_{s^{(m)}} \int_0^{s^{(m)}} \tilde{G}^{-p}_{r^{(m)}}  d\tilde{V}^{p}_r\right]}_{\alpha^{1,1,p}_s}+\\
 \sigma^{-2} N^{-1} \underbrace{  \Exp^{\gamma^{\emp(V_n)}}\left[ \tilde{\Lambda}_s^p(\tilde{G})\left( \tilde{G}^p_s-\tilde{G}^p_{s^{(m)}} \right) \int_0^{s^{(m)}} \tilde{G}^{-p}_{r^{(m)}}  d\tilde{V}^{p}_r\right]}_{\alpha^{1,2,p}_s}+\\
 \sigma^{-2} N^{-1} \underbrace{ \Exp^{\gamma^{\emp(V_n)}}\left[\tilde{\Lambda}_s^p (\tilde{G}) \tilde{G}^p_s \int_0^{s^{(m)}} \left( \tilde{G}^{-p}_r - \tilde{G}^{-p}_{r^{(m)}}\right) d\tilde{V}^{p}_r \right]}_{\alpha^{1,3,p}_s} +\\
 \sigma^{-2} N^{-1} \underbrace{\Exp^{\gamma^{\emp(V_n)}}\left[ \tilde{\Lambda}_s^p (\tilde{G}) \tilde{G}^p_s \int_{s^{(m)}}^s \tilde{G}^{-p}_r d\tilde{V}^{p}_r \right]}_{\alpha^{1,4,p}_s},
 \end{multline}
 so that 
 \[
 \alpha^1_s \leq \frac{4}{\sigma^4} \frac{1}{N^4} \sum_{j=1}^4 \sum_{p \in I_n} | \alpha^{1,j,p}_s |^2.
 \]
 We prove that for any $M > 0$, for any $m$ sufficiently large, we have
\[
 \lsup{n} \frac{1}{N} \log Q^n \left( \sup_{s \in [0, T]}  \frac{1}{N^4}  \sum_{p \in I_n} | \alpha^{1,j,p}_s |^2  \geq \frac{\epsilon \sigma^2}{48 T C } \right) \leq -M \ j=1,\cdots,4.
\]
 The proofs are somewhat similar. They all rely upon the use of 
 Proposition~\ref{prop:gamma tilde p t}, Corollary~\ref{cor:tilde Gp independent tilde Gq}, Lemma~\ref{Lemma Fourier Bound on R Jn}, Isserlis'  and Cramer's Theorems.
 Let $0 \leq v \leq m$ be such that $s^{(m)}=v \eta_m$. For the rest of the proof we define 
 \begin{equation}\label{eq:Bdef}
 B:=\frac{\epsilon \sigma^2}{48 T C }.
 \end{equation}
 {\bf Proof for $\alpha^{1,1,p}_s$}\\ 
From \eqref{eq:alpha1jpdef} we have
 \[
 \alpha^{1,1,p}_s = \Exp^{\gamma^{\emp(V_n)}}\left[\left( \tilde{\Lambda}_s^p(\tilde{G})-\tilde{\Lambda}_{s^{(m)}}^p(\tilde{G})\right)\tilde{G}^p_{s^{(m)}} \int_0^{s^{(m)}} \tilde{G}^{-p}_{r^{(m)}}  d\tilde{V}^{p}_r\right].
 \]
 {\bf Step 1: An upper bound for $\left| \tilde{\Lambda}^p_{s}(\tilde{G}) -  \tilde{\Lambda}^p_{s^{(m)}}(\tilde{G})  \right|$}\\
 We recall the definition of $ \tilde{\Lambda}^p_{s}(\tilde{G})$:
\[
\tilde{\Lambda}^p_{s}(\tilde{G}) = \frac{e^{-\frac{u_p}{N \sigma^2} \int_0^{s} \left| \tilde{G}^p_u  \right|^2 \, du}}
{\Exp^{\gamma^{\emp(V_n)}} \left[ e^{-\frac{u_p}{N \sigma^2} \int_0^{s} \left| \tilde{G}^p_u  \right|^2 \, du}  \right]}:=\frac{X_p(s)}{\Exp^{\gamma^{\emp(V_n)} }\left[ X_p(s) \right]},
\]
with $u_p=1$ if $p \neq 0$ and $u_0=1/2$, see \eqref{eq:Lambdapt}. 
We then use the Lipschitz continuity of $x \to e^{-x}$ for $x \geq 0$:
\[
\left| e^{-x} - e^{-y} \right| \leq | x - y |,
\]
to obtain
	\begin{multline*}\label{eq:lambdapdiffbound}
\tilde{\Lambda}^p_{s}(\tilde{G}) -  \tilde{\Lambda}^p_{s^{(m)}}(\tilde{G}) =
\dfrac{X_p(s)-X_p(s^{(m)})}{\Exp^{\gamma^{\emp(V_n)}} \left[ X_p(s) \right]} - \frac{X_p(s^{(m)})}{\Exp^{\gamma^{\emp(V_n)}} \left[ X_p(s^{(m)}) \right]}
\left(1 - \dfrac{\Exp^{\gamma^{\emp(V_n)}} \left[ X_p(s^{(m)}) \right]}{\Exp^{\gamma^{\emp(V_n)}} \left[ X_p(s) \right]}\right),
\end{multline*}
\begin{multline}
\left|\tilde{\Lambda}^p_{s}(\tilde{G}) -  \tilde{\Lambda}^p_{s^{(m)}}(\tilde{G})\right| 
\leq 
\dfrac{\frac{u_p}{N \sigma^2}   \int_{s^{(m)}}^s \left| \tilde{G}^{p}_u \right|^2  \, du }{\Exp^{\gamma^{\emp(V_n)}} \left[ X_p(s) \right]} 
+ \tilde{\Lambda}^p_{s^{(m)}}(\tilde{G})\dfrac{\frac{u_p}{N \sigma^2} \Exp^{\gamma^{\emp(V_n)}} \left[  \int_{s^{(m)}}^s \left| \tilde{G}^{p}_u \right|^2  \, du  \right]}{\Exp^{\gamma^{\emp(V_n)}} \left[ e^{-\frac{u_p}{N \sigma^2} \int_0^{s} \left| \tilde{G}^{p}_u  \right|^2 \, du}  \right] }.
\end{multline}
We therefore have to find a strictly positive lower bound for $\Exp^{\gamma^{\emp(V_n)}} \left[ e^{-\frac{u_p}{N \sigma^2} \int_0^{s} \left| \tilde{G}^{p}_u  \right|^2 \, du}  \right] $ and show that there exists a positive constant $D$, independent of $p$ and $N$ such that
\begin{equation}\label{eq:minorExps}
0 < D \leq \Exp^{\gamma^{\emp(V_n)}} \left[ e^{-\frac{u_p}{N \sigma^2} \int_0^{s} \left| \tilde{G}^{p}_u  \right|^2 \, du}  \right] \leq 1 < \infty.
\end{equation}
Indeed, since $x \to e^{-x}$ is convex, Jensen's inequality commands that
\[
e^{-  \frac{u_p}{N \sigma^2} \int_0^{s} \Exp^{\gamma^{\emp(V_n)}} \left[ \left| \tilde{G}^{p}_u  \right|^2 \right]\,du } \leq \Exp^{\gamma^{\emp(V_n)}} \left[ e^{-\frac{u_p}{N \sigma^2} \int_0^{s} \left| \tilde{G}^{p}_u  \right|^2 \, du}  \right].
\]
	According to Lemma~\ref{lem:tildeGpGq}
	\[
	\Exp^{\gamma^{\emp(V_n)}} \left[ \left| \tilde{G}^{p}_u  \right|^2 \right] = \sum_{k,l \in I_n} \tilde{R}_\J(p,l-k)f(Z^k_u)f(Z^l_u)
	\leq N  \sum_{\ell \in I_n} \left|\tilde{R}_\J(p,\ell)\right|.
	\]
Next we recall that
\[
\tilde{R}_\J(p,\ell) = \sum_{k \in I_n} R_\J(k,\ell) F_N^{-pk},
\]
and, from,
\[
\left| \tilde{R}_\J(p,\ell)  \right| \leq \sum_{k \in I_n} \left| R_\J(k,\ell) \right| \leq b_\ell \sum_{k \in I_n} a_k,
\]
it follows from \eqref{eq:RJdef} and \eqref{eq:ab}
\[
\sum_{\ell \in I_n} \left| \tilde{R}_\J(p,\ell)  \right|  \leq ab.
\]
Finally
\[
\frac{u_p}{N \sigma^2} \int_0^{s} \Exp \left[ \left| \tilde{G}^{p}_u  \right|^2  \right]\, du  \leq \frac{u_p ab T}{\sigma^2} \leq \frac{ab T}{\sigma^2}
\]
and \eqref{eq:minorExps} is proved with
\(
D = e^{-\frac{ab T}{\sigma^2}}.
\).
Going back to \eqref{eq:lambdapdiffbound} and since $u_p \leq 1$, we have
\begin{equation}\label{eq:majdeslambdatilde}
\left| \tilde{\Lambda}^p_{s}(\tilde{G}) -  \tilde{\Lambda}^p_{s^{(m)}}(\tilde{G})  \right| \leq 
\frac{1}{ND\sigma^2}\left( \int_{s^{(m)}}^s \left| \tilde{G}^{p}_u \right|^2  \, du    +   \tilde{\Lambda}^p_{s^{(m)}}(\tilde{G}) \Exp^{\gamma^{\emp(V_n)}} \left[  \int_{s^{(m)}}^s \left| \tilde{G}^{p}_u \right|^2  \, du  \right] \right).
\end{equation}
\textbf{Step 2: upper bound for $\alpha^{1,1,p}_s$:}\\
From the definition of $\alpha^{1,1,p}_s$ in \eqref{eq:alpha1jpdef} and
\eqref{eq:majdeslambdatilde}, we have
 \begin{multline*}
 | \alpha^{1,1,p}_s |^2 \leq  \frac{2}{N^2 D^2 \sigma^4} \Bigg( \Exp^{\gamma^{\emp(V_n)}} \left[ \left( \int_{s^{(m)}}^s \left| \tilde{G}^{p}_u \right|^2  \, du \right)  \left|  \tilde{G}^p_{s^{(m)}} \right|\, \left| \int_0^{s^{(m)}} \tilde{G}^{-p}_{r^{(m)}}  d\tilde{V}^{p}_r \right| \right]^2\\
+\Exp^{\gamma^{\emp(V_n)}} \left[  \int_{s^{(m)}}^s \left| \tilde{G}^{p}_u \right|^2  \, du  \right]^2 \Exp^{\gamma^{\emp(V_n)}} \left[ \tilde{\Lambda}^p_{s^{(m)}}(\tilde{G}) \left|  \tilde{G}^p_{s^{(m)}} \right|\, \left| \int_0^{s^{(m)}} \tilde{G}^{-p}_{r^{(m)}}  d\tilde{V}^{p}_r \right|  \right]^2.
 \Bigg)
 \end{multline*}
By Cauchy-Schwarz again,
 \begin{multline*}
 | \alpha^{1,1,p}_s |^2 \leq  \frac{2}{N^2 D^2 \sigma^4}
 \Exp^{\gamma^{\emp(V_n)}} \left[ \left( \int_{s^{(m)}}^s \left| \tilde{G}^{p}_u \right|^2  \, du \right)^2 \right] \Bigg(
 \Exp^{\gamma^{\emp(V_n)}} \left[ \left|  \tilde{G}^p_{s^{(m)}} \right|^2\, \left| \int_0^{s^{(m)}} \tilde{G}^{-p}_{r^{(m)}}  d\tilde{V}^{p}_r \right|^2  \right]\\
+ \Exp^{\gamma^{\emp(V_n)}} \left[ \tilde{\Lambda}^p_{s^{(m)}}(\tilde{G}) \left|  \tilde{G}^p_{s^{(m)}} \right|^2 \right]\ 
 \Exp^{\gamma^{\emp(V_n)}} \left[ \tilde{\Lambda}^p_{s^{(m)}}(\tilde{G}) \left| \int_0^{s^{(m)}} \tilde{G}^{-p}_{r^{(m)}}  d\tilde{V}^{p}_r \right|^2 \right]
 \Bigg).
 \end{multline*}
 Applying once more Cauchy-Schwarz to the integral in the first factor in the right hand side we obtain
 \begin{multline}\label{eq:isserli1}
 | \alpha^{1,1,p}_s |^2 \leq  \frac{2}{ N^2 D^2 \sigma^4} (s-s^{(m)})  \left( \int_{s^{(m)}}^s \Exp^{\gamma^{\emp(V_n)}} \left[ \left| \tilde{G}^{p}_u \right|^4 \right] \, du  \right)\\
 \times \Bigg(
 \Exp^{\gamma^{\emp(V_n)}} \left[ \left|  \tilde{G}^p_{s^{(m)}} \right|^2\, \left| \int_0^{s^{(m)}} \tilde{G}^{-p}_{r^{(m)}}  d\tilde{V}^{p}_r \right|^2  \right]+\\
 \Exp^{\gamma^{\emp(V_n)}} \left[ \tilde{\Lambda}^p_{s^{(m)}}(\tilde{G}) \left|  \tilde{G}^p_{s^{(m)}} \right|^2 \right]\ 
 \Exp^{\gamma^{\emp(V_n)}} \left[ \tilde{\Lambda}^p_{s^{(m)}}(\tilde{G}) \left| \int_0^{s^{(m)}} \tilde{G}^{-p}_{r^{(m)}}  d\tilde{V}^{p}_r \right|^2 \right]
 \Bigg).
 \end{multline}
 {\bf Step 3: Apply Isserlis' Theorem}\\
 We recall Isserlis' formula for four centered Gaussian variables $X_k,\,k=1,\cdots,4$
 \begin{equation}\label{eq:Isserlis4}
\Exp \left[ X_1 X_2 X_3 X_4  \right] = \Exp \left[ X_1 X_2  \right] \Exp \left[ X_3 X_4  \right] +\Exp \left[ X_1 X_3  \right] \Exp \left[ X_2 X_4  \right] +\Exp \left[ X_1 X_4  \right] \Exp \left[ X_2 X_3  \right].
 \end{equation}
 For the first factor of the first term in the right hand side of \eqref{eq:isserli1} we let $X_1=X_2=\tilde{G}^p_u$ and $X_3=X_4=X_1^*=X_2^*=\tilde{G}^{-p}_u$. By Lemma \ref{lem:tildeGpGq} we have 
 \[
  \Exp^{\gamma^{\emp(V_n)}} \left[ X_1 X_2  \right] = \Exp^{\gamma^{\emp(V_n)}} \left[ X_3 X_4  \right] = 0,
 \]
 	if $p \neq 0$, and by Corollary \ref{cor:boundgamma}, and $0 \leq f \leq 1$
 \[
\max \left\{\Exp^{\gamma^{\emp(V_n)}} \left[ X_1 X_2  \right], \Exp^{\gamma^{\emp(V_n)}} \left[ X_3 X_4  \right]   \right\}  \leq Nab
 \]
 if $p=0$, as well as
	\[
 \max_{j=1,2, k=3,4}\Exp^{\gamma^{\emp(V_n)}}\left[X_jX_k\right] \leq N a b
 \]
 for all $p \in I_n$, so that 
 \[
 \int_{s^{(m)}}^s \Exp^{\gamma^{\emp(V_n)}} \left[ \left| \tilde{G}^{p}_u \right|^4 \right] \, du  \leq 3 (ab)^2 N^2 (s-s^{(m)}), \quad \forall p \in I_n.
 \]
 For the 
 	second factor of the first term 
 we use again \eqref{eq:Isserlis4} with  $X_1=\tilde{G}^p_{s^{(m)}}$, $X_2=X_1^*=\tilde{G}^{-p}_{s^{(m)}}$, $X_3=\int_0^{s^{(m)}} \tilde{G}^{-p}_{r^{(m)}}  d\tilde{V}^{p}_r$ and $X_4=X_3^*= \int_0^{s^{(m)}} \tilde{G}^{p}_{r^{(m)}}  d\tilde{V}^{-p}_r$. 
 By Lemma \ref{lem:tildeGpGq} again we have
 \[
 \Exp^{\gamma^{\emp(V_n)}} \left[ X_1 X_4  \right] = \Exp^{\gamma^{\emp(V_n)}} \left[ X_2 X_3  \right] = 0,
 \]
 if $p \neq 0$
 and, by Corollary \ref{cor:boundgamma}, and $0 \leq f \leq 1$
 \[
  \Exp^{\gamma^{\emp(V_n)}} \left[ X_1 X_2  \right] \leq N ab, \quad \Exp^{\gamma^{\emp(V_n)}}  \left[ X_3 X_4  \right] \leq ab \sum_{k \in I_n} \left| \int_0^{s^{(m)}} f(V^k_{r^{(m)}}) d\tilde{V}^{p}_r  \right|^2,
 \]
 as well as 
 \[
 \max \left\{  \Exp^{\gamma^{\emp(V_n)}} \left[ X_1 X_4  \right], \Exp^{\gamma^{\emp(V_n)}} \left[ X_2 X_3  \right]   \right\} \leq Nab
 \]
 if $p=0$. Furthermore, for the same reasons,
\[
\max\left(\left|\Exp^{\gamma^{\emp(V_n)}} \left[ X_1 X_3  \right]\right|,\left|\Exp^{\gamma^{\emp(V_n)}} \left[ X_3 X_4  \right]\right|\right)
\leq 
ab \sqrt{N} \left(  \sum_{k \in I_n} \left| \int_0^{s^{(m)}} f(V^k_{r^{(m)}}) d\tilde{V}^{p}_r  \right|^2 \right)^{1/2},
\] 
so that  
\[
\Exp^{\gamma^{\emp(V_n)}} \left[ \left|  \tilde{G}^p_{s^{(m)}} \right|^2\, \left| \int_0^{s^{(m)}} \tilde{G}^{-p}_{r^{(m)}}  d\tilde{V}^{p}_r \right|^2  \right]
\leq  3(ab)^2 N \sum_{k \in I_n} \left| \int_0^{s^{(m)}} f(V^k_{r^{(m)}}) d\tilde{V}^{p}_r  \right|^2.
\]
 By Lemma \ref{Lemma Fourier Bound on R Jn} and $0 \leq f \leq 1$
 \[
 \Exp^{\gamma^{\emp(V_n)}} \left[ \tilde{\Lambda}^p_{s^{(m)}}(\tilde{G}) \left|  \tilde{G}^p_{s^{(m)}} \right|^2 \right] \leq N C_\J,
 \]
 and
 \[
 \Exp^{\gamma^{\emp(V_n)}} \left[ \tilde{\Lambda}^p_{s^{(m)}}(\tilde{G}) \left| \int_0^{s^{(m)}} \tilde{G}^{-p}_{r^{(m)}}  d\tilde{V}^{p}_r \right|^2 \right] \leq
 C_\J \sum_{k \in I_n} \left| \int_0^{s^{(m)}} f(V^k_{r^{(m)}}) d\tilde{V}^{p}_r  \right|^2,
 \]
 so that  the second factor of the second term in the right hand side of \eqref{eq:isserli1} is upper bounded by $ (C_\J)^2 N \sum_{k \in I_n} \left| \int_0^{s^{(m)}} f(V^k_{r^{(m)}}) d\tilde{V}^{p}_r  \right|^2$.\\
 {\bf Step 4: Wrapping things up}\\
 Bringing all this together we find that
 \[
  \frac{1}{N^4} \sum_{p \in I_n} | \alpha^{1,1,p}_s |^2 \leq   
A \frac{1}{N^3} (s-s^{(m)})^2  \sum_{k \in I_n} \sum_{p \in I_n}\left| \int_0^{s^{(m)}} f(V^k_{r^{(m)}}) d\tilde{V}^{p}_r  \right|^2,
 \]
 for some positive constant $A$, and by Parseval's Theorem
 \[
  \frac{1}{N^4} \sum_{p \in I_n} | \alpha^{1,1,p}_s |^2 \leq  
 A \frac{1}{N^2} (s-s^{(m)})^2  \sum_{k \in I_n} \sum_{l \in I_n}\left( \int_0^{s^{(m)}} f(V^k_{r^{(m)}}) dV^l_r  \right)^2.
 \]
 Next we use Corollary \ref{Corollary Measure Representation} to write
 \[
 dV^l_r = \sigma dW^l_r + \sigma \theta^l_r dr,\quad l \in I_n,
 \]
 from which follows that
 \begin{multline*}
   \frac{1}{N^4} \sum_{p \in I_n} | \alpha^{1,1,p}_s |^2 \leq  
  \frac{2A}{N^2} (s-s^{(m)})^2 \Bigg(\sum_{k \in I_n} \sum_{l \in I_n}\left( \int_0^{s^{(m)}} f(V^k_{r^{(m)}}) dW^l_r  \right)^2 +\\
 \sum_{k \in I_n} \sum_{l \in I_n}\left( \int_0^{s^{(m)}} f(V^k_{r^{(m)}}) \theta^l_r dr  \right)^2
 \Bigg).
 \end{multline*}
 %for some constant $A$, and, 
 By Cauchy-Schwarz  and $0 \leq f \leq 1$, one has
 \[
 \left( \int_0^{s^{(m)}} f(V^k_{r^{(m)}}) \theta^l_r dr  \right)^2 \leq T \int_0^{s^{(m)}} (\theta^l_r)^2 dr.
 \]
 So that,
 \begin{multline*}
   \frac{1}{N^4} \sum_{p \in I_n} | \alpha^{1,1,p}_s |^2 \leq  
  \frac{2A}{N^2} (s-s^{(m)})^2 \left(\sum_{k \in I_n} \sum_{l \in I_n}\left( \int_0^{s^{(m)}} f(V^k_{r^{(m)}}) dW^l_r  \right)^2 +
NT  \sum_{l \in I_n}\int_0^{s^{(m)}} (\theta^l_r)^2 dr \right). 
 \end{multline*}
 We can conclude with Lemmas \ref{lem:BG} and \ref{Lemma: bound theta}.
 
 We provide the details. Since $s-s^{(m)} \leq T/m$,
 \begin{multline*}
 Q^n\left( \sup_{s \in [0,T]} \frac{1}{N^4} \sum_{p \in I_n} | \alpha^{1,1,p}_s |^2 \geq B   \right) \leq \\
 Q^n\left(\sup_{s \in [0,T]} \frac{1}{N^2 m^2} \sum_{k \in I_n} \sum_{l \in I_n}\left( \int_0^{s^{(m)}} f(V^k_{r^{(m)}}) dW^l_r  \right)^2 \geq \frac{B}{4T^2 A} \right) + \\
 Q^n\left(\sup_{s \in [0,T]} \frac{1}{N m^2} \sum_{l \in I_n}\int_0^{s^{(m)}} (\theta^l_r)^2 dr \geq \frac{B}{4T^3 A} \right),
 \end{multline*}
where  $B$ is defined in \eqref{eq:Bdef}.
 The logarithm of the left hand side is less than or equal to twice the maximum of the logarithms of the two terms in the right hand side.
 
 For the first term, writing $E:=\frac{B}{4T^2 A}$,  we have
 \begin{multline*}
 \log Q^n\left(\sup_{s \in [0,T]}\frac{1}{N^2 m^2} \sum_{k \in I_n} \sum_{l \in I_n}\left( \int_0^{s^{(m)}} f(V^k_{r^{(m)}}) dW^l_r  \right)^2 \geq
 E \right) =\\ \log Q^n\left(\sup_{s \in [0,T]} \frac{1}{m} \frac{1}{2N} \sum_{k \in I_n} \sum_{l \in I_n}\left( \int_0^{s^{(m)}} f(V^k_{r^{(m)}}) dW^l_r  \right)^2 \geq Nm\frac{E}{2} \right).
 \end{multline*}
 Now let $\zeta_s$ be the submartingale
 \[
 \zeta_s = \exp \left( \frac{1}{m} \frac{1}{2N} \sum_{k \in I_n} \sum_{l \in I_n}\left( \int_0^{s^{(m)}} f(V^k_{r^{(m)}}) dW^l_r  \right)^2  \right).
 \]
 By Doob's submartingale inequality we have
 \begin{multline*}
 Q^n\left(\sup_{s \in [0,T]} \frac{1}{m} \frac{1}{2N} \sum_{k \in I_n} \sum_{l \in I_n}\left( \int_0^{s^{(m)}} f(V^k_{r^{(m)}}) dW^l_r  \right)^2 \geq Nm\frac{E}{2} \right) = 
 Q^n \left( \sup_{s \in [0,T]} \zeta_s \geq \exp \left( Nm\frac{E}{2}  \right) \right)\\ \leq
 \exp  \left(- Nm\frac{E}{2}  \right) \Exp^{Q_n} \left[  \zeta_T  \right].
 \end{multline*}
 The application of Lemma~\ref{lem:BG} with $\varepsilon^2=\frac{1}{m}$ yields
 \[
 \log Q^n\left(\sup_{s \in [0,T]} \frac{1}{N^2 m^2} \sum_{k \in I_n} \sum_{l \in I_n}\left( \int_0^{s^{(m)}} f(V^k_{r^{(m)}}) dW^l_r  \right)^2 \geq
 E \right) \leq -Nm\frac{E}{2}-\frac{N}{4} \log (1- 4\frac{T}{m}),
 \]
 indicating that we can find $m$ large enough such that
 \begin{equation}\label{eq:1stalpha11}
 \lsup{n} \frac{1}{N} \log Q^n\left(\sup_{s \in [0,T]} \frac{1}{N^2 m^2} \sum_{k \in I_n} \sum_{l \in I_n}\left( \int_0^{s^{(m)}} f(V^k_{r^{(m)}}) dW^l_r  \right)^2 \geq
 E \right) \leq -M.
 \end{equation}
 For the second term, writing $E:=\frac{B}{4T^3 A}$,  we have
 \[
 Q^n\left(\sup_{s \in [0,T]} \frac{1}{N m^2} \sum_{l \in I_n}\int_0^{s^{(m)}} (\theta^l_r)^2 dr \geq E \right) \leq
 Q^n\left(\frac{1}{N} \sup_{r \in [0, T]} \sum_{l \in I_n} (\theta^l_r)^2 \geq \frac{E}{T} m^2 \right),
 \]
 and Lemma \ref{Lemma: bound theta} shows that, given $M  > 0$, we can find $m$ large enough such that
 \begin{equation}\label{eq:2ndalpha11}
 \lsup{n}\frac{1}{N} Q^n\left(\sup_{s \in [0,T]} \frac{1}{N m^2} \sum_{l \in I_n}\int_0^{s^{(m)}} (\theta^l_r)^2 dr \geq E \right) \leq -M .
 \end{equation}
 The combination of \eqref{eq:1stalpha11} and \eqref{eq:2ndalpha11} shows that for all $M>0$, for $m$ large enough
 \[
 \lsup{n} \frac{1}{N} \log Q^n \left(  \sup_{s \in [0,T]}  \frac{1}{N^4}  \sum_{p \in I_n} | \alpha^{1,1,p}_s |^2  \geq B \right) \leq -M,
 \]
 where $B$ being defined in \eqref{eq:Bdef}.
 
 The proof for $\alpha^{1,2,p}_s$ is very similar to that for $\alpha^{1,3,p}_s$ which we give now.

 {\bf Proof for $\alpha^{1,3,p}_s$}\\
  {\bf Step 1: An upper bound for $\left| \int_0^{s^{(m)}} \left( \tilde{G}^{-p}_r - \tilde{G}^{-p}_{r^{(m)}} \right)  d\tilde{V}^{p}_r \right|^2$}\\
 From \eqref{eq:alpha1jpdef} we have
 \[
\alpha^{1,3,p}_s =
 \Exp^{\gamma^{\emp(V_n)}}\left[\tilde{\Lambda}_s^p (\tilde{G}) \tilde{G}^p_s \int_0^{s^{(m)}} \left( \tilde{G}^{-p}_r  - \tilde{G}^{-p}_{r^{(m)}} \right) d\tilde{V}^{p}_r  \right].
  \]
  This commands, by Cauchy-Schwarz, that
  \[
  \left| \alpha^{1,3,p}_s  \right|^2 \leq \\
  \Exp^{\gamma^{\emp(V_n)}}\left[\tilde{\Lambda}_s^p (\tilde{G}) | \tilde{G}^p_s|^2 \right] \times \Exp^{\gamma^{\emp(V_n)}}\left[\tilde{\Lambda}_s^p(\tilde{G}) \left| \int_0^{s^{(m)}}\left( \tilde{G}^{-p}_r - \tilde{G}^{-p}_{r^{(m)}} \right) d\tilde{V}^{p}_r \right|^2 \right].
  \]
  By Lemma \ref{Lemma Fourier Bound on R Jn} and $0 \leq f \leq 1$
  \[
  \Exp^{\gamma^{\emp(V_n)}}\left[\tilde{\Lambda}_s^p (\tilde{G}) | \tilde{G}^p_s|^2 \right]  \leq C_\J 
  \sum_{j \in I_n}  f(V^j_s)^2 \leq N C_\J.
  \]
 By Lemma \ref{Lemma Fourier Bound on R Jn} again,
  \[
  \Exp^{\gamma^{\emp(V_n)}}\left[\tilde{\Lambda}_s^p(\tilde{G}) \left| \int_0^{s^{(m)}} \left( \tilde{G}^{-p}_r - \tilde{G}^{-p}_{r^{(m)}} \right)  d\tilde{V}^{p}_r \right|^2 \right] \leq
  C_\J \sum_{j \in I_n} \left| \int_0^{s^{(m)}} (f(V^j_r)-f(V^j_{r^{(m)}})) d\tilde{V}^{p}_r \right|^2 .
  \]
  By Parseval's Theorem
  \[
 C_\J \sum_{p \in I_n} \sum_{j \in I_n} \left| \int_0^{s^{(m)}} (f(V^j_r)-f(V^j_{r^{(m)}}) d\tilde{V}^{p}_r \right|^2 =
  N C_\J \sum_{j, k \in I_n} \left( \int_0^{s^{(m)}} (f(V^j_r)-f(V^j_{r^{(m)}}) dV^{k}_r \right)^2.
  \]
  and therefore
  \[
  \sum_{p \in I_n} \left| \alpha^{1,3,p}_s  \right|^2 \leq \\
  (C_\J)^2 N^2 \sum_{j, k \in I_n} \left( \int_0^{s^{(m)}} (f(V^j_r)-f(V^j_{r^{(m)}})) dV^{k}_r \right)^2.
  \]
  By \eqref{eq: theta SDE}-\eqref{eq: theta SDE1}  and Cauchy-Schwarz
  \begin{multline*}
  (C_\J)^2 N^2 \sum_{j, k \in I_n} \left| \int_0^{s^{(m)}} (f(V^j_r)-f(V^j_{r^{(m)}})) dV^{k}_r \right|^2 \leq \\
  A N^2 \left( \sum_{j, k \in I_n} \left( \int_0^{s^{(m)}} (f(V^j_r)-f(V^j_{r^{(m)}})) dW^{k}_r \right)^2 +
   \sum_{j, k \in I_n} \left( \int_0^{s^{(m)}} (f(V^j_r)-f(V^j_{r^{(m)}})) \theta^{k}_r \, dr \right)^2 \right),
  \end{multline*}
  for some constant $A > 0$, so that we have established that
 \begin{multline*}
   \frac{1}{N^4} \sum_{p \in I_n} \left| \alpha^{1,3,p}_s  \right|^2 \leq 
   A  \frac{1}{N^2} \Bigg( \sum_{j, k \in I_n} \left( \int_0^{s^{(m)}} (f(V^j_r)-f(V^j_{r^{(m)}})) dW^{k}_r \right)^2 +\\
   \sum_{j, k \in I_n} \left( \int_0^{s^{(m)}} (f(V^j_r)-f(V^j_{r^{(m)}})) \theta^{k}_r \, dr \right)^2 \Bigg).
  \end{multline*}
  By Cauchy-Schwarz on the second integral
  \begin{multline*}
   \frac{1}{N^4} \sum_{p \in I_n} \left| \alpha^{1,3,p}_s  \right|^2 \leq 
   A \frac{1}{N^2} \Bigg( \sum_{j, k \in I_n} \left( \int_0^{s^{(m)}} (f(V^j_r)-f(V^j_{r^{(m)}})) dW^{k}_r \right)^2 +\\
   \left( \sum_{j \in I_n}  \int_0^{s^{(m)}} (f(V^j_r)-f(V^j_{r^{(m)}}))^2 dr \right) \left( \sum_{k  \in I_n} \int_0^{s^{(m)}} (\theta^{k}_r)^2 \, dr \right)  \Bigg).
  \end{multline*}
  So that
  \begin{multline}\label{eq:alpha13p}
  Q^n\left( \sup_{s \in [0,T]}  \frac{1}{N^4} \sum_{p \in I_n} | \alpha^{1,3,p}_s |^2 \geq B  \right) \leq \\
  Q^n\left( \sup_{s \in [0,T]} \frac{1}{N^2} \sum_{j, k \in I_n} \left( \int_0^{s^{(m)}} (f(V^j_r)-f(V^j_{r^{(m)}})) dW^{k}_r \right)^2 \geq E \right) +\\
  Q^n\left(  \frac{1}{N} \left( \sum_{j \in I_n}  \int_0^{T} (f(V^j_r)-f(V^j_{r^{(m)}}))^2 dr \right)  \frac{1}{N} \left( \sum_{k  \in I_n} \int_0^{T} (\theta^{k}_r)^2 \, dr \right) \geq E \right),
  \end{multline}
 where $E= B/(2A) $.\\
{\bf Step 2: Upper bounding the second term in the right hand side of \eqref{eq:alpha13p}}\\
 Let $h(m): \N^* \to \R^+$ be such that $\lim_{m \to \infty} h(m) =0$. $h$ is specified later. The second term in the right hand side is dealt with as follows
 \begin{multline*}
 Q^n\left( \underbrace{ \frac{1}{N} \sum_{j \in I_n}  \int_0^{T} (f(V^j_r)-f(V^j_{r^{(m)}}))^2 dr}_{S_1} \times  \underbrace{\frac{1}{N}\sum_{k  \in I_n} \int_0^{T} (\theta^{k}_r)^2 \, dr }_{S_2} \geq E \right) =\\
Q^n\left( \mathbbm{ 1}_{S_1 > h(m)} S_1 S_2 +  \mathbbm{1}_{S_1 \leq h(m)} S_1 S_2 \geq E \right)
 \leq \\
 Q^n\left( \mathbbm{ 1}_{S_1 > h(m)} S_1 S_2 \geq E/2 \right)+Q^n\left(\mathbbm{1}_{S_1 \leq h(m)} S_1 S_2 \geq E/2 \right) \leq \\
 Q^n\left(S_1 > h(m) \right)+Q^n\left( S_2 \geq \frac{E}{2 h(m)} \right).
 \end{multline*}
 The term $Q^n\left( S_2 \geq \frac{E}{2 h(m)} \right)$ can be dealt with  Lemma \ref{Lemma: bound theta} since $\lim_{m \to \infty} h(m) =0$.
 Consider next the term $S_1:= \frac{1}{N} \sum_{j \in I_n}  \int_0^{T} (f(V^j_r)-f(V^j_{r^{(m)}}))^2 dr$. By the Lipschitz continuity of $f$, \eqref{eq: theta SDE}, Cauchy-Schwarz, and $r - r^{(m)} \leq T/m$ we have 
 \begin{align*}
 S_1 &\leq  \frac{1}{N} \sum_{j \in I_n} \int_0^{T} \left( V^j_r - V^j_{r^{(m)}} \right)^2 \, dr =  \frac{\sigma^2}{N} \sum_{j \in I_n} \int_0^{T} \left(  W^j_r - W^j_{r^{(m)}} + \int_{r^{(m)}}^r  \theta^j_s \,ds \right)^2\,dr  \\
 &\leq
  \frac{2 \sigma^2}{N} \sum_{j \in I_n} \int_0^{T} \left( W^j_r - W^j_{r^{(m)}}\right)^2 \, dr +  \frac{2 \sigma^2}{N} \sum_{j \in I_n} \int_0^{T} \left(  \int_{r^{(m)}}^r  \theta^j_s \,ds \right)^2\,dr  \\
&\leq   \frac{2 \sigma^2}{N} \sum_{j \in I_n} \int_0^{T} \left( W^j_r - W^j_{r^{(m)}}\right)^2 \, dr +  \frac{2 \sigma^2}{N} \sum_{j \in I_n} \int_0^{T} \left((r-r^{(m)})  \int_{r^{(m)}}^r  (\theta^j_s)^2  \,ds \right)\,dr \\
&\leq  \frac{2 \sigma^2}{N} \sum_{j \in I_n} \int_0^{T} \left( W^j_r - W^j_{r^{(m)}}\right)^2 \, dr + \frac{2\sigma^2 T^3}{m^2} \frac{1}{N} \sup_{s \in [0,T]} \sum_{j \in I_n} \left(  \theta^j_s \right)^2.
 \end{align*}
 We conclude that
 \begin{align*}
 Q^n \left( S_1 > h(m)   \right) \leq 
 Q^n \left( \frac{1}{N} \sum_{j \in I_n} \int_0^{T} \left( W^j_r - W^j_{r^{(m)}}\right)^2 \, dr \geq h(m)/(4 \sigma^2) \right)\\+
 Q^n\left(  \frac{1}{N}  \sum_{j \in I_n} \sup_{s \in [0,T]}\left(  \theta^j_s \right)^2 \geq \frac{1}{4 \sigma^2 T^3} m^2 h(m)\right).
 \end{align*}
 The second term in the right hand side of the previous inequality is dealt with  Lemma \ref{Lemma: bound theta}, provided that $\lim_{m \to \infty} m^2 h(m) = \infty$.
 
Regarding the first term, decomposing the integral, we have
\begin{align*}
\frac{1}{N} \sum_{j \in I_n} \int_0^{T} \left( W^j_r - W^j_{r^{(m)}}\right)^2 \, dr &=\frac{1}{N} \sum_{j \in I_n} \sum_{v=0}^{m-1} \int_{v \eta_m}^{(v+1)\eta_m} \left( W^j_r - W^j_{v \eta_m}\right)^2 \, dr\\
& =\frac{1}{N} \sum_{j \in I_n} \sum_{v=0}^{m-1} \int_{0}^{\eta_m} \left( W^{j,v}_r \right)^2 \, dr,
\end{align*}
where \((W^{j,v}_s)_{j,v}\) are independent Brownian motions.
\begin{align*}
\int_{0}^{\eta_m} \left( W^{j,v}_r \right)^2 dr & = \int_0^1 \left(W^{j,v}_{r\eta_m}\right)^2 \eta_m dr
= \int_0^1 \left(\dfrac{1}{\sqrt{\eta_m}}W^{j,v}_{r\eta_m}\right)^2 (\eta_m)^2 dr.
\end{align*}
We set \(\widehat{W}^{j,v}_{r} = \frac{1}{\sqrt{\eta_m}}W^{j,v}_{r\eta_m}\).
Thanks to the scaling property of the Brownian motion, 
\((\widehat{W}^{j,v}_{r})_{j,v}\) are independent Brownian motions,
so that
\[
\int_{0}^{\eta_m} \left( W^{j,v}_r \right)^2 dr = \eta_m^2\int_0^1 \left(\widehat{W}^{j,v}_r\right)^2dr.
\]
We deduce
\begin{multline*}
Q^n\left( \frac{1}{N} \sum_{j \in I_n} \int_0^{T} \left( W^j_r - W^j_{r^{(m)}}\right)^2 \, dr \geq h(m)/(4 \sigma^2) \right)\\
= Q^n \left(\frac{1}{Nm} \sum_{j \in I_n} \sum_{v=0}^{m-1}
\int_0^1 \left(\widehat{W}^{j,v}_r\right)^2dr \geq mh(m)\dfrac{1}{4T^2\sigma^2}\right).
\end{multline*}
This forces us to choose $h$ in such a way that $\lim_{m \to \infty} m h(m) = \infty$, e.g. $h(m)=1/\sqrt{m}$. Note that this implies that $\lim_{m \to \infty} m^2 h(m) = \infty$. In order to apply Cramer's Theorem, we require that the random variable $\int_0^1 \left(\widehat{W}^{j,v}_r\right)^2\,dr$ has exponential moments. This existence is due to the fact that, through Jensen's Inequality,
\begin{align*}
\Exp\bigg[\exp\bigg(\frac{1}{4}\int_0^1 (W_s)^2 ds \bigg) \bigg] \leq \Exp\bigg[ \int_0^1 \exp\bigg(\frac{1}{4}(W_s)^2\bigg)ds \bigg] \\
=  \int_0^1 \Exp\bigg[ \exp\bigg(\frac{1}{4}(W_s)^2\bigg) \bigg]  \leq \Exp\bigg[ \exp\bigg(\frac{1}{4}(W_1)^2\bigg) \bigg]  < \infty.
\end{align*}
{\bf Step 3: Upper bounding the first term in the right hand side of \eqref{eq:alpha13p}}\\
 In order to deal with the first term in the right hand side of \eqref{eq:alpha13p} we have to control the term $\frac{1}{N} \sup_{s \in [0, T]} \sum_{j \in I_n} (f(V^j_s)-f(V^j_{s^{(m)}}))^2$.
 In order to do this, we define the set
\[
K_{\kappa,n} = \left\lbrace V : \frac{1}{N}\sup_{s\in [0,T]}\sum_{j\in I_n}\big(f(V^j_s)-f(V^j_{s^{(m)}})\big)^2 \geq \frac{\kappa T}{m}\right\rbrace \subset \T^{N}.
\]
The following Lemma, whose proof is left to the reader, indicates that, for $\kappa$ large enough, the probability of this event is exponentially small for large $n$.
\begin{lemma}\label{lem:Kkappan}
For all $M > 0$, for $\kappa > 0$ large enough,
\begin{equation}\label{eq: intermediate 1}
\lsup{n} \frac{1}{N} \log Q^n \big( K_{\kappa,n} \big) \leq - M.
\end{equation}
\end{lemma}
Using this Lemma we write
\begin{multline}\label{eq:alpha1intbound}
\lsup{n} \frac{1}{N} \log Q^n \left( \sup_{s\in [0,T]}  \frac{1}{N^2} \sum_{j, k \in I_n} \left( \int_0^{s^{(m)}} (f(V^j_r)-f(V^j_{r^{(m)}})) dW^{k}_r \right)^2 \geq E  \right) \leq \\
\max \Bigg\{ \lsup{n} \frac{1}{N} \log Q^n \left( K_{\kappa,n} \right),\\
\lsup{n} \frac{1}{N} \log Q^n \left( K_{\kappa,n}^c \cap \left\{ \sup_{s\in [0,T]}  \frac{1}{N^2} \sum_{j, k \in I_n} \left( \int_0^{s^{(m)}} (f(V^j_r)-f(V^j_{r^{(m)}})) dW^{k}_r \right)^2 \geq E \right\}  \right) \Bigg\} \leq \\
\max \Bigg\{ -M,\\
 \lsup{n} \frac{1}{N} \log Q^n \left( K_{\kappa,n}^c \cap \left\{ \sup_{s\in [0,T]}  \frac{1}{N^2} \sum_{j, k \in I_n} \left( \int_0^{s^{(m)}} (f(V^j_r)-f(V^j_{r^{(m)}})) dW^{k}_r \right)^2 \geq E \right\}  \right) \Bigg\},
\end{multline}
where $\kappa$ is large enough so that \eqref{eq: intermediate 1} holds. 
Note that 
\begin{multline*}
 \sup_{s\in [0,T]}  \frac{1}{N^2} \sum_{j, k \in I_n} \left( \int_0^{s^{(m)}} (f(V^j_r)-f(V^j_{r^{(m)}})) dW^{k}_r \right)^2 \geq E \Longleftrightarrow \\
 \sup_{s\in [0,T]}  \frac{h(m)}{2 N} \sum_{j, k \in I_n} \left( \int_0^{s^{(m)}} \sqrt{\frac{2m}{\kappa T}}(f(V^j_r)-f(V^j_{r^{(m)}})) dW^{k}_r \right)^2 \geq \frac{h(m) m N E}{\kappa T},
\end{multline*}
where $h: \N \to \R^{+}$ is monotonically decreasing toward 0.
Now let $\zeta_s$ be the submartingale
\begin{equation}\label{eq:xit}
\zeta_s = \exp \left(   \frac{h(m)}{2 N} \sum_{j, k \in I_n} \left( \int_0^{s^{(m)}} \sqrt{\frac{2m}{\kappa T}}(f(V^j_r)-f(V^j_{r^{(m)}})) dW^{k}_r \right)^2 \right).
\end{equation}
 Through Doob's submartingale inequality,
\begin{multline}\label{eq:intermediate 2}
Q^n \bigg(K^c_{\kappa,n} \cap \left\{  \sup_{s\in [0,T]}  \frac{1}{N^2} \sum_{j, k \in I_n} \left( \int_0^{s^{(m)}} (f(V^j_r)-f(V^j_{r^{(m)}})) dW^{k}_r \right)^2 \geq E \right\} \bigg) \\
\leq \Exp^{Q^n}\big[ \zeta_T\cap K^c_{\kappa,n} \big] \exp \left(-\frac{h(m) m N E}{\kappa T} \right).
\end{multline}
Choosing, e.g. $h(m)=1/\sqrt{m}$ we can apply Lemma \ref{lem:BG} with $\varepsilon ^2 = h(m)$ and obtain $\Exp^{Q^n}\big[ \zeta_T\cap K^c_{\kappa,n} \big] \leq \left(1 - 4 \frac{T}{m} \right)^{-N/4}$. Hence, upon taking $m\to \infty$, we find that
\begin{multline*}
\lsup{n}\frac{1}{N} \log Q^n \left(K^c_{\kappa,n} \cap \left\{  \sup_{s\in [0,T]}  \frac{1}{N^2} \sum_{j, k \in I_n} \left( \int_0^{s^{(m)}} (f(V^j_r)-f(V^j_{r^{(m)}})) dW^{k}_r \right)^2 \geq E \right\} \right) \leq \\ -M .
\end{multline*}
We have established that for $m$ large enough
\[
\lsup{n} \frac{1}{N} \log Q^n \left(\sup_{s \in [0,T]} \frac{1}{N^4} \sum_{p \in I_n} \left| \alpha^{1,3,p}_s  \right|^2 \geq B \right) \leq  -M,
\]
where $B$ is defined in \eqref{eq:Bdef}.
\\
 {\bf Proof for $\alpha^{1,4,p}_s$}\\
   We next consider $\alpha^{1,4,p}_s$ in \eqref{eq:alpha1jpdef}.
  As in the previous derivations, by Corollary \ref{Corollary Measure Representation}, Cauchy-Schwarz inequality and Parseval's theorem, we write
 \begin{equation}\label{eq:beta3}
   \frac{1}{N^4} \sum_{p \in I_n}  \left| \alpha^{1,4,p}_s   \right|^2  \leq 
  A \frac{1}{N^2} \left( \sum_{j, k \in I_n} \left( \int_{s^{(m)}}^s f(V^j_r) dW^{k}_r \right)^2 + 
    \sum_{j, k \in I_n} \left( \int_{s^{(m)}}^s f(V^j_r) \theta^{k}_r \, dr\right)^2 \right),
  \end{equation}
 for some constant $A > 0$, independent of $n,\,m$.
In the remaining of this Appendix we neglect for simplicity the drift part, i.e. the second term in the right hand side of the previous equation,  since this can be dealt with similarly to the above by the use of Lemma \ref{Lemma: bound theta} or \ref{Lemma: bound eta}.

  From \eqref{eq:beta3}, neglecting the drift term, and letting $E:=  B/A$, we write
  \begin{multline*}
  \frac{1}{N} \log Q^n \left( \sup_{s \in [0,T]} A \frac{1}{N^2}  \sum_{j, k \in I_n} \left| \int_{s^{(m)}}^s f(V^j_r) dW^{k}_r \right|^2 \geq  B   \right) = \\
   \frac{1}{N} \log Q^n \left( \sup_{s \in [0,T]} \frac{1}{N^2}  \sum_{j, k \in I_n} \left| \int_{s^{(m)}}^s f(V^j_r) dW^{k}_r \right|^2 \geq  E  \right) = \\
  \frac{1}{N} \log Q^n  \left( \sup_{0 \leq u \leq m-1} \sup_{s \in [u \eta_m,(u+1) \eta_m]} \frac{1}{N^2}  \sum_{j, k \in I_n} \left| \int_{s^{(m)}}^s f(V^j_r) dW^{k}_r \right|^2 \geq  E   \right) \leq  \\
    \frac{1}{N} \log \left( m \sup_{0 \leq u \leq m-1} Q^n  \left(  \sup_{s \in [u \eta_m,(u+1) \eta_m]} \frac{1}{N^2}  \sum_{j, k \in I_n} \left| \int_{s^{(m)}}^s f(V^j_r) dW^{k}_r \right|^2 \geq  E  \right) \right) = \\
    \frac{1}{N} \log \left( m \sup_{0 \leq u \leq m-1} Q^n  \left(  \sup_{s \in [u \eta_m,(u+1) \eta_m]} \zeta_s \geq \exp \left(  \frac{ N h(m) E }{2}  \right) \right) \right),
  \end{multline*}
  where
  \[
  \zeta_s = \exp \left( \frac{ h(m) }{2N} \sum_{j, k \in I_n} \left| \int_{s^{(m)}}^s f(V^j_r) dW^{k}_r \right|^2 \right).
  \]
The function $h: \N \to \R^+$ is increasing and is defined just below.
  Since $\zeta_s$ is a submartingale for $s \in [u \eta_m,(u+1) \eta_m]$, by Doob's submartingale inequality,
\begin{align*}
Q^n  \left(  \sup_{s \in [u \eta_m,(u+1)\eta_m)]} \zeta_s \geq \exp \left(  \frac{ N h(m) E }{2}  \right) \right) &\leq \exp \left(  - \frac{ N h(m) E }{2}  \right) \Exp^{Q^n}\big[\zeta_{(u+1)\eta_m} \big].
\end{align*}
We apply Lemma \ref{lem:BG} with $\varepsilon = \sqrt{h(m)}$, $T = \eta_m $ and conclude that, if $h(m) \leq \frac{m}{4T}$ for $m$ large enough, e.g. $h(m) = \sqrt{m}$,
\[
\Exp^{Q^n}\big[\zeta_{(u+1)\eta_m} \big] \leq (1-4 h(m) \eta_m)^{-N/4},
\]
and therefore
\begin{multline*}
\frac{1}{N} \log Q^n  \left( \sup_{s \in [0,T]} A \frac{1}{N^2}  \sum_{j, k \in I_n} \left| \int_{s^{(m)}}^s f(V^j_r) dW^{k}_r \right|^2 \geq  B   \right) \leq \\
\frac{ \log m }{ N } - \frac{  h(m) E}{2 }-\frac{1}{4} \log (1-4 h(m) \eta_m).
\end{multline*}
We have established that for all $M>0$, for $m$ large enough
\[
\lsup{N}\frac{1}{N} \log Q^n (\sup_{s \in [0,T]} \frac{1}{N^4} \sum_{p \in I_n} \left| \alpha^{1,4,p}_s  \right|^2 \geq  B )) \leq  -M,
\]
and hence proved the Lemma.
 \end{proof}
%%%%%

%%%%%%%%%
\begin{proof}[{\bf Proof of Lemma~\ref{lem:alpha2}}]
The salient point in the proof is the use of the difference of the correlation functions $K_{\emp(V_n)}$ and $K_{\emp(V_n)}^{q_m}$, defined in Appendix \ref{app:covdiscrete}, over the sets $I_n \times I_{q_m}$ and $I_n \times I_n$. We remind the reader that $q_m$ is defined at the start of Section \ref{Section Definition Tilde Psi qm rm}. The proof shows that it is possible to choose $m$ and $q_m$ as functions of $n$ as stated in the Lemma. Assume that $s^{(m)} = v \eta_m$, $v = 0,\cdots, m-1$.\\
{\bf Step 1: Finding an upper bound of $\alpha^2_{v \eta_m}$ in terms of $\bar{K}_{\emp(V_n)}-\bar{K}_{\emp(V_n)}^{q_m}$}\\
In detail \eqref{eq:tildetheta-eta} implies that  
\begin{multline*}
\alpha^2_{v \eta_m} =  \frac{5}{N^2\sigma^4} \sum_{p \in I_n}  \left| \left( \left( \bar{\tilde{L}}^{ p}_{\emp(V_n)} - \bar{\tilde{L}}^{q_m,  p}_{\emp(V_n)} \right) \delta \tilde{V}^p \right)(v \eta_m) \right|^2 = \\
\frac{5}{N^2\sigma^4}\sum_{p\in I_n}  \bigg|
 \sum_{w=0}^{v} \left( \tilde{L}^{  p}_{\emp(V_n)}(v \eta_m,w \eta_m)-
 \tilde{L}^{q_m,  p}_{\emp(V_n)}(v \eta_m,w \eta_m) \right) \delta \tilde{V}^p_w \bigg|^2.
\end{multline*}
Next, by Cauchy-Schwarz on the $w$ index 
	\begin{align}
\nonumber
\alpha^2_{v \eta_m} & \leq 
\frac{5}{N^2\sigma^4} \sum_{p\in I_n} \sum_{w=0}^{v} \left| \tilde{L}^{ p}_{\emp(V_n)}(v \eta_m,w \eta_m)-
\tilde{L}^{q_m, p}_{\emp(V_n)}(v \eta_m,w \eta_m) \right|^2
\times \sum_{w=1}^{v}\left| \delta \tilde{V}^p_w \right|^2\\
&\leq	
\frac{5}{N^2\sigma^4}
\sum_{p\in I_n} \sum_{w=0}^{v} \left| \tilde{L}^{  p}_{\emp(V_n)}(v \eta_m,w \eta_m)-
\tilde{L}^{q_m,  p}_{\emp(V_n)}(v \eta_m,w \eta_m) \right|^2  \times \sum_{p\in I_n} \sum_{w=1}^{v} \left|\delta \tilde{V}^p_w\right|^2.
\label{eq:boundalpha2}
\end{align}
By \eqref{eq:Ltildebar}, for all $s,\,u \in [0, t]$ and for all $t \in [0, T]$, we have 
\[
\tilde{L}^{ p}_{\emp(V_n)}(s , u) - \tilde{L}^{q_m,  p}_{\emp(V_n)}(s , u) = \sigma^2 \left( {\rm Id} + \sigma^{-2} \bar{\tilde{K}}^{q_m,  p}_{\emp(V_n)} \right)^{-1}(s , u) - \sigma^2 \left( {\rm Id} + \sigma^{-2} \bar{\tilde{K}}^{ p}_{\emp(V_n)} \right)^{-1}(s , u).
\]
By the identity $A^{-1}-B^{-1}=A^{-1}(B-A)B^{-1}$
\begin{multline*}
\sigma^2 \left( {\rm Id} + \sigma^{-2} \bar{\tilde{K}}^{q_m,  p}_{\emp(V_n)} \right)^{-1}(s , u) - \sigma^2 \left( {\rm Id} + \sigma^{-2} \bar{\tilde{K}}^{ p}_{\emp(V_n)} \right)^{-1}(s , u)=\\
\left( {\rm Id} + \sigma^{-2} \bar{\tilde{K}}^{q_m,  p}_{\emp(V_n)} \right)^{-1}  \circ \left( \bar{\tilde{K}}^{ p}_{\emp(V_n)} -  \bar{\tilde{K}}^{q_m,  p}_{\emp(V_n)} \right) \circ \left( {\rm Id} + \sigma^{-2} \bar{\tilde{K}}^{ p}_{\emp(V_n)} \right)^{-1} (s , u),
\end{multline*}
where $\circ$ indicates the composition of the operators. By Remark~\ref{rem:opcomposition} in Appendix \ref{app:covariances}
we have 
\begin{multline*}
\tilde{L}^{ p}_{\emp(V_n)}(s , u) - \tilde{L}^{q_m,  p}_{\emp(V_n)}(s , u) = \\
 \int_0^t \int_0^t \left( {\rm Id} + \sigma^{-2} \bar{\tilde{K}}^{q_m,  p}_{\emp(V_n)} \right)^{-1}(s, x) \left( \tilde{K}^{ p}_{\emp(V_n)} -  \tilde{K}^{q_m,  p}_{\emp(V_n)} \right)(x,y) \left( {\rm Id} + \sigma^{-2} \bar{\tilde{K}}^{ p}_{\emp(V_n)} \right)^{-1}(y, u) \, dx\, dy,
\end{multline*}
for all $s,\,u \in [0, t]$ and for all $t \in [0, T]$.

We recall further that\footnote{This comes from the fact that, say for an $n \times n$ matrix $A$, but this is also true for general linear operators,
\[
\norm{A}_{max} = \max_{i, j} | A_{ij} | \leq \norm{A}_2 = \sigma_{max}(A) \text{ the largest singular value of } A
\]}
\begin{align}
\left( {\rm Id} + \sigma^{-2} \bar{\tilde{K}}^{q_m,   p}_{\emp(V_n)} \right)^{-1}(s, x)  & \leq 1, \nonumber\\
%\]
\mbox{ and }\quad
%\[
\left( {\rm Id} + \sigma^{-2} \bar{\tilde{K}}^{ p}_{\emp(V_n)} \right)^{-1}(s, x)  &\leq 1, \label{eq:boundondft}
%\]
\end{align}
we conclude that
\[
 \left| \tilde{L}^{ p}_{\emp(V_n)}(v \eta_m,w \eta_m) - \tilde{L}^{q_m, p}_{\emp(V_n)}(v \eta_m,w \eta_m) \right| \leq
 \int_0^{v \eta_m} \int_0^{v \eta_m}  \left|  \left( \tilde{K}^{ p}_{\emp(V_n)} -  \tilde{K}^{q_m,  p}_{\emp(V_n)} \right)(x,y)  \right| \, dx\,dy,
\]
and, by Cauchy-Schwarz, and $v \leq m$,
\[
 \left| \tilde{L}^{ p}_{\emp(V_n)}(v \eta_m,w \eta_m) - \tilde{L}^{q_m,  p}_{\emp(V_n)}(v \eta_m,w \eta_m) \right|^2 \leq
 T^2 \int_0^T \int_0^T \left|  \left( \tilde{K}^{ p}_{\emp(V_n)} -  \tilde{K}^{q_m,  p}_{\emp(V_n)} \right)(x,y)  \right|^2 \, dx\,dy,
\]
so that, by \eqref{eq:boundalpha2}, 
\begin{multline*}
\alpha^2_{v \eta_m} \leq  
\frac{5 T^2 }{N^2\sigma^4} \sum_{p\in I_n}  \int_0^T \int_0^T  \left|  \left( \tilde{K}^{ p}_{\emp(V_n)} -  \tilde{K}^{q_m,  p}_{\emp(V_n)} \right)(x,y)  \right|^2 \, dx\,dy \times  \sum_{p\in I_n} \sum_{w=1}^{v} \left|\delta \tilde{V}^p_w\right|^2  ,
\end{multline*}
and, by Parseval's theorem,
\[
\alpha^2_{v \eta_m} \leq  
\frac{5 T^2}{\sigma^4} \sum_{k \in I_n}  \int_0^T \int_0^T  \left(  \left( K^{  k}_{\emp(V_n)}  -  K^{q_m,  k}_{\emp(V_n)}\right)(x,y)  \right)^2 \, dx\,dy   \times  \sum_{k \in I_n}  \sum_{w=1}^{v} \left|\delta V^k_w\right|^2.
\]
{\bf Step 2: Choose $m$ and $q_m$ as functions of $n$}\\
We observe that $K^{q_m,  k}_{\emp(V_n)}$ is equal to $K^{  k}_{\emp(V_n)}$ over the set $I_{q_m}$ and to 0 over the complement of $I_{q_m}$ in $I_n$, their common value being 
\[
K^{ k}_{\emp(V_n)}(x,y)=\sum_{h \in I_n} R_\J(k,h)\frac{1}{N} \sum_{l \in I_n} f(V^l_x) f(V^{l+h}_y),
\]
so that we have
\[
\alpha_{v \eta_m}^2 \leq \frac{5T^2}{\sigma^4}   \sum_{k \in I_n \setminus I_{q_m}}   \int_0^T \int_0^T  \left( \sum_{h \in I_n}R_\J(k,h)\frac{1}{N} \sum_{l \in I_n} f(V^l_x) f(V^{l+h}_y)\right)^2 \, \, dx\,dy  \times  \sum_{k \in I_n}  \sum_{w=1}^{v} \left|\delta V^k_w\right|^2.
\]
Because $0 \leq f \leq 1$ we have
\[
\alpha^2_{v \eta_m} \leq  \frac{5T^4}{\sigma^4}  \sum_{k \in I_n \setminus I_{q_m}} \left( \sum_{h \in I_n}\left| R_\J(k,h)\right| \right)^2 \times  \sum_{k \in I_n} \sum_{w=1}^{v} \left|\delta V^k_w\right|^2.
\]
Define
\[
\psi(n,q_m):= \sum_{k \in I_n \setminus I_{q_m}} \left( \sum_{h \in I_n}\left| R_\J(k,h)\right| \right)^2.
\]
By choosing $q_m$ as a function of $m$, and $m$ as a function of $n$, $\psi(n,q_m)$ can be made arbitrarily small for large $n$ and $m$ . 
We have 
\[
\alpha^2_{v \eta_m} \leq  \frac{5T^4}{\sigma^4}  \psi(n,q_m)  \sum_{k \in I_n}  \sum_{w=1}^{v} \left|\delta V^k_w\right|^2.
\]
As before, we neglect the contribution of the drift term in \eqref{eq: theta SDE} and write that, for $m$, $n$ large enough
\[
\alpha^2_{v \eta_m} \leq  \frac{5T^4}{\sigma^4}  \psi(n,q_m)  \sum_{k \in I_n}  \sum_{w=1}^{v} \left|\delta W^k_w\right|^2.
\]
Let us define
\begin{equation}\label{eq:white}
\delta W^k_w:=\sqrt{\frac{T}{m}} \, \xi^{w,k},\, k \in I_n,\, w=1,\cdots,m,
\end{equation}
 where the $\xi^{w,k}$s are i.i.d. $\mathcal{N}(0,1)$.
Using \eqref{eq:white}, we have
\[
\sup_{v =0,\cdot,m-1} \alpha^2_{v \eta_m} \leq   \frac{5T^5}{\sigma^4}  N \psi(n,q_m)\frac{1}{Nm}  \sum_{k\in I_n}  \sum_{w=1}^{m} (\xi^{w,k})^2.
\]
Define 
$\varphi(n,m):= 5T^5  N \psi(n,q_m)/\sigma^4$ 
and assume that we have chosen $\psi(n,q_m)$ such that $\lim_{n,m \to \infty} \varphi(n,m)=0$. 
\begin{remark}\label{rem:choice of m}
Because of \eqref{eq:akbl} we have 
\[
\psi(n,q_m) \leq A b^2 \sum_{k = q_m+1}^n \frac{1}{k^4} \leq A b^2 (n-q_m) \frac{1}{(q_m+1)^4},
\]
for some $A > 0$ independent of $n$ and $m$, and therefore
\[
\varphi(n,m) \leq B (2n+1)(n-q_m)\frac{1}{(q_m+1)^4}
\]
with $B=A b^2T^5$.  Now choose $q_m=ng(m)$ with $g(m) \leq 1$. It follows that
\[
 (2n+1) (n-q_m) \frac{1}{(q_m+1)^4}=\frac{1}{n^2}(2+\frac{1}{n})(1-g(m))\frac{1}{\left(g(m)+\frac{1}{n}\right)^4}.
\]
	At this step, any choice of \(g \leq 1\) yields to $\lim_{n,m \to \infty} \varphi(n,m)=0$.
\end{remark}
\noindent
{\bf Step 3: Apply Cramer's Theorem and conclude}\\
Next we set $A:=\frac{\epsilon}{3 T C \sigma^2}$ and have
\begin{equation}\label{eq:int114}
Q^n(  \sup_{v =0,\cdot,m-1}  \alpha^2_{v \eta_m} \geq \frac{\epsilon}{3 T C \sigma^2} ) %\leq Q^n \left(\varphi(n,m) \frac{1}{Nm}  \sum_{k\in I_n}  \sum_{w=1}^{m} (\xi^{w,k})^2  \geq A\right)=\\
 = Q^n\left( \frac{1}{Nm} \sum_{k\in I_n} \sum_{w=1}^m (\xi^{w,k})^2 \geq  \frac{A}{ \varphi(n,m)}  \right).
\end{equation}
Since $\lim_{n,m \to \infty} \varphi(n,m)=0$ we can choose $n_0$ and $m_0$ such that $ \frac{A}{\varphi(n,m)}  > 1$ for $n \geq n_0$ and $m \geq m_0$, 1 being the mean of $ (\xi^{w,k})^2$. Let $\rho:= \frac{A}{\varphi(n_0,m_0)} $. We have
\[
Q^n\left( \frac{1}{Nm} \sum_{k\in I_n} \sum_{w=1}^m (\xi^{w,k})^2 \geq  \frac{A}{ \varphi(n,m)}  \right) \leq Q^n\left( \frac{1}{Nm} \sum_{k\in I_n} \sum_{w=1}^m (\xi^{w,k})^2 \geq \rho \right),
\]
as soon as $n \geq n_0$ and $m \geq m_0$.
We conclude thanks to Cramer's Theorem. We state in the following Lemma
	a version adapted to our setting.
\begin{lemma}\label{lem:cramer}
Let $\xi^{w,k}$, $w=0,\cdots,m-1$, $k \in I_n$, be a sequence of i.i.d. $\mathcal{N}(0,1)$ random variables under $Q^n$, and $\rho > 1$. There exists $\alpha > 0$ depending on $\rho$ such that
\[
\lsup{n} \frac{1}{N} \log Q^n\left( \frac{1}{Nm} \sum_{k\in I_n} \sum_{w=0}^{m-1} (\xi^{w,k})^2 \geq \rho \right) \leq -m \alpha.
\]
\end{lemma}
\begin{proof}
	See \cite[Th. 2.2.3]{dembo-zeitouni:97}.
\end{proof}
According to this Lemma there exists $\alpha(\rho) > 0$ such that
\[
\frac{1}{N} \log Q^n \left( \frac{1}{Nm} \sum_{k\in I_n} \sum_{w=1}^m (\xi^{w,k})^2 \geq \frac{A}{ \varphi(n,m)} \right) \leq 
-m \alpha(\rho).
\]
Combining this with \eqref{eq:int114} we obtain
\[
\frac{1}{N} \log Q^n \left(\sup_{s \in [0,T]} \alpha^2_{s^{(m)}} \geq  \frac{\epsilon}{3 T C \sigma^2 } \right) \leq  -m \alpha(\rho),
\]
as soon as $n \geq n_0$ and $m \geq m_0$.
This completes the proof.
\end{proof}
%%%%%%
%%%%%%

\begin{proof}[{\bf Proof of Lemma \ref{lem:alpha3}}]
The proof is based on a comparison of the length $N$ DFTs of a sequence of length $N$ and of the same sequence of length $Q_m$ padded with $N-Q_m$ zeroes followed by the use of Cramer's Theorem, i.e. Lemma \ref{lem:cramer}.\\
{\bf Step 1: Fourier analysis}\\
We have, with $s^{(m)}=v \eta_m$,
\[
\alpha^3_{v \eta_m} = \frac{5}{N^2} \sum_{p \in I_n}  \left| \sigma^{-2} ( \bar{\tilde{L}}^{  p}_{\emp(V^m_n)}\delta \tilde{V}^{m, p})(v \eta_m)-\tilde{\theta}^{m, p}_s\right|^2.
\]
By equations \eqref{eq:Ltildebar} and \eqref{eq:Kqmp} 
\begin{align*}
 (\bar{\tilde{L}}^{ p}_{\emp(V^m_n)}\delta \tilde{V}^{m, p})(v \eta_m) & =  \sigma^2 \sum_{w=0}^{v} \delta \tilde{V}^{m,p}_w-  \sigma^2 \sum_{w=0}^{v} (1+\sigma^{-2}\bar{\tilde{K}}^{  p}_{\emp(V^m_n)})^{-1}(v \eta_m,w \eta_m) \delta \tilde{V}^{m, p}_w  \\
& =\sigma^2 \sum_{w=0}^{v} \delta \tilde{V}^{m,p}_w-  \sigma^2 \sum_{w=0}^{v} (1+\sigma^{-2}\bar{\tilde{K}}^{ v \eta_m}_{\emp(V^m_n)}(\frac{2 \pi p}{N}))^{-1}(v \eta_m,w \eta_m) \delta \tilde{V}^{m, p}_w.
\end{align*}
By \eqref{eq:vmtthetamt} and \eqref{eq:Ktilde} we have
\[
\theta^{m,j}_s = \sigma^{-2} \sum_{k \in I_{q_m}} \sum_{w=0}^{v} L^{q_m,k}_{\emp(V^m_n)}(v \eta_m,w \eta_m)\delta V^{m, k+j}_w.
\]
Taking the length $N$ DFT of both sides and using Lemma \ref{lem:translation} we obtain for $p \in I_n$
\[
\tilde{\theta}^{m,p}_s = \sigma^{-2} \sum_{k \in I_{q_m}} F_N^{kp} \sum_{w=0}^{v} L^{q_m,k}_{\emp(V^m_n)}(v \eta_m,w \eta_m)\delta \tilde{V}^{m, p}_w,
\]
where $F_N=e^{\frac{2\pi i}{N}}$. The relation
\[
L^{q_m,k}_{\emp(V^m_n)}(v \eta_m,w \eta_m) = \frac{1}{Q_m} \sum_{q \in I_{q_m}} \tilde{L}^{q_m,q}_{\emp(V^m_n)}(v \eta_m,w \eta_m) F_{Q_m}^{qk},
\]
where $F_{Q_m}=e^{\frac{2i\pi }{Q_m}}$, implies 
\[
\tilde{\theta}^{m, p}_s=
\sigma^{-2} Q_m^{-1}\sum_{k,q\in I_{q_m}} F_N^{kp} F_{Q_m}^{kq}\sum_{w=0}^{v}\tilde{L}^{q_m,  q}_{\emp(V^m_n)}(v \eta_m,w \eta_m)\delta \tilde{V}^{m, p}_w.
\]
According to \eqref{eq:Lmubarn1},
\[
\bar{\tilde{L}}^{q_m,  q}_{\emp(V^m_n)}=\sigma^2 \left( {\rm Id} - ({\rm Id}+\sigma^{-2} \bar{\tilde{K}}^{q_m,  q}_{\emp(V^m_n)})^{-1}\right) = 
\sigma^2 \left( {\rm Id} - \left( {\rm Id}+\sigma^{-2} \bar{\tilde{K}}^{q_m}_{\emp(V^m_n)}\left( \frac{2\pi q}{Q_m} \right) \right)^{-1}\right),
\]
so that we have,
	using $\sum_{q \in I_{q_m}} F_{Q_m}^{kq}=Q_m \delta_k$, where $\delta_k=1$ if $k=0$ and 0 otherwise.
	 And therefore $\sum_{k,q\in I_{q_m}} F_N^{kp} F_{Q_m}^{kq}=Q_m$,
\begin{multline*}
\tilde{\theta}^{m, p}_s= \\
  \sum_{w=0}^{v}\delta \tilde{V}^{m, p}_w-  \sum_{w=0}^{v} \sum_{k \in I_{q_m}} \frac{1}{2\pi} \sum_{q \in I_{q_m}} e^{-ik (\frac{2\pi q}{Q_m}-\frac{2\pi p}{N})}\left( {\rm Id}+\sigma^{-2} \bar{\tilde{K}}^{q_m}_{\emp(V^m_n)}\left( \frac{2 \pi q}{Q_m} \right) \right)^{-1}(v \eta_m,w \eta_m) \frac{2\pi}{Q_m}\delta \tilde{V}^{m, p}_w.
\end{multline*}
We conclude that
\begin{multline*}
\sigma^{-2} ( \bar{\tilde{L}}^{q_m, p}_{\emp(V^m_n)}\delta \tilde{V}^{m, p})(v \eta_m)-\tilde{\theta}^{m, p}_s = \\
\sum_{w=0}^{v} \Bigg( \sum_{k \in I_{q_m}} \frac{1}{2\pi} \sum_{q \in I_{q_m}} e^{-ik (\frac{2\pi q}{Q_m}-\frac{2\pi p}{N})} \left( {\rm Id}+\sigma^{-2} \bar{\tilde{K}}^{q_m}_{\emp(V^m_n)}\left( \frac{2 \pi q}{Q_m} \right) \right)^{-1}(v \eta_m,w \eta_m)\frac{2\pi }{Q_m} - \\
  \left( {\rm Id}+\sigma^{-2} \bar{\tilde{K}}^{q_m}_{\emp(V^m_n)}\left( \frac{2 \pi p}{N} \right) \right)^{-1}(v \eta_m,w \eta_m) \Bigg) \delta \tilde{V}^{m, p}_w.
\end{multline*}
With a slight abuse of notation and ignoring the time dependency for the moment we write
\[
\left( {\rm Id}+\sigma^{-2} \bar{\tilde{K}}^{q_m}_{\emp(V^m_n)}\left( \frac{2 \pi q}{Q_m} \right) \right)^{-1}(v \eta_m,w \eta_m) =
\frac{1}{1 + \sigma^{-2} \bar{\tilde{K}}^{q_m}_{\emp(V^m_n)}\left( \frac{2 \pi q}{Q_m} \right)},
\]
and 
\[
\left( {\rm Id}+\sigma^{-2} \bar{\tilde{K}}^{q_m}_{\emp(V^m_n)}\left( \frac{2 \pi p}{N} \right) \right)^{-1}(v \eta_m,w \eta_m) =
\frac{1}{1 + \sigma^{-2} \bar{\tilde{K}}^{q_m}_{\emp(V^m_n)}\left( \frac{2 \pi p}{N} \right)}.
\]
Because
\[
\frac{1}{1 + \sigma^{-2}\bar{\tilde{K}}^{q_m}_{\emp(V^m_n)}\left( \frac{2 \pi p}{N} \right)}=\int_{-\pi}^\pi \frac{\delta(\varphi - \frac{2\pi p}{N})}{1 + \sigma^{-2} \bar{\tilde{K}}^{q_m}_{\emp(V^m_n)}\left(\varphi\right)}\,d\varphi,
\]
and
\[
\frac{1}{2\pi} \sum_{k \in \Z} e^{- ik (\varphi - \frac{2\pi p}{N})} = \delta(\varphi - \frac{2\pi p}{N}),
\]
we have
\begin{multline*}
\sum_{k \in I_{q_m}} \frac{1}{2\pi} \sum_{q \in I_{q_m}} \frac{e^{-ik (\frac{2\pi q}{Q_m}-\frac{2\pi p}{N})}}{1 + \sigma^{-2} \bar{\tilde{K}}^{q_m}_{\emp(V^m_n)}\left( \frac{2 \pi q}{Q_m} \right)}\frac{2\pi}{Q_m}-\frac{1}{1 + \sigma^{-2}\bar{\tilde{K}}^{q_m}_{\emp(V^m_n)}\left( \frac{2 \pi p}{N} \right)}=\\
\sum_{k \in I_{q_m}} \frac{1}{2\pi} \sum_{q \in I_{q_m}} \frac{e^{-ik (\frac{2\pi q}{Q_m}-\frac{2\pi p}{N})}}{1 + \sigma^{-2} \bar{\tilde{K}}^{q_m}_{\emp(V^m_n)}\left( \frac{2 \pi q}{Q_m} \right)}\frac{2\pi}{Q_m}- 
\int_{-\pi}^\pi \frac{\delta(\varphi - \frac{2\pi p}{N})}{1 + \sigma^{-2} \bar{\tilde{K}}^{q_m}_{\emp(V^m_n)}\left(\varphi\right)}\,d\varphi=\\
\sum_{k \in I_{q_m}} e^{2 \pi ik \frac{ p}{N}} \left(\frac{1}{2\pi} \sum_{q \in I_{q_m}} \frac{e^{-2\pi ik \frac{ q}{Q_m}}}{1 + \sigma^{-2} \bar{\tilde{K}}^{q_m}_{\emp(V^m_n)}\left( \frac{2 \pi q}{Q_m} \right)}\frac{2\pi}{Q_m} -\frac{1}{2\pi}
\int_{-\pi}^\pi  \frac{e^{- ik \varphi}}{1 + \sigma^{-2} \bar{\tilde{K}}^{q_m}_{\emp(V^m_n)}\left(\varphi\right)}\,d\varphi\right)-\\
\sum_{k \in \Z - I_{q_m}} e^{2 \pi ik \frac{ p}{N}}\frac{1}{2\pi} \int_{-\pi}^\pi  \frac{e^{-ik \varphi}}{1 + \sigma^{-2} \bar{\tilde{K}}^{q_m}_{\emp(V^m_n)}\left(\varphi\right)}\,d\varphi.
\end{multline*}
Define
\[
\forall \varphi \in [-\pi,\,\pi], \quad h(\varphi):= \frac{e^{-ik \varphi}}{1 + \sigma^{-2} \bar{\tilde{K}}^{q_m}_{\emp(V^m_n)}\left(\varphi\right)} 
\quad \text{ and } \quad
\Delta \varphi \:= \frac{2 \pi}{Q_m},
\]
and write
\[
\frac{1}{2\pi} \sum_{q \in I_{q_m}} \frac{e^{-2\pi ik \frac{ q}{Q_m}}}{1 + \sigma^{-2} \bar{\tilde{K}}^{q_m}_{\emp(V^m_n)}\left( \frac{2 \pi q}{Q_m} \right)}\frac{2\pi}{Q_m}=
\frac{1}{2\pi} \sum_{q=0}^{2q_m} h(-\pi+\frac{\Delta \varphi}{2}+q \Delta \varphi).
\]
This shows that the first term in the left hand side of the previous equations is the Riemann sum, corresponding to the midpoint rule, approximating $\int_{-\pi}^\pi h(\varphi)\,d\varphi$. This implies  that
\[
\left| \frac{1}{2\pi} \sum_{q \in I_{q_m}} \frac{e^{-2\pi ik \frac{ q}{Q_m}}}{1 + \sigma^{-2} \bar{\tilde{K}}^{q_m}_{\emp(V^m_n)}\left( \frac{2 \pi q}{Q_m} \right)}\frac{2\pi}{Q_m} -\frac{1}{2\pi}
\int_{-\pi}^\pi  \frac{e^{- ik \varphi}}{1 + \sigma^{-2} \bar{\tilde{K}}^{q_m}_{\emp(V^m_n)}\left(\varphi\right)}\,d\varphi \right| \leq \frac{ D}{Q_m^2},
\]
where $D$ is a positive constant that depends on the maximum value of the magnitude of the second order derivative of $h$ over the interval $[-\pi,\,\pi]$, hence bounded. Therefore we have proved that 
\begin{multline*}
\left| \sum_{k \in I_{q_m}} e^{2 \pi ik \frac{ p}{N}} \left(\frac{1}{2\pi} \sum_{q \in I_{q_m}} \frac{e^{-2\pi ik \frac{ q}{Q_m}}}{1 + \sigma^{-2} \bar{\tilde{K}}^{q_m}_{\emp(V^m_n)}\left( \frac{2 \pi q}{Q_m} \right)}\frac{2\pi}{Q_m} -\frac{1}{2\pi}
\int_{-\pi}^\pi  \frac{e^{- ik \varphi}}{1 + \sigma^{-2} \bar{\tilde{K}}^{q_m}_{\emp(V^m_n)}\left(\varphi\right)}\,d\varphi \right)\right| \leq \\
 \frac{D}{Q_m},\,\forall p \in I_n.
\end{multline*}
We now consider the term $\frac{1}{2\pi} \int_{-\pi}^\pi  \frac{e^{-ik \varphi}}{1 + \sigma^{-2} \bar{\tilde{K}}^{q_m}_{\emp(V^m_n)}\left(\varphi\right)}\,d\varphi $. It is the $k$th coefficient in the Fourier series of the periodic function $\varphi \to \frac{1}{1 + \sigma^{-2} \bar{\tilde{K}}^{q_m}_{\emp(V^m_n)}\left(\varphi\right)}$. Since $1 + \sigma^{-2} \bar{\tilde{K}}^{q_m}_{\emp(V^m_n)}\left(\varphi\right)$ is positive, three times differentiable with a bounded third order derivative, see Lemma \ref{lem:FTKmut},  a standard result in Fourier analysis indicates that this coefficient is $\smallO{1/|k|^3}$. Since $\sum_{h =|k|}^\infty \frac{1}{h^3}$ is of order $\mathcal{O}(1/k^2)$, we conclude that for $Q_m$ large enough
\[
\left| \sum_{k \in \Z - I_{q_m}} e^{2 \pi ik \frac{ p}{N}}\frac{1}{2\pi} \int_{-\pi}^\pi  \frac{e^{-ik \varphi}}{1 + \sigma^{-2} \bar{\tilde{K}}^{q_m}_{\emp(V^m_n)}\left(\varphi\right)}\,d\varphi \right| \leq \frac{D}{Q_m^2},
\]
for some constant $D > 0$.

Reintroducing the time dependency, and by Cauchy-Schwarz on the $w$ index, we have therefore proved that for $Q_m$ large enough

\begin{multline*}
\Bigg|  \sum_{w=0}^{v} \Bigg( \sum_{k \in I_{q_m}} \frac{1}{2\pi} \sum_{q \in I_{q_m}} e^{-ik (\frac{2\pi q}{Q_m}-\frac{2\pi p}{N})} \left( {\rm Id}+\sigma^{-2} \bar{\tilde{K}}^{q_m}_{\emp(V^m_n)}\left( \frac{2 \pi q}{Q_m} \right) \right)^{-1}(v \eta_m,w \eta_m)\frac{2\pi }{Q_m} - \\
  \left( {\rm Id}+\sigma^{-2} \bar{\tilde{K}}^{q_m}_{\emp(V^m_n)}\left( \frac{2 \pi p}{N} \right) \right)^{-1}(v \eta_m,w \eta_m) \Bigg) \delta \tilde{V}^{m, p}_w \Bigg| \leq \\
 \Bigg( \sum_{w=0}^{v} \Bigg| \sum_{k \in I_{q_m}} \frac{1}{2\pi} \sum_{q \in I_{q_m}} e^{-ik (\frac{2\pi q}{Q_m}-\frac{2\pi p}{N})} \left( {\rm Id}+\sigma^{-2} \bar{\tilde{K}}^{q_m}_{\emp(V^m_n)}\left( \frac{2 \pi q}{Q_m} \right) \right)^{-1}(v \eta_m,w \eta_m)\frac{2\pi }{Q_m} - \\
  \left( {\rm Id}+\sigma^{-2} \bar{\tilde{K}}^{q_m}_{\emp(V^m_n)}\left( \frac{2 \pi p}{N} \right) \right)^{-1}(v \eta_m,w \eta_m)  \Bigg|^2 \Bigg)^{1/2} \times \left( \sum_{w=0}^{v} | \delta \tilde{V}^{m, p}_w |^2 \right)^{1/2} \leq \\
\frac{D}{Q_m} \left( \sum_{w=0}^{v} | \delta \tilde{V}^{m, p}_w |^2 \right)^{1/2},
\end{multline*}
for some constant $D > 0$,
and therefore that
\[
\alpha^3_{v \eta_m} = \frac{5}{N^2} \sum_{p \in I_n}  \left| \sigma^{-2} ( \bar{\tilde{L}}^{q_m, p}_{\emp(V^m_n)}\delta \tilde{V}^{m, p})(v \eta_m)-\tilde{\theta}^{m, p}_s\right|^2 \leq \frac{ 5 D^2 }{N^2 Q_m^2} \sum_{p \in I_n} \sum_{w=0}^{v} | \delta \tilde{V}^{m, p}_w |^2,
 \]
 so that, by Parseval's theorem
 \[
 \alpha^3_{v \eta_m}  \leq \frac{ 5D^2 }{N Q_m^2} \sum_{k \in I_n}\sum_{w=0}^{v} | \delta V^{m, k}_w |^2.
 \]
 {\bf Step 2: Apply Cramer's Theorem and conclude}\\
As in previous proofs, Lemma \ref{Lemma: bound eta} allows us to neglect the contribution of the drift terms $\theta^{m,p}$ in the above so that we are interested in upper bounding the probability that the quantity $\frac{ 5 D^2 \sigma^2 }{N Q_m^2} \sum_{k \in I_n}\sum_{w=0}^{v} | \delta W^{k}_w |^2 = \frac{ 5 D^2 \sigma^2 T}{N m Q_m^2} \sum_{k \in I_n} \sum_{w=0}^v \left( \xi^{w,k}  \right)^2$ is larger than $\frac{\epsilon}{3 T C \sigma^2}$.
Following the same strategy as in the end of the proof of Lemma \ref{lem:alpha2}, we choose $m_0$ such that $\rho:= \frac{Q_{m_0}^2 \epsilon}{ 15  C T^2 D^2 \sigma^{4}  } > 1$. Applying again Lemma \ref{lem:cramer} shows that there exists $\alpha(\rho) > 0$ such that
\[
\frac{1}{N} \log Q^n (\sup_{s \in [0,T]} \alpha^3_{s^{(m)}} \geq  \frac{\epsilon}{3 T C \sigma^2}) \leq  -m \alpha(\rho),
\]
as soon as $m \geq m_0$.
This completes the proof.
 \end{proof}

%%%
%%%
\begin{proof}[{\bf Proof of Lemma \ref{lem:alpha4}}]\ \\
The proof uses the idea of  writing an upper bound of $\alpha^4_{v\eta_m}$ as a sum of three terms and upper bounding each of the three terms. We only provide the proof for one of the three terms, the one requiring the more work.\\
{\bf Step 1: An upper bound for $\alpha^4_{v\eta_m}$}\\
%We recall that $\eta_m = T/m$ so that $t_v^{(m)}=v \eta_m$.
We go back to the initial definition of $\bar{\tilde{L}}^{q_m, p}_{\emp(V_n)}$ and $\bar{\tilde{L}}^{q_m, p}_{\emp(V_n^m)}$, see \eqref{eq:tildeLdefdis}, to write the expression for $\alpha^4_{v\eta_m}$ in \eqref{eq:tildetheta-eta} as
\begin{multline}
\alpha^4_{v\eta_m} = \frac{5}{N^2\sigma^4} \sum_{p \in I_n} \left| \left( \left( \bar{\tilde{L}}^{q_m, p}_{\emp(V_n)}-\bar{\tilde{L}}^{q_m, p}_{\emp(V_n^m)}  \right)\delta \tilde{V}^p \right)(v\eta_m) \right|^2 = \\
\frac{5}{N^4\sigma^4}  \sum_{p\in I_n}  \left| \Exp^{\gamma^{\emp(V_n)}}\left[ \tilde{\Lambda}^p_{v \eta_m}(\tilde{G}^{c,m}) \tilde{G}^{c,m, -p}_{v\eta_m} \int_0^{v\eta_m} \tilde{G}^{c,m,p}_{r^{(m)}} \, d\tilde{V}^p_r - \tilde{\Lambda}^p_{v \eta_m}(\tilde{G}^{m}) \tilde{G}^{m, -p}_{v\eta_m} \int_0^{v\eta_m} \tilde{G}^{m,p}_{r^{(m)}} \, d\tilde{V}^{p}_r\right]  \right|^2, \label{eq:alpha4}
\end{multline}
where $\tilde{G}^{c,m, p}$ is the length $N$ DFT obtained by padding with $N-Q_m$ zeros the length $Q_m$ stationary periodic sequence
\begin{equation}\label{eq:Gnmit}
G^{c,m, j}_t = \sum_{k \in I_n} J^{j k}_{n, m} f(V^k_t),\,j \in I_{q_m},
\end{equation}
and $\tilde{G}^{m, p}$ is the length $N$ DFT obtained by padding with $N-Q_m$ zeros the length $Q_m$ stationary periodic sequence
\begin{equation}\label{eq:Gmit}
G^{m, j}_t = \sum_{k \in I_n} J^{j k}_{n, m} f(V^{m, k}_t),\,j \in I_{q_m}.
\end{equation}
The coefficients $(J^{j k}_{n,m })_{j \in I_{q_m},\,k \in I_n}$ are defined in \eqref{eq:Hdef} and \eqref{lccl}.
In order to proceed, we upper bound the right  hand side of \eqref{eq:alpha4} by a sum of three terms 
\begin{multline*}
\alpha^4_{v\eta_m} \leq 
\frac{15 }{N^4\sigma^{4}}  \sum_{p\in I_n}  \underbrace{\left| \Exp^{\gamma^{\emp(V_n)}}\left[ \left(\tilde{\Lambda}^p_{v \eta_m}(\tilde{G}^{c,m}) -  \tilde{\Lambda}^p_{v \eta_m}(\tilde{G}^{m})  \right) \tilde{G}^{c,m, -p}_{v\eta_m} \int_0^{v\eta_m} \tilde{G}^{c,m,p}_{r^{(m)}} \, d\tilde{V}^p_r \right]  \right|^2}_{\alpha^{4,1,p}_{v \eta_m}} + \\
\frac{15 }{N^4\sigma^4}  \sum_{p\in I_n}  \underbrace{\left| \Exp^{\gamma^{\emp(V_n)}}\left[ \tilde{\Lambda}^p_{v \eta_m}(\tilde{G}^{m}) \left( \tilde{G}^{c,m, -p}_{v\eta_m} -  \tilde{G}^{m, -p}_{v\eta_m} \right) \int_0^{v\eta_m} \tilde{G}^{c,m,p}_{r^{(m)}} \, d\tilde{V}^p_r  \right]  \right|^2}_{\alpha^{4,2,p}_{v \eta_m}}+\\
\frac{15 }{N^4\sigma^{4}}  \sum_{p\in I_n}  \underbrace{\left| \Exp^{\gamma^{\emp(V_n)}}\left[ \tilde{\Lambda}^p_{v \eta_m}(\tilde{G}^{m}) \tilde{G}^{m, -p}_{v\eta_m} \int_0^{v\eta_m} \left( \tilde{G}^{c,m,p}_{r^{(m)}} -  \tilde{G}^{m,p}_{r^{(m)}} \right) \, d\tilde{V}^{p}_r \right]  \right|^2}_{\alpha^{4,3,p}_{v \eta_m}},
\end{multline*}
and show that for any $M > 0$, all $m \in \N$, there exists a constant $\mathfrak{c} > 0$ such that for all $\epsilon \leq \exp(-\mathfrak{c}T) \delta^2/T$, all $0 \leq u \leq m$ and  all $0 \leq v \leq u$ 
\[
\lsup{n}\frac{1}{N} \log Q^n\left(\frac{15}{N^4\sigma^4} \sum_{p \in I_n} \alpha^{4,j,p}_{v\eta_m} \geq \frac{\epsilon \mathfrak{c}}{3 T C \sigma^2} \exp \left (v \eta_m\mathfrak{c} \right)  \text{ and } \tau(\epsilon,\mathfrak{c}) \geq u\eta_m\right) \leq -M\ j=1,2,3.
\]
The proofs are somewhat similar. We provide a proof for the most complicated term corresponding to $j=1$ and leave it to the reader to provide proofs for the cases $j=2,3$.\\
{\bf Step 2: Upper bounding $\left| \tilde{\Lambda}^p_{v \eta_m}(\tilde{G}^{c,m}) -  \tilde{\Lambda}^p_{v \eta_m}(\tilde{G}^{m})  \right|$}\\
We first recall the definitions of $ \tilde{\Lambda}^p_{v \eta_m}(\tilde{G}^{c,m})$ and $\tilde{\Lambda}^p_{v \eta_m}(\tilde{G}^{m})$:
\[
\tilde{\Lambda}^p_{v \eta_m}(\tilde{G}^{c,m}) = \frac{e^{-\frac{u_p}{N \sigma^2} \int_0^{v \eta_m} \left| \tilde{G}^{c,m,p}_s  \right|^2 \, ds}}
{\Exp^{\gamma^{\emp(V_n)}} \left[ e^{-\frac{u_p}{N \sigma^2} \int_0^{v \eta_m} \left| \tilde{G}^{c,m,p}_s  \right|^2 \, ds}  \right]},
\]
and
\[
\tilde{\Lambda}^p_{v \eta_m}(\tilde{G}^{m}) = \frac{e^{-\frac{u_p}{N \sigma^2} \int_0^{v \eta_m} \left| \tilde{G}^{m,p}_s  \right|^2 \, ds}}
{\Exp^{\gamma^{\emp(V_n)}} \left[ e^{-\frac{u_p}{N \sigma^2} \int_0^{v \eta_m} \left| \tilde{G}^{m,p}_s  \right|^2 \, ds}  \right]},
\]
with $u_p=1$ if $p \neq 0$ and $u_0=1/2$, see \eqref{eq:Lambdapt}.
First note
\begin{multline*}
 \tilde{\Lambda}^p_{v \eta_m}(\tilde{G}^{c,m}) -  \tilde{\Lambda}^p_{v \eta_m}(\tilde{G}^{m}) =  
\frac{e^{-\frac{u_p}{N \sigma^2} \int_0^{v \eta_m} \left| \tilde{G}^{c,m,p}_s  \right|^2 \, ds} - e^{-\frac{u_p}{N \sigma^2} \int_0^{v \eta_m} \left| \tilde{G}^{c,m}_s  \right|^2 \, ds}}
{\Exp^{\gamma^{\emp(V_n)}} \left[ e^{-\frac{u_p}{N \sigma^2} \int_0^{v \eta_m} \left| \tilde{G}^{c,m,p}_s  \right|^2 \, ds}  \right]}\\
 +\tilde{\Lambda}^p_{v \eta_m}(\tilde{G}^{m}) \dfrac{\Exp^{\gamma^{\emp(V_n)}} \left[ e^{-\frac{u_p}{N \sigma^2} \int_0^{v \eta_m} \left| \tilde{G}^{m,p}_s  \right|^2 \, ds}  \right] - \Exp^{\gamma^{\emp(V_n)}} \left[ e^{-\frac{u_p}{N \sigma^2} \int_0^{v \eta_m} \left| \tilde{G}^{c,m,p}_s  \right|^2 \, ds}  \right]}{\Exp^{\gamma^{\emp(V_n)}} \left[ e^{-\frac{u_p}{N \sigma^2} \int_0^{v \eta_m} \left| \tilde{G}^{c,m,p}_s  \right|^2 \, ds}  \right]}.
\end{multline*}
Now, as in the proof of Lemma~\ref{lem:alpha1}, we use the Lipschitz continuity of $x \to e^{-x}$ for $x \geq 0$:
\[
\left| e^{-x} - e^{-y} \right| \leq | x - y |,
\]
to obtain
\begin{align*}
\left| \tilde{\Lambda}^p_{v \eta_m}(\tilde{G}^{c,m}) -  \tilde{\Lambda}^p_{v \eta_m}(\tilde{G}^{m})  \right| 
&\leq 
\frac{u_p}{N \sigma^2} \frac{  \left| \int_0^{v\eta_m} \left( \left| \tilde{G}^{c,m,p}_s  \right|^2 -  \left| \tilde{G}^{m,p}_s  \right|^2 \right)\, ds  \right|}{\Exp^{\gamma^{\emp(V_n)}} \left[ e^{-\frac{u_p}{N \sigma^2} \int_0^{v \eta_m} \left| \tilde{G}^{m,p}_s  \right|^2 \, ds}  \right]  } 
\\
& \quad +\tilde{\Lambda}^p_{v \eta_m}(\tilde{G}^{m}) \frac{u_p}{N \sigma^2} \frac{\Exp^{\gamma^{\emp(V_n)}} \left[  \left| \int_0^{v\eta_m} \left( \left| \tilde{G}^{c,m,p}_s  \right|^2 -  \left| \tilde{G}^{m,p}_s  \right|^2 \right)\, ds  \right|  \right]}{\Exp^{\gamma^{\emp(V_n)}} \left[ e^{-\frac{u_p}{N \sigma^2} \int_0^{v \eta_m} \left| \tilde{G}^{m,p}_s  \right|^2 \, ds}  \right] }.
\end{align*}
Because $u_p = 1$ or $1/2$ and 
\[
0 < D \leq \Exp^{\gamma^{\emp(V_n)}} \left[ e^{-\frac{u_p}{N \sigma^2} \int_0^{v \eta_m} \left| \tilde{G}^{m,p}_s  \right|^2 \, ds}  \right] \leq 1 < \infty
\]
for some constant $D$  independent of $p$, $m$ and $N$ (see the proof of Lemma \ref{lem:alpha1}). So, we have
\begin{multline}\label{eq:int1lem321}
\left| \tilde{\Lambda}^p_{v \eta_m}(\tilde{G}^{c,m}) -  \tilde{\Lambda}^p_{v \eta_m}(\tilde{G}^{m})  \right| \leq \\
\frac{1}{ND\sigma^2 } 
\Bigg(  \left| \int_0^{v\eta_m} \left( \left| \tilde{G}^{c,m,p}_s  \right|^2 -  \left| \tilde{G}^{m,p}_s  \right|^2 \right)\, ds  \right|  + \\
\tilde{\Lambda}^p_{v \eta_m}(\tilde{G}^{m})  \Exp^{\gamma^{\emp(V_n)}} \left[  \left| \int_0^{v\eta_m} \left( \left| \tilde{G}^{c,m,p}_s  \right|^2 -  \left| \tilde{G}^{m,p}_s  \right|^2 \right)\, ds  \right|  \right]
 \Bigg).
\end{multline}
Given two complex numbers $x$ and $y$ with complex conjugates $x^*$ and $y^*$, it is clear that
\[
| \,|x|^2-|y^2|\, | = | (x-y) x^* + y(x^*-y^*) | \leq | x - y | \, ( | x^* | + | y | ) = | x - y | \, ( | x | + | y | ),
\]
and therefore, by Cauchy-Schwarz,
\begin{multline}\label{eq:int2lem321}
\left| \int_0^{v\eta_m} \left( \left| \tilde{G}^{c,m,p}_s  \right|^2 -  \left| \tilde{G}^{m,p}_s  \right|^2 \right)\, ds  \right| \leq
\left( \int_0^{v\eta_m} \left| \tilde{G}^{c,m,p}_s -  \tilde{G}^{m,p}_s   \right|^2 \, ds \right)^{1/2} \times \\
\left( 2 \int_0^{v\eta_m} \left( \left| \tilde{G}^{c,m,p}_s \right|^2 + \left| \tilde{G}^{m,p}_s \right|^2 \right) \, ds \right)^{1/2}.
\end{multline}
Combining \eqref{eq:int1lem321} and \eqref{eq:int2lem321} we obtain
\begin{multline*}
\left| \Exp^{\gamma^{\emp(V_n)}}\left[ \left(\tilde{\Lambda}^p_{v \eta_m}(\tilde{G}^{c,m}) -  \tilde{\Lambda}^p_{v \eta_m}(\tilde{G}^{m})  \right) \tilde{G}^{c,m, -p}_{v\eta_m} \int_0^{v\eta_m} \tilde{G}^{c,m,p}_{r^{(m)}} \, d\tilde{V}^p_r \right]  \right|^2 \leq \\
\frac{4}{N^2 D^2 \sigma^4 } \Bigg( \Exp^{\gamma^{\emp(V_n)}} \Bigg[ \left( \int_0^{v\eta_m} \left| \tilde{G}^{c,m,p}_s -  \tilde{G}^{m,p}_s   \right|^2 \, ds \right)^{1/2} \times
\left( \int_0^{v\eta_m} \left( \left| \tilde{G}^{c,m,p}_s \right|^2 + \left| \tilde{G}^{m,p}_s \right|^2 \right) \, ds \right)^{1/2} \times \\\left| \tilde{G}^{c,m, p}_{v\eta_m} \right| \,\left| \int_0^{v\eta_m} \tilde{G}^{c,m,p}_{r^{(m)}} \, d\tilde{V}^p_r \right| \Bigg] \Bigg)^2 +\\
\frac{4}{N^2 D^2 \sigma^4 } \left(  \Exp^{\gamma^{\emp(V_n)}} \left[  \left( \int_0^{v\eta_m} \left| \tilde{G}^{c,m,p}_s -  \tilde{G}^{m,p}_s   \right|^2 \, ds \right)^{1/2} \times 
\left( \int_0^{v\eta_m} \left( \left| \tilde{G}^{c,m,p}_s \right|^2 + \left| \tilde{G}^{m,p}_s \right|^2 \right) \, ds \right)^{1/2}  \right]  \right)^2 \times \\
\left( \Exp^{\gamma^{\emp(V_n)}} \left[ \tilde{\Lambda}^p_{v \eta_m}(\tilde{G}^{m})  \left| \tilde{G}^{c,m, p}_{v\eta_m} \right| \,\left| \int_0^{v\eta_m} \tilde{G}^{c,m,p}_{r^{(m)}} \, d\tilde{V}^p_r \right|    \right] \right)^2.
\end{multline*}

Three applications of Cauchy-Schwarz dictate
\begin{multline*}
\alpha^{4,1,p}_{v \eta_m} \leq 
\frac{4}{N^2 D^2 \sigma^4 } \Exp^{\gamma^{\emp(V_n)}} \left[  \int_0^{v\eta_m} \left| \tilde{G}^{c,m,p}_s - \tilde{G}^{m,p}_s   \right|^2 \, ds \right] \times \\
\Bigg(    
\Exp^{\gamma^{\emp(V_n)}} \left[ \int_0^{v\eta_m} \left( \left| \tilde{G}^{c,m,p}_s \right|^2 + \left| \tilde{G}^{m,p}_s \right|^2 \right) ds \left| \tilde{G}^{c,m, p}_{v\eta_m} \right|^2 \,\left| \int_0^{v\eta_m} \tilde{G}^{c,m,p}_{r^{(m)}} \, d\tilde{V}^p_r \right|^2 \right] + \\
\Exp^{\gamma^{\emp(V_n)}} \left[ \int_0^{v\eta_m} \left( \left| \tilde{G}^{c,m,p}_s \right|^2 + \left| \tilde{G}^{m,p}_s \right|^2 \right) \, ds  \right]
\times \Exp^{\gamma^{\emp(V_n)}} \left[ \tilde{\Lambda}^p_{v \eta_m}(\tilde{G}^{m})  \left| \tilde{G}^{c,m, p}_{v\eta_m} \right|^2  \right] \times \\
\Exp^{\gamma^{\emp(V_n)}} \left[   \tilde{\Lambda}^p_{v \eta_m}(\tilde{G}^{m}) \left| \int_0^{v\eta_m} \tilde{G}^{c,m,p}_{r^{(m)}} \, d\tilde{V}^p_r \right|^2 \right]
\Bigg)\\
\leq \frac{E}{N^2}\Exp^{\gamma^{\emp(V_n)}} \left[  \int_0^{v\eta_m} \left| \tilde{G}^{c,m,p}_s - \tilde{G}^{m,p}_s   \right|^2 \, ds \right] (A_1 + A_2),
\end{multline*}
with $E:=\dfrac{4}{D^2 \sigma^4 }$ and
\begin{align*}
A_1 & :=\Exp^{\gamma^{\emp(V_n)}} \left[ \int_0^{v\eta_m} \left( \left| \tilde{G}^{c,m,p}_s \right|^2 + \left| \tilde{G}^{m,p}_s \right|^2 \right) \, ds  \right]
\times \Exp^{\gamma^{\emp(V_n)}} \left[ \tilde{\Lambda}^p_{v \eta_m}(\tilde{G}^{m})  \left| \tilde{G}^{c,m, p}_{v\eta_m} \right|^2  \right]  \\
&\quad \quad \times\Exp^{\gamma^{\emp(V_n)}} \left[   \tilde{\Lambda}^p_{v \eta_m}(\tilde{G}^{m}) \left| \int_0^{v\eta_m} \tilde{G}^{c,m,p}_{r^{(m)}} \, d\tilde{V}^p_r \right|^2 \right]\\
A_2 & :=\Exp^{\gamma^{\emp(V_n)}} \left[ \int_0^{v\eta_m} \left( \left| \tilde{G}^{c,m,p}_s \right|^2 + \left| \tilde{G}^{m,p}_s \right|^2 \right) ds \left| \tilde{G}^{c,m, p}_{v\eta_m} \right|^2 \,\left| \int_0^{v\eta_m} \tilde{G}^{c,m,p}_{r^{(m)}} \, d\tilde{V}^p_r \right|^2 \right]\\
& = \int_0^{v\eta_m} \Exp^{\gamma^{\emp(V_n)}} \left[ \left( \left| \tilde{G}^{c,m,p}_s \right|^2 + \left| \tilde{G}^{m,p}_s \right|^2 \right)\left| \tilde{G}^{c,m, p}_{v\eta_m} \right|^2 \,\left| \int_0^{v\eta_m} \tilde{G}^{c,m,p}_{r^{(m)}} \, d\tilde{V}^p_r \right|^2 \right]\,ds.
\end{align*}

{\bf Step 3: Upper bounding $A_1$}\\
Using equations \eqref{eq:Gnmit},  \eqref{eq:Gmit} and 
Corollary~\ref{cor:boundgamma} we have
\[
\Exp^{\gamma^{\emp(V_n)}} \left[ \int_0^{v\eta_m} \left( \left| \tilde{G}^{c,m,p}_s \right|^2 + \left| \tilde{G}^{m,p}_s \right|^2 \right) \, ds  \right]
\leq 2 a b T N.
\]
By Lemma~\ref{Lemma Fourier Bound on R Jn} we have
\[
\Exp^{\gamma^{\emp(V_n)}} \left[ \tilde{\Lambda}^p_{v \eta_m}(\tilde{G}^{m})  \left| \tilde{G}^{c,m, p}_{v\eta_m} \right|^2  \right] \leq C_\J N,
\]
and
\[
\Exp^{\gamma^{\emp(V_n)}} \left[   \tilde{\Lambda}^p_{v \eta_m}(\tilde{G}^{m}) \left| \int_0^{v\eta_m} \tilde{G}^{c,m,p}_{r^{(m)}} \, d\tilde{V}^p_r \right|^2 \right] \leq
C_\J \sum_{k \in I_n} \left|  \int_0^{v\eta_m} f(V^k_{r^{(m)}}) \, d\tilde{V}^p_r  \right|^2,
\]
so that
\begin{equation}\label{eq:eqB}
A_1 \leq 2 a b (C_\J)^2 T N^2 \sum_{k \in I_n} \left|  \int_0^{v\eta_m} f(V^k_{r^{(m)}}) \, d\tilde{V}^p_r  \right|^2.
\end{equation}
{\bf Step 4: Upper bounding $A_2$ by Isserlis' Theorem}\\
Upperbounding the second term, $A_2$, requires the use of Isserlis' Theorem. 
In order to do this, we recall Isserlis' formula for six centered Gaussian random variables $(X_k)_{k=1,\cdots,6}$. For simplicity we write $ \Exp^\gamma$ for $\Exp^{\gamma^{\emp(V_n)}}$.
\begin{equation}
\Exp^\gamma \left[ X_1 X_2 X_3 X_4 X_5 X_6 \right] = \frac{1}{48}\sum_{\sigma \in S^6}
\Exp^\gamma \left[ X_{\sigma(1)} X_{\sigma(2)}\right]\Exp^\gamma \left[ X_{\sigma(3)} X_{\sigma(4)}\right]\Exp^\gamma \left[ X_{\sigma(5)} X_{\sigma(6)}\right],
\end{equation}
where \(S^6\) denotes the set of permutations of \(\{1,2,\cdots,6\}\).
Now if $X_{k+1}=X_k^*$, $k=1,3,5$, this reads
\begin{multline} \label{eq:isserlis}
    \Exp^\gamma \left[ |X_1|^2 |X_3|^2 |X_5|^2 \right] =  \Exp^\gamma \left[ |X_1|^2 \right] \Exp^\gamma \left[ |X_3|^2 \right] \Exp^\gamma \left[ |X_5|^2 \right] + \Exp^\gamma \left[ |X_1|^2 \right]  \left| \Exp^\gamma \left[ X_3 X_5 \right] \right|^2+\\
\Exp^\gamma \left[ |X_1|^2 \right] \left| \Exp^\gamma \left[ X_3 X_5^* \right] \right|^2 + \Exp^\gamma \left[ |X_5|^2 \right] \left| \Exp^\gamma \left[ X_1 X_3 \right] \right|^2 + \Exp^\gamma \left[ X_1 X_3 \right] \Exp^\gamma \left[ X_1^* X_5 \right] \Exp^\gamma \left[ X_3^* X_5^* \right] +\\
 \Exp^\gamma \left[ X_1 X_3 \right] \Exp^\gamma \left[ X_1^* X_5^* \right] \Exp^\gamma \left[ X_3^* X_5 \right] + \Exp^\gamma \left[ | X_5 |^2 \right] \left| \Exp^\gamma \left[ X_1 X_3^* \right] \right|^2  + \Exp^\gamma \left[ X_1 X_3^* \right] \Exp^\gamma \left[ X_1^* X_5 \right] \Exp^\gamma \left[ X_3 X_5^* \right] + \\
 \Exp^\gamma \left[ X_1 X_3^* \right] \Exp^\gamma \left[ X_1^* X_5^* \right] \Exp^\gamma \left[ X_3 X_5 \right] +
      \Exp^\gamma \left[ X_1 X_5 \right] \Exp^\gamma \left[ X_1^* X_3 \right] \Exp^\gamma \left[ X_3^* X_5^* \right] + \Exp^\gamma \left[ X_1 X_5 \right] \Exp^\gamma \left[ X_1^* X_3^* \right] \Exp^\gamma \left[ X_3 X_5^* \right] +\\
     \Exp^\gamma \left[ | X_3 |^2 \right]  \left| \Exp^\gamma \left[ X_1 X_5 \right] \right|^2  +
       \Exp^\gamma \left[ X_1 X_5^* \right] \Exp^\gamma \left[ X_1^* X_3 \right] \Exp^\gamma \left[ X_3^* X_5 \right] + \Exp^\gamma \left[ X_1 X_5^* \right] \Exp^\gamma \left[ X_1^* X_3^* \right] \Exp^\gamma \left[ X_3 X_5 \right] + \\
   \Exp^\gamma \left[ | X_3 |^2 \right]      \left| \Exp^\gamma \left[ X_1 X_5^* \right] \right|^2.
\end{multline}
We let
\begin{align*}
X_1 &=  \tilde{G}^{c,m,p}_s \quad \text{ or } \quad 
X_1 = \tilde{G}^{m,p}_s,\\
X_3 &=  \tilde{G}^{c,m, p}_{v\eta_m},\\
X_5 &= \int_0^{v\eta_m} \tilde{G}^{c,m,p}_{r^{(m)}} \, d\tilde{V}^p_r.
\end{align*}
Note that we have
\begin{align*}
X_2 = X_1^* & = \tilde{G}^{c,m,-p}_s\quad \text{ or } \quad   X_1^* = \tilde{G}^{m,-p}_s,\\
X_4 = X_3^* & =  \tilde{G}^{c,m, -p}_{v\eta_m},\\
X_6 = X_5^* & = \int_0^{v\eta_m} \tilde{G}^{c,m,-p}_{r^{(m)}} \, d\tilde{V}^{p}_r.
\end{align*}
Thanks to these identifications and using Corollary \ref{cor:boundgamma} we have
\begin{align*}
\Exp^\gamma \left[ |X_1|^2 \right] \leq a b N,\ \Exp^\gamma \left[ |X_3|^2 \right] &\leq a b N, \ \Exp^\gamma \left[ |X_5|^2 \right] \leq a b  \sum_{k \in I_n} \left|  \int_0^{v\eta_m} f(V^k_{r^{(m)}}) \, d\tilde{V}^p_r  \right|^2,\\
\max_{i=1,2,3,4, j=5,6}
\left| \Exp^\gamma \left[ X_i X_j \right] \right|& \leq a b \sqrt{N} \left(  \sum_{k \in I_n} \left|  \int_0^{v\eta_m} f(V^k_{r^{(m)}}) \, d\tilde{V}^p_r  \right|^2 \right)^{1/2},\\
\max_{i=1,2, j=3,4}
\left| \Exp^\gamma \left[ X_i X_j \right] \right|& \leq a b N.
\end{align*}
All fifteen terms  in the right hand side of \eqref{eq:isserlis} are upper-bounded by 
\[(a b)^3  N^2 \sum_{k \in I_n} \left|  \int_0^{v\eta_m} f(V^k_{r^{(m)}}) \, d\tilde{V}^p_r  \right|^2,
\]
 so that
\[
A_2 \leq 15 (a b)^3 T N^2 \sum_{k \in I_n} \left|  \int_0^{v\eta_m} f(V^k_{r^{(m)}}) \, d\tilde{V}^p_r  \right|^2.
\]
{\bf Step 5 Express the upper bound on $ \alpha^{4,1,p}_{v \eta_m} $ using the stopping time $\tau(\epsilon,\mathfrak{c})$}\\
Using \eqref{eq:eqB}, and returning to the notation $\Exp^{\gamma^{\emp(V_n)}}$
\[
\alpha^{4,1,p}_{v \eta_m} \leq 
D \Exp^{\gamma^{\emp(V_n)}} \left[  \int_0^{v\eta_m} \left| \tilde{G}^{c,m,p}_s - \tilde{G}^{m,p}_s   \right|^2 \, ds \right]  \times  \sum_{k \in I_n} \left|  \int_0^{v\eta_m} f(V^k_{r^{(m)}}) \, d\tilde{V}^p_r  \right|^2
\]
for some positive constant $D$ independent of $n$ and $m$.
By Corollary \ref{cor:boundgamma} and the Lipschitz continuity of $f$ 
\begin{align*}
\Exp^{\gamma^{\emp(V_n)}} \left[  \int_0^{v\eta_m} \left| \tilde{G}^{c,m,p}_s - \tilde{G}^{m,p}_s   \right|^2 \, ds \right]  &\leq
ab  \int_0^{v\eta_m} \sum_{k \in I_n} \left( f(V^k_s) -   f(V^{m,k}_s) \right)^2 ds \\
& \leq 
 ab  \int_0^{v\eta_m} \sum_{k \in I_n} \left( V^k_s -   V^{m,k}_s \right)^2 ds,
\end{align*}
so that  we have
\[
\frac{15}{N^4\sigma^4} \sum_{p \in I^n} \alpha^{4,1,p}_{v \eta_m} \leq  \frac{D}{N^4} 
\int_0^{v\eta_m} \sum_{k \in I_n} \left( V^k_s -   V^{m,k}_s \right)^2 ds \times 
\sum_{k \in I_n} \sum_{p \in I_n} \left|  \int_0^{v\eta_m} f(V^k_{r^{(m)}}) \, d\tilde{V}^p_r  \right|^2
\]
for some positive constant $D$.
By  Parseval's theorem on the $p$ index 
\[
\frac{15}{N^4\sigma^4} \sum_{p \in I_n} \alpha^{4,1,p}_{v \eta_m} \leq 
 \frac{D}{N^2}
\int_0^{v\eta_m} \frac{1}{N} \sum_{k \in I_n} \left( V^k_s -   V^{m,k}_s \right)^2 ds \times 
\sum_{k \in I_n} \sum_{l \in I_n} \left(  \int_0^{v\eta_m} f(V^k_{r^{(m)}}) \, dV^l_r  \right)^2.
\]
We next use the relation
\[
dV^l_r = \sigma dW^l_r + \sigma \theta^l_r \, dr
\]
to write
\begin{multline*}
\frac{15 }{N^4\sigma^4} \sum_{p \in In} \alpha^{4,1,p}_{v \eta_m} \leq  \frac{D}{N^2} \int_0^{v\eta_m} \frac{1}{N} \sum_{k \in I_n} \left( V^k_s -   V^{m,k}_s \right)^2 ds \times
\sum_{k \in I_n} \sum_{l \in I_n} \left(  \int_0^{v\eta_m} f(V^k_{r^{(m)}}) \, dW^l_r  \right)^2 \\
+ \frac{D}{N^2} \int_0^{v\eta_m} \frac{1}{N} \sum_{k \in I_n} \left( V^k_s -   V^{m,k}_s \right)^2 ds \times
\sum_{k \in I_n} \sum_{l \in I_n} \left(  \int_0^{v\eta_m} f(V^k_{r^{(m)}}) \, \theta^l_r \,dr \right)^2,
\end{multline*}
where we have included the constant $\sigma^2$ into $D$.

Since, if $\tau(\epsilon,\mathfrak{c}) \geq u\eta_m$, by \eqref{eq:deftau} we have
\[
\frac{1}{N} \sum_{k \in I_n} \left( V^k_s -   V^{m,k}_s \right)^2 \leq \epsilon \exp (s \mathfrak{c} )
\]
for all $s \leq u\eta_m$, we conclude that
\[
\lsup{n} \frac{1}{N} \log Q^n \left( \frac{15}{N^4\sigma^4} \sum_{p \in I_n} \alpha^{4,1,p}_{v \eta_m} \geq \frac{\epsilon \mathfrak{c}}{3 T C \sigma^2} \exp \left (v \eta_m\mathfrak{c} \right)  \text{ and } \tau(\epsilon,\mathfrak{c}) \geq u\eta_m\right)
\]
is upperbounded by twice the larger of the two terms
\begin{align}\label{eq:firstterm}
\lsup{n} \frac{1}{N} \log Q^n \left(  \frac{D}{N^2} \int_0^{v\eta_m} e^{s \mathfrak{c}} \, ds \times 
\sum_{k \in I_n} \sum_{l \in I_n} \left(  \int_0^{v\eta_m} f(V^k_{r^{(m)}}) \, dW^l_r  \right)^2 
\geq \frac{ \mathfrak{c}\exp \left (v \eta_m\mathfrak{c} \right)}{6 T C \sigma^2}  \right)\\
\label{eq:secondterm}
\lsup{n} \frac{1}{N} \log Q^n \left(\frac{D}{N^2} \int_0^{v\eta_m} e^{s \mathfrak{c}} \, ds \times 
\sum_{k \in I_n} \sum_{l \in I_n} \left(  \int_0^{v\eta_m} f(V^k_{r^{(m)}}) \, \theta^l_r \, dr \right)^2 
\geq \frac{ \mathfrak{c}\exp \left (v \eta_m\mathfrak{c} \right)}{6 T C \sigma^2}  \right).
\end{align}
{\bf Step 6: conclude by the use of Lemmas \ref{lem:BG} and \ref{Lemma: bound theta}}\\
Since \({\exp (v \eta_m \mathfrak{c}) - 1} \leq {\exp \left(v \eta_m\mathfrak{c} \right) }\), we can upper bound \eqref{eq:firstterm} by
\[
\lsup{n} \frac{1}{N} \log Q^n \left(   \frac{1/\mathfrak{c}}{2N}  
\sum_{k \in I_n} \sum_{l \in I_n} \left(  \int_0^{v\eta_m} f(V^k_{r^{(m)}}) \, dW^l_r  \right)^2 
\geq \frac{N \mathfrak{c}}{12 T C D \sigma^2} \right).
\]
By the exponential Tchebycheff inequality
\begin{multline*}
Q^n \left(  \frac{1/\mathfrak{c}}{2N}  
\sum_{k \in I_n} \sum_{l \in I_n} \left(  \int_0^{v\eta_m} f(V^k_{r^{(m)}}) \, dW^l_r  \right)^2 
\geq \frac{N \mathfrak{c}}{12 T C D \sigma^2} \right) \leq \\
\exp \left( - \frac{N \mathfrak{c}}{12 T C D \sigma^2}  \right)  \Exp^{Q^n} \left[
\exp \left( \frac{1/\mathfrak{c}}{2N}  
\sum_{k \in I_n} \sum_{l \in I_n} \left(  \int_0^{v \eta_m} f(V^k_{r^{(m)}}) \, dW^l_r  \right)^2 \right) \right].
\end{multline*}
In order to apply Lemma~\ref{lem:BG} to the above expectation we require
\[
\frac{1}{\sqrt{\mathfrak{c}}} < \frac{\sqrt{m}}{2 \sqrt{vT}}
\]
for $v=0,\cdots,u$ and this is certainly satisfied if
\[
\frac{1}{\sqrt{\mathfrak{c}}} < \frac{1}{2\sqrt{T}}.
\]
Lemma \ref{lem:BG} then commands that
\[
\Exp^{Q^n} \left[
\exp \left( \frac{1/\mathfrak{c}}{2N}  
\sum_{k \in I_n} \sum_{l \in I_n} \left(  \int_0^{v \eta_m} f(V^k_{r^{(m)}}) \, dW^l_r  \right)^2 \right) \right] \leq
\left(1- 4\frac{vT}{m\mathfrak{c}}\right)^{-N/4},
\]
and hence
\[
\Exp^{Q^n} \left[ \exp \left( \frac{1/\mathfrak{c}}{2N}  
\sum_{k \in I_n} \sum_{l \in I_n} \left(  \int_0^{v \eta_m} f(V^k_{r^{(m)}}) \, dW^l_r  \right)^2 \right) \right] \leq
\left(1- 4\frac{T}{\mathfrak{c}}\right)^{-N/4}.
\]
Therefore we have
\begin{multline*}
\frac{1}{N} \log Q^n \left(  \frac{1/\mathfrak{c}}{2N}  
\sum_{k \in I_n} \sum_{l \in I_n} \left(  \int_0^T f(V^k_{r^{(m)}}) \, dW^l_r  \right)^2 
\geq \frac{N \mathfrak{c}}{12 T C D \sigma^2}\right) \leq \\
-\mathfrak{c} \frac{1}{12 T C D \sigma^2} - \frac{1}{4} \log \left(1- 4\frac{T}{\mathfrak{c}}\right).
\end{multline*}
We conclude that for $\mathfrak{c}$ large enough, for all positive $M$s and for all $v=0,\cdots,u$ \eqref{eq:firstterm} is less than $-M$.

Along similar lines, we can upperbound \eqref{eq:secondterm} by
\[
\lsup{n} \frac{1}{N} \log Q^n \left(   \frac{1}{N^2}  
\sum_{k \in I_n} \sum_{l \in I_n} \left(  \int_0^{v\eta_m} f(V^k_{r^{(m)}}) \theta^l_r \, dr \right)^2 
\geq \frac{ \mathfrak{c}^2}{6 T C D \sigma^2} \right),
\]
and, by Cauchy-Schwarz, by
\[
\lsup{n} \frac{1}{N} \log Q^n \left(     
\left( \frac{1}{N} \sum_{k \in I_n} \int_0^{v\eta_m} \left( f(V^k_{r^{(m)}})\right)^2 \, dr \right) \times \left( \frac{1}{N}\sum_{l \in I_n}   \int_0^{v\eta_m}  \left( \theta^l_r \right)^2\, dr  \right)
\geq \frac{ \mathfrak{c}^2}{6 T C D \sigma^2} \right).
\]
Since $0 \leq f \leq 1$ and $0 \leq v \eta_m \leq T$, this is also upperbounded by
\[
\lsup{n} \frac{1}{N} \log Q^n \left( \frac{1}{N}\sum_{l \in I_n}   \int_0^T  \left( \theta^l_r \right)^2\, dr 
\geq \frac{ \mathfrak{c}^2}{6 T^2 C D \sigma^2} \right) \leq  \lsup{n} \frac{1}{N} \log Q^n \left( \frac{1}{N} \sup_{r \in [0,T]} \sum_{l \in I_n} \left( \theta^l_r \right)^2 \geq \frac{ \mathfrak{c}^2}{6 T^3 C D \sigma^2} \right),
\]
and Lemma \ref{Lemma: bound theta} allows us to conclude.
\end{proof}
%%%
%%%
%%%
%%%

%%%%%%%

%%%%%%%%
%%%%%%%%
\section{Proof of Lemma \ref{lem:limit equations}}\label{app:stochvolterra}
We give the proof of Lemma \ref{lem:limit equations}.
\begin{proof}[{\bf Proof of Lemma \ref{lem:limit equations}}]\ \\
Equation \eqref{eq: limit equations} resembles a Volterra equation of the second kind. As previously, we ignore for the sake of simplicity  the upper time index in $L_\mu$ and $K_\mu$.\\
{\bf Step 1: Construction of the sequence of processes $(\Phi^{i,n}_t)_{i \in \Z, n \in \N_*}$}\\
We proceed as in the case of the deterministic Volterra equations by constructing the following sequence of processes
\begin{align*}
\forall j \in \mathbb{Z},\quad V_t^{j, 0 }& =  \sigma W_t^j \\
V_t^{j, 1} & =  \sigma W_t^j + \sigma^{-1} \int_0^t \left(  \sum_{i \in \Z} \int_0^s L_{\mu}^{ i}(s, u) \,dV_u^{i+j, 0} \right)\,ds \\
& =  \sigma W_t^j +  \int_0^t \left(  \sum_{i \in \Z} \int_0^s L_{\mu}^{ i}(s, u) \,dW_u^{i+j} \right)\,ds,
\end{align*}
where the infinite sum is the \(L^2\) limit of the finite sums.
The existence of this limit is guaranteed by Proposition~\ref{prop:Lkmuregular}. 
%which implies that
%\[
%\sup_{s,u\in[0,T]}\sum_{i \in \Z}  \left( L^{i}_{\mu}(s,u)   \right)^2  \leq C
%\]
%which is sufficient to ensure the convergence as \(N_1\) goes to infinity of the Gaussian processes 
%\[
%\int_0^t  \sum_{i = -N_1}^{N_1} \int_0^s L_{\mu}^{ i}(s, u) \,dW_u^{i+j}\,ds.
%\]
We then compute the following difference
\begin{equation}\label{eq:B1}
V_t^{j, 1} - V_t^{j, 0}  =   \int_0^t \left( \sum_{i \in \Z} \int_0^s L_{\mu}^{ i}(s, u) \,dW_u^{i+j} \right)\,ds =: \psi^{j, 1}_t.
\end{equation}
Using \eqref{eq:B1} we write formally
\begin{align}
\nonumber V_t^{j, 2} &= \sigma W_t^j + \sigma^{-1} \int_0^t 
\sum_{i \in \Z} \int_0^s L_{\mu}^{ i}(s, u) \,dV_u^{i+j,1} 
\,ds  \\
& =  V_t^{1, j} +
\sigma^{-1} \int_0^t  \sum_{i \in \Z} \int_0^s L_{\mu}^{ i}(s, u) d \psi_u^{i+j, 1}\,ds. \label{eq:20180622}
\end{align}
Again, the convergence of the infinite sum is obtained by the study of
the sequence of variances of Gaussian processes. 
Applying the Young's convolution theorem \cite[Theorem 4.15]{brezis:10}, thanks to Proposition~\ref{prop:Lkmuregular},
we deduce 
\[
\sup_{0\leq v\leq u\leq s\leq T}\sum_{l\in\Z}\left(\sum_{i \in\Z}L_{\mu}^{ i}(s, u)L_{\mu}^{l-i}(u, v) \right)^2 < \infty.
\]
We deduce easily the existence of the limit in \eqref{eq:20180622}.
We write now
\[
\psi^{j, 2}_t := V_t^{j, 2} - V_t^{j, 1} =  \sigma^{-1} \int_0^t  \sum_{i \in \Z} \int_0^s L_{\mu}^{ i}(s, u) d \psi_u^{i+j, 1}\,ds,
\]
and hence
\[
\frac{d \psi^{j, 2}_t}{d t} = \sigma^{-1} \sum_{i \in \Z} \int_0^t L_{\mu}^{ i}(t, s) \, \frac{d \psi^{i+j, 1}_s}{d s} \, ds.
\]
Iterating this process one finds that
\[
V_t^{j, n} - V_t^{j, n-1} := \psi^{j, n}_t,
\]
where $\psi^{j, n}_t$ is such that
\[
\frac{d \psi^{j, n}_t}{d t} =  \sigma^{-1} \sum_{i \in \Z} \int_0^t L_{\mu}^{ i}(t, s) \, \frac{d \psi^{i+j, n-1}_s}{d s} \, ds,\,n \geq 2.
\]
Define
\[
\Phi^{j, n}_t = \frac{d \psi^{j, n}_t}{d t}, \, n \geq 1.
\]
This sequence of processes satisfies
\begin{align}\label{eq:Phijn}
\Phi^{j, n}_t & =  \sigma^{-1} \sum_{i \in \Z} \int_0^t L_{\mu}^{i}(t, s) \, \Phi^{i+j, n-1}_s \, ds,\, n \geq 2\\
%\end{align}
\mbox{ and } \quad %also
%\begin{align}
\label{eq:Bnt}
\sum_{k = 1}^p \psi^{j, k}_t & = V^{j, p}_t - V^{j, 0}_t = V^{j, p}_t - \sigma W^j_t = \sum_{k = 1}^p \int_0^t \Phi^{j, k}_s \, ds.
\end{align}
{\bf Step 2: Analysis of the sequence $(\Phi^{j,k}_t)_{j \in \Z, k \in \N_*}$}\\
We now analyze the sequence $(\Phi^{j, k}_t)_{k \geq 1}$.
First we note that
\begin{align}
\nonumber
\Phi^{j, 2}_t & =  \sigma^{-1} \sum_{i \in \Z} \int_0^t L_{\mu}^{i}(t, s) \, \Phi^{i+j, 1}_s \, ds,\\
\text{ with } \quad 
%\begin{equation}
\label{eq:Phi1}
\Phi^{j, 1}_s &=  \sum_{i \in \Z} \int_0^s L^{ i}_{\mu}(s, u) \, dW_u^{i+j}.
%\end{equation}
\end{align}
Consider next $\Phi^{j, 3}_t$. We write, using \eqref{eq:Phijn},
\begin{equation}\label{eq:Phi3}
\Phi^{j, 3}_t =  \sigma^{-1} \sum_{i \in \Z} \int_0^t L_{\mu}^{ i}(t, s) \, \Phi^{i+j, 2}_s \, ds = \sigma^{-2} \sum_{i, l \in \Z} \int_0^t L_{\mu}^{i}(t, s) \left( \int_0^s L_{\mu}^{ l}(s, u) \Phi^{l+i+j, 1}_u \, du \right) \, ds.
\end{equation}
Letting $\ell=l+i$ we have
\[
\Phi^{j, 3}_t = \sigma^{-2} \sum_{i, \ell \in \Z} \int_0^t L_{\mu}^{i}(t, s) \left( \int_0^s L_{\mu}^{ \ell-i}(s, u) \Phi^{\ell+j, 1}_u \, du \right) \, ds 
\]
and note that this can be rewritten as
\[
\Phi^{j, 3}_t = \sigma^{-2} \sum_{i, \ell \in \Z} \int_0^t  \left( \int_s^t L_{\mu}^{i}(t, u) L_{\mu}^{ \ell-i}(u, s) \, du \right)  \Phi^{\ell+j, 1}_s \, ds,
\]
by exchanging the order of integration. It follows for \(k \geq 2\) that
\begin{equation}\label{eq:Phint}
\Phi^{j, k}_t = \sigma^{-(k-1)} \sum_{\ell \in \Z} \int_0^t L_{{\mu}, k-1}^{\ell}(t, s) \, \Phi^{\ell+j, 1}_s \, ds,
\end{equation}
with
\begin{equation}\label{eq:Lmun}
L_{{\mu}, p+1}^{i}(t, s) = \sum_{l \in \Z} \int_s^t L_{\mu}^{l}(t, u) L_{{\mu}, p}^{ i-l}(u, s) \, du \quad p \geq 1
\end{equation}
and
\begin{equation}\label{eq:Lmu1}
L_{{\mu}, 1}^i = L_{\mu}^i.
\end{equation}
{\bf Step 3: Formal definition of the solution}\\
It follows from \eqref{eq:Bnt} and \eqref{eq:Phint} that
\[
V^{j, p}_t = \sigma W_t^j + \int_0^t \Phi_s^{j,1} \, ds+ \sigma^{-1} \int_0^t \left( \int_0^s \left(\sum_{i \in \Z}  \sum_{k = 1}^{p-1}  \sigma^{-(k-1)} L_{{\mu}, k}^{ i}(s, u) \right) \, \Phi^{i+j, 1}_u \, du \right) \, ds.
\]
If the series $\sum_{k = 1}^p \sigma^{-(k-1)} L_{{\mu}, k}^{ i}(s, u)$ is convergent for all $i \in \Z$, we can formally define a solution by
\begin{equation}\label{eq:easysol}
V_t^j = \sigma W_t^j +  \int_0^t \Phi_s^{j,1} \, ds + \sigma^{-1} \sum_{i \in \Z} \int_0^t \left( \int_0^s M_{\mu}^i(s, u) \Phi^{i+j, 1}_u \, du \right) \, ds,
\end{equation}
where
\begin{equation}\label{eq:Hi}
M_{\mu}^i(s, u) = \lim_{p \to \infty} \sum_{k = 1}^p \sigma^{-(k-1)} L_{{\mu}, k}^{ i}(s, u),
\end{equation}
is called the resolvent kernel.

This reads, because of \eqref{eq:Phi1},
\begin{multline}\label{eq:20180330_1}
V_t^j = \sigma W_t^j + 
\sum_{i \in \Z}  \int_0^t \left( \int_0^s L^{ i}_{\mu}(s, u) \, dW_u^{i+j} \right) \, ds+\\
\sigma^{-1} \sum_{i \in \Z}\int_0^t \left( \int_0^s M_{\mu}^i(s, u) \left(  \sum_{l \in \Z} \int_0^u L^{l}_{\mu}(u, v) \, dW_v^{i+l+j} \right) \, du \right) \, ds.
\end{multline}
Letting $\ell=l+i$ we have
\begin{multline*}
V_t^j = \sigma W_t^j + \sum_{i \in \Z}  \int_0^t \left( \int_0^s L^{ i}_{\mu}(s, u) \, dW_u^{i+j} \right) \, ds+\\
 \sigma^{-1} \sum_{i, \ell \in \Z}\int_0^t \left( \int_0^s M_{\mu}^i(s, u) \left(  \int_0^u L^{\ell-i}_{\mu}(u, v) \, dW_v^{\ell+j} \right) \, du \right) \, ds.
\end{multline*}
{\bf Step 4: Proof of the convergence of \eqref{eq:Hi}}\\
We prove the convergence of the right hand side of \eqref{eq:Hi}.
Note that \eqref{eq:Lmun} is a convolution with respect to the spatial index:
\[
L_{{\mu}, p+1}^{i}(t, s) = \int_s^t \left( L_{\mu}^{\cdot}(t, u) \star L_{{\mu}, p}^{\cdot}(u, s) \right)^i \, du.
\]
Applying Young's convolution theorem \cite[Theorem 4.15]{brezis:10}, thanks to Proposition \ref{prop:Lkmuregular},  and Cauchy-Schwarz we conclude that
\begin{multline} \label{eq:young}
\sum_{l \in \Z} \left| L_{{\mu}, p+1}^{l}(t, s) \right| \leq 
\int_s^t \sum_{l \in \Z} \left| L_{{\mu}}^{l}(t, u) \right|
\times \sum_{l \in \Z} \left| L_{{\mu}, p}^{ l}(u, s) \right| \, du \leq \\
\left( \int_s^t \left( \sum_{l \in \Z} \left| L_{{\mu}}^{ l}(t, u) \right| \right)^2 \, du \right)^{1/2} \times 
\left( \int_s^t \left( \sum_{l \in \Z} \left| L_{{\mu}, p}^{ l}(u, s) \right| \right)^2 \, du \right)^{1/2}.
\end{multline}
Applying this for $p=1$ we obtain, according to \eqref{eq:Lmu1}
\begin{multline}\label{eq:Lmu2}
\sum_{l \in \Z} \left| L_{{\mu}, 2}^{ l}(t, s) \right| \leq \left( \int_s^t \left( \sum_{l \in \Z} \left| L_{{\mu}}^{ l}(t, u) \right| \right)^2 \, du \right)^{1/2} 
%\times 
\left( \int_s^t \left( \sum_{l \in \Z} \left| L_{{\mu}}^{ l}(u, s) \right| \right)^2 \, du \right)^{1/2} \\
\leq 
\left( \int_0^T \left( \sum_{l \in \Z} \left| L_{{\mu}}^{l}(t, u) \right| \right)^2 \, du \right)^{1/2} 
%\times 
\left( \int_0^T \left( \sum_{l \in \Z} \left| L_{{\mu}}^{l}(u, s) \right| \right)^2 \, du \right)^{1/2}
=: A(t) B(s).
\end{multline}
Both $A(t)$ and $B(s)$ are finite by Proposition~\ref{prop:Lkmuregular}.
Applying \eqref{eq:young} for $p=2$ we obtain, using \eqref{eq:Lmu2}
\begin{align}
\left(\sum_{l \in \Z} \left| L_{{\mu}, 3}^{l}(t, s) \right|\right)^2 &\leq 
\int_0^T \left( \sum_{l \in \Z} \left| L_{{\mu}}^{l}(t, u) \right| \right)^2 \, du  \times 
\int_s^t \left( \sum_{l \in \Z} \left| L_{{\mu}, 2}^{l}(u, s) \right| \right)^2 \, du \nonumber \\
& \leq
A^2(t) B^2(s) \int_s^t A^2(u) \, du.\label{eq:Lmu3}
\end{align}
Applying \eqref{eq:young} for $p=3$ we obtain, using \eqref{eq:Lmu3}
\begin{align}\nonumber
\left(\sum_{l \in \Z} \left| L_{{\mu}, 4}^{l}(t, s) \right|\right)^2 &\leq 
 \int_0^T \left( \sum_{l \in \Z} \left| L_{{\mu}}^{ l}(t, u) \right| \right)^2 \, du  \times 
 \int_s^t \left( \sum_{l \in \Z} \left| L_{{\mu}, 3}^{ l}(u, s) \right| \right)^2 \, du \\
 & \leq 
A^2(t) B^2(s) \int_s^t A^2(u) \int_s^u A^2(v) \, dv \, du.
\label{eq:Lmu4}
\end{align}
In general we can write
\begin{equation}\label{eq:Lkplus2}
\left(\sum_{l \in \Z} \left| L_{{\mu}, k+2}^{ l}(t, s) \right|\right)^2  \leq  A^2(t) B^2(s) F_k(t,s),\,k=1,\,2,\,3, \cdots
\end{equation}
where
\begin{align}
F_1(t,s) & =  \int_s^t A^2(u) \, du \label{eq:F1} \\
F_2(t,s) & =  \int_s^t A^2(u) F_1(u,s)\, du \nonumber\\
 & \vdots \nonumber \\
 F_k(t,s) & = \int_s^t A^2(u) F_{k-1}(u,s) \, du \label{eq:Fk}.
\end{align}
We claim that
\begin{equation}\label{eq:FkF1}
F_k(t,s) = \frac{1}{k !} \left(F_1(t,s) \right)^k.
\end{equation}
This is true for $k=1$. By induction, assume it holds for $k-1$, then by \eqref{eq:Fk} we have
\begin{multline*}
F_k(t,s) = \int_s^t A^2(u) F_{k-1}(u,s) \, du= \frac{1}{(k-1) !} \int_s^t A^2(u) \left( F_1(u,s) \right)^{k-1} \, du = \\
\frac{1}{(k-1) !} \int_s^t \left( F_1(u,s) \right)^{k-1} \frac{\partial F_1(u,s)}{\partial u} \, du = 
\frac{1}{k !} \left[ \left(F_1(u,s) \right)^k  \right]_{u=s}^{u=t} = \frac{1}{k !} \left(F_1(t,s) \right)^k.
\end{multline*}
Next, by \eqref{eq:F1} we have
\[
0 \leq F_1(t,s) \leq \int_0^T A^2(u) \,du = \int_0^T \left( \int_0^T  \sum_{l \in \Z} \left| L_{{\mu}}^{l}(u, v) \right| \right)^2 \, dv \ \,du \leq C^2
\]
for some constant $C > 0$ by Proposition~\ref{prop:Lkmuregular}. By \eqref{eq:Lkplus2} and \eqref{eq:FkF1} we conclude that
\begin{equation}\label{eq:sumMi}
\sum_i \sigma^{-(k + 1)}\left| L_{{\mu}, k+2}^{ i}(t, s) \right| \leq \sigma^{-1} \frac{(\sigma^{-1} C )^{k}}{\sqrt{k !}} A(t) B(s),
\end{equation}
which implies
\begin{equation}\label{eq:Mi}
\sigma^{-(k + 1)}\left| L_{{\mu}, k+2}^{ i}(t, s) \right| \leq \sigma^{-1} \frac{(\sigma^{-1} C )^{k}}{\sqrt{k !}} A(t) B(s)
\end{equation}
for all $i \in \Z$.
and, since the series $z^k/\sqrt{k!}$ is absolutely convergent for all complex $z$, \eqref{eq:Mi} shows that the right hand side of \eqref{eq:Hi} is absolutely and uniformly convergent so that $M^i_{\mu}(t,s)$ is well-defined for all $i \in \Z$, continuous and uniformly bounded w.r.t. to $i$, and \eqref{eq:sumMi} shows that the series $M^i_{\mu}(t,s)$ is absolutely convergent, so that we have obtained \eqref{eq:Msol}.

\noindent
{\bf Step 5: Existence and uniqueness of the solution}\\
We then prove that  \eqref{eq:Msol} is a solution to \eqref{eq: limit equations} and that it is unique.
Indeed, \eqref{eq:Msol} implies
\begin{multline}\label{eq:dMsol}
dV_u^{i+j} = \sigma dW_u^{i+j} + \sum_{k \in \Z}   \left( \int_0^u L^{k}_{\mu}(u, v) \, dW_v^{k+i+j} \right) \, du+\\
 \sigma^{-1} \sum_{k, \ell \in \Z} \left( \int_0^u M_{\mu}^k(u, v) \left(  \int_0^v L^{\ell-k}_{\mu}(v, w) \, dW_w^{\ell+i+j} \right) \, dv \right) \, du,
\end{multline}
and \eqref{eq: limit equations} can be rewritten
\begin{equation}\label{eq: limit equations1}
		V^j_t = \sigma W^j_t + \sigma^{-1} \sum_{i\in \Z} \int_0^t \left( \int_0^s L_{\mu}^{ i}(s,u)dV^{i+j}_u \right) ds .
\end{equation}
Replacing the value of $dV_u^{i+j}$ given by \eqref{eq:dMsol} in the right hand side of \eqref{eq: limit equations1} we obtain
\[
V^j_t = \sigma W^j_t + \sigma^{-1}(A+B+C)
\]
with
\begin{equation}\label{eq:A}
A=\sigma \sum_{i \in \Z} \int_0^t \left( \int_0^s  L_{\mu}^{ i}(s,u) dW^{i+j}_u \right)\,ds,
\end{equation}
and, according to the definition \eqref{eq:Phi1} of $\Phi^{j,1}$,
\begin{align}
B & =\sum_{i,k \in \Z} \int_0^t \left( \int_0^s L_{\mu}^{ i}(s,u) \left( \int_0^u L_\mu^{k}(u,v) dW^{k+i+j}_v \right)\,du  \right)\,ds \nonumber \\
& = \sum_{i \in \Z} \int_0^t \left( \int_0^s L_{\mu}^{ i}(s,u) \Phi_u^{i+j,1} \right)\,ds.\label{eq:B}
\end{align}
Next we find that, using again \eqref{eq:Phi1},
\begin{align*}
C & = \sigma^{-1} \sum_{i,k,l \in \Z} \int_0^t \left( \int_0^s L_{\mu}^{ i}(s,u) \left( \int_0^u M^k(u,v) \left( \int_0^v L_\mu^{l-k}(v,w) \, dW_w^{l+i+j}  \right) \, dv \right)\, du  \right) \, ds \\
& = \sigma^{-1} \sum_{i,k \in \Z} \int_0^t \left( \int_0^s L_{\mu}^{ i}(s,u) \left( \int_0^u M^k(u,v)  \Phi_v^{k+i+j,1} \,dv \right)\,du \right) \,ds.
\end{align*}
Exchanging the order of integration and applying $k \to k+i$ yields
\begin{align*}
C & =  \sigma^{-1} \sum_{i,k \in \Z} \int_0^t \left( \int_0^s \left( \int_v^s L_{\mu}^{ i}(s,u) M^k(u,v)\,du \right) \Phi_v^{k+i+j,1} \,dv \right)\,ds \\
& = \sigma^{-1} \sum_{i,k \in \Z} \int_0^t \left( \int_0^s \left( \int_v^s L_{\mu}^{ i}(s,u) M^{k-i}(u,v)\,du \right) \Phi_v^{k+j,1} \,dv \right)\,ds.
\end{align*}
Using the definition \eqref{eq:Hi} of $M^k$ and rearranging terms
\[
C = \sigma^{-1} \sum_{k \in \Z}  \sum_{l=1}^\infty \int_0^t \left( \int_0^s \sigma^{-(l-1)} \left( \sum_{i \in \Z} \int_v^s L_{\mu}^{ i}(s,u) L_{\mu,l}^{k-i}(u,v) \,du \right) \Phi_v^{k+j,1} \,dv \right)\,ds.
\]
Because \eqref{eq:Lmun} this reads
\[
C = \sigma^{-1} \sum_{k \in \Z}   \int_0^t \left( \int_0^s \sum_{l=1}^\infty \left( \sigma^{-(l-1)} L_{\mu,l+1}^{k}(s,v) \right) \Phi_v^{k+j,1} \,dv \right)\,ds,
\]
and since, because of \eqref{eq:Hi},
\[
\sum_{l=1}^\infty \left( \sigma^{-(l-1)} L_{\mu,l+1}^{k}(s,v) \right) = \sigma \left(  M^k(s,v) - L^{k}(s,v) \right)
\]
we end up with
\begin{equation}\label{eq:C}
C = \sum_{k \in \Z} \int_0^t \left( \int_0^s M^k(s,v) \Phi_v^{k+j,1} \,dv \right)\,ds -  \sum_{k \in \Z} \int_0^t \left( \int_0^s L^{k}(s,v) \Phi_v^{k+j,1} \,dv \right)\,ds.
\end{equation}
Combining equations \eqref{eq:A}, \eqref{eq:B} and \eqref{eq:C} we find
\[
\sigma^{-1} ( A+B+C) = \sum_{i \in \Z} \int_0^t \left( \int_0^s  L_{\mu}^{ i}(s,u) dW^{i+j}_u \right)\,ds + \sigma^{-1} \sum_{k \in \Z} \int_0^t \left( \int_0^s M^k(s,v) \Phi_v^{k+j,1} \,dv \right)\,ds,
\]
and therefore that $\sigma W^j_t+\sigma^{-1} (A+B+C)$ is equal to the right hand side of \eqref{eq:easysol}.
We have proved that \eqref{eq:Msol} is a solution to \eqref{eq: limit equations}.

Uniqueness is obtained by noting that if two solutions $V_{1,t}$ and $V_{2,t}$ exist, there difference $V_t=V_{1,t}-V_{2,t}$ must satisfy the deterministic homogeneous Volterra equation of the second type
\[
V^j_t = \sigma^{-1} \sum_{i \in \Z} \int_0^t \int_0^s L_\mu^i (s,u) \, dV^{i+j}_u\,ds,
\]
for which it is easily proved that the only solution is the null solution.
\end{proof}
%%%%%%%%%%%
%%%%%%%%%%%
\section{Proof of Lemma \ref{lem:alphaj}} \label{app:alphajs}
Lemma \ref{lem:alphaj} follows from the following four Lemmas.
\begin{lemma}\label{lem:alphajone}
For all $\varepsilon > 0$, there exists $m_0(\varepsilon)$ in $\N$ such that for all $m \geq m_0$
\[
\Exp \left[ \sup_{s \in [0,t]} \left| \alpha_s^{j,1} \right|  \right] \leq C \varepsilon 
\]
for some positive constant $C$ independent of $j$.
\end{lemma}
\begin{lemma}\label{lem:alphajtwo}
For all $\varepsilon > 0$, there exists $m_0(\varepsilon)$ in $\N$ such that for all $m \geq m_0$
\[
\Exp \left[ \sup_{s \in [0,t]} \left| \alpha_s^{j,2} \right|  \right] \leq C \varepsilon 
\]
for some positive constant $C$ independent of $j$.
\end{lemma}
\begin{lemma}\label{lem:alphajthree}
For all $\varepsilon > 0$, there exists $m_0(\varepsilon)$ in $\N$ such that for all $m \geq m_0$
\[
\Exp \left[ \sup_{s \in [0,t]} \left| \alpha_s^{j,3} \right|  \right] \leq C \varepsilon 
\]
for some positive constant $C$ independent of $j$.
\end{lemma}
\begin{lemma}\label{lem:alphajfour}
For all $\varepsilon > 0$, there exists $m_0(\varepsilon)$ in $\N$ such that for all $m \geq m_0$
\[
\Exp \left[ \sup_{s \in [0,t]} \left| \alpha_s^{j,4} \right|  \right] \leq C \varepsilon 
\]
for some positive constant $C$ independent of $j$.
\end{lemma}
Lemma \ref{lem:limit equations} allows us to rewrite the $\alpha^{j, k}_t$s, $k=1,2,3,4$ as follows.
\begin{multline}\label{eq:alphaj1W}
\alpha^{j, 1}_t = \sigma \underbrace{\sum_{i \in I_{q_m}} \int_0^t ( L_{\mu}^{ i}(t,s) - L_{\mu}^{ i}(t^{(m)},s^{(m)}) ) \, dW^{i+j}_s  \, }_{\alpha_t^{j,1,1}} +\\
\sigma^{-1}  \underbrace{\sum_{i, k, L \in I_{q_m}}  \int_0^t ( L_{\mu}^{i}(t,s) - L_{\mu}^{ i}(t^{(m)},s^{(m)}) ) \left( \int_0^s M_{\mu}^k(s, u) \left(  \int_0^u L^{L-k}_{\mu}(u, v) \, dW_v^{L+i+j} \right) \, du \right) \, ds }_{\alpha_t^{j,1,2}}.
\end{multline}
Lemma~\ref{lem:alphajone} then follows from the following two Lemmas.
\begin{lemma}\label{lem:alphajoneone}
For all $\varepsilon > 0$, there exists $m_0(\varepsilon)$ in $\N$ such that for all $m \geq m_0$
\[
\Exp \left[ \sup_{s \in [0,t]} \left| \alpha_s^{j,1,1} \right|  \right] \leq C \varepsilon 
\]
for some positive constant $C$ independent of $j$.
\end{lemma}
\begin{lemma}\label{lem:alphajonetwo}
For all $\varepsilon > 0$, there exists $m_0(\varepsilon)$ in $\N$ such that for all $m \geq m_0$
\[
\Exp \left[ \sup_{s \in [0,t]} \left| \alpha_s^{j,1,2} \right|  \right] \leq C \varepsilon 
\]
for some positive constant $C$ independent of $j$.
\end{lemma}

\begin{proof}[Proof of Lemma \ref{lem:alphajoneone}]
The proof is based upon recognizing that
\[
S^{j}_t := \sum_{i \in I_{q_m}} \int_0^t ( L_{\mu_*}^{i}(t,s) - L_{\mu_*}^{i}(t^{(m)},s^{(m)}) ) \, dW^{i+j}_s
\]
is a  continuous martingale with quadratic variation
\[
\left \langle S^{j}_\cdot \right\rangle_t=\sum_{i \in I_{q_m}} \int_0^t  ( L_{\mu_*}^{i}(t,s) - L_{\mu_*}^{i}(t^{(m)},s^{(m)}) )^2\, ds,
\]
because of the independence of the Brownian motions. 

So that we have
\[
\sup_{s \in [0,t]} \left| \alpha_s^{j,1,1} \right| 
=  \sup_{s \in [0,t]} \left| S^j_s \right|.
\]
By Burkholder-Davis-Gundy's inequality we have
\[
\Exp \left[ \sup_{s \in [0,t]} \left| S^j_s \right|  \right] \leq C_1 \Exp \left[  \left \langle S^{j}_\cdot \right\rangle_t^{1/2}  \right] \leq \left( \sum_{i \in I_{q_m}} \int_0^t  ( L_{\mu_*}^{ i}(t,s) - L_{\mu_*}^{i}(t^{(m)},s^{(m)}) )^2\, ds \right)^{1/2}.
\]
This is upperbounded by
\[
\left( \int_0^t \sum_{i \in \Z} ( L_{\mu_*}^{i}(t,s) - L_{\mu_*}^{i}(t^{(m)},s^{(m)}) )^2\, ds \right)^{1/2},
\]
which, by Parseval's Theorem is equal to
\[
\frac{1}{\sqrt{2 \pi}} \left( \int_0^t \int_{-\pi}^\pi \left|  \bar{\tilde{L}}_{\mu_*}(\varphi)(t,s)-  \bar{\tilde{L}}_{\mu_*}(\varphi)(t^{(m)},s^{(m)})\right|^2 \, d\varphi \, ds \right)^{1/2}.
\]
The relation
\[
 \bar{\tilde{L}}_{\mu_*} = \sigma^2 \left({\rm Id}- \left( {\rm Id} + \sigma^{-2} \bar{\tilde{K}}_{\mu_*} \right)^{-1} \right)
\]
dictates that
\[
\left|  \tilde{L}_{\mu_*}(\varphi)(t,s)-  \tilde{L}_{\mu_*}(\varphi)(t^{(m)},s^{(m)})\right|^2 =  
\sigma^4 \left|  \left( {\rm Id} + \sigma^{-2} \tilde{K}_{\mu_*}(\varphi) \right)^{-1}(t,s) -   \left( {\rm Id} + \sigma^{-2} \tilde{K}_{\mu_*}(\varphi) \right)^{-1}(t^{(m)},s^{(m)})\right|^2.
\]
By the Lipschitz continuity of the application $A \to ({\rm Id}+ A)^{-1}$, for $A$ a positive operator, we obtain that
\[
\left|  \tilde{L}_{\mu_*}(\varphi)(t,s)-  \tilde{L}_{\mu_*}(\varphi)(t^{(m)},s^{(m)}) \right|^2 \leq
C \left| \tilde{K}_{\mu_*}(\varphi)(t,s) - \tilde{K}_{\mu_*}(\varphi)(t^{(m)},s^{(m)}) \right|^2
\]
for some positive constant $C$.
Next we write
\[
 \tilde{K}_{\mu_*}(\varphi)(t,s) = \sum_{k \in \Z} \tilde{R}_\J(\varphi,k) \int_{\T^\Z} f(v^0_t) f(v^k_s)\, d\mu_*(v),
\]
from which it follows that
\begin{multline*}
\left|  \tilde{K}_{\mu_*}(\varphi)(t,s)  - \tilde{K}_{\mu_*}(\varphi)(t^{(m)},s^{(m)})\right| \leq 
\sum_{k \in \Z} \left| \tilde{R}_\J(\varphi,k) \right| \left| \int_{\T^\Z} \left( f(v^0_t)f(v^k_s) - f(v^0_{t^{(m)}})f(v^k_{s^{(m)}})  \right) \, d\mu_*(v) \right| = \\
\sum_{k \in \Z} \left| \tilde{R}_\J(\varphi,k) \right| \left| \int_{\T^\Z} \left( \left( f(v^0_t)-f(v^0_{t^{(m)}} ) \right) f(v^k_s)+ \left( f(v^k_s)-f(v^k_{s^{(m)}} ) \right) f(v^0_{t^{(m)}}) \right)  \, d\mu_*(v) \right|.
\end{multline*}
Because $0 \leq f \leq 1$ 
\[
\left|  \tilde{K}_{\mu_*}(\varphi)(t,s)  - \tilde{K}_{\mu_*}(\varphi)(t^{(m)},s^{(m)})\right| \leq 
\sum_{k \in \Z} \left| \tilde{R}_\J(\varphi,k) \right| 
\int_{\T^\Z} |f(v^0_t)-f(v^0_{t^{(m)}} )| + |f(v^k_s)-f(v^k_{s^{(m)}} )| \, d\mu_*(v).
\]
	By stationarity, we have
\begin{align*}
&\hspace{-50pt}\int_{\T^\Z} |f(v^0_t)-f(v^0_{t^{(m)}} )| + |f(v^k_s)-f(v^k_{s^{(m)}} )| \, d\mu_*(v)\\
&=\int_{\T^\Z} |f(v^0_t)-f(v^0_{t^{(m)}} )| + |f(v^0_s)-f(v^0_{s^{(m)}} )| \, d\mu_*(v)\\
&\leq 2 \int_{\T^\Z}  \sup_{0\leq t_1,t_2 \leq T, |t_2 - t_1|\leq \eta_m}|v^0_{t_2} - v^0_{t_1}|d\mu_*(v)\\
&\leq \epsilon
\end{align*}
for \(m\) large enough.
Thus, we have
\[
\left|  \tilde{K}_{\mu_*}(\varphi)(t,s)  - \tilde{K}_{\mu_*}(\varphi)(t^{(m)},s^{(m)})\right| \leq C \varepsilon 
\]
for some positive constant $C$, since $\sum_{k \in \Z} \left| \tilde{R}_\J(\varphi,k) \right| \leq D$ for some positive constant $D$ independent of $\varphi$, and therefore, as announced,
\[
\Exp \left[ \sup_{s \in [0,t]} \left|  \alpha_s^{j,1,1}  \right| \right] \leq C \varepsilon,
\]
for some positive constant $C$.
\end{proof}
\begin{proof}[Proof of Lemma \ref{lem:alphajonetwo}]
We have
\begin{multline*}
\Exp \Bigg[ \Bigg| \sum_{i, k, L \in I_{q_m}}  \int_0^t ( L_{\mu_*}^{i}(t,s) - L_{\mu_*}^{i}(t^{(m)},s^{(m)}) )\\
 \left( \int_0^s M_{\mu_*}^k(s, u) \left(  \int_0^u L^{L-k}_{\mu_*}(u, v) \, dW_v^{L+i+j} \right) \, du \right) \, ds   \Bigg| \Bigg] \leq \\
 \Exp \Bigg[  \int_0^t  \int_0^s \sum_{i \in I_{q_m}} \left|  L_{\mu_*}^{i}(t,s) - L_{\mu_*}^{ i}(t^{(m)},s^{(m)}) \right| \\
  \sum_{k \in I_{q_m}} \left| M_{\mu_*}^k(s, u) \right| \left|   \int_0^u \sum_{L \in I_{q_m}} L^{L-k}_{\mu_*}(u, v) \, dW_v^{L+i+j} \right| \, du  \, ds   \Bigg] \leq \\
 \int_0^t  \int_0^s \sum_{i \in I_{q_m}} \left|  L_{\mu_*}^{i}(t,s) - L_{\mu_*}^{ i}(t^{(m)},s^{(m)}) \right|  \sum_{k \in I_{q_m}} \left| M_{\mu_*}^k(s, u) \right| \\
 \Exp \left[ \sup_{u \in [0,t]} \left|   \int_0^u \sum_{L \in I_{q_m}} L^{L-k}_{\mu_*}(u, v) \, dW_v^{L+i+j} \right| \right]\, du  \, ds.
\end{multline*}
Because $\int_0^u \sum_{L \in I_{q_m}} L^{L-k}_{\mu_*}(u, v) \, dW_v^{L+i+j}$ is a continuous martingale, the B\"urkholder-Davis-Gundy inequality, Parseval's Theorem, and Proposition~\ref{prop:Lkmuregular} dictate
\begin{multline*}
\Exp \left[  \sup_{u \in [0,t]} \left|   \int_0^u \sum_{L \in I_{q_m}} L^{L-k}_{\mu_*}(u, v) \, dW_v^{L+i+j} \right| \right] \leq C_1 \left( \int_0^t \sum_{L \in I_{q_m}} \left( L^{L-k}_{\mu_*}(t, v) \right)^2 dv \right)^{1/2} \leq \\
 C_1 \left( \int_0^t \sum_{L \in \Z} \left( L^{L}_{\mu_*}(t, v) \right)^2 dv \right)^{1/2} =  \frac{C_1}{\sqrt{2\pi}}\left( \int_0^t  \int_{-\pi}^\pi \left| \tilde{L}(\varphi)( t, v) \right|^2 \, d\varphi \, dv \right)^{1/2} \leq D
\end{multline*}
for some positive constant $D$.
Next we have
\[
\sum_{k \in I_{q_m}} \left| M_{\mu_*}^k(u, v) \right| \leq \sum_{k \in \Z} \left| M_{\mu_*}^k(u, v) \right| \leq E
\]
for some positive constant $E$, so that
\[
\Exp \left[ \sup_{s \in [0,t]} \left| \alpha_s^{j,1,2} \right|  \right] \leq D E T^2 \sup_{s, u \in [0,t]} \sum_{i \in I_{q_m}} \left|  L_{\mu_*}^{i}(s,u) - L_{\mu_*}^{i}(s^{(m)},u^{(m)}) \right|.
\]
Because of Lemma \ref{lem:upperboundseries} below there exists a positive convergent series $A=(a_i)_{i \in \Z}$ such that for all $\varepsilon > 0$ there exists $m_0(\varepsilon)$ such that for all $m \geq m_0$
\[
\left|  L_{\mu_*}^{i}(s,u) - L_{\mu_*}^{ i}(s^{(m)},u^{(m)}) \right| \leq \varepsilon a_i
\]
for all $s, u \in [0,t]$. This proves the Lemma.
\end{proof}
\begin{lemma}\label{lem:upperboundseries}
Let $\bar{O}$ be an operator on $L^2(\Z,[0,T])$ defined by the continuous kernels $O^i(t,s)$, $i \in \Z$. There exists a positive convergent series $A=(a_i)_{i \in \Z}$ such that for all $\varepsilon > 0$ there exists $m_0(\varepsilon)$ such that for all $i \in \Z$ and for all $m \geq m_0$
\[
\left|  O^i(s,u) - O^i(s^{(m)},u^{(m)}) \right| \leq \varepsilon a_i
\]
for all $s, u \in [0,t]$.
\end{lemma}
\begin{proof} 
We proceed by contradiction. Assume that for all positive convergent series $A=(a_i)_{i \in \Z}$ there exists $i_0 \in \Z$, $s_0, u_0 \in [0,t]$ and $\varepsilon > 0$ such that for all $m \in \N^*$
\[
\varepsilon a_{i_0} < \left|  O^{i_0}(s_0,u_0) - O^{i_0}(s_0^{(m)},u_0^{(m)}) \right|.
\]
 Choosing $m$ large enough and by the continuity of $O^{i_0}(s.u)$ w.r.t. $(s,u)$ we obtain a contradiction.
\end{proof}
We proceed with the term $\alpha^{j, 2}_t $:
\begin{multline}\label{eq:alphaj2W}
\alpha^{j, 2}_t = \sigma \underbrace{\sum_{i \in I_{q_m}}  \int^t_{t^{(m)}}  L_{\mu_*}^{ i}(t^{(m)},s^{(m)}) \, dW^{i+j}_s  }_{\alpha^{j, 2,1}_t } +\\
\sigma^{-1} \underbrace{\sum_{i, k, L \in I_{q_m}}  \int^t_{t^{(m)}}  L_{\mu_*}^{ i}(t^{(m)},s^{(m)})  
\left( \int_0^s M_{\mu_*}^k(s, u) \left(  \int_0^v L^{L-k}_{\mu_*}(u, v) \, dW_v^{L+i+j} \right) \, du \right) \, ds }_{\alpha^{j, 2,2}_t }.
\end{multline}
Lemma \ref{lem:alphajtwo} then follows from the following two Lemmas.
\begin{lemma}\label{lem:alphajtwoone}
For all $\varepsilon > 0$, there exists $m_0(\varepsilon)$ in $\N$ such that for all $m \geq m_0$
\[
\Exp \left[ \sup_{s \in [0,t]} \left| \alpha_s^{j,2,1} \right|  \right] \leq C \varepsilon 
\]
for some positive constant $C$ independent of $j$.
\end{lemma}
\begin{lemma}\label{lem:alphajtwotwo}
For all $\varepsilon > 0$, there exists $m_0(\varepsilon)$ in $\N$ such that for all $m \geq m_0$
\[
\Exp \left[ \sup_{s \in [0,t]} \left| \alpha_s^{j,2,2} \right|  \right] \leq C \varepsilon 
\]
for some positive constant $C$ independent of $j$.
\end{lemma}
\begin{proof}[Proof of Lemma \ref{lem:alphajtwoone}]
The proof is very similar to that of Lemma \ref{lem:alphajoneone}. As in this Lemma it is based upon recognizing that
\[
S^{j}_t := \sum_{i \in I_{q_m}} \int_{t^{(m)}}^t  L_{\mu_*}^{ i}(t,s)  \, dW^{i+j}_s
\]
is a  continuous martingale with quadratic variation
\[
\left \langle S^{j}_\cdot \right\rangle_t=\sum_{i \in I_{q_m}} \int_{t^{(m)}}^t  ( L_{\mu_*}^{ i}(t,s) )^2\, ds,
\]
because of the independence of the Brownian motions. 

We have
\[
\Exp \left[ \sup_{s \in [0,t]} \left| \alpha_s^{j,2,1} \right|  \right] = \sigma \Exp \left[ \sup_{s \in [0,t]} \left| S^j_s \right| \right],
\]
and, by Burkholder-Davis-Gundy's inequality
\begin{multline*}
\Exp \left[ \sup_{s \in [0,t]} \left| \alpha_s^{j,2,1} \right|  \right] \leq C_1 \sigma \left( \sum_{i \in I_{q_m}} \int_{t^{(m)}}^t  ( L_{\mu_*}^{i}(t,u) )^2\, du \right)^{1/2} \leq \\
C_1 \sigma \left( \int_{t^{(m)}}^t \sum_{i \in \Z} ( L_{\mu_*}^{i}(t,u) )^2\, du \right)^{1/2}= \frac{C_1 \sigma}{\sqrt{2\pi}}\left( \int_{t^{(m)}}^t  \int_{-\pi}^\pi \left| \tilde{L}_{\mu_*}(\varphi)( t, u) \right|^2 \, d\varphi \, du \right)^{1/2}.
\end{multline*}
The fact that $ \int_{-\pi}^\pi \left| \tilde{L}_{\mu_*}(\varphi)( t, w) \right|^2\,d\varphi \leq C$ for some positive constant $C$ uniformly in $t$, $w$, follows from Proposition~\ref{prop:Lkmuregular} and ends the proof.
\end{proof}
\begin{remark}
The proof of Lemma \ref{lem:alphajtwotwo} is very similar and left to the reader.
\end{remark}
Next we write
\begin{multline}\label{eq:alphaj3W}
\alpha^{j, 3}_t = \sigma \underbrace{\sum_{i \in I_{q_m}} \int_0^t \left( L_{\mu_*}^{ i}(t^{(m)},s^{(m)})  - L_{\emp(V^m_n)}^{q_m, i}(t^{(m)}, s^{(m)}) \right) \, dW^{i+j}_s }_{\alpha^{j,3,1}_t} +\\
\sigma^{-1} \sum_{i, k, L \in I_{q_m}}  \int_0^t \left( L_{\mu_*}^{i}(t^{(m)},s^{(m)})  - L_{\emp(V^m_n)}^{q_m,  i}(t^{(m)}, s^{(m)}) \right) \\
\left( \int_0^s M_{\mu_*}^k(s, u) \left(  \int_0^u L^{ L-k}_{\mu_*}(u, v) \, dW_v^{L+i+j} \right) \, du \right) \, ds,
\end{multline}
and define
\begin{multline*}
\alpha^{j,3,2}_t:=\sum_{i, k, L \in I_{q_m}}  \int_0^t \left( L_{\mu_*}^{ i}(t^{(m)},s^{(m)})  - L_{\emp(V^m_n)}^{q_m,  i}(t^{(m)}, s^{(m)}) \right) \\
\left( \int_0^s M_{\mu_*}^k(s, u) \left(  \int_0^u L^{L-k}_{\mu_*}(u, v) \, dW_v^{L+i+j} \right) \, du \right) \, ds.
\end{multline*}
Lemma~\ref{lem:alphajthree} then follows from the next two Lemmas.
\begin{lemma}\label{lem:alphajthreeone}
For all $\varepsilon > 0$, 
\[
\Exp \left[ \sup_{s \in [0,t]} \left| \alpha_s^{j,3,1} \right|  \right] \leq C \varepsilon 
\]
for some positive constant $C$ independent of $j$, for all $m,n$ large enough.
\end{lemma}
Similarly we have
\begin{lemma}\label{lem:alphajthreetwo}
For all $\varepsilon > 0$, 
\[
\Exp \left[ \sup_{s \in [0,t]} \left| \alpha_s^{j,3,2} \right|  \right] \leq C \varepsilon 
\]
for some positive constant $C$ independent of $j$, for all $m,n$ large enough.
\end{lemma}
\begin{proof}[Sketch of a proof of Lemma \ref{lem:alphajthreeone}]\ \\
We note that $S^{j,m}_t:=\sum_{i \in I_{q_m}} \int_0^t \left( L_{\mu_*}^{i}(t^{(m)},s^{(m)})  - L_{\emp(V^m_n)}^{q_m,  i}(t^{(m)}, s^{(m)}) \right) \, dW^{i+j}_s $ is a martingale. Hence, by the B\"urkholder-Davis-Gundy inequality,
\[
\Exp \left[ \sup_{s \in [0,t]} | S^{j,m}_s | \right] \leq C_1 \Exp \left[ \langle  S^{j,m}_\cdot  \rangle_t^{1/2} \right] \leq C_1 \Exp \left[ \langle  S^{j,m}_\cdot  \rangle_t \right]^{1/2}.
\]
By the independence of the Brownian motions
\[
\langle  S^{j,m}_\cdot \rangle_t = \sum_{i \in I_{q_m}} \int_0^t \left( L_{\mu_*}^{i}(t^{(m)},u^{(m)})  - L_{\emp(V^m_n)}^{q_m,  i}(t^{(m)}, u^{(m)}) \right)^2 \, du,
\]
and therefore, by Cauchy-Schwarz
\[
\Exp \left[ \sup_{s \in [0,t]} | S^{j,m}_s | \right] \leq C_1 \left(  \int_0^t \sum_{i \in I_{q_m}} \Exp \left[ \left( L_{\mu_*}^{ i}(t^{(m)},u^{(m)})  - L_{\emp(V^m_n)}^{q_m, i}(t^{(m)}, u^{(m)}) \right)^2  \right] \, du \right)^{1/2}.
\]
By Proposition \ref{prop:Amununif} $\left| L_{\mu_*}^{i}(t^{(m)},u^{(m)})  - L_{\emp(V^m_n)}^{q_m, i}(t^{(m)}, u^{(m)}) \right| \leq  D_t(\mu_*, \emp(V^m_n)) \, \smallO{1/|i|^3}$, where $D_t$ is the Wasserstein distance between the two measures $\mu_*$ and $\emp(V^m_n)$,  we conclude that
\[
\Exp \left[ \sup_{s \in [0,t]} | S^{j,n,m}_s | \right] \leq C_1 T^{1/2} \Exp \left[ D_t(\mu_*, \emp(V^m_n)) \right] \sum_{i \in I_{q_m}} \smallO{1/|i|^3} \leq C \Exp \left[ D_t(\mu_*, \emp(V^m_n)) \right]
\]
for a constant $C > 0$. This concludes the proof of the Lemma since Lemma \ref{lem:limit law of Qmn} implies that $\lim_{m,n \to \infty} \Exp \left[ D_t(\mu_*, \emp(V^m_n)) \right] = 0$.
\end{proof}
The proof of Lemma~\ref{lem:alphajthreetwo} is very similar and left to the reader.
So is the proof of Lemma~\ref{lem:alphajfour}.

\end{document}